\documentclass[11pt]{amsart}

\usepackage{amsmath, amsthm, amssymb} 
\numberwithin{equation}{section}

\usepackage{amsthm}
\usepackage{hyperref}
\usepackage[all]{xy}
\usepackage[usenames,dvipsnames]{xcolor}

\usepackage{mleftright}

\makeatletter
\newtheoremstyle{indented}
{}
{}
{\addtolength{\@totalleftmargin}{.0em}
	\addtolength{\linewidth}{.0em}
	\parshape 1.em \linewidth}
{}
{\itshape}
{.}
{.5em}
{}
\makeatother

\theoremstyle{indented}

\usepackage{overpic}

\newtheorem{theo}{Theorem}[section]
\newtheorem{lemma}[theo]{Lemma}
\newtheorem{prop}[theo]{Proposition}
\newtheorem{prpl}[theo]{Proposal}
\newtheorem{hypo}[theo]{Hypothesis}
\newtheorem{corollary}[theo]{Corollary}
\newtheorem{defn}[theo]{Definition}
\newtheorem{warning}[theo]{Warning}

\newtheorem{eg}[theo]{Example}

\newtheorem{rmk}[theo]{Remark}

\newtheorem{conj}[theo]{Conjecture}

\newtheorem{manualtheoreminner}{Theorem}
\newenvironment{manualtheorem}[1]{%
	\renewcommand\themanualtheoreminner{#1}%
	\manualtheoreminner
}{\endmanualtheoreminner}

\usepackage[left=2.5cm,right=2.5cm,top=3.5cm]{geometry}

\pdfoutput=1
\usepackage{amsmath}
\usepackage{amssymb}
\usepackage{wrapfig}
\usepackage{bbm}
\usepackage{arydshln}
\usepackage{multirow}
\usepackage{mathtools}
\usepackage{blkarray}
\usepackage{setspace}
\usepackage{dsfont}
\usepackage{float}
\usepackage{subcaption}
\usepackage{braket}
\usepackage[utf8]{inputenc}
\usepackage[english]{babel}
\usepackage[all]{xy}
\usepackage{tikz-cd}
\usepackage{mathtools}
\usepackage{tikz}

\makeatletter
\newcommand\mathcircled[1]{%
  \mathpalette\@mathcircled{#1}%
}
\newcommand\@mathcircled[2]{%
  \tikz[baseline=(math.base)] \node[draw,circle,inner sep=3pt] (math) {$\m@th#1#2$};%
}
\makeatother

\allowdisplaybreaks

\newcommand{\C}{{\mathbb C}}

\DeclareMathOperator{\Bun}{Bun}

\DeclareMathOperator{\Hom}{Hom}
\DeclareMathOperator{\Vect}{Vect}
\DeclareMathOperator{\DGV}{\DD(\K)}
\DeclareMathOperator{\ob}{ob}
\DeclareMathOperator{\Ext}{Ext}
\newcommand{\Mod}{{\textup{-Mod}}}
\newcommand{\modd}{{\textup{-mod}}}
\newcommand{\Bimod}{{\textup{-BiMod}}}

\usepackage{cjhebrew}
\newcommand{\Fp}{{\text{\<p>}}}

\newcommand{\K}{{\bb K}}

\newcommand{\bb}{\mathbb}
\newcommand{\mc}{\mathcal}

\newcommand{\mf}{\mathfrak}

\newcommand{\aaa}{{\textbf{a}}}
\newcommand{\bbb}{{\textbf{b}}}

\newcommand{\dd}{{\textbf{d}}}
\newcommand{\Gl}{{\textup{Gl}}}
\newcommand{\sh}{{\textup{sh}}}
\newcommand{\DD}{{\textup{D}}}
\newcommand{\CC}{{\textup{C}}}
\newcommand{\DC}{{\textup{D}_c}}
\newcommand{\D}{{\mc D}}

\newcommand{\cc}{{c}}
\newcommand{\q}{{\textup{q}}}

\newcommand{\Tr}{\textup{Tr}}

\newcommand{\Crit}{\textup{Crit}}
\newcommand{\Quot}{\textup{Quot}}
\newcommand{\Hilb}{\textup{Hilb}}
\newcommand{\perf}{\textup{perf}}

\newcommand{\Coh}{\textup{Coh}}
\newcommand{\QC}{\textup{QCoh}}
\newcommand{\op}{\textup{op}}

\newcommand{\fd}{\textup{fd}}
\newcommand{\cs}{\textup{cs}}
\newcommand{\Sym}{\textup{Sym}}
\newcommand{\id}{{\textup{id}}}
\newcommand{\Alg}{\textup{Alg}}
\newcommand{\Ass}{\textup{Ass}}
\newcommand{\Perv}{\textup{PervCoh}}

\newcommand{\Per}{\textup{Perv}}
\newcommand{\coker}{\textup{coker}}
\newcommand{\cyc}{\textup{cyc}}
\newcommand{\fg}{\textup{fg}}
\newcommand{\tw}{\textup{tw}}
\newcommand{\Tw}{\textup{Tw}}

\newcommand{\cn}{\textup{c}}
\newcommand{\gr}{\textup{gr}}
\newcommand{\nil}{\textup{nil}}

\newcommand{\pt}{\textup{pt}}
\newcommand{\Z}{\bb Z}
\newcommand{\DGVect}{\textup{DGVect}}
\newcommand{\Ainf}{A_\infty}
\newcommand{\Filt}{\textup{Filt}}
\newcommand{\Triang}{\textup{Triang}}
\newcommand{\triang}{\textup{triang}}
\newcommand{\Thick}{\textup{Thick}}
\newcommand{\thick}{\textup{thick}}
\newcommand{\filt}{\textup{filt}}
\newcommand{\red}{\textup{red}}
\newcommand{\Pic}{\textup{Pic}}
\newcommand{\Perf}{{\textup{Perf}}}
\newcommand{\Fd}{{\textup{Fd}}}
\newcommand{\Fg}{{\textup{Fg}}}
\newcommand{\fr}{{\textup{fr}}}
\newcommand{\Fr}{{\textup{Fr}}}
\newcommand{\ZZ}{{\mathbb{Z}}}
\newcommand{\Zz}{{\mathcal{Z}}}
\newcommand{\nn}{{r}}
\newcommand{\p}{{\textup{p}}}

\newcommand{\V}{{\bb V}}

\newcommand{\f}{{\textup{f}}}
\newcommand{\DT}{{\textup{DT}}}
\newcommand{\PT}{{\textup{PT}}}
\newcommand{\rk}{{\textup{rk}}}
\newcommand{\NCDT}{{\textup{NCDT}}}
\newcommand{\VW}{{\textup{VW}}}
\newcommand{\M}{{{M}}}
\newcommand{\rr}{{{\textbf{r}}}}
\newcommand{\qq}{{{\textbf{q}}}}
\newcommand{\g}{{{\mathfrak{g}}}}
\newcommand{\gl}{{{\mathfrak{gl}}}}
\newcommand{\spl}{{{\mathfrak{sl}}}}
\newcommand{\glh}{{{\widehat{\mathfrak{gl}}}}}
\newcommand{\gh}{{{\widehat{\mathfrak{g}}}}}
\newcommand{\rrr}{{{r}}}
\newcommand{\xx}{{{\textbf{x}}}}

\newcommand{\ee}{{e}}

\newcommand{\Spec}{{\textup{Spec}\ }}
\newcommand{\Spf}{{\textup{Spf}\ }}
\newcommand{\Homi}{{\underline{\textup{Hom}}}}
\newcommand{\Endi}{{\underline{\textup{End}}}}
\newcommand{\End}{{\textup{End}}}
\newcommand{\Pq}{{{\bb Z [\![\qq]\!]}}}
\newcommand{\lP}{{ [\![ }}
\newcommand{\rP}{{ ]\!] }}
\newcommand{\lL}{{ (\!( }}
\newcommand{\rL}{{ )\!) }}
\newcommand{\one}{{\textbf{1}}}
\newcommand{\sss}{{\textup{ss}}}

\usepackage{ytableau}

\begin{document}

\title[]{Perverse coherent extensions on Calabi-Yau threefolds and representations of cohomological Hall algebras}

\author{Dylan Butson}
\address[D.B.]{Mathematical Institute, University of Oxford, Andrew Wiles Building, Radcliffe Observatory Quarter (550), Woodstock Road, Oxford, OX2 6GG}
\email{dylan.butson@maths.ox.ac.uk}
\author{Miroslav Rap\v{c}\'{a}k}
\address[M.R.]{Building 4, Meyrin site, CERN, 1211 Geneva 23, Switzerland}
\email{miroslav.rapcak@cern.ch}

\begin{abstract}
For $Y\to X$ a toric Calabi-Yau threefold resolution and $M\in \DD^b\Coh(Y)^T$ satisfying some hypotheses, we define a stack $\mf M(Y,M)$ parameterizing \emph{perverse coherent extensions} of $M$, iterated extensions of $M$ and the compactly supported perverse coherent sheaves of Bridgeland. We define framed variants $\mf M^\f(Y,M)$, prove that they are equivalent to stacks of representations of framed quivers with potential $(Q^\f,W^\f)$, and deduce natural monad presentations for these sheaves. Moreover, following Soibelman we prove that the homology $H_\bullet(\mf M^{\f,\zeta}(Y,M),\varphi_{W^\f})$ of the space of $\zeta$-stable, $\f$-framed perverse coherent extensions of $M$, with coefficients in the sheaf $\varphi_{W^\f}$ of vanishing cycles for $W^\f$, is a representation of the Kontsevich-Soibelman cohomological Hall algebra of $Y$.
	
For $M=\mc O_Y[1]$, $\mf M^{\f}(Y,M)$ is the stack of perverse coherent systems of Nagao-Nakajima, so $\bb V_Y^\zeta=H_\bullet(\mf M^{\f,\zeta}(Y,M),\varphi_{W^\f})$ is the DT/PT series of $Y$ for $\zeta=\zeta_{\DT/\PT}$ by Szendroi and \emph{loc. cit.}, and we conjecture that $\V_Y^{\zeta_\NCDT}$ is the vacuum module for the quiver Yangian of Li-Yamazaki. For $M=\mc O_S[1]$ with $S\subset Y$ a divisor, $\mf M^{\f}(Y,M)$ provides a definition in algebraic geometry for Nekrasov's spiked instanton variant of the ADHM construction, and analogous variants of the constructions of Kronheimer-Nakajima, Nakajima-Yoshioka, and Finkelberg-Rybnikov. We conjecture that $H_\bullet(\mf M^{\f,\zeta}(Y,M),\varphi_{W^{\f}})$ is the vacuum module of the vertex algebra $\V(Y,S)$ defined by the \mbox{authors} in a companion paper, generalizing the AGT conjecture to this setting. For $Y\to X=\{xy-z^mw^n\}$, this gives a geometric approach to the relationship between $W$-algebras and Yangians for affine $\gl_{m|n}$.
\end{abstract}

\maketitle

\begingroup
\hypersetup{hidelinks}
\tableofcontents
\endgroup

\vspace*{-.8cm}
\section{Introduction}

This paper is intended to contribute to a family of interconnected mathematical ideas at the intersections of geometric representation theory, enumerative algebraic geometry, low dimensional topology, and integrable systems, which follow predictions from supersymmetric quantum field theory and string theory. We are especially interested in a group of results following the conjectures of Alday-Gaiotto-Tachikawa \cite{AGT}, and their independent proof by Schiffmann-Vasserot \cite{SV}, Maulik-Okounkov \cite{MO}, and Braverman-Finkelberg-Nakajima \cite{BFN5}.

The goal of this paper is to construct certain moduli spaces of sheaves on Calabi-Yau threefolds, and representations on their homology groups of algebras of Hecke modifications of these sheaves, which conjecturally identify with modules over certain vertex algebras and affine quantum groups, generalizing the AGT conjecture and several related predictions of string theory.

In this introduction, we will explain these goals and the results of this paper as follows:

\begin{itemize}
	
	\item In Section \ref{agintrosec}, we explain the motivation from algebraic geometry: to construct moduli spaces of coherent sheaves on Calabi-Yau threefolds $Y$ which are equivalent to spaces of representations of a framed quiver with potential via a monad description, and in particular give models for moduli spaces of instantons on divisors $S\subset Y$, as in Theorem \ref{Athm}.
	
	\item In Section \ref{rtintrosec}, we explain the motivation from geometric representation theory: to construct representations on the cohomology groups of these moduli spaces of certain infinite dimensional associative algebras, as in Theorem \ref{Bthm}, and identify these with particular modules.
	
	\item In Section \ref{stintrosec}, we explain the motivation from the relevant string theory constructions.

	\item In Section \ref{summarysec}, we give a concrete summary of the results of each section in this paper.
	
\end{itemize}

\subsection{Motivation from algebraic geometry}\label{agintrosec}
The first goal of this paper is to construct certain moduli spaces of sheaves on Calabi-Yau threefolds $Y$ which admit descriptions as moduli spaces of representations of quivers with potential. The motivating example is the space of perverse coherent systems on $Y$ introduced by Nagao-Nakajima \cite{NN}, and studied systematically in \cite{Nag}. This space was defined to give a geometric interpretation to the results of Szendroi \cite{Sz1}, who computed the Donaldson-Thomas (DT) invariants \cite{DT}, Panharipande-Thomas (PT) invariants \cite{PT}, and generalizations thereof, for the resolved conifold $Y=|\mc O_{\bb P^1}(-1)\oplus \mc O_{\bb P^1}(-1)|$ using the moduli space of representations of a certain quiver with potential discovered in \cite{KlW}. The definition of the moduli space of perverse coherent systems on $Y$ is in terms of the notion of perverse coherent sheaves on $Y$ introduced in \cite{Bdg1}, but we will begin by explaining the simplest example of such a moduli space, the case $Y=\bb C^3$, for which the category of perverse coherent sheaves on $Y$ is simply the usual abelian category of coherent sheaves.

A (rank 1, compactly supported) perverse coherent system on $\C^3$ is by definition a pair $(F,s)$ where $F\in\Coh_\cs(\C^3)$ is a compactly supported coherent sheaf on $\C^3$ and $s:\mc O_{\C^3}\to F$ is a map of coherent sheaves. The stack parameterizing these objects is evidently equivalent to that of finite dimensional representations of the quiver with potential
\[ Q = \begin{tikzcd}
	\boxed{\bb C}\arrow[r, "I"] &
	\mathcircled{V}\arrow[out=340,in=20,loop,swap,"B_3"]
	\arrow[out=220,in=260,loop,swap,"B_2"]
	\arrow[out=100,in=140,loop,swap,"\normalsize{B_1}"]
\end{tikzcd} \quad\quad \text{with}\quad\quad W = B_1[B_2,B_3] \ .\]
Indeed, the critical locus equations for the potential $W$ require that the endomorphisms $B_i:V\to V$ commute, so that $V$ defines a compactly supported coherent sheaf $F$ on $\C^3$, and thus we can identify the stack of representations of the underlying unframed quiver with potential with the stack of compactly supported coherent sheaves on $\C^3$. Similarly, the image of the unit $1\in\bb C$ under the map $I:\C \to V$ determines the image of the unit section $1\in \mc O_{\C^3}$ under a unique map $s:\mc O_{\C^3}\to F$.

The open substack of cyclic representations of this quiver with potential, which is precisely the stable locus with respect to a choice of King stability condition $\zeta$, corresponds to the condition that the map $s$ is surjective, so that the moduli space of $\zeta$-stable representations is equivalent to the Hilbert scheme $\Hilb_n(\bb C^3)$ of zero dimensional subschemes of $\C^3$ of length $n=\dim V$, noting that we have a short exact sequence
\[ \mc I \to \mc O_{\bb C^3} \xrightarrow{s} F \] 
so that the data of the coherent sheaf $F$ and the map $s$ are equivalent to that of an ideal sheaf $\mc I$ in $\mc O_{\bb C^3}$ of codimension $n$. This moduli space is the simplest example of a DT theory moduli space, and is the prototypical example of the class of moduli spaces considered in this paper.

Finally, we mention another equivalent interpretation of this moduli space, which is the one that our construction most naturally generalizes. In the derived category $\DD^b\Coh(\C^3)$, we can rewrite the above exact triangle equivalently as
\begin{equation}\label{DTextensioneqn}
	 F \to \mc I[1] \to \mc O_{\bb C^3}[1]  \ , 
\end{equation}
which we interpret geometrically as follows: the ideal sheaves $\mc I$ being parameterized by $\Hilb_n(\C^3)$, viewed as objects in the derived category $\mc I[1]\in \DD^b\Coh(\C^3)$ concentrated in cohomological degree $-1$, can equivalently be characterized as certain extensions of the object $\mc O_{\bb C^3}[1]$ by a compactly supported coherent sheaf $F\in \Coh_\cs(\C^3)$ concentrated in cohomological degree $0$. Even away from the stable locus, the data of the map $s$ in the definition of perverse coherent system is evidently equivalent to a representative of the extension class under the identification
\[\Hom^1(\mc O_{\bb C^3}[1],F)= \Hom^0(\mc O_{\bb C^3},F) \ .\]

In fact, there exist exotic t-structures on the triangulated category $\DD^b\Coh(\C^3)$ with hearts containing both $\Coh_\cs(\C^3)$ and $\mc O_{\bb C^3}[1]$, called perverse coherent t-structures in the sense of Deligne, as recalled in \cite{ArBez}. Thus, we can understand the moduli space of perverse coherent systems as simply describing the space of extensions between objects in a certain abelian category, and this is the perspective that we will generalize below. For more general threefolds $Y$ on which the perverse coherent sheaves in the sense of Bridgeland are distinct from the usual coherent heart, this requires a new family of t-structures that combines these two notions, which we call Bridgeland-Deligne perverse coherent t-structures.

The primary motivation for the generalization of the moduli space of perverse coherent systems on a threefold $Y$ that we introduce in this paper is to construct moduli spaces of sheaves supported on algebraic surfaces $S$ or divisors $S\subset Y$ which are also equivalent to representations of framed quivers with potential. Perhaps the most famous example of any construction relating quiver representations and sheaves is that of Atiyah-Drinfeld-Hitchin-Manin \cite{ADHM}, which provides an isomorphism between the moduli space of stable, framed instantons of rank $r$ and charge $n$ on $\bb R^4$ and the space of stable representations of dimension $\dim V=n$ of the quiver 
\[   \begin{tikzcd}\boxed{\bb C^r}  \arrow[r, shift left=0.5ex, "I"] & \arrow[l, shift left=0.5ex, "J"]
	\mathcircled{V}	\arrow[out=70,in=110,loop,swap,"B_1"]
	\arrow[out=250,in=290,loop,swap,"\normalsize{B_2}"]\end{tikzcd}
\quad\quad\text{with relations}\quad\quad  [B_1,B_2]+IJ=0   \ .\]

It is well known that in the case $r=1$, the only non-trivial instantons are purely singular and thus the moduli space of these objects is equivalent to the Hilbert scheme $\Hilb_n(\C^2)$ after choosing a complex structure to identify $\bb R^4\cong \C^2$. The induced identification between stable, rank $1$ representations of the ADHM quiver of dimension $n$ and points in $\Hilb_n(\C^2)$ is exactly in parallel with the $\C^3$ case explained above: one can show that the stable locus in the stack of quiver representations is contained in the locus where the map $J:V\to \bb C$ vanishes, after which the argument from the $\C^3$ case above applies \emph{mutatis mutandis}.

It is not difficult to see that this construction can be recast in terms of sheaves on $\C^3$ supported on $\C^2$, and that the ADHM quiver admits a corresponding natural generalization to a quiver with potential. The Hilbert scheme of zero dimensional, length $n$ subschemes of $\C^2$ is by definition the quot scheme $\Quot_n(\C^2,\mc O_{\C^2})$ of length $n$ quotients over $\C^2$ of the structure sheaf $\mc O_{\C^2}$, which can be equivalently interpreted as the quot scheme $\Quot_n(\C^3,\iota_*\mc O_{\C^2})$ of length $n$ quotients over $\C^3$ of the pushforward of the structure sheaf $\mc O_{\C^2}$ along the inclusion $\iota:\C^2 \to \C^3$. Similarly, since the ADHM quiver is a doubled quiver in the sense of Nakajima \cite{Nak1}, it admits a canonical enhancement to a tripled quiver with potential in the sense of Ginzburg \cite{Ginz}, given by
\begin{equation}\label{3dadhmeqn}
	\begin{tikzcd}\boxed{\C^r} \arrow[r, shift left=0.5ex, "I"] & \arrow[l, shift left=0.5ex, "J"]
	\mathcircled{V}\arrow[out=340,in=20,loop,swap,"B_3"]
	\arrow[out=220,in=260,loop,swap,"B_2"]
	\arrow[out=100,in=140,loop,swap,"\normalsize{B_1}"] \end{tikzcd}
\quad\quad\text{and}\quad\quad   W = B_1[B_2,B_3] + B_3I J \ .  
\end{equation}
These two reformulations evidently correspond to each other under an equivalence of the type explained above. Indeed, the critical locus equations for the potential $W$ imply not only that $B_1,B_2,I$ and $J$ satisfy the ADHM relations, but additionally that
\[ [B_1,B_3]=[B_2,B_3]=0 \quad\quad \text{and} \quad\quad B_3I=0 \ .\]
Together, this implies that on the stable locus $B_3$ vanishes identically on $V$, and thus $V$ should be interpreted as a compactly supported coherent sheaf on $\C^3$ which is annihilated by the coordinate function $z_3$ corresponding to $B_3$, and thus supported on $\C^2\subset \C^3$. In turn, the map $\mc O_{\C^3}\to V$ corresponding to $I$ evidently factors through a map $s:\iota_*\mc O_{\C^2} \to V$.

In these terms, the goal of our construction is to systematically generalize the preceding example in which $Y=\C^3$ and $S=\C^2$ to give analogous models in algebraic geometry for moduli spaces of instantons on divisors $S$ in Calabi-Yau threefolds $Y$. However, in order to better motivate the details of our approach, we briefly discuss some of the subtleties which appear already in this simple setting and must be resolved in the general construction.

One natural question when comparing with the example of perverse coherent systems on $\C^3$ is to give an analogous geometric description of the full stack of representations of the quiver with potential of Equation \ref{3dadhmeqn}. In particular, away from the stable locus the map $J:V\to \C$ is not in general zero, so the moduli space is evidently parameterizing more than just general maps $s:\mc O_{\C^2}\to V$, but there is not such an evident geometric interpretation for the map $J$ as for $I$. Moreover, this issue is even more apparent in rank $r>1$, in which case the map $J$ is not in general zero even on the stable locus. Further, the well-known description in algebraic geometry of the moduli space of stable representations of the ADHM quiver of higher rank, as in \cite{NakLec}, requires a choice of projective compactification of $S=\C^2$, but we would like to generalize this construction to surfaces $S$ in threefolds $Y$ for which there is not evidently a canonical choice of projective compactification. We now explain the resolution of these subtleties in the present setting:

Even in the well-known setting of \emph{loc. cit.}, where the moduli space of stable representations of the usual ADHM quiver is identified with the moduli space of torsion free sheaves $\mc E$ on $\bb P^2$ equipped with a trivialization $\varphi:\mc E|_{\bb P^1_\infty}\xrightarrow{\cong} \mc O_{\bb P^1_\infty}$ on the line $\bb P^1_\infty \subset \bb P^2$ at infinity, there is surprisingly little discussion in the literature of the full stack of representations. In fact, the only explicit reference we could find which gives a geometric description of the stack of representations of the ADHM quiver is in Section 5 of \cite{BFG}, where the result is attributed to Drinfeld. The description is as the stack of torsion free perverse coherent sheaves on $\bb P^2$, equipped with a trivialization at infinity as above, where the former are defined to be complexes of coherent sheaves
\begin{equation}\label{pervTFSeqn}
	 E = \left[ E^{-1}[1] \xrightarrow{d} E^0  \right] \ , 
\end{equation}
such that $H^{-1}E$ is torsion free and $H^0E$ has zero dimensional support. We now reinterpret this description analogously to the interpretation of perverse coherent systems in terms of extensions explained above.

First, note that these complexes of sheaves are again in the heart of a perverse coherent t-structure on $\bb P^2$ in the sense of Deligne, as the name suggests. For simplicity, we begin by considering the stable locus on which the differential $d$ is surjective, so that the objects being parameterized are just the usual torsion free coherent sheaves $H^{-1} E=\mc E[1]$ thought of as concentrated in cohomological degree $-1$. It is well known that every such sheaf fits into a canonical short exact sequence
\[ \mc E\to (\mc E^\vee)^\vee \to  F \]
where the reflexive hull $(\mc E^\vee)^\vee$ determines a rank $r$ vector bundle on $\bb P^2$, and $F$ is a coherent sheaf with zero dimensional support concentrated in cohomological degree $0$.

Note that $F$ must be supported on the complement $\C^2=\bb P^2\setminus \bb P^1_\infty$ of the line at infinity, by the existence of the isomorphism $\varphi$, and similarly for $H^0E$ in the unstable case. Thus, the restriction $\tilde{\mc E}= \mc E|_{\C^2}$ and the restriction of the double dual $( \tilde{\mc E}^\vee)^\vee$ also fit into a short exact sequence
\[\tilde{ \mc E}\to (\tilde{ \mc E}^\vee)^\vee \to  F   \ . \]
The rank $r$ vector bundle $(\tilde{ \mc E}^\vee)^\vee $ on $\C^2$ is necessarily trivializable, and
the choice of trivialization at infinity of $\mc E$ determines a trivialization at infinity for $(\mc E^\vee)^\vee$ which extends uniquely to a trivialization of $(\tilde{ \mc E}^\vee)^\vee[1]$. Thus, the space of sheaves of the form $\tilde{\mc E}[1]$ is identified with the space of extensions
\begin{equation}\label{pervcohextADHMeqn}
	  F \to \tilde{\mc E}[1] \to  \mc O_{\C^2}^{\oplus r}[1] \ ,
\end{equation}
which is the desired analogue of Equation \ref{DTextensioneqn}. In fact, the object $\mc O_{\C^2}^{\oplus r}[1]$ is itself an iterated extension of the object $\mc O_{\C^2}[1]$, so that we can further identify the space of sheaves $\tilde{\mc E}[1]$ with the space of iterated extensions of compactly supported coherent sheaves $F\in\Coh_\cs(\C^3)$ with $r$ copies of $\mc O_{\C^2}[1]$, such that $F$ is a subobject and the underlying iterated extension of $\mc O_{\C^2}[1]$ is equipped with an isomorphism to the object $\mc O_{\C^2}^{\oplus r}[1]$.

The generalization of this description to the unstable locus simply corresponds to dropping the requirement that $F$ occurs as a subobject in the iterated extension. For example, in the simplest case that $F=\iota_*\mc O_\pt$ is given by the structure sheaf of a single reduced point in $\C^2$, there is a unique non-trivial extension class
\begin{equation}\label{beqn1}
	 \mc O_{\C^2}[1] \to E \to \mc O_\pt \quad\quad \text{represented by}\quad\quad \xymatrixcolsep{3pc}
\xymatrixrowsep{3pc}
\vcenter{\xymatrix{ \mc{O}_{\C^2} \ar[d]^{1} \ar[r] &    \mc{O}_{\C^2}^2\ar[r] & \mc{O}_{\C^2}  \\
	\mc{O}_{\C^2}  &  &} } \ , 
\end{equation}
where the top row is given by the Koszul resolution of $\mc O_\pt$. The totalization of this map of complexes provides a representative $E$ of the extension class, and this complex of sheaves is the prototypical example of the desired geometric interpretation of a representation of the ADHM quiver with $J\neq 0$. Equivalently, we can view this as a representative of the corresponding class in $\Hom^2( \mc O_\pt,\mc O_{\C^2})$, as
\[ \mc O_{\C^2} \to E^{-1} \to E^0 \to \mc O_\pt  \quad\quad \text{where}\quad\quad \begin{cases} E^{-1} & =\coker\left[ \mc O_{\C^2} \to \mc O_{\C^2}^3\right]  \\  E^0 & =\mc O_{\C^2} \end{cases} \ ,\]
which is the corresponding prototypical example of a perverse coherent torsion free sheaf, in the sense of Equation \ref{pervTFSeqn}, for which the differential $d$ is not surjective.

We can extend this argument to explicitly parameterize the entire moduli space of such sheaves in terms of representations of the ADHM quiver. The corresponding prototypical extension of the form in Equation \ref{pervcohextADHMeqn} is represented analogously in terms of the Koszul resolution of $\mc O_\pt$ by the map of complexes
\begin{equation}\label{beqn2} \xymatrixcolsep{3pc}
\xymatrixrowsep{3pc}
\vcenter{\xymatrix{  &  & \mc{O}_{\C^2}\ar[d]^{1}   \\  \mc{O}_{\C^2}  \ar[r] &    \mc{O}_{\C^2}^2\ar[r] & \mc{O}_{\C^2} 
	}} \ . \end{equation}
Moreover, note that any compactly supported coherent sheaf $F\in \Coh_\cs(\C^2)$ of length $n$ is itself an iterated extension of the structure sheaves of $n$ points in $\C^2$, and again there are natural representatives of the two elementary non-trivial extensions in $\Hom^1(\mc O_\pt,\mc O_\pt)\cong \C^2$ given by
\begin{equation}\hspace*{-.5cm}\label{beqn3}  \xymatrixcolsep{3pc}
\xymatrixrowsep{3pc}
\vcenter{\xymatrix{ & \mc{O}_{\C^2}  \ar[r] \ar[d]^{\scriptsize{\begin{bmatrix} 1 \\ 0
	\end{bmatrix}}} &    \mc{O}_{\C^2}^2\ar[r]  \ar[d]^{\scriptsize{\begin{bmatrix} 0 & 1
	\end{bmatrix}}} & \mc{O}_{\C^2}  \\  \mc{O}_{\C^2}  \ar[r] &    \mc{O}_{\C^2}^2\ar[r] & \mc{O}_{\C^2} 
}} \quad\quad\text{and}\quad\quad \xymatrixcolsep{3pc}
\xymatrixrowsep{3pc}
\vcenter{\xymatrix{ & \mc{O}_{\C^2}  \ar[r] \ar[d]^{\scriptsize{\begin{bmatrix} 0 \\ -1
\end{bmatrix}}} &    \mc{O}_{\C^2}^2\ar[r]  \ar[d]^{\scriptsize{\begin{bmatrix} 1 & 0
\end{bmatrix}}} & \mc{O}_{\C^2}  \\  \mc{O}_{\C^2}  \ar[r] &    \mc{O}_{\C^2}^2\ar[r] & \mc{O}_{\C^2} 
}} \ . \end{equation}
Thus, we can express a general sheaf $E$ which is isomorphic to an iterated extension of $n$ copies of $\mc O_\pt$ and $r$ copies of $\mc O_{\C^2}[1]$, and equipped with an isomorphism between the underlying iterated extension of $\mc O_{\C^2}[1]$ and $\mc O_{\C^2}^{\oplus r}[1]$, as a complex of the form
\begin{equation}\label{ADHMmonadeqn}
	  \xymatrixcolsep{9pc}
\xymatrixrowsep{.2pc}\vcenter{\xymatrix{ & \mc O_{\C^2}\otimes V & \\ & \oplus & \\ \mc O_{\C^2}\otimes V \ar[r]^{\scriptsize{\begin{bmatrix} B_1-z_1 \\ B_2-z_2 \\ J
	\end{bmatrix}}}  & \mc O_{\C^2}\otimes V \ar[r]^{\scriptsize{\begin{bmatrix} z_2-B_2 & B_1-z_1 & I
\end{bmatrix}}} & \mc O_{\C^2}\otimes V \\ & \oplus & \\
	& \mc O_{\C^2}\otimes \C^r }} \ ,
\end{equation}
where $V$ is an $n$ dimensional vector space representing the multiplicity space of the copies of $\mc O_\pt$ used in the construction. This is precisely the monad presentation of Equation 2.6 in \cite{NakLec} for arbitrary maps $B_1$, $B_2$, $I$, and $J$ satisfying the ADHM equations, and in particular we deduce that this moduli space of sheaves is equivalent to the full stack of representations of the ADHM quiver.

The preceding arguments all apply analogously to identify the stack of representations of the quiver with potential of Equation \ref{3dadhmeqn} with that parameterizing complexes of coherent sheaves isomorphic to an iterated extension of $\mc O_\pt$ and $\iota_*\mc O_{\C^2}[1]$ in an appropriate category of perverse coherent sheaves on $\C^3$ in the sense of Deligne, equipped with an isomorphism of the underlying iterated extension of $\iota_*\mc O_{\C^2}[1]$ with $\iota_*\mc O_{\C^2}^{\oplus r}[1]$. This is the prototypical example of the spaces of coherent sheaves studied in this paper, which we call (rank $r$, trivially framed) \emph{perverse coherent extensions} of $\iota_*\mc O_{\C^2}[1]$. We remark that the object which we have referred to as \emph{the} underlying iterated extension of $\iota_*\mc O_{\C^2}[1]$ is not in general well-defined, but we give general hypotheses sufficient to ensure that it is, which can be verified directly in the examples of interest; we will ignore this subtlety for the remainder of the introduction.

The reinterpretation of the framing data as a choice of isomorphism between the underlying iterated extension of $\iota_*\mc O_{\C^2}[1]$ and the trivial extension $\iota_*\mc O_{\C^2}^{\oplus r}[1]$ may at first seem ad hoc, but in fact it immediately leads to a natural family of generalizations of the above correspondence, which play a crucial role in our desired applications in geometric representation theory. We define a framing structure $\f$ to be a fixed iterated extension $H_\f$ of $r$ copies of $\iota_*\mc O_{\C^2}[1]$, and we can consider the analogous moduli space of sheaves but equipped with an isomorphism to $H_\f$ instead of the trivial extension $\iota_*\mc O_{\C^2}^{\oplus r}[1]$, which we call $\f$\emph{-framed} perverse coherent extensions.

In present setting, such extensions are determined by a nilpotent Higgs field $\phi_\f$ of rank $r$ on $\C^2$, and if we assume that $\phi_\f$ is given by the constant nilpotent matrix $A_\f\in \gl_r$, the corresponding stack of $\f$-framed perverse coherent extensions of $\iota_*\mc O_{\C^2}[1]$ is equivalent to that of representations of the modified quiver with potential
\begin{equation}\label{CCDSeqn}
 \begin{tikzcd}\boxed{V_\infty} 	\arrow[out=160,in=200,loop,swap,"\bullet" marking, "A_\f"]  \arrow[r, shift left=0.5ex, "I"] & \arrow[l, shift left=0.5ex, "J"]
	\mathcircled{V}\arrow[out=340,in=20,loop,swap,"B_3"]
	\arrow[out=220,in=260,loop,swap,"B_2"]
	\arrow[out=100,in=140,loop,swap,"\normalsize{B_1}"] \end{tikzcd}
\quad\quad\text{with}\quad\quad    W = B_1[B_2,B_3] + B_3I J + I A_\f J \ ;
\end{equation}
arrows marked with a $\bullet$ such as that labeled by $A_\f$ in this quiver are fixed, in the sense that they are not part of the data required to specify a representation of the quiver, but are chosen as part of the input data for defining the quiver with potential itself, and modify the equations on the collection of linear maps corresponding to the unmarked arrows defining a quiver representation. This quiver with potential was studied recently in \cite{CCDS} in precisely the geometric representation theory context in which we are interested, and their results were our initial motivation for studying this generalized notion of framing.

We can also consider framed perverse coherent extensions in the presence of more than one object at once, which provide analogous models in algebraic geometry for moduli spaces of instantons on divisors which are not necessarily irreducible. Indeed, considering the shifted structure sheaves of the three coordinate divisors in $\C^3$ simultaneously, we obtain a description in algebraic geometry for the stack of representations of the spiked instantons variant of the ADHM quiver introduced in \cite{NekP}, the quiver with potential defined by
\begin{equation}\label{Nekeqn}
	\begin{tikzcd}
	& & \boxed{V_\infty^2} \ar[dl, shift left=0.5ex, "I_2"] \\
	\boxed{V_\infty^3} \arrow[r, shift left=0.5ex, "I_3"] & \arrow[l, shift left=0.5ex, "J_3"]
	\mathcircled{V}\arrow[out=340,in=20,loop,swap,"B_3"]
	\arrow[out=220,in=260,loop,swap,"B_2"]
	\arrow[out=100,in=140,loop,swap,"\normalsize{B_1}"]  \ar[ur, shift left=0.5ex, "J_2"] \ar[dr, shift left=0.5ex, "J_1"] \\
	& & \boxed{V_\infty^1} \ar[ul, shift left=0.5ex, "I_1"]
\end{tikzcd}  \quad\quad \text{with}\quad\quad   W_M^\f = B_1[B_2,B_3] + \sum_{k=1}^3 B_k I_k J_k\ .
\end{equation}

\noindent The above is the quiver with potential corresponding to the trivial framing structure, but for any choice of iterated extension of the framing objects there is a corresponding modification of the potential analogous to that in Equation \ref{CCDSeqn} above.

We now briefly describe our generalization of the three dimensional variants of the ADHM construction described above to a larger class of algebraic surfaces $S$, or more generally divisors, in Calabi-Yau threefolds $Y$. We must assume that $Y$ occurs as a resolution $f:Y\to X$ such that $\dim f^{-1}(\{x\})\leq 1$ for each $x\in X$ and $f_* \mc O_Y \cong \mc O_X$, where we use $f_*$ to denote the full derived pushforward functor, and for concreteness we restrict our attention to the case that $X$ is an affine, toric threefold singularity with a unique $T$-fixed point $x\in X$. The former conditions are precisely those under which the perverse coherent t-structure in the sense of Bridgeland is defined on $\DD^b\Coh(Y)$, and the basic idea is to modify the definitions above by replacing the compactly supported coherent sheaves $\Coh_\cs(Y)$ on $Y=\C^3$ used in the preceding constructions with the complexes of sheaves with compactly supported cohomology contained in the alternative heart $\Perv_\cs(Y)\subset \Perv(Y)$.

The category of compactly supported perverse coherent sheaves $\Perv_\cs(Y)$ has a natural collection of simple generators, generalizing the role of the structure sheaves of points in $\Coh_\cs(\C^3)$ in the previous constructions. In the simplest example where the fibre $C=f^{-1}(\{x\})$ over the fixed point $x\in X$ is reduced and irreducible, so that it is isomorphic as a scheme to $\bb P^1$, the compactly supported perverse coherent sheaves on $Y$ are generated by the structure sheaves of points in $Y\setminus C$ together with the objects $F_0=\iota_*\mc O_{\bb P^1}$ and $F_1\iota_* \mc O_{\bb P^1}(-1)[1]$, where $\iota:C\to Y$ is the inclusion of the exceptional curve. For simplicity, we restrict our attention to the formal completion $\hat Y$ of $Y$ along $C$, for which the category is generated only by the latter two objects.

We can again identify the moduli stack of objects in this category with the stack of representations of an unframed quiver with potential, in analogy with the description of $\Coh_\cs(\C^3)$ above, where the quiver is defined to have two vertices corresponding to the objects $F_0$ and $F_1$, and arrows determined by the non-trivial extension classes between these objects in the threefold $Y$. For example, in the case that $Y=|\mc O_{\bb P^1}(-1)\oplus \mc O_{\bb P^1}(-1)|$, it is given by
\[
\begin{tikzcd}
	\mathcircled{V_0} \arrow[r, bend left=25 ] \arrow[r, bend left=40 ,  "A\ C"]  & \arrow[l, bend left=25 ] \arrow[l, bend left=40 ,  "B\ D"]  \mathcircled{V_1}
\end{tikzcd}
\quad\quad \text{and}\quad\quad
W= ABCD-ADCB \ ,
\]
which is precisely the unframed variant of the quiver with potential used in our motivating example of the study of the DT theory of the conifold in \cite{Sz1}, originally discovered in string theory by Klebanov-Witten \cite{KlW}. We will comment further on the physical derivation below in Section \ref{stintrosec}, which explains the motivation from string theory for our constructions.

A typical example is the description of the structure sheaf $\mc O_y$ of a point $y$ on the exceptional curve $C$ as the quiver representation corresponding to an extension
\[    \mc O_{\bb P^1} \to \mc O_y \to \mc O_{\bb P^1}(-1)[1]   \ ; \]
this was one of the original motivations for the definition of the category of perverse coherent sheaves in \cite{Bdg1}, and can be understood as a reinterpretation in the derived category of the classical description of Beilinson \cite{Bei2} of coherent sheaves on $\bb P^1$ as representations of a quiver.

More generally, we consider the stack of iterated extensions of these objects together with a fixed number of copies of an auxiliary object $M\in \DD^b\Coh(Y)$, such as $\mc O_S[1]$ for a reduced divisor $S\subset Y$, equipped with an isomorphism of the underlying iterated extension of $\mc O_S[1]$ with some fixed choice of such an extension. This is the general definition of the stack of perverse coherent extensions of $M$, and under some hypotheses on the object $M$, we prove that it is equivalent to the stack of representations of a framed quiver with potential, and that this equivalence is implemented explicitly by a monad presentation, as for the examples in the case $Y=\C^3$ explained above.

In the case $Y=\widetilde{A_{m-1}}\times \bb A^1 \to X=\{ xy-z^m\}\times \bb A^1$, the product of $\bb A^1$ with a resolution $\widetilde{A_{m-1}}$ of the singularity $A_{m-1}=\{ xy-z^m\}$, and $M=\mc O_S[1]$ for $S=\widetilde{A_{m-1}}$, the resulting quiver is a framed, tripled variant of the affine Dynkin quiver of type $A_{m-1}$. In the case $m=2$ for example, with framing structure determined by a nilpotent endomorphism $G_\f\in \gl_r$, the resulting quiver with potential is given by
\[ \hspace*{-1cm} Q_M^\f = \quad \begin{tikzcd}
	\boxed{V_\infty}\arrow[d,shift left=0.5ex, "I"] \arrow[out=160,in=200,loop, "\bullet" marking,swap,"G_\f"] \\ 
	\arrow[out=160,in=200,loop,swap,"E"] \ar[u,shift left=0.5ex, "J"] \mathcircled{V_0} \arrow[r, bend left=25 ] \arrow[r, bend left=40 ,  "A\ C"]  & \arrow[l, bend left=25 ] \arrow[l, bend left=40 ,  "B\ D"] \mathcircled{V_1}\arrow[out=340,in=20,loop,swap,"F"]
\end{tikzcd} \quad\quad \text{and}\quad\quad W_M^\f =E(BC-DA)+ F(AD-CB) + EIJ - IG_\f J  \ . \] 
In the case $G_\f=0$, this is precisely the three dimensional analogue of the Nakajima quiver variety given by the doubled affine Dynkin quiver, the moduli spaces of representations of which was identified with moduli space of instantons on $\widetilde{A_{m-1}}$ in the results of Kronheimer-Nakajima \cite{KrNak}. We give a three dimensional variant of this identification which extends to the full stack of quiver representations, via a corresponding variant of the monad description of \emph{loc. cit.}.

In the case $Y=|\mc O_{\bb P^1}(-1)\oplus \mc O_{\bb P^1}(-1)| \to X = \{ xy-zw\}$ and $M=\mc O_S[1]$ for $S=|\mc O_{\bb P^1} (-1)|$, the quiver with potential is given by 
\[\begin{tikzcd}
	\boxed{V_\infty}\arrow[d,shift left=0.5ex, "I"] \\ 
	\mathcircled{V_0} \arrow[r, bend left=25 ] \arrow[r, bend left=40 ,  "A\ C"]  & \arrow[l, bend left=25 ] \arrow[l, bend left=40 ,  "B\ D"] \mathcircled{V_1}\ar[ul, bend right=40, swap,"J"]    
\end{tikzcd} \quad\quad \text{and}\quad\quad   W = ABCD-ADBC + IJC \ . \]
This is a three dimensional analogue of the quiver with relations studied in \cite{NY1}, generalizing the results of the thesis \cite{KingTh} of King describing vector bundles to torsion free sheaves on $|\mc O(-1)|$. However, we note that this example involves an additional subtlety which was absent in the preceding examples: since the potential has quartic terms rather than only cubic, the critical locus equations will have cubic terms rather than just quadratic, and thus the differentials in the monad description must have some quadratic dependence on the linear maps defining the quiver representation. Indeed, the three dimensional monad description produced by our construction in this case is given by
	\begin{equation}\label{NYmonadeqn}
	\xymatrix @R=-.3pc @C=-1.5pc{ & {\scriptsize{\begin{pmatrix}  1 & -B \\ z & -D \\ -A & x \\ - C & y \\  0 & J \end{pmatrix}}} & && {\scriptsize{\begin{pmatrix}zy-DC & CB-y & 0 & zB-D & 0\\ D A - zx & x-BA & D-zB & 0 & I \\ 0 & yA-xC & yz - CD & AD-xz & 0 \\ xC-yA & 0 & CB-y & x-AB & 0 \\ 0 & 0 & 0 & J & 0\end{pmatrix}}} && && \scriptsize{\begin{pmatrix} x & y & B & D & I \\ A & C & 1 & z & 0 \end{pmatrix}}&\\
		\mc O \otimes V_0   & & \mc O(1)^2 \otimes V_0  && &&\mc O(1)^2 \otimes V_0 &&  &&\mc O  \otimes V_0 	\\
		\oplus&  \longrightarrow & \oplus && \longrightarrow && \oplus && \longrightarrow && \oplus \\
		\\
		\mc O(1) \otimes V_1  & & \mc O^2 \otimes V_1 &&& &    \mc O^2 \otimes V_1 &&&& \mc O(1) \otimes V_1 \\ 
		&   & \oplus &&  && \oplus &&  &&  \\
		& &\mc O(1) \otimes V_\infty && && \mc O \otimes V_\infty } \ .
\end{equation}

Evidently a monad description involving a non-linear dependence on the linear maps defining the quiver representation can not be derived from the straightforward argument we gave to obtain the ADHM monad description in Equation \ref{ADHMmonadeqn} above, in which we added differentials to the Koszul resolutions of the constituent objects of the iterated extension corresponding to the representatives of the extension classes between the simple objects, as such differentials evidently depend linearly on the extension classes and in turn the linear maps in the quiver representation.

In fact, we observe that the condition that the potential is cubic corresponds to the requirement that the path algebra of the resulting quiver with relations is Koszul. In general, the path algebra need not be Koszul and equivalently its Koszul dual algebra, which is given by the $\Ext$ algebra of the simple objects $F_i$, will in general be an $A_\infty$ algebra rather than a strict associative algebra. Following this observation, we use the analogy with the classical setting of \cite{BGS} to construct the monad description as the image, under the derived equivalence of module categories between the Koszul dual algebras, of a combinatorial description of the corresponding heart of the category of modules over the $\Ext$ algebra. In the Koszul case, this heart is given by the category of linear complexes of projective modules, as explained in \emph{loc. cit.}, and in the general $\Ainf$ case we generalize this description using the twisted objects construction proposed by Kontsevich in \cite{KonHMS}, and established carefully in \cite{Lef}, allowing us to deduce a general formula for the desired non-linear monad description of perverse coherent extensions.

We summarize the above discussion by stating the first main theorem of the paper: Let $Y\to X$ be a toric Calabi-Yau threefold resolution and $M\in \DD^b\Coh(Y)^T$ an object in the coherent derived category, equivariant with respect to the structure torus $T$, and satisfying some hypotheses as outlined above. Then given a choice of a framing structure $\f$, there exists an algebraic stack $\mf M^\f(Y,M)$ parameterizing $\f$-framed perverse coherent extensions of $M$, and we have:

\begin{manualtheorem}{A}[\ref{stackpotthm}]\label{Athm} 
	There is a canonical framed quiver with potential $(Q_M^\f,W_M^\f)$ and an equivalence of algebraic stacks
	\begin{equation*}
		\mf M(Q_M^\f,W_M^\f) \xrightarrow{\cong} \mf M^{\f}(Y,M) \ ,
	\end{equation*}
	where the induced equivalence of groupoids of $\bb K$ points is defined on objects by
	\begin{equation*}\hspace*{-1cm}
		(V_i, B_{ij})   \mapsto \left(	\tilde H:=  \bigoplus_{i\in I_M}K(I_i)\otimes  V_i \ ,\ d_B:= \sum_{k\in \bb Z,\  i,i_2,...,i_{k-1},j \in I_M}  K (\rho_k^{\Sigma\oplus \Sigma_\infty}(\cdot, b_{i,i_2},...,b_{i_{k-1}, j})^N)\otimes  ( B_{i,i_2}  ...  B_{i_{k-1}, j} )  \right)  \ .
	\end{equation*}
\end{manualtheorem}

This general formula determines the monad descriptions of Equations \ref{ADHMmonadeqn} and \ref{NYmonadeqn}, for example. The pair $(V_i,B_{ij})$ denotes a representation of the framed quiver with potential $(Q_M^\f,W_M^\f)$, determined by vector spaces $V_i$ and linear maps $B_{ij}:V_i\to V_j$, the objects $K(I_i)\xrightarrow{\cong} F_i$ are canonical Koszul-type projective resolutions of the compactly supported simple objects $F_i\in \Perv_\cs(Y)$, and $d_B$ denotes the deformation of the differential determined by the quiver representation in terms of certain $\Ainf$ module structure maps $(\rho_k^{\Sigma\oplus \Sigma_\infty})_{k\in\bb Z}$ for the $\Ext$ algebras $\Sigma=\Hom(F,F)$, and the representatives of extension classes $b_{i,j}\in \Hom^1(F_i,F_j)$ such as those of Equations \ref{beqn1}, \ref{beqn2} and \ref{beqn3}.

\subsection{Motivation from geometric representation theory}\label{rtintrosec} The primary motivation for the results of this paper, and in particular for the detailed study of the moduli spaces of coherent sheaves on threefolds described in the preceding section, is to generalize a conjecture from the seminal paper \cite{AGT} of Alday-Gaiotto-Tachikawa (AGT) and its proof by Schiffmann-Vasserot in \cite{SV}. This generalization predicts relationships between the enumerative geometry of coherent sheaves on surfaces and threefolds, and the representation theory of certain vertex algebras and affine quantum groups, mediated by familiar mechanisms from geometric representation theory.

To begin, we explain the geometric origin of the vector spaces on which we will define these representations. Recall that the original motivation for the definition of perverse coherent systems in \cite{NN} was to provide a geometric explanation for the computations in \cite{Sz1} of the DT invariants, PT invariants, and generalizations thereof, for the threefold $Y=|\mc O_{\bb P^1}(-1)\oplus \mc O_{\bb P^1}(-1)|$, in terms of the stack of representations of a quiver with potential. In fact, the computation was of the cohomological analogues of these invariants, as we now explain:

For each choice of stability condition $\zeta$, \emph{loc. cit.} compute the homology of the space of $\zeta$-stable representations of the framed quiver with potential $(Q^{\f}_Y,W^{\f}_Y)$ determined by the object $\mc O_Y[1]$ with its unique framing structure $\f$, or equivalently the space of $\zeta$-stable, trivially framed, perverse coherent extensions of $\mc O_Y[1]$, with coefficients in the sheaf of vanishing cycles $\varphi_{W^{\f}_Y}$ for $W^{\f}_Y$,
\[ \bb V^\zeta_Y = H_\bullet( \mf M^{\f,\zeta}(Y,\mc O_{Y}[1]),\varphi_{W^{\f}_Y}) = \bigoplus_{\dd \in \bb N^{Q_Y}} H_\bullet( \mf M_\dd^{\f,\zeta}(Y,\mc O_{Y}[1]),\varphi_{W^{\f}_Y})  \ \]
where $\mf M_\dd^{\f,\zeta}(Y,\mc O_{Y}[1])$ denotes the connected component corresponding to representations of dimension $\dd=(d_i)_{i\in I} \in \bb N^{Q_Y}$. Then, it was explained in \emph{loc. cit.} that the corresponding generating functions of the Euler characteristics of the graded components
\[ \mc Z^\zeta_Y(\qq) =  \sum_{\dd \in \bb N^{Q_Y}} \qq^\dd \chi(H_\bullet( \mf M_\dd^{\f,\zeta}(Y,\mc O_{Y}[1]),\varphi_{W^{\f}_Y})) \  \in \Pq  \ , \]
for appropriate stability conditions $\zeta_\DT$ and $\zeta_\PT$, compute the DT series and PT series of $Y$, 
\[ \mc Z_Y^{\zeta_\DT}(\qq) = \mc Z_Y^\DT(\tilde \qq) \quad\quad \text{and}\quad \quad   \mc Z_Y^{\zeta_\PT}(\qq) = \mc Z_Y^\PT(\tilde \qq)  \ ,\]
for some appropriate identifications of parameters. For example, in the case $Y=\C^3$ the DT series is given by the MacMahon function
\[ \mc Z_{\C^3}^{\DT}(q) = \prod_{k=1}^\infty \frac{1}{(1-q^k)^k}=M(q) \quad\quad \text{for}\quad\quad q=-q_0  \ ,\]
as conjectured in \cite{MNOP}; this computation follows from a fixed point counting argument by general results of \cite{BehFan}, and was successively generalized in \cite{Sz1}, \cite{Young1}, \cite{NN}, and \cite{Nag}.

For any object $M\in \DD^b\Coh(Y)^T$ satisfying appropriate hypotheses and a fixed choice of framing structure $\f$, we can consider analogous spaces of $\zeta$-stable representations of the associated framed quiver with potential, or equivalently $\zeta$-stable, $\f$-framed perverse coherent extensions of $M$, and define the corresponding cohomological invariants
\begin{equation}\label{moduleintroeqn}
	 \bb V^{\f,\zeta}(M) = \bigoplus_{\dd \in \bb N^{Q_Y}} \bb V^{\f,\zeta}_\dd(M) \quad\quad  \bb V^{\f,\zeta}_\dd(M)= H_\bullet(\mf M^{\f,\zeta}_\dd(Y,M), \varphi_{W_M^\f})  \ ,
\end{equation}
and their associated generating functions. We will be especially interested in the case $M=\mc O_S[1]$ for $S\subset Y$ a divisor, as we have discussed in the preceding section.

Let $S\subset Y$ be an effective Cartier divisor, $S^\red$ the underlying reduced scheme, and let $S_d$ for $d\in \mf D_S$ denote the irreducible components of $S^\red$, so that $\mc O_{S^\red}^\sss$ the semisimplification of $\mc O_{S^\red}$ is
\[ \mc O_{S^\red}^\sss = \bigoplus_{d\in \mf D_S} \mc O_{S_d}   \quad\quad  \text{and we have} \quad\quad  [S] = \sum_{d\in \mf D_S} r_d [S_d]  \]
for some tuple $\rr_S=(r_d)_{d\in\mf D_S}$ of positive integers $r_d\in \bb N$. Consider the space of perverse coherent extensions of $\mc O_{S^\red}^\sss[1]$, and note that $\mc O_S[1]$ is by definition an iterated extension of the objects $\mc O_{S_d}[1]$ which each occur with multiplicity $r_d$, so that the object $\mc O_S[1]$ itself determines a framing structure $\f_S$ of rank $\rr_S$. Thus, for each stability condition $\zeta$ on the corresponding framed quiver with potential, we define the associated cohomological invariant
\[ \bb V^\zeta_S = \bigoplus_{\dd\in \bb N^{V_{Q_Y}}} \V^\zeta_{S,\dd} \quad\quad \V_{S,\dd}^\zeta= H_\bullet(\mf M_\dd^{\f_S,\zeta}(Y,\mc O_{S^\red}^\sss[1] ),\varphi_{W^{\f_S}}) \ .   \]
For an appropriate choice of stability condition $\zeta=\zeta_\VW$, we let $\bb V_S=\bb V_S^{\zeta_\VW}$ and introduce the corresponding generating function
\[ \mc Z^\VW_S(\q) = \sum_{\dd \in \bb N^{Q_Y}} \qq^\dd \chi(H_\bullet(\mf M_\dd^{\f_S,\zeta_\VW}(Y,\mc O_{S^\red}^\sss[1] ),\varphi_{W^{\f_S}})) \  \in \Pq  \ , \]
which we interpret as a local analogue of the Vafa-Witten partition function proposed in \cite{VW}. Indeed, in cases where $S^\red$ is given by an irreducible smooth surface, this will be a generating function for the Euler characteristics of moduli spaces of instantons on $S$, as we explain below. We also consider the analogous vector spaces
\[ \bb V_S^0 =\bigoplus_{\dd\in \bb N^{V_{Q_Y}}} \V^0_{S,\dd} \quad\quad \V_{S,\dd}^0= H_\bullet(\mf M_\dd^{0_S,\zeta_\VW}(Y,\mc O_{S^\red}^\sss[1] ),\varphi_{W^{0_S}}) \ ,  \]
determined by the trivial framing structure $0_S$ of rank $\rr_S$.

We now consider the simplest example of this construction: let $Y=\C^3$ and $S=\C^2$, so that the corresponding quiver with potential is given by the three dimensional analogue of the ADHM quiver, as in Equation \ref{3dadhmeqn}. One crucial aspect of the analogy between this quiver with potential and the usual ADHM quiver is that they are related by \emph{dimensional reduction}, in the sense of Appendix A of \cite{Dav1} for example. It is proved in \emph{loc. cit.} that there is a natural isomorphism
\begin{equation}\label{dimredeqn}
	 H_\bullet( \mf M_\dd(Q,W),\varphi_W ) \xrightarrow{\cong} H_\bullet(\mf M_\dd(\tilde Q,R) )  \ ,
\end{equation}
between the homology of the stack of representations of the quiver with potential $(Q,W)$, with coefficients in the sheaf of vanishing cycles determined by $W$, and the ordinary Borel-Moore homology of the corresponding dimensionally reduced quiver with relations $(\tilde Q,R)$. In the case at hand, this implies we have isomorphisms
\[  \bb V_{\C^2,n} = H_\bullet( \mf M_n^{\f,\zeta}(\C^3,\mc O_{\C^2}[1]), \varphi_{W^\f}) \xrightarrow{\cong } H_\bullet(\Hilb_n(\C^2)) \ ,  \]
so that we can identify
\[ \bb V_{\C^2} = \bigoplus_{n\in \bb N} H_\bullet(\Hilb_n(\C^2)) \quad\quad \text{and}\quad\quad \mc Z^\VW_{\C^2}(q) = \prod_{k=1}^\infty \frac{1}{1-q^k} = q^{\frac{1}{24}} \eta(q)^{-1} \ ,\]
where the latter can again be computed directly by a fixed point counting argument.

In fact, the preceding well known formula for the generating function of the Euler characteristics of Hilbert schemes of points on surfaces admits a more structured algebraic refinement, which was discovered independently by Grojnowski \cite{Groj} and Nakajima \cite{Nak}. It was observed in \emph{loc. cit.} that there exist natural correspondences
\begin{equation}\label{Nakcorreqn}
	 \vcenter{\xymatrix{ & \Hilb_{n,n+k} \ar[dr]^p \ar[dl]_q & \\ \Hilb_n && \Hilb_{n+k} }} \quad\quad \text{inducing}\quad\quad \alpha_{-k}^n=p_*\circ q^*: H_\bullet(\Hilb_n(\C^2)) \to H_\bullet(\Hilb_{n+k}(\C^2)) \ ,
\end{equation}
for $k\in \bb Z$ and $n\in \bb N$ in the compatible range. Taking a sum over $n\in \bb N$, we obtain operators
\begin{equation}\label{Nakopeqn}
	 \alpha_k = \sum_{n\in \bb N} \alpha_k^n : \bb V_{\C^2} \to \bb V_{\C^2}  \ ,
\end{equation}
for each $k\in \bb Z$ which satisfy the relations implicit in the following theorem:
\begin{theo}\label{Nakthm} \cite{Groj,Nak} There exists a natural representation
	\[ \mc U(\pi) \to \End( \bb V_{\C^2}) \quad\quad \text{defined by} \quad\quad b_k \mapsto \alpha_k   \ ,\]
of the algebra of modes $\mc U(\pi)$ of the Heisenberg vertex algebra $\pi$ on the vector space $\V_{\C^2}$, such that $\V_{\C^2}$ is identified with the vacuum module $\pi_0$ for the Heisenberg vertex algebra.
\end{theo}

\noindent In particular, note that as a corollary of the latter part of the statement, we have an identification
\begin{equation}\label{VWC2introeqn}
P_q(\pi) = \mc Z^\VW_{\C^2}(q) 
\end{equation}
of the Poincare polynomial of the vacuum module of the vertex algebra $\pi$ and the rank 1 local Vafa-Witten invariant of $\C^2$, which immediately implies the above computation of the latter.

The preceding theorem and the resulting identification of Equation \ref{VWC2introeqn} are the prototype of our desired results relating the enumerative geometry of sheaves on algebraic surfaces to the representation theory of vertex algebras, and the goal of this paper is to extend these to more general divisors $S$ in Calabi-Yau threefolds $Y$ corresponding to more interesting vertex algebras, and provide geometric constructions of certain classes of modules for these algebras. 

It will be important in the generalization to replace all of the above constructions by their equivariant analogues with respect to the action of the subtorus $\tilde T$ of the toric structure torus $T$ of the Calabi-Yau threefold $Y$ which preserves the Calabi Yau structure, as well as a subtorus $T_\f$ of the group of symmetries of the framing structure $\f$, the product of which we will denote by $A=\tilde T\times T_\f$. In particular, we consider the vector spaces $\bb V$ over the field of fractions $F$ of $H_A^\bullet(\pt)$, and similarly with all of the linear algebraic constructions above. The analogous statement of Theorem \ref{Nakthm} above gives an action of the family of vertex algebras over $F$ given by the Heisenberg algebra at level
\begin{equation}\label{heisleveleqn}
	 \kappa=-\frac{1}{\hbar_i\hbar_j} \quad\quad\text{where}\quad\quad H_{\tilde T}^\bullet(\pt)=\bb K[\hbar_1,\hbar_2,\hbar_3]/(\hbar_1+\hbar_2+\hbar_3) \ , 
\end{equation}
and $i$ and $j$ are determined by the choice of toric divisor $\C^2$ in $\C^3$.

In a sense, the first such generalization of this construction actually predated the above Theorem, and was also discovered by Nakajima in the seminal paper \cite{Nak1}: among many other influential results, this paper constructs an action of the universal enveloping algebra of a Kac-Moody Lie algebra $\mc U(\gh)$, which is indeed the algebra of modes of the closely related vertex algebra $V_\kappa(\g)$, on the homology of the moduli space of instantons on the resolved du Val singularity of the corresponding ADE type; for concreteness we will restrict our attention to type $A$ in the following. These results were another key step in the family of generalizations we propose in the present work, and in a sense we can understand the preceding example as simply the $\gl_1$ analogue of this more general construction, but this example is also slightly degenerate in the following sense: the action of the torus $\tilde T$ on $\widetilde{A_{m-1}}$ viewed as a toric divisor in $\widetilde{A_{m-1}}\times \bb A^1$ factors through a one dimensional quotient, so there is not a parameterization of the generic level as in Equation \ref{heisleveleqn}. In fact, the level is determined by rank of the instantons, which otherwise does not modify the resulting algebra.

Another construction of an action of $\mc U(V_\kappa(\g))$ was given by Braverman in \cite{Brav1}, on a moduli space of sheaves on $\bb P^1\times \bb P^1$ with structure group corresponding to the ADE type, a reduction of structure to a Borel along $\{0\}\times \bb P^1$, and trivialization along $(\{\infty\}\times \bb P^1) \cup (\bb P^1\times \{\infty\})$. Although the level of the action was not explicitly computed in \emph{loc. cit.}, it was implicitly determined by the other results therein, and given by
\begin{equation}\label{agtleveleqn}
 \kappa = -h^\vee -\frac{\hbar_2}{\hbar_1}    \ . 
\end{equation}
This result would later be reinterpreted as a special case of a generalization of the AGT conjecture to subprincipal $W$ algebras, and we will give a geometric approach to the more general conjecture which can be understood as a reformulation and generalization of the results of \emph{loc. cit.}.

A fundamental breakthrough in the study of analogous vertex algebra actions on moduli spaces of sheaves of higher rank came from physics in the paper \cite{AGT} mentioned above. Among other observations which have not yet been as carefully codified into mathematics, the authors conjectured that the vertex algebra which acts analogously on the $\tilde T$-equivariant cohomology of the moduli space of instantons of rank $r$ on $\C^2_{xy}$ is given the principal affine $W$ algebra $W^\kappa_{f_{\textup{prin}}}(\gl_r)$, a quantum Hamiltonian reduction of the affine algebra $V_\kappa(\gl_r)$, at level $\kappa$ as in Equation \ref{agtleveleqn}, which we will refer to as the AGT conjecture.

The AGT conjecture as stated in type A was proved independently by Schiffmann-Vasserot \cite{SV} and Maulik-Okounkov \cite{MO} using quite different constructions. This paper follows the approach of \cite{SV}, but as we will later speculate it is possible some ingredients of the approach of \cite{MO} will be useful in completing the proof of the general conjecture. Another distinct proof was given for general ADE type by Braverman-Finkelberg-Nakajima in \cite{BFN5}. We now state the result in terms of the notions we have introduced above:

\begin{theo}\label{SVAGTthm}\cite{SV,MO,BFN5} There exists a natural representation
	\[ \mc U(W^\kappa_{f_{\textup{prin}} }(\gl_r)) \to \End(\V_{r[\C^2]}^0)  \ ,\]
	such that $\V_{r[\C^2]}^0$ is identified with the universal Verma module $\bb M_r$ for $W^\kappa_{f_{\textup{prin}} }(\gl_r)$.
\end{theo}

In \cite{CCDS}, the authors extended this result to give an analogous geometric construction of the vacuum module as the cohomology of the moduli space of stable representations of the variant of the ADHM construction in Equation \ref{CCDSeqn} with $A_\f$ given by a principal nilpotent, following many of the calculations previously established in \cite{SV}. Note that this is precisely the framing structure determined by the divisor $r[\C^2]$, and thus in the notation we have introduced above, we have:

\begin{theo}\label{CCDSAGTtheo}\cite{CCDS} There exists a natural representation
	\[ \mc U(W^\kappa_{f_{\textup{prin}} }(\gl_r)) \to \End_F(\V_{r[\C^2]})  \ ,\]
	such that $\V_{r[\C^2]}$ is identified with the vacuum module for $W^\kappa_{f_{\textup{prin}} }(\gl_r)$.
\end{theo}

In particular, we obtain the identification
\[ P_q(W^\kappa_{f_{\textup{prin}} }(\gl_r))  = \prod_{j=1}^r \prod_{k=1}^\infty \frac{1}{1-q^{k+j}}=\mc Z^\VW_{r[\C^2]}(q) \ ,\]
where we have used the usual abuse of notation denoting the vacuum module by the underlying vertex algebra $W^\kappa_{f_{\textup{prin}} }(\gl_r)$. These results are also closely related to the cohomological variant of \cite{FFJMM} and similarly the results of \cite{BerFM} in their simpler specialization to case $n=0$. The formulation of Schiffmann-Vasserot is also analogously related to unpublished results of Feigin-Tsymbaliuk as described in \cite{Ts1}, and in turn previous results of Feigin and collaborators in the context of quantum toroidal algebras, as we will explain in more detail below.

The result was also generalized in \cite{RSYZ} to the case that $Y=\C^3$ and $S$ is given by an arbitrary toric divisor $S_{L,M,N}=L[\C^2_{yz}]+M[\C^2_{xy}]+N[\C^2_{xz}]$, again closely following the approach of \cite{SV}, giving an action of the Gaiotto-Rapcak vertex algebras $Y_{L,M,N}$ introduced in \cite{GaiR} on the moduli space of stable representations of the quiver with potential of Equation \ref{Nekeqn}, such that it is identified with the analogue of the Verma module for this algebra.

The proposal that there should be analogous vertex algebras corresponding to more general divisors $S$ and threefolds $Y$ was considered already in the original paper of Gaiotto-Rapcak \cite{GaiR}, and explored in some detail in \cite{PrR}, but the analogue of the AGT conjecture in this setting was not known in general, as the relevant descriptions of moduli spaces generalizing the spiked instantons construction of Nekrasov to threefolds $Y$ other than $\C^3$ had not been constructed previously.

There is also a closely related conjecture of Feigin-Gukov \cite{FeiG} that there should exist vertex operator algebras $\text{VOA}[M_4,\gl_r]$ associated to four manifolds $M_4$, analogously generalizing the AGT conjecture. This appears to coincide with the predictions of \cite{GaiR} discussed in the preceding paragraph in the case that the underlying reduced scheme $S^\red$ is irreducible and smooth, $M_4$ is the analytification of $S^\red$, and $r\in \bb N$ the multiplicity of $S^\red$ in $S$.

However, a mathematical definition of these vertex algebras was also not known in general, for a divisor $S$ nor four-manifold $M_4$, and relatively few examples were known for $r \geq 2$. In the companion paper \cite{Bu}, we give a general combinatorial construction of vertex algebras $\V(Y,S)$ as the kernel of screening operators acting on lattice vertex algebras determined by the data of the GKM graph of $Y$ and a Jordan-Holder filtration of $\mc O_S$ with subquotients structure sheaves $\mc O_{S_d}$ of the divisors $S_d$ occuring as irreducible, reduced components of $S$. This construction reproduces many interesting vertex algebras, conjecturally including $\mc W$ superalgebras $\mc W_{f_0,f_1}(\gl_{m|n})$ and genus zero class S chiral algebras $\V(\bb P^1,\gl_m;f_1,...,f_k)$ with $k\leq 2$ marked points, each in type A and for general nilpotents $f_i$, and appears to satisfy the predictions of \cite{GaiR}, \cite{PrR}, and \cite{FeiG}. In particular, we formulate the following analogue of the AGT conjecture in this setting:

\begin{conj}\label{VWintroconj} There exists a natural representation
\begin{equation}\label{repstrmapeqn}
		\rho: \mc U(\V(Y,S)) \to \End_F(\V_S )
\end{equation}
	of the algebra of modes $\mc U(\V(Y,S))$ of the vertex algebra $\V(Y,S)$ on the vector space $\V_S$, such that $\V_S$ is identified with the vacuum module $\bb V_{(Y,S)}$ for the vertex algebra $\V(Y,S)$.
\end{conj}

\noindent As a consequence of this conjecture, we expect equality between the local Vafa-Witten series $\mc Z^\VW_S(q)$ of the divisor $S$ defined above and the Poincare series $P_q$ of the vacuum module $\V_{(Y,S)}$, that is
\begin{equation}\label{VWintroeqn}
	 \mc Z^\VW_S(q) =  P_q(\V_{(Y,S)})  \ \in \bb Z \lP q \rP  \ .
\end{equation}

We have not been able to prove the preceding conjecture at the time of writing. However, we do provide a construction of the conjectural representation structure map of Equation \ref{repstrmapeqn} as a special case of Theorem \ref{Bthm} below, and we can verify directly that this satisfies the conjecture in the case that $S$ is a smooth subvariety of $Y$ and the corresponding moduli space parameterizes sheaves of rank $1$ on $S=S^\red$, as we discuss further below. It is not possible to directly verify that the resulting endomorphisms satisfy the relations of the algebra of modes $\mc U(\V(Y,S))$ in general, as even in the case of rank $r$ sheaves on $\C^2$ considered in \cite{SV} this is not possible. However, the construction of the vertex algebras $\V(Y,S)$ we give in \cite{Bu}, as the joint kernels of certain collections of screening operators, provides a family of compatible free field realizations of these vertex algebras which are in some sense defined to allow for a proof of Conjecture \ref{VWintroconj} generalizing that of \cite{SV}, as we will explain further below. We hope to complete the proof following this approach in future work.

We now describe the construction of the representation of $\mc U(W^\kappa_{f_{\textup{prin}} }(\gl_r))$ from \cite{SV}, and our generalization of the construction for arbitrary divisors $S$ in threefolds $Y$ of the class considered here. The construction of \emph{loc. cit.} is in essence the same as the original construction of Nakajima defined above, but it proceeds in stages and constructs actions of several auxiliary algebras which are necessary to provide the abstract proof of the desired relations implicit in Theorem \ref{SVAGTthm} above.

From the geometric perspective we have advocated for in Section \ref{agintrosec} above, the correspondences of Nakajima in Equation \ref{Nakcorreqn} above can be understood as adding $k$ copies of the structure sheaf $\mc O_\pt$ of a point in $\C^2$ to an iterated extension of $n$ copies of $\mc O_\pt$ together with a single copy of $\mc O_{\C^2}[1]$, or removing $-k$ copies of $\mc O_\pt$ if $k<0$. In fact, the relations between these correspondences can be understood more primitively in the absence of the auxiliary copy of $\mc O_{\C^2}[1]$.

As we have explained in Section \ref{agintrosec}, the stack of compactly supported coherent sheaves of length $n$ on $\C^2$ is equivalent to that of $n$ dimensional representations of the unframed quiver with relations underlying the ADHM quiver, given by a single vertex $Q_{\C^2}$ with two loops $B_1,B_2$ satisfying $[B_1,B_2]=0$, so that we have
\[\mf M_n(\C^2) = \left[ C_n / \Gl_n  \right]  \quad\quad \text{where}\quad\quad  C_n=\{ (B_1,B_2) \in \gl_n^{\times 2}\ | \ [B_1,B_2]=0 \} \ ,\]
and we denote the corresponding Borel-Moore homology groups by
\[\mc H(\C^2) =  \bigoplus_{n\in \bb N}\mc H_n(\C^2) = \bigoplus_{n\in \bb N} H_\bullet( \mf M_n(\C^2)) \ . \]

There are analogous correspondences between the spaces $\mf M_n(\C^2)$, defined by the stacks of short exact sequences $\mf M_{k,l}(\C^2)$ of representations of the unframed quiver of dimension $k+l$, with a subobject of dimension $k$ and quotient object of dimension $l$,
\[\vcenter{\xymatrix{ & \mf M_{k,l}(\C^2) \ar[dr]^p \ar[dl]_q  & \\ \mf M_{k}(\C^2) \times \mf M_{l}(\C^2)  && \mf M_{k+l}(\C^2) }}\]
which induce maps
\[ p_*\circ q^*: \mc H_k(\C^2) \otimes \mc H_l(\C^2) \to \mc H_{k+l}(\C^2)  \quad\quad \text{and thus}\quad\quad m: \mc H(\C^2)^{\otimes 2} \to \mc H(\C^2) \ ,  \]
which defines an associative algebra structure on $\mc H(\C^2)$; the resulting algebra $\mc H(\C^2)$ is called the \emph{preprojective cohomological Hall algebra} of $\C^2$.

Further, the stacks of short exact sequences of iterated extensions of $n+k$ copies of the structure sheaf of a point and a single copy of $\mc O_{\C^2}[1]$, with subobject given by an iterated extension of $k$ copies of $\mc O_\pt$, define analogous correspondences
\begin{equation}\label{cohacoreqn}
	\vcenter{\xymatrix{ & \Hilb_{k;n}(\C^2) \ar[dr]^p \ar[dl]_q  & \\ \mf M_{k}(\C^2) \times \Hilb_n(\C^2)  && \mf \Hilb_{n+k}(\C^2) }}
\end{equation}
and induce a representation of $\mc H(Y)$ on $\bb V_{\C^2}$, which we denote by
\[ \rho_{\C^2} : \mc H(\C^2) \to \End_F( \V_{\C^2})  \ . \]
From this perspective, the Nakajima operator $\alpha_{-1}:\V_{\C^2} \to \V_{\C^2}$ of Equation \ref{Nakopeqn} is given by the action of the fundamental class of $[\mf M_1(\C^2)] \in \mc H_1(\C^2)$ under this representation. More generally, the image of the spherical subalgebra $\mc{SH}(\C^2)\subset \mc H(\C^2)$ generated by $\mc H_1(\C^2)$ includes the Nakajima operators $\alpha_k$ for all $k> 0$, and we obtain an isomorphism
\[ \mc U(\pi)_+  \xrightarrow{\cong} \rho_{\C^2}( \mc{SH}(\C^2))  \ , \]
between the positive half $\mc U(\pi)_+$ of the algebra of modes $\mc U(\pi)$ of the Heisenberg vertex algebra $\pi$, and the image of the spherical subalgebra $\mc{SH}(\C^2)$ under the above representation $\rho_{\C^2}$.

In fact, the analogous results for higher rank sheaves on $\C^2$ were proved in \cite{SV}, giving a representation
\[ \rho_{r[\C^2]}^0 : \mc{H}(\C^2) \to \End_F(\V_{r[\C^2]}^0) \quad\quad \text{inducing}\quad\quad \mc U(W^\kappa_{f_{\textup{prin}} }(\gl_r))_+ \xrightarrow{\cong} \rho_{r[\C^2]}^0( \mc{SH}(\C^2))   \ .\]
Similarly, the action of the opposite algebra $\mc{SH}(\C^2)^\op$ by adjoints with respect to the equivariant integration pairing on $\bb V_{r[\C^2]}^0$ determines the action of the negative half of the algebra, and this provides the construction of the representation structure map in Theorem \ref{SVAGTthm} above.

This approach was also generalized in \cite{RSYZ} to the case of an arbitrary toric divisor $S_{L,M,N}$ in $Y=\C^3$, as mentioned above. The authors consider the analogous cohomological Hall algebra constructed from the homology of the stack of representations of the corresponding unframed quiver with potential, or equivalently the stack of compactly supported coherent sheaves on $\C^3$, with coefficients in the sheaf of vanishing cycles
\[ \mc H(\C^3) = \bigoplus_{n\in\bb N} \mc H_n(\C^3) \quad\quad \mc H_n(\C^3)= H_\bullet(\mf M_n(\C^3),\varphi_{W_\C^3}) \ , \] which was introduced in general by Kontsevich-Soibelman in \cite{KS1}, where it was proved that the analogous correspondences always define an associated algebra structure in this setting. Moreover, it was proved in \cite{RSYZ} that the analogous correspondences define a representation of $\mc H(\C^3)$ on the analogous homology groups $\V_{S_{L,M,N}}$ of the moduli spaces of stable representations of the framed quiver with potential of Equation \ref{Nekeqn} with coefficients in the sheaf of vanishing cycles, and that the analogous extension to the opposite algebra by adjoints determines a representation of the algebra of modes $\mc U(Y_{L,M,N})$ of the vertex algebra $Y_{L,M,N}$ introduced in \cite{GaiR}. In particular, in the case that $L=N=0$, this construction reduces exactly to the construction of \cite{SV} explained above: the dimensional reduction equivalence of Equation \ref{dimredeqn} extends to an identification of modules under an analogous isomorphism of associative algebras $\mc H(\C^3)\cong \mc H(\C^2)$, which follows from the results of \cite{YZ1} and the appendix to \cite{RenS}.

In general, for a Calabi-Yau threefold resolution $Y\to X$ satisfying the hypotheses outlined above, the construction of \cite{KS1} again provides an associative algebra structure on the homology of the stack of representations of the corresponding unframed quiver with potential, or equivalently the stack of compactly supported perverse coherent sheaves in the sense of Bridgeland \cite{Bdg1}, with coefficients in the sheaf of vanishing cycles determined by the potential. We denote the resulting algebra by $\mc H(Y)$ and call it the Kontsevich-Soibelman cohomological Hall algebra of $Y$.

For an object $M\in \DD^b\Coh(Y)^T$ satisfying the hypotheses required to define the stack $\mf M^\f(Y,M)$ of $\f$-framed perverse coherent extensions of $M$, and a compatible stability condition $\zeta$, recall from Equation \ref{moduleintroeqn} that we define the associated cohomological invariant by
\[ \V^{\f,\zeta}(M) = \bigoplus_{\dd \in \bb N^{Q_Y}} H_\bullet(\mf M^{\f,\zeta}_\dd(Y,M), \varphi_{W_M^\f}) \ . \]

We prove the following general result, following the approach of Soibelman from \cite{Soi}:

\begin{manualtheorem}{B}[\ref{cohareptheo}]\label{Bthm}
	There exists a natural representation
	\[ \rho_M: \mc H(Y) \to \End( \V^{\f,\zeta}(M) )  \]
	of the Kontsevich-Soibelman cohomological Hall algebra $\mc H(Y)$ on $\V^{\f,\zeta}(M) $.
\end{manualtheorem}

In particular, in the case that $M=\mc O_{S^\red}^\sss[1]$ with framing structure $\f$ of rank $\rr_S$ given by $\f_S$ or $0_S$, as defined above, we obtain:

\begin{corollary}\label{Nakopcoro} There exist natural representations
	\[ \rho_S: \mc H(Y) \to \End_F( \V_S) \quad\quad \text{and}\quad\quad \rho_S^0:\mc H(Y) \to \End_F(\V_S^0)  \ . \]
\end{corollary}

Moreover, this representation can be restricted to a spherical subalgebra $\mc{SH}(Y)\subset \mc H(Y)$ or extended to a representation of the opposite algebra $\mc{SH}^\op(Y)$ by adjoints, and our expectation is that this provides a construction of the desired representation of Conjecture \ref{VWintroconj}, via an identification
\[ \mc U(\V(Y,S))_+ \xrightarrow{\cong}  \rho_S(\mc{SH}(Y)) \ . \]
In particular, in the case $S=r[\C^2]$ the representation $\rho_{r[\C^2]}^0$ is that used in the proof of the AGT conjecture in \cite{SV}, and the variant $\rho_{r[\C^2]}$ is that used in \cite{CCDS}, as described above.

We emphasize one key feature of our construction which differs significantly from previous constructions of vertex algebras corresponding to general algebraic surfaces $S$ in the spirit of the AGT conjecture: the moduli spaces $\mf M^{\f,\zeta}(Y,\mc O_{S^\red}^\sss[1])$ typically parameterize spaces of sheaves on surfaces with components corresponding to a lattice of possible values of the first Chern class $c_1(\mc E)$ of the sheaves, in addition to subcomponents corresponding to the possible values of $c_2(\mc E)$ usually considered. Moreover, the correspondences used in the construction of the representation in Theorem \ref{Bthm} are interpreted geometrically as modifying the sheaves along curve classes within the divisor in a way that relates components of the moduli space with distict values of $c_1(\mc E)$.

Similarly, the construction of the vertex algebras $\V(Y,S)$ in the companion paper \cite{Bu}, in the case that $S$ is a smooth subvariety of $Y$ and the corresponding moduli space is of sheaves of rank $1$ on $S=S^\red$, gives a lattice vertex algebra $V_{H_2(S;\Z)}$ tensored with the Heisenberg algebra $\pi_{H_0(S)}$, where the components of the lattice extension correspond to components of the moduli space with distinct possible values of $c_1(\mc E)$, in contrast to \cite{Nak} which constructs the Heisenberg subalgebra $\pi_{H_\bullet(S)}$. In fact, a similar construction was outlined in the final chapter of \cite{NakLec} in the rank 1 case, and our construction of $\V(Y,S)$ can be understood as providing a higher rank generalization of the proposal of \emph{loc. cit.}. As we mentioned following Conjecture \ref{VWintroconj} above, we can verify the conjecture by direct computation in the rank 1 case, as the relations in the algebra of modes of the lattice vertex algebra are given by explicit formulas; a proof of this will appear in future work.

In higher rank, the construction of $\V(Y,S)$ is again somewhat manifestly as a lattice type vertex algebra extension, conjecturally corresponding to Hecke modifications of higher rank sheaves along curve classes, of the $W$-algebras corresponding to modifications of the sheaves at points of the surface. This is in contrast to the construction of \cite{Neg1}, which induces an action on a moduli space of sheaves with fixed first Chern class.

Finally, we discuss the role of the affine Yangian $\mc Y(\glh_1)$ of $\gl_1$ in the proof of the AGT conjecture, which occurs in both \cite{SV} and \cite{MO}, and its putative generalization in the present context. As we have explained, the action of the positive and negative halves of the algebra of modes $\mc U(W^\kappa_{f_{\textup{prin}} }(\gl_r))$ on the module $\V^0_{r[\C^2]}$ are given by the image of $\mc{SH}(\C^3)$ and $\mc{SH}(\C^3)^\op$ under the representation $\rho_{r[\C^2]}^0:\mc H(\C^3) \to \End_F(\bb V_{r[\C^2]})$ above, for each $r\in \bb N$. It is proved in \cite{SV} that these extend to the action of an algebra which was later identified with the affine Yangian $\mc Y(\glh_1)$ of $\gl_1$. This algebra admits a triangular decomposition
\[ \mc Y(\glh_1) =\mc Y(\glh_1)_- \otimes \mc Y(\glh_1)_0 \otimes \mc Y(\glh_1)_+ \quad\quad \text{with}\quad\quad \mc Y(\glh_1)_+ \cong \mc{SH}(\C^3) \quad\quad \mc Y(\glh_1)_-\cong \mc{SH}(\C^3)^\op  \ , \]
such that the representations of $\mc{SH}(Y)$ and $\mc{SH}(Y)^\op$ above extend to define a representation
\[ \rho_{r[\C^2]}: \mc Y(\glh_1) \to \End_F(\V_{r[\C^2]})  \]
for each $r\in \bb N$, inducing a surjection
\[ \mc Y(\glh_1) \to \mc U(W^\kappa_{f_{\textup{prin}} }(\gl_r)) \ ,\]
up to issues of completions which will we will continue to ignore in the remaining discussion below.

This additional observation played an important role in the proof of the AGT conjecture in \cite{SV}, following a crucial suggestion of Nakajima: The basic idea is that the Verma modules $\bb V^0_{r[\C^2]}$ under consideration satisfy the factorization property $\bb V^0_{r[\C^2]} \cong \V_{\C^2}^{\otimes r}$ and there exists a compatible coproduct $\Delta:\mc Y(\glh_1) \to \mc Y(\glh_1)^{\otimes 2}$ in the sense that the following diagram commutes
\[ \vcenter{\xymatrixcolsep{7pc} \xymatrix{\mc Y(\glh_1)  \ar[r]^{\Delta} \ar[d]^{\rho_{r[\C^2]}} &  \mc Y(\glh_1)^{\otimes 2} \ar[d]^{\rho_{r_1[\C^2]}^0\otimes \rho_{r_2[\C^2]}^0}  \\  \End(\V^0_{r[\C^2]}) \ar[r]^{\cong} & \End(\V_{r_1[\C^2]}^0)\otimes  \End(\V_{r_2[\C^2]}^0)    }} \ .\]
Moreover, the image of the induced embedding
\[ W^\kappa_{f_{\textup{prin}} }(\gl_r) \to W^{\kappa_1}_{f_{\textup{prin}} }(\gl_{r_1})\otimes \mc W^{\kappa_2}_{f_{\textup{prin}} }(\gl_{r_2}) \]
is characterized by the kernel of a screening operator acting on the latter tensor product of vertex algebras. The proof of this fact uses the compatible structure of the Feigin-Frenkel resolutions of \cite{FF1} to reduce to the case $r_1=r_2=1$, for which the calculation can be directly checked, and this implies the action factors through $ \mc U(W^\kappa_{f_{\textup{prin}} }(\gl_r))$ in general; this is essentially the proof of Theorem \ref{SVAGTthm} from \cite{SV}. A physical explanation of this mechanism was also explained in \cite{GaiR2}.

In the case of a general threefold $Y$, it is natural to expect there exists a universal algebra $\mc Y(Y)$ admitting an analogous triangular decomposition
\[ \mc Y(Y) =\mc Y(Y)_- \otimes \mc Y(Y)_0 \otimes \mc Y(Y)_+ \quad\quad \text{with}\quad\quad \mc Y(Y)_+ \cong \mc{SH}(Y) \quad\quad \mc Y(Y)_-\cong \mc{SH}(Y)^\op  \ , \]
such that there exist representations of $\mc Y(Y)$ on the vector spaces $\V_S$ factoring through those of Conjecture \ref{VWintroconj}. Indeed, from the construction of the seminal paper of \cite{Nak2} and its cohomological variant \cite{Var}, together with the dimensional reduction isomorphisms between critical and preprojective cohomological Hall algebras of the type proved in \cite{YZ1} and the appendix to \cite{RenS}, it is natural to expect that in the case $Y=Y_{m,0}=\widetilde{A_{m-1}}\times \C$ the corresponding algebra is given by the affine Yangian $\mc Y(\glh_m)$ of $\gl_m$, in the sense defined for $\spl_m$ in \cite{Guay}.

In fact, based on considerations from string theory which we will discuss in Section \ref{stintrosec} below, as well as the many existing results in mathematics we have mentioned, Costello \cite{CosMSRI} proposed a family of conjectures about geometric actions of affine Yangian type quantum groups in this setting. One observation of \emph{loc. cit.} was that while the positive and negative parts $\mc Y(Y)_\pm$ of the putative algebra are determined by the universal algebra $\mc{SH}(Y)$, the structure of the algebra $\mc Y_M(Y)$ and in particular its Cartan subalgebra $\mc Y_M(Y)_0$ should, as the notation suggests, depend on the choice of object $M\in \DD^b\Coh(Y)^T$, as well as the framing structure $\f$, that determines the homology group $\bb V^{\f,\zeta}(M)$ on which we would like to construct the representation
\begin{equation}\label{Yangacteqn}
	\rho_M:\mc Y_M(Y) \to \End_F(\bb V^{\f,\zeta}(M)) \ . 
\end{equation}
In particular, it was conjectured in \emph{loc. cit.} that these variants correspond to shifted variants of affine Yangian type quantum groups, in the sense of \cite{BrKl} and \cite{FKPRW}. 

In addition to the main result quoted in Theorem \ref{Bthm} above, we outline a construction of a representation of this larger algebra, following the approach of \cite{RSYZ} and in turn \cite{SV}. We also attempt to explain the resulting computations in language closely paralleling that introduced in the string theory papers of Li-Yamazaki \cite{GL1} and Galakov-Li-Yamazaki \cite{GLY}, which appear to have provided computations that establish some of the conjectures stated here, and we discuss these relationships further in Section \ref{stintrosec} below. We hope that this paper will help to translate the results of \emph{loc. cit.} into geometric representation theory.

Let us conclude this section of the introduction by mentioning the compatibility of the preceding approach with the definition of the vertex algebras $\V(Y,S)$ given in the companion paper \cite{Bu}, towards the proof of Conjecture \ref{VWintroconj}, as alluded to above. In the setting of Corollary \ref{Nakopcoro}, the conjectural extension of the representation $\rho_S^0$ is expected to give a surjection
\begin{equation}\label{MVmapeqn}
	 \mc Y_S(Y) \to \mc U (\V(Y,S))  \ , 
\end{equation}
induced by the representations of these algebras on the common vector space $\V_S^0$. Further, we have
\[ \bb V_S^0 \cong \bigotimes_{d\in \mf D_S} \bb V_{S^d}^{\otimes r_d} \quad\quad \text{so that}\quad\quad \bb V_S^0 \cong \V_{S_1}^0\otimes \V_{S_2}^0 \ , \]
if $S=S_1+S_2$. Moreover, if we pick a composition series of $\mc O_{S}$, $\mc O_{S_1}$ and $\mc O_{S_2}$ in terms of the structure sheaves $\mc O_{S_d}$ of the irreducible components $S_d$ of $S^\red$, compatible in the sense that
\[ \mc O_{S_1} \to \mc O_S \to  \mc O_{S_2}   \  , \]
defines a short exact sequence compatible with the filtrations, then we obtain
\begin{equation}\label{voaembeqn}
	 \bb V(Y,S) \to \V(Y,S_1)\otimes \V(Y,S_2)  
\end{equation}
an embedding of vertex algebras, with image characterized by the kernel of a screening operator; this is essentially by definition in the construction of the companion paper \cite{Bu}. Thus, it is natural to speculate that there exist coproduct maps making the following diagram commute
\[ \vcenter{\xymatrixcolsep{7pc} \xymatrix{\mc Y_S(Y)  \ar[r]^{\Delta_{S_1,S_2}} \ar[d]^{\rho_{S}^0} &  \mc Y_{S_1}(Y)\otimes \mc Y_{S_2}(Y)  \ar[d]^{\rho_{S_1}^0\otimes \rho_{S_2}^0}  \\  \End(\V^0_{S}) \ar[r]^{\cong} & \End(\V_{S_1}^0)\otimes  \End(\V_{S_2}^0)    }} \ ,\]
and such that the induced embedding of vertex algebras is given by that of Equation \ref{voaembeqn} above.

These conjectural coproducts appear to be the affine analogues of the coproducts constructed in \cite{BrKl} and \cite{FKPRW}, and as we will explain in the final Section \ref{MVsec}, the maps of the form of Equation \ref{MVmapeqn} above for an appropriate choice of divisor in a threefold $Y_{m,n}$ resolving the singularity $X_{m,n}=\{xy-z^mw^n\}$ are expected to give a geometric construction of Brundan-Kleshchev type isomorphisms between affine $W$-superalgebras in type A and quotients of shifted Yangians for the affine superalgebras $\glh_{m|n}$. We hope to give a geometric construction of the coproducts and complete the proof of Conjecture \ref{VWintroconj} following this approach in future work.

\subsection{Motivation from string theory}\label{stintrosec} The broader context for the results of this paper are a family of interconnected mathematical ideas at the intersections of enumerative algebraic geometry, low dimensional topology, geometric representation theory, and integrable systems, which follow predictions from supersymmetric quantum field theory and string theory. In particular, the conjectures of Alday-Gaiotto-Tachikawa from \cite{AGT}, their generalizations, and adaptations thereof in mathematics, are a central example of these ideas and are the primary motivation for the present paper, as we have explained in the preceding section.

In this section, we will recall the well-known physical arguments behind the AGT conjecture, as well as its generalization in \cite{GaiR}, \cite{PrR}, and \cite{FeiG}, and give a cartoon of the physics argument to explain its connections with cohomological Hall algebras and affine Yangian type quantum groups, following the series of paper \cite{Cos1}, \cite{Cos2}, \cite{Cos3}, as well as \cite{CWY}, \cite{LiYam}, and \cite{GLY}. We will mention some further references which have influenced our thinking, but make no attempt to give systematic attributions for the physics arguments below, which we emphasize are only a cartoon to motivate the mathematical results of this paper, as well as the many further conjectures we state and various expected connections between them.

The contemporary perspective on the AGT conjecture begins with the existence of a family of special quantum field theories in six real spacetime dimensions, admitting an ADE classification, and defined by having symmetries governed by the six dimensional $\mc N=(2,0)$ superconformal algebra. The definition of these theories is poorly understood even by the standards of theoretical physics, but it is widely accepted that they exist based on an impressive web of compatible connections between better understood supersymmetric quantum field theories in lower dimensions, and corresponding relationships with topics in the pure mathematics canon from homological knot invariants \cite{ALR} to the quantum geometric Langlands correspondence \cite{AFO}. Indeed, the vertex algebras $\V(Y,S)$ considered here, defined in the companion paper \cite{Bu}, are closely related to both these subjects.

The six dimensional $\mc N=(2,0)$ superconformal field theories admit a holomorphic-topological twist, a certain deformation of the quantum field theory defined on six dimensional manifolds that are the product of a smooth algebraic curve $\Sigma$ and a four manifold $M_4$. This deformation depends only on the underlying holomorphic and topological structures of $\Sigma$ and $M_4$, respectively, so that any quantities computed in these theories can be understood as invariants of these objects.

The AGT conjecture is an example of a general pattern of mathematical predictions arising in quantum field theory from the operation of \emph{dimensional reduction}, a kind of integration or pushforward of quantum field theories along fibrations, which uses the quantum field theory on the total space to encode the geometry of the fibres into a quantum field theory on the base. The total integral or pushforward to a point of a quantum field theory determines concrete mathematical objects such as numbers or vector spaces, which define invariants of the mathematical input data as above, and this operation satisfies an analogue of Fubini's theorem that canonically identifies invariants computed in two distinct ways, by reducing from a product of spaces to the point in different orders. This often leads to identifications of pairs of numbers or vector spaces, as well as various related algebraic or categorical structures, of completely distinct mathematical origins.

In the setting of the six dimensional $\mc N=(2,0)$ theory of type $\g$ considered above, this implies a sort of commutativity of the following schematic diagram, which we explain below:
\begin{equation}\label{AGTeqn}
	 \vcenter{\xymatrixcolsep{5pc}\xymatrix{ \textup{6d $\mc N=(2,0)$ type $\g$ on $\Sigma\times M_4$}  \ar@{~>}[r]^{\int_{\Sigma}} \ar@{~>}[d]^{\int_{M_4}}  & \textup{4d $\mc N=2$ class $\mc S[\Sigma,\g]$ on $M_4$} \ar@{~>}[d]^{\int_{M_4}}  \\
\textup{2d $\mc N=(0,2)$ class $\mc T[M_4,\g]$ on $\Sigma$} \ar@{~>}[r]^{\int_{\Sigma}}  & \textup{Invariants of $\g$, $\Sigma$, and $M_4$} }} \ .
\end{equation}

The top horizontal arrow of the diagram in Equation \ref{AGTeqn} indicates dimensional reduction along $\Sigma$, which takes the six dimensional $\mc N=(2,0)$ superconformal field theory of type $\g$ we have considered above and produces a four dimensional $\mc N=2$ superconformal field theory. This was discovered in the seminal paper of Gaiotto \cite{GaiS}, which already explained many of the key structural features of the correspondence, including identifying the space of superconformal deformations of the four dimensional theory with the Deligne-Mumford compactification of the moduli space of curves, so that the action of the mapping class group determines dualities among these theories.

For the holomorphic-topological twist of the six dimensional theory considered above, the reduction to four dimensions induces the topological twist of the corresponding four dimensional $\mc N=2$ theories generalizing the one considered in the case of pure gauge theory by Witten in \cite{WitTQFT} to explain the Donaldson invariants \cite{Don}. Thus, we expect that the invariants of $\g$, $\Sigma$ and $M_4$ computed from this perspective should include a family of decorated Donaldson-type invariants of the four manifold $M_4$, determined by a choice of curve $\Sigma$ and ADE type $\g$.

The left vertical arrow of the diagram in Equation \ref{AGTeqn} indicates dimensional reduction along $M_4$, and produces a two dimensional $\mc N=(0,2)$ supersymmetric quantum field theory $T[M_4,\g]$ on the Riemann surface $\Sigma$, as was studied with similar applications in \cite{GGP} and \cite{DGP}. The holomorphic-topological twist of the six dimensional theory induces a holomorphic twist in two dimensions, so that the algebra of observables of this theory is expected to define a vertex operator algebra, denoted by $\textup{VOA}[M_4,\g]$ in \cite{FeiG}. In this context, the natural invariant associated to an algebraic curve $\Sigma$ by the vertex operator algebra $\textup{VOA}[M_4,\g]$ is the space of conformal blocks on $\Sigma$.

Thus, the commutativity of the diagram in Equation \ref{AGTeqn} can be understood concretely as predicting a relationship between the Donaldson-type invariants of the four manifold $M_4$, determined by a choice of curve $\Sigma$ and ADE type $\g$, and the conformal blocks of the vertex operator algebra $\textup{VOA}[M_4,\g]$ on $\Sigma$. The exact mathematical statement is difficult to formulate in general, and does not appear in the literature to our knowledge, but seems to be of the following general form:

\begin{prpl}\label{metaAGTprpl}
The partition function on $M_4$ of the class $\mc S(\Sigma)$ theory of type $\g$ as the curve $\Sigma$ varies can be identified with a canonical section of the sheaf of conformal blocks of $\textup{VOA}[M_4,\g]$ over $\overline{\mf M}_{g,n}$.
\end{prpl}

Several aspects of this phenomenon were discovered in the special case that $\Sigma=E$ is an elliptic curve in the seminal paper of Vafa-Witten \cite{VW} prior to almost all of the preceding references. In this case, the class $\mc S$ theory for $\Sigma$ identifies with four dimensional $\mc N=4$ superconformal gauge theory of type $\g$, where the parameter on moduli space of elliptic curves determines the complexified gauge coupling $\tau$ of the gauge theory and the dualities corresponding to the mapping class group include the famous $S$-duality autoequivalence of four dimensional $\mc N=4$ gauge theory, which maps $\tau\mapsto -\frac{1}{\tau}$ generalizing the classical Montonen-Olive electromagnetic duality and provides a physical explanation for the quantum geometric Langlands correspondence \cite{KW}.

Thus, Vafa-Witten explained that the partition functions $\mc Z^\VW_{M_4,\g}(q)$ of four dimensional $\mc N=4$ gauge theory in the Donaldson-Witten twist, given by the generating functions for Euler characteristics of components of the moduli spaces of instantons on four manifolds $M_4$, should define modular forms. They also observed that from the perspective of the alternative reduction, these modular forms should correspond to the characters of modules $\bb V_{M_4,\g}$ for the vertex algebra $\textup{VOA}[M_4,\g]$, viewed as sections of the sheaf of conformal blocks over $\overline{\mf M}_{1,1}$. Concretely, they conjectured
\begin{equation}\label{VW4meqn}
	 \mc Z^\VW_{M_4,\g}(q) = P_q(\bb V_{M_4,\g})  \ \in \bb Z\lP q\rP \ ,
\end{equation}
where $P_q$ denotes the Poincare polynomial with respect to the conformal grading. The analogue of this statement for divisors $S$ in Calabi-Yau threefolds $Y$ is precisely the equality of Equation \ref{VWintroeqn}.

Another crucial insight leading to the AGT conjecture came in the context of four dimensional $\mc N=2$ supersymmetric gauge theories. In the seminal papers \cite{SW1} and \cite{SW2}, Seiberg and Witten explained that the low energy physics of $\mc N=2$ theories on the Coulomb branch is controlled by certain complex integrable systems in terms of the \emph{Seiberg-Witten prepotential}, a quantity related to the variation of Hodge structure on the cohomology of the fibres of the integrable system.

In the celebrated paper \cite{Nek1}, Nekrasov showed that the Seiberg-Witten prepotential of a four dimensional $\mc N=2$ theory can be computed in terms of the instanton partition function on $\bb R^4$, defined using equivariance with respect to an $S^1\times S^1$ action. Nekrasov explained that this construction had a physical interpretation in terms of the \emph{$\Omega$-background}, a notion of $S^1$-equivariance for quantum field theories with a localization mechanism analogous to equivariant cohomology, such that certain quantities in $S^1$-equivariant theories can be computed using induced quantum field theories on the $S^1$-fixed points, and the latter appear naturally quantized with respect to $\hbar\in H_{S^1}^\bullet(\pt)$. This insight turned out to feature prominently in fascinating connections between supersymmetric quantum field theory and quantization of integrable systems, such as those of Nekrasov-Shatashvili \cite{NS1}, \cite{NS2}, and Nekrasov-Witten \cite{NW}, closely related to our setting of interest. We note that a mathematical formulation of the $\Omega$-background construction for factorization algebras was given in Chapter 3 of the thesis \cite{BTh} of the first author, following the preprints \cite{Bu1,Bu2}.

In the context of the diagram in Equation \ref{AGTeqn}, this insight implies that it sensible to consider a local model for the four manifold given by $M_4=\bb R^4_{\hbar_1,\hbar_2}$, flat space $\bb R^4$ in the $\Omega$-background induced by the $S^1\times S^1$ action, and moreover the instanton counting invariants computed by the right hand side of the diagram in this case should be related to the quantization of the Seiberg-Witten integrable system. In the class case of class $\mc S$ theories of type $\g$ determined by $\Sigma$, the relevant integrable system is the Hitchin system on $T^\vee\Bun_G(\Sigma)$, the quantization of which plays a central role in the geometric Langlands correspondence, as described in \cite{BD2}.

The key breakthrough of Alday-Gaiotto-Tachikawa in \cite{AGT} was the realization that the vertex algebra that features on the left side of the diagram in Equation \ref{AGTeqn} in this context, which can in retrospect be denoted $\textup{VOA}[\bb R^4_{\hbar_1,\hbar_2},\g]$ in the language of \cite{FeiG} we have followed above, is the closely related principal affine $W$-algebra $W_{\rho_\textup{prin}}^\kappa(\g)$ of type $\g$, that is
\[\textup{VOA}[\bb R^4_{\hbar_1,\hbar_2},\g]=W_{\rho_\textup{prin}}^\kappa(\g) \quad\quad\text{where}\quad\quad \kappa = - h^\vee - \frac{\hbar_2}{\hbar_1}  \ .\]
Moreover, the authors explained the interpretation of the Nekrasov partition functions for class $\mc S$ theories as conformal blocks of this vertex algebra on the corresponding curve, which was the original observation of the general type formulated in Proposal \ref{metaAGTprpl} above.

Shortly afterwards, it was explained by Gaiotto in \cite{Gst} that there is also a sensible variant of the above conjectures which is essentially local on the curve $\Sigma$. The restriction to $\bb C$ or $\bb C^\times$ of the observables of a conformal field theory are identified with a module $\V_{M_4,\g}$ of the vertex algebra $\textup{VOA}[M_4,\g]$ and its associative algebra of modes $\mc U(\textup{VOA}[M_4,\g])$, respectively; these are the Hilbert space and algebra of operators of a quantum mechanical system on $\bb R$ given by reducing along $S^1$ in the formal identification $\bb C^\times= S^1\times \bb R$, as indicated by the bottom horizontal arrow in the following analogue of the diagram in Equation \ref{AGTeqn}:
\begin{equation}\label{mathAGTeqn}
	\vcenter{\xymatrixcolsep{5pc}\xymatrix{ \textup{6d $\mc N=(2,0)$ type $\g$ on $\bb C^\times \times M_4$}  \ar@{~>}[r]^{\int_{S^1}} \ar@{~>}[d]^{\int_{M_4}}  & \textup{5d $\mc N=2$ type $\g$ on $\bb R\times M_4$} \ar@{~>}[d]^{\int_{M_4}}  \\
			\textup{2d $\mc N=(0,2)$ class $\mc T[M_4,\g]$ on $\bb C^\times$} \ar@{~>}[r]^{\int_{S^1}}  & \textup{1d SQM $\mc T[S^1\times M_4,\g]$ on $\bb R$} }} \ .
\end{equation}

The analogous commutativity of the preceding diagram implies that the Hilbert space and algebra of operators of this quantum mechanical system can be computed in terms of the five dimensional $\mc N=2$ supersymmetric gauge theory on $\bb R \times M_4$ induced by the reduction on $S^1$ of the six dimensional $\mc N=(2,0)$ theory of type $\g$. The holomorphic topological twist of the latter induces a five dimensional analogue of the Donaldson-Witten twist, for which the Hilbert space for the theory reduced on a four manifold $M_4$ is given by the homology of the moduli space of instantons on $M_4$, and the algebra of observables is generated by instanton operators, which add or remove charge from instanton particles in the five dimensional gauge theory propagating along the $\bb R$ direction.

The preceding discussion can be formally summarized by the following:

\begin{prpl} There exists a natural representation
	\[ \mc U(\textup{VOA}[M_4,\g]) \to \End ( H_\bullet( \mf M(M_4,\g)) ) \ ,\]
	of the algebra of modes $ \mc U(\textup{VOA}[M_4,\g])$ of the vertex algebra $\textup{VOA}[M_4,\g]$ on the homology $ H_\bullet( \mf M(M_4,\g))$ of the moduli space of $\g$-instantons on $M_4$, such that the latter is identified with the module $\bb V_{M_4,\g}$.
\end{prpl}

This is a schematic generalization of the common mathematical statement of the AGT conjecture, such as that quoted in Theorem \ref{SVAGTthm} above from \cite{SV} which corresponds to the special case that $M_4=\bb R^4$ so that $\textup{VOA}[M_4,\g]=W_{\rho_\textup{prin}}^\kappa(\g)$ and the relevant moduli space of instantons is given by the ADHM construction. In the abelian case, this generalization was essentially discovered by Nakajima, in the sense explained in \cite{NakInst} for example. Even the paper of Vafa-Witten \cite{VW} followed the results of Nakajima \cite{Nak1} in the relevant example.

Unfortunately, even after the inspiring new proposals for this program in \cite{FeiG}, relatively few concrete examples of these vertex algebras were known in the non-abelian case for interesting four manifolds $M_4$. However, an alternative approach in a closely related setting was considered previously by Gaiotto-Rapcak in the seminal paper \cite{GaiR}, which allowed for more direct gauge theory computations of the resulting vertex algebras. The six dimensional $\mc N=(2,0)$ theory for $\gl_r$ is essentially by definition the theory of quantum fluctuations of a stack of $r$ distinct six dimensional extended objects called \emph{M5-branes}, inside of an eleven dimensional spacetime supporting a mysterious variant of string theory called \emph{M-theory}. It was observed in \emph{loc. cit.} that there exist vertex algebras corresponding to two dimensional theories defined on $\Sigma$ as above, constructed by reducing in the $\Omega$-background a configuration of M5 branes supported on a toric divisor $S$ in a toric Calabi-Yau threefold $Y$, rather than a four manifold $M_4$. In summary, we consider
\begin{equation}\label{6dMTeqn}
	 \textup{6d $\mc N=(2,0)$ on $\bb C^\times \times S$}  \quad \subset \quad  \textup{M-theory on $\bb R \times \bb C^\times  \times \bb C \times Y$}  
\end{equation}
and perform the analogous reductions of Equation \ref{mathAGTeqn}. In this setting, we note that rather than being labeled by an ADE type $\g$, the six dimensional $\mc N=(2,0)$ theory occurs in type A, with rank on each reduced, irreducible component of $S$ determined by the multiplicity of this component in $S$. This is the motivation for the vertex algebras $\V(Y,S)$ defined in the companion paper \cite{Bu}, for which the generalization of the AGT conjecture is our Conjecture \ref{VWintroconj} above.

The computations of Gaiotto-Rapcak \cite{GaiR}, and their generalizations in Prochazka-Rapcak \cite{PrR}, used the type IIB duality frame in which the geometry of the threefold $Y$ determines a $(p,q)$-web of type II fivebranes, and the M5-branes wrapping toric divisors $S$ in $Y$ determine D3-branes stretched between the walls of the $(p,q)$-web according to their position on the boundary of the moment polytope of the threefold $Y$. In this setting, the configuration could be analyzed in terms of junctions of boundary conditions and domain walls between the four dimensional $\mc N=4$ gauge theories on the D3 branes, giving concrete predictions for the vertex algebras in terms of constructions in the representation theory of affine Lie algebras, following \cite{GaiW}. This lead to mathematical constructions of some examples of these algebras, such as in \cite{CrG} and in turn \cite{ACF}, but the analogue of the AGT conjecture was not formulated in these examples.

Motivated by closely related ideas in string theory \cite{NekBPS}, Nekrasov discovered the quiver with potential of Equation \ref{Nekeqn} above, which appeared to describe moduli spaces of instantons supported on the toric divisors $S_{L,M,N}$ in $Y=\C^3$, a generalization of the ADHM construction that he called the space of \emph{spiked instantons} \cite{NekP}. Motivated by this, the results of \cite{RSYZ} proved the generalization of the AGT conjecture in this setting, closely following the methods of \cite{SV}, but using the moduli space of spiked instantons in place of the usual moduli space of instantons described by the ADHM construction. These results of \cite{RSYZ} were the prototype for the generalization of the AGT conjecture stated in Conjecture \ref{VWintroconj}, though we note that a description of the moduli space of spiked instantons in terms of algebraic geometry was not provided in \emph{loc. cit.}, nor was the generalization to divisors $S$ in more interesting threefolds $Y$.

Towards stating more refined predictions about this general correspondence, including understanding the ingredients of the proof of the AGT conjecture given in \cite{SV} and in turn \cite{RSYZ}, and how we should expect them to generalize in the case of $\V(Y,S)$, we give a cartoon of the relevant string theory arguments, following a proposal of Costello developed in the series of papers \cite{Cos1}, \cite{Cos2}, and \cite{Cos3}, and described in general in his lecture \cite{CosMSRI}; we also follow the more recent results of \cite{CWY}, \cite{LiYam} and \cite{GLY}. The basic idea is to consider lifts of the compatible reductions of the diagram in Equation \ref{mathAGTeqn}, along the inclusion of the six dimensional $\mc N=(2,0)$ theory into M-theory as in Equation \ref{6dMTeqn}, which are summarized in the following diagram:
\begin{equation}\label{stringAGTeqn}
	\vcenter{\xymatrixcolsep{3pc}\xymatrix{ \textup{Twisted M-theory on $\bb R \times \bb C^\times  \times \bb C \times Y$}  \ar@{~>}[r]^{\int_{S^1}} \ar@{~>}[d]^{\int_{Y}}  & \textup{Twisted type IIA on $\bb R^2 \times \bb C \times Y$} \ar@{~>}[d]^{\int_{Y}}  \\
			\textup{5d Chern-Simons for $\g_Y$ on $\bb R \times \bb C^\times \times \bb C$} \ar@{~>}[r]^{\int_{S^1}}  & \textup{4d Chern-Simons for $\gh_Y$ on $\bb R^2 \times \bb C$} }} \ .
\end{equation}

The field theories in the diagram of Equation \ref{mathAGTeqn} all occur as theories of observables local to a defect supported along a submanifold of the spacetime underlying the corresponding object in the diagram of Equation \ref{stringAGTeqn}. We begin by describing the top horizontal arrow and the resulting expectations for computations in the twisted type IIA setting.

The conjectural geometric description of the vacuum module $\bb V_S$ of the vertex algebra $\V(Y,S)$, in terms of the homology of the corresponding moduli space of instantons supported on the divisor $S$, was derived in the context of the diagram of Equation \ref{mathAGTeqn} as the Hilbert space of the Donaldson-Witten twist of the five dimensional $\mc N=2$ theory on $\bb R \times M_4$. In fact, it was explained in \cite{WitADHM} that the moduli spaces of instantons in gauge theories on stacks of relatively heavy, non-compact branes can be interpreted as moduli spaces of bound states of these fixed branes with the spectrum of infinitesimally massive, compactly supported branes. The reformulation of the ADHM construction explained in Section \ref{agintrosec} and its generalization in Theorem \ref{Athm} is a mathematical realization of this perspective, and provides a systematic approach to the computations of Aspinwall-Katz \cite{AsK}.

Similarly, in this context the representations of the Kontsevich-Soibelman cohomological Hall algebra $\mc H(Y)$ constructed in Theorem \ref{Bthm} can be understood as defining the action of a universal BPS algebra associated to the type IIA string theory on $Y$, induced by correspondences between moduli spaces of bound states of D-branes defined by adjoining additional compactly supported branes, as in the correspondence of Equation \ref{cohacoreqn}. Moreover, the choice of auxiliary brane $M\in \DD^b\Coh(Y)$ conjecturally determines an algebra $\mc Y_M(Y)$ with triangular decomposition
\begin{equation}\label{repeatyangeqn}
	 \mc Y_M(Y) = \mc{SH}(Y)^\op \otimes \mc H^0_M \otimes \mc{SH}(Y)  \quad\quad \text{and}\quad\quad \mc Y_M(Y) \to \End_F(\bb V_M) 
\end{equation}
a representation as in Equation \ref{Yangacteqn}, for which we outline a construction in Section \ref{quiveryangsec} following \cite{RSYZ}. This should be interpreted as a working definition of a mathematical avatar for the BPS algebras introduced by Li-Yamazaki in \cite{LiYam}, and generalized in \cite{GLY}, under the name (shifted) \emph{quiver Yangians}; these algebras were defined to act on the moduli spaces of bound states of D-branes in $Y$, in precisely the way we have explained above, and we hope to provide a more careful statement of the relation to \emph{loc. cit.} in future work.

In the case that the auxiliary branes are a configuration of D4 branes in the twisted IIA frame determined by an M5 brane wrapping the divisor $S$, corresponding to $M=\mc O_{S^\red}^\sss[1]\in\DD^b\Coh(Y)$ and framing structure $\f_S$ of rank $\rr_S$ as described above, we expect that the resulting representation of Equation \ref{repeatyangeqn} factors through a map 
\begin{equation}\label{repeatmveqn}
	 \mc Y_S(Y) \to \mc U(\V(Y,S))  \ ,
\end{equation}
as in Equation \ref{MVmapeqn}, from this universal algebra $\mc Y_S(Y)$ to the algebra of modes $\mc U(\V(Y,S))$.

We now consider the alternative order of reduction, indicated by the left vertical arrow in the diagram of Equation \ref{stringAGTeqn}, which we expect to give an analogous description of these objects in terms of representation theory. A model for this reduction was calculated in \cite{Cos2} in a holomorphic-topological twist of M-theory, and in the presence of an $\Omega$-background acting by the rank two subtorus of the toric action of $T$ on $Y$ which preserves the Calabi-Yau structure. The resulting five dimensional theory was computed to be a non-commutative deformation of a five dimensional cousin of Chern-Simons theory of type $\g=\g_Y$ determined by the threefold $Y$. In the case that $Y=Y_{m,n}$ is a resolution of $X_{m,n}=\{xy-z^mw^n\}$, Costello conjectured that the corresponding Lie algebra was $\g_{Y_{m,n}}=\gl_{m|n}$, and detailed calculations were done in the $\gl_1$ case in \cite{GaiO} and \cite{OZ}.

The intended implications of these results are informed by a general proposal of Costello originating in \cite{Cos1}, which explains relationships between various classes of quantum groups and Chern-Simons type theories in various dimensions. In the classic paper on the physical origin of the Jones polynomial \cite{WitJones}, Witten explains that the category of line operators in three dimensional Chern-Simons theory is identified with representations of the quantum enveloping algebra $U_q\g$. Costello explained that there is an analogous relationship between line operators in a four dimensional holomorphic-topological variant of Chern-Simons theory and representations of the Yangian $Y_\hbar\g$. Moreover, he explained that the quantum groups $U_q\g$ and $Y_\hbar \g$ arising in these theories are the Koszul dual algebras $\mc A|_{\ell}^!$ of the algebras of local operators $\mc A$ of the quantum field theories restricted along a line $\ell$, the locus along which the defects being classified are supported. These are the universal algebras controlling defects with a specified locus of support, in the sense that coupling a defect in the theory supported along the line $\ell$ is equivalent to a map
\begin{equation}\label{kozmapeqn}
	 \mc A|_{\ell}^! \to \mc B 
\end{equation}
to the putative algebra of observables $\mc B$ on the defect, so that taking $\mc B=\End(V)$ determines the correspondence between defects and representations of these algebras. Further, the five dimensional variant of Chern-Simons theory, or equivalently four dimensional Chern-Simons theory in affine type, is meant to describe the observables of a quantum gravitational theory in the twisted context, as in the twisted holographic principle of Costello-Li \cite{CosL}, and thus the Koszul dual algebra is deformed by a twisted analogue of backreaction to an alternate algebra $\widetilde{\mc A}(Y)|_{\ell_M}^!$, which depends on the precise defect supported on $\ell$ corresponding to a configuration of branes determined by $M$. Thus, the commutativity of the diagram of Equation \ref{stringAGTeqn} conjecturally induces an identification
\[ \widetilde{\mc A}(Y)|_{\ell_M}^! \cong \mc Y_M(Y)\]
such that the maps of Equation \ref{kozmapeqn} induce those of Equation \ref{repeatmveqn} and more generally \ref{repeatyangeqn}.

In the case that $Y=Y_{m,n}$ so that $\g_{Y_{m,n}}=\gl_{m|n}$ as described above, this implies a conjectural identification between the infinite dimensional associative algebras $\mc Y_M(Y_{m,n})$ corresponding to various compatible objects $M\in \DD^b\Coh(Y_{m,n})^T$ and various shifts of the affine Yangian for $\gl_{m|n}$. This is the motivation for the Conjectures \ref{Yangactconj} and \ref{Yangsurfconj} in the main text, for example.

In this general context, Costello explained that the additional data on the algebra of observables $\mc A$ forgotten in its restriction to $\ell$, related to its factorization structure in the transverse directions to the support of the defect, determines the generalized commutativity data for the monoidal structure on the category of representations of the Koszul dual quantum group $\mc A|_{\ell}^!$.  We hope that this will lead to a better understanding of the relationship between the crucial coproducts on the affine Yangian of $\gl_1$ that are required for the proof of the AGT conjecture in \cite{SV}, their proposed generalization discussed at the end of Section \ref{stintrosec}, and the more conceptual constructions of the R-matrix \cite{MO}, coproduct \cite{Dav1}, and primitives \cite{DavMein} in closely related contexts.

Finally, we mention that the main applications discussed in Sections \ref{DTsec} and \ref{VWsec}, and especially their relationship outlined in Section \ref{MVsec}, can indeed be understood in the context of the twisted holographic principle of Costello-Li \cite{CosL} mentioned above. In particular, the isomorphism of Conjecture \ref{twholoconj} is of the typical form of their general conjecture.

\subsection{Summary of results}\label{summarysec} We now give a concrete summary of the results of this paper:

In Section \ref{presec}, we review some relevant preliminaries for later reference and in order to fix notation. We recommend the reader skip this section and return to it only later as necessary.

In Section \ref{unextsec}, we explain the unframed analogue of Theorem \ref{Athm}, a folklore theorem stating the equivalence between compactly supported perverse coherent sheaves on certain toric Calabi-Yau threefolds and finite dimensional representations of corresponding quivers with potential, which follows from results of Bridgeland \cite{Bdg1} and Van den Bergh \cite{VdB1} together with some general facts about triangulated categories. We explain a formula for the natural monad presentation for these complexes of sheaves, using an analogy with the role of the Koszul duality derived equivalences of the BGG category $\mc O$ in \cite{BGS}. A more detailed overview is given in Section \ref{unextovsec}.

In Section \ref{extsec}, we generalize the results of the preceding section to describe analogous abelian categories of coherent sheaves generated by the compactly supported perverse coherent sheaves together with an auxiliary object $M\in \DD^b\Coh(Y)^T$ contained in the heart of a compatible Bridgeland-Deligne perverse coherent t-structure. We describe the moduli spaces of objects in these categories, and framed variants thereof, in terms of spaces of representations of framed quivers with potential via a natural monad formalism, proving Theorem \ref{Athm}. A more detailed overview is given in Section \ref{Extoverviewsec}.

In Section \ref{cohareptotalsec}, we recall the Kontsevich-Soibelman cohomological Hall algebra $\mc H(Y)$ of a quiver with potential corresponding to a threefold $Y$, and give the construction of the representation of $\mc H(Y)$ on the homology $\bb V^{\f,\zeta}(M)$ of the moduli space of $\zeta$-stable, $\f$-framed perverse coherent extensions of $M$, proving Theorem \ref{Bthm}. A more detailed overview is given in Section \ref{coharepintrosec}.

In Section \ref{funsec}, we explain the primary intended applications of the results of the previous sections: In Section \ref{DTsec}, we explain the relationship with perverse coherent systems of \cite{NN} and the results of \cite{Sz1}, and for $Y=Y_{m,n}$ state the conjectural relationship to the affine Yangian of $\gl_{m|n}$. In Section \ref{VWsec}, we explain the local Vafa-Witten invariants of divisors $S\subset Y$ defined by Theorem \ref{Athm}, and the relationship of the representation of the cohomological Hall algebra constructed in Theorem \ref{Bthm} in this case to the generalization of the AGT conjecture and the vertex algebras $\V(Y,S)$ defined in the companion paper \cite{Bu}. In Section \ref{MVsec}, we explain the relationship between these results and its potential application to an analogue of the Brundan-Kleshchev isomorphism between affine W-superalgebras and truncated shifted Yangians for affine $\gl_{m|n}$.

\newpage

\section{Preliminaries}\label{presec} In this section, we recall some preliminary results for later reference and in order to fix notation. We recommend the reader skip this section and return to it only later as necessary.

\subsection{$A_\infty$ algebras and their module categories}\label{Ainfalgsec}

Let $S$ be a finite dimensional commutative $\K$ algebra.
\begin{defn}
	A $\Z$-graded $A_\infty$ algebra over $S$ is $A$ is a $\Z^2$-graded $S$ module
\[ A=\bigoplus_{i,j\in \Z } A^i_j[-i]\langle -j \rangle \quad \quad\textup{together with} \quad\quad m_n: A^{\otimes n} \to A[2-n]  \]
for $n\geq 1$, satisfying the usual $A_\infty$ relations.
\end{defn}
When $A^i_j=0$ for $i\neq 0$, or $j\neq 0$, this is equivalent to a graded associative algebra structure, or usual $A_\infty$ algebra, respectively, and when $m_n=0$ for $n\geq 3$ this is equivalent to a (graded) DG-algebra.

We define morphisms of graded $A_\infty$ algebras as usual, and assume that all $A_\infty$ algebras and morphisms of such are graded and strictly unital, in the sense of \cite{LPWZ}. In particular, all (strictly unital, graded) $A_\infty$ algebras are equipped with a canonical strict (strictly unital) map $u:S \to A$.

\begin{defn} An augmentation on a (strictly unital, $\Z$-graded) $A_\infty$ algebra is a (strictly unital) map $\epsilon:A\to S$ of $A_\infty$ algebras over $S$.
\end{defn}

\begin{defn} A (right, $\Z$-graded) $A_\infty$ module $M$ over a $\Z$-graded $A_\infty$ algebra $A$ is a $\Z^2$-graded vector space
	\[ M = \bigoplus_{i,j\in \Z} M^i_j[-i]\langle -j\rangle  \quad\quad \text{together with} \quad\quad \rho_n: M\otimes  A^{\otimes n-1} \to M[2-n] \]
for $n\geq 1$, satisfying the usual $A_\infty$ relations.
\end{defn}
We assume in addition that all $A_\infty$ modules and morphisms of such are graded and strictly unital, and let $\DD_\ZZ(A)$ denote the triangulated category given by the derived category of (strictly unital, $\Z$-graded, right) $A_\infty$ modules over $A$ under (strictly unital) maps of $\Ainf$ modules. In particular, if $A$ is a plain (graded) associative algebra, this agrees with the usual derived category of (graded) $A$ modules. Similarly, we let $\DD_\ZZ^?(A)$ for $?=b,+,-$ denote the usual bounded variants, and $\DD^?(A)$ the usual derived category of (strictly unital, but not graded) $A_\infty$ modules.

\begin{defn}\label{thickdefn}
	A subcategory $\mc C \subset \mc D$ of a triangulated category $\mc D$ is called thick if it is a strictly full, triangulated subcategory which is closed under taking direct summands.
\end{defn}

For any collection of objects $\{M_i\}$ of $\D=\DD_\ZZ(A)$, the thick subcategory $\thick (M_i)$ generated by $\{M_i\}$ is the minimal thick subcategory of $\mc D$ containing all of the objects $M_i$ and their graded shifts. In particular, we let ${\textup{D}_\perf}(A)=\thick(A)$ denote the thick subcategory generated by the rank one free module and its graded shifts, and $\DD_\fd(A)$ the thick subcategory generated by modules $M\in \DD_\ZZ(A)$ with finite dimensional cohomology; the ungraded variants are denoted $\Thick(M_i)$, $\DD_\Perf(A)$, and $\DD_\Fd(A)$, respectively.

\subsection{Tilting objects and derived equivalences}\label{tiltingsec} Throughout this section, let $\D$ be a $\K$-linear, algebraic triangulated category in the sense of Keller \cite{Kel1}. The primary examples of triangulated categories we consider in the present work all fall within this class.

\begin{defn}\label{tiltingdefn}
	An object $T\in \D$ is called a \emph{tilting} object if it satisfies the following conditions:
	\begin{itemize}
		\item $T$ is compact, that is, $\Hom_{\D}(T,\cdot):\D\to \DGV$ preserves coproducts;
		\item $T$ generates $\D$, that is, $\Hom_{\D}(T,C)=0$ implies $C=0$.
	\end{itemize}
If in addition $T$ satisfies the condition that the DG associative algebra $\Hom_{\D}(T,T)$ has cohomology concentrated in degree $0$, then $T$ is called a \emph{classical} tilting object.
\end{defn}

We will use the following result of Keller throughout:

\begin{theo}\cite{Kel2}\label{Morita}
	Let $T$ be a tilting object in an algebraic triangulated category admitting set-indexed coproducts, and let $\Lambda = \Hom_{\D}(T,T)$ be the DG associative algebra of (derived) endomorphisms of $T$. Then there is a triangle equivalence
	\[ \Psi_T: \D \xrightarrow{\cong} \DD( \Lambda ) \] 
	such that composition with the forgetful functor $\ob:\DD(\Lambda ) \to \DGV$ is given by
	 \[\ob\circ \Psi_T =  \Hom_{\D}(T,\cdot): \D \to \DGV  \ , \]
	and with inverse equivalence given by
	\[(\cdot)\otimes_{\Lambda}T:\DD( \Lambda ) \to \D  \ . \]
	 Further, this restricts to define equivalences
\begin{equation}\label{dgmoritaeqntriang}
		\Psi_T: \Triang(T) \xrightarrow{\cong} \Triang(\Lambda ) \quad\quad\text{and}\quad\quad 	 \Psi_T: \Thick(T) \xrightarrow{\cong} \DD_\Perf(\Lambda ) \  ,
\end{equation}
where $\Triang(T)$ denotes the minimal triangulated subcategory of $\D$ containing $T$, similarly for $\Triang(\Lambda )$, and $\Thick(T)$ is the minimal thick subcategory of $\D$ containing $T$.
\end{theo}

In summary, we obtain mutually inverse triangle equivalences
\[ \Hom_{\D}(T,\cdot):\D \xymatrix{  \ar@<.5ex>[r]^{\cong} & \ar@<.5ex>[l]} \DD(\Lambda) : (\cdot)\otimes_{\Lambda}T  \ . \]
Further, we note these equivalences identify the object $T\in \D$ with the rank 1 free module $\Lambda \in \DD(\Lambda)$.

The primary application of exceptional collections in the present work is to constructing well-behaved tilting objects, which uses the following result: Recall that a collection of objects $T_i$ \emph{classically generates} a category $\mc C$ if $\thick(T_i)=\mc C$.

\begin{theo} \cite{BonV}
	Let $\D$ be a compactly generated triangulated category. Then a set of compact objects classically generates $\D^\cc$ if and only if it generates $\D$.
\end{theo}

Let $X$ be a smooth variety and $E$ in $\DD^b\Coh(X)$ a classical generator. Then from the preceding theorem, we have:

\begin{corollary}
$E$ is a tilting object for $\DD \QC(X)$.
\end{corollary}

Let $\Lambda=\Hom_{\DD^b\Coh(X)}(E,E)$ be the DG algebra of endomorphisms of $E$. Applying the Morita theory from \cite{Kel2} recalled in Theorem \ref{Morita}, we have:

\begin{corollary} There are triangle equivalences
\begin{align*}
 \DD\QC(X) \xrightarrow{\cong} \DD(\Lambda ) &  & \DD^b\Coh(X) \xrightarrow{\cong} \DD_\Perf(\Lambda ) \end{align*}
intertwining the forgetful functor $\DD(\Lambda \Mod) \to \DGV$ with $\Hom_{\QC(X)}(E,\cdot)$.
\end{corollary}

\subsection{Perverse coherent sheaves and non-commutative resolutions}\label{NCCRsec}

In this section, we recall descriptions of categories of coherent sheaves in terms of non-commutative algebras, following \cite{Bdg1} and \cite{VdB1}, \cite{VdB2} throughout. Throughout this section, let $f:Y\to X$ a projective map of Noetherian schemes satisfying the conditions
\begin{enumerate}
	\item $f_*\mc O_Y = \mc O_X$, and
	\item the fibres of $f$ are at most one dimensional.
\end{enumerate}

We recall the notation $\Homi_{\DD^b\Coh(X)}(E,F)\in \DD^b\Coh(Y)$ for the internal $\Hom$ object in $\DD^b\Coh(Y)$. In particular, for any object $E\in \DD^b\Coh(Y)$, the internal endomorphism algebra
\[  \underline{\Lambda}_E =\Homi_{\DD^b\Coh(Y)}(E,E) \in \Alg_\Ass(\DD^b\Coh(Y)) \]
defines a DG associative algebra object internal to $\DD^b\Coh(Y)$, and we have
$$ \Homi_{\DD^b\Coh(Y)}(E,F) \in\underline{\Lambda}_E \Mod(\DD^b\Coh(Y)) $$
for each $F\in \DD^b\Coh(Y)$. Similarly, we have 
\[ f_*\underline{\Lambda}_E\in \Alg_\Ass(\DD^b\Coh(X)) \quad\quad \text{and}\quad\quad f_*\Homi_{\DD^b\Coh(Y)}(E,F) \in (f_*\underline{\Lambda}_E)\Mod(\DD^b\Coh(X)) \ .\]
In this context the first main result of \cite{VdB1}, following \cite{Bdg1}, is as follows:

\begin{theo}\label{derivedVdBequiv}\cite{VdB1}
	There is a vector bundle $E$ on $Y$ such that for $\mc A= f_*\Homi_{\DD^b\Coh(Y)}(E,E)$ there are mutually inverse equivalences of categories
	\[f_*\Homi_{\DD^b\Coh(Y)}(E,\cdot):\xymatrix{ \DD^b\Coh(Y) \ar@<.5ex>[r]^{\cong} & \ar@<.5ex>[l]  \DD^b \mc A\Mod} : f^{-1}(\cdot)\otimes_{f^{-1}(\mc A)} E \ . \] 
\end{theo}

Note we are considering $\mc A \in \Alg_\Ass(\DD^b\Coh(X))$ as above and use the simplified notation $\DD^b\mc A\Mod:= \mc A\Mod(\DD^b\Coh(X))$. We will give a concrete description of the vector bundle $E$ and explain the details of this equivalence in the more specific setting of Section \ref{Kozpatsec} below.

It is natural to ask what the standard $t$-structure on $\DD^b \mc A\Mod$ corresponds to on $\DD^b\Coh(X)$. Towards giving a geometric description, we recall the Beilinson-Bernstein-Deligne theorem on gluing t-structures along recollements: 

\begin{theo}\label{BBD}\cite{BBD} Consider three triangulated categories with exact functors
	\[ \mc D_F \xrightarrow{i_{*}} \mc D \xrightarrow{j^{*}} \mc D_U \]
	satisfying the following conditions:
	\begin{enumerate}
		\item there exist left and right adjoints $i^*,i^!$ and $j_!,j_*$ to both $i_{*}$ and $j^*$, respectively;
		\item $j^*\circ i_* = 0$;
		\item for each $E\in \mc D$, there exist exact triangles
		\[ j_!j^*E \to E  \to i_*i^*E \xrightarrow{[1]} \quad\quad \text{and}\quad\quad i_*i^! E \to E \to j_*j^* E \xrightarrow{[1]} \quad\quad \textup{; and} \]
		\item $i_*,j_!,j_*$ are fully faithful.
	\end{enumerate}

	Then for each pair of $t$-structures $(\mc D^{\leq 0}_F,\mc D^{\geq 0}_F)$ and $(\mc D^{\leq 0}_U,\mc D^{\geq 0}_U)$ on $\mc D_F$ and $\mc D_U$, we have that
	\[ \mc D^{\leq 0} :=\{ E \in \D | j^*E \in \D_U^{\leq 0}\text{ and } i^*E \in \D_F^{\leq 0} \} \quad\quad \text{and}\quad\quad
		\mc D^{\geq 0}  =\{ E \in \D | j^*E \in \D_U^{\geq 0}\text{ and } i^!E \in \D_F^{\geq 0} \} 
	\]
define a $t$-structure on $\mc D$.
\end{theo}

\begin{defn}
	A \emph{recollement} of triangulated categories is a triple of triangulated categories with functors satisfying the hypotheses of the preceding theorem.
\end{defn}

\begin{rmk}
	The notation in the statement of Theorem \ref{BBD} is motivated by the example of the recollement given by the derived categories of constructible sheaves on complementary closed and open embeddings $i:F \to X$ and $j:U=X\setminus F  \to X$.
\end{rmk}

For our application, we take $\D=\DD^b\Coh(Y)$, $\D_U=\DD^b\Coh(X)$, and $\D_F$ the full subcategory
\[ \mc C = \{ E\in \DD^b\Coh(Y)\ |\ f_*E=0 \} \ .\]
It was observed in \cite{Bdg1} that these categories constitute a recollement; we outline the proof to fix notation, and for completeness as it was omitted in \emph{loc. cit.}:

\begin{prop}\label{Bdgrecprop}\cite{Bdg1} The categories and functors
	\begin{equation}
		\mc C \xrightarrow{\iota} \DD^b\Coh(Y) \xrightarrow{f_*}\DD^b\Coh(X) \ , \label{threc}
	\end{equation}
	define a recollement of triangulated categories, where $\iota:\mc C \to \DD^b\Coh(Y)$ is the inclusion of the full subcategory.
\end{prop}
\begin{proof} By definition, we have $f_*\circ \iota=0$. The functor $f_*=f_!$ by projectivity, and admits left and right adjoints $f^*$ and $f^!$, respectively, by coherent duality. Moreover, for each object $E\in \DD^b\Coh(Y)$ we have
\[ 	f_*f^* E  \cong E \otimes f_* \mc O_Y \cong E  \ ,  \]
by the hypothesis that $f_*\mc O_Y=\mc O_X$, and similarly $f_*f^!E\cong E$. It follows similarly that $f^*$ and $f^!$ are fully faithful.

Further, for each object $E\in \DD^b\Coh(Y)$ define
\[ \tilde C_E = \textup{cone}\left[ f^*f_*E \to E \right] \quad\quad \text{and}\quad\quad C_E=\textup{cocone}\left[ E \to f^!f_*E \right] \]
so that we have exact triangles
\begin{equation}\label{Bdgrecseqeqn}
	 f^*f_*E \to E \to \tilde C_E \xrightarrow{[1]} \quad\quad\text{and}\quad\quad  C_E \to E \to f^!f_* E \xrightarrow{[1]}  \ .
\end{equation}
By the preceding paragraph, we have $f_*C_E=f_*\tilde C_E=0$, so that $C_E,\tilde C_E \in \mc C$, and thus these exact triangles imply left and right admissibility, respectively, of the inclusion of the full subcategory $\iota: \mc C \to \DD^b\Coh(Y)$. In particular, the left and right adjoints
\[ \iota^L,\iota^R:\DD^b\Coh(Y) \to \mc C\quad\quad \text{are defined by}\quad\quad E \mapsto \tilde C_E, C_E \ , \] respectively, and by construction we have the desired exact triangles.
\end{proof}

Fix an integer $k\in \bb Z$ and consider the $k^{th}$ shift of the standard $t$-structure on $\mc C$ inherited as a subcategory of $\mc D$:
\[ \mc C^{k,\leq 0 }:= \mc C^{\leq k} \quad\quad \text{and}\quad\quad \mc C^{k,\geq 0}:= \mc C^{\geq k} \ . \]

\begin{defn}\label{pervcohdef} The perverse coherent $t$-structure is defined by
\begin{align}
		\DD^b\Coh(Y)^{k,\leq 0} & := \{ E \in \DD^b\Coh(Y) \ |\  f_*E \in \DD^b\Coh(X)^{\leq 0}\text{ and } \iota^L E \in \mc C^{k,\leq 0} \} &  ,\\
	\DD^b\Coh(Y)^{k,\geq 0} &:=\{ E \in \DD^b\Coh(Y) \  |  \  f_*E \in \DD^b\Coh(X)^{\geq 0}\text{ and } \iota^R E \in \mc C^{k,\geq 0} \} & .
\end{align}
\end{defn}

\noindent Note that this indeed defines a t-structure, by Proposition \ref{Bdgrecprop} and Theorem \ref{BBD}.

Towards giving a more concrete description, we have:
\begin{prop} For each $E\in  $, the conditions $\iota^L E \in \mc C^{k,\leq 0}$ and $\iota^R E \in \mc C^{k,\geq 0}$ are equivalent to the conditions
	\[ \Hom(E,C)=0 \text{ for each $C\in \mc C^{>k}$} \quad\quad \text{and}\quad\quad  \Hom(C,E)=0 \text{ for each $C\in \mc C^{<k}$} \  ,  \]
respectively.
\end{prop}
\begin{proof} By definition $\mc C^{\leq k}= {^{\perp}}{\mc C}^{>k}$, so that $\iota^L E \in \mc C^{k,\leq 0}$ if and only if $\Hom(\iota^L E,C)=0$ for each $C\in {\mc C}^{>k}$. Moreover, by construction we have $ \Hom(\tilde C_E, C) = \Hom(E,C) $ for each object $C\in \mc C$ since $\mc C = {^\perp}(f^*\DD^b\Coh(X))$. The proof of the equivalence of the two latter conditions follows by the dual argument.
\end{proof}

In summary, we obtain the concrete description of the perverse coherent $t$-structure:
	\begin{align*}
		\DD^b\Coh(Y)^{k,\leq 0} & := \{ E \in \DD^b\Coh(Y) \ |\  f_*E \in \DD^b\Coh(X)^{\leq 0}\text{ and } \Hom(E,C)=0 \text{ for each $C\in \mc C^{>k}$} \} &  ,\\
		\DD^b\Coh(Y)^{k,\geq 0} &:=\{ E \in \DD^b\Coh(Y) \  |  \  f_*E \in \DD^b\Coh(X)^{\geq 0}\text{ and }  \Hom(C,E)=0 \text{ for each $C\in \mc C^{<k}$} \} & .
	\end{align*}

\begin{defn}\label{PervCohdefn} The category of perverse coherent sheaves on $f:Y\to X$ is defined by
	\[\Perv^{k}(Y/X) : = \DD^b\Coh(Y)^{k,\leq 0} \cap  \DD^b\Coh(Y)^{k,\geq 0}  \ .\]
\end{defn}
In particular, we let\[ \Perv(Y/X)=\Perv^{-1}(Y/X) \]denote the category of perverse coherent sheaves for the integer $k=-1$, and use the term perverse coherent sheaves on $f:Y\to X$ to refer to this case, unless specified otherwise. Also, we will often drop the dependence on $f$ and $X$ from the notation and write simply \[\Perv(Y):=\Perv(Y/X) \ .\]

In this case, we obtain the following further simplification:

\begin{corollary}\label{PervCohcoro} The category of perverse coherent sheaves on $f:X\to Y$ is given by
\[\hspace*{-1.5cm} \Perv(Y)= \{ E \in \DD^{[-1,0]}\Coh(Y) \ | \ R^0f_*H^{-1}E=0,\ R^1f_* H^0E=0,\ \Hom( H^0E,C)=0 \text{ for any $C\in \mc C^\heartsuit$} \} \ . \]
In particular, the right hand side is an abelian category.
\end{corollary}
\begin{proof}
	See Lemma 3.2 in \cite{Bdg1}.
\end{proof}

We can now state the result of Bridgeland and Van den Bergh on the image of the heart $\mc A\Mod\subset \DD^b\mc A\Mod$ under the equivalence of Theorem \ref{derivedVdBequiv}:

\begin{theo} \cite{Bdg1,VdB1}\label{PCohequiv} The equivalence of Theorem \ref{derivedVdBequiv} restricts to an equivalence of abelian categories
	\[ \Perv(Y) \xrightarrow{\cong} \mc A\Mod  \ . \]
\end{theo}

\subsection{Quivers, path algebras, and Koszul duality for $A_\infty$ algebras}\label{quiversec}

\begin{defn}\label{quiverdefn}
	A finite, bigraded \emph{quiver} $Q$ is a finite set of vertices $V_Q$, a finite set of edges $E_Q=\sqcup_{k,j\in \bb Z} (E_Q)^k_j$, and source and target maps $s,t:E_Q\to V_Q$. We write
	$$ E_Q(v,w) = \bigsqcup_{k,j\in \bb Z} (E_Q)^k_j(v,w)$$
	for the space of edges from $v$ to $w$, defined by $s^{-1}(v)\cap t^{-1}(w)$; the elements of $(E_Q)^k_j$ are called edges of \emph{bidegree} $(k,j)$, as they determine by definition the bidegree of the corresponding summand of the $\bb K$-span of the edge set
	\[ \bb K\langle E_Q \rangle := \bigoplus_{k,j\in \Z} \bb K\langle E_Q \rangle^k_j [-k] \langle -j \rangle  \ ,\]
	where $\bb K\langle E_Q\rangle^k_j=\bb K\langle (E_Q)^k_j\rangle$ is the $\bb K$-span of the space of bidegree $(k,j)$ edges.
	
	A \emph{path} $p$ of length $n$ in a quiver $Q$ is a sequence of edges $e_n,...,e_1\in E_Q$ such that $t(e_i)=s(e_{i+1})$ for $i=1,...,n-1$. The vertices $s(e_1)$ and $t(e_n)$ are the \emph{source} and \emph{target} of the path, and the bidegree of the path is sum of the bidegrees of the edges in the corresponding sequence.
\end{defn}

We say that a quiver \emph{plain} if its edge set is concentrated in bidegree $(0,0)$, \emph{plain graded} if its edge set is concentrated in bidegree $(0,\bullet)$ and \emph{cohomologically graded} if its edge set is concentrated in bidegree $(\bullet,0)$. 

Concatenation of paths is defined as usual as the concatenation of the corresponding sequences, whenever it yields a new path. We also formally define the set of length zero paths to be the set of vertices of $Q$, and define concatenation of a fixed path with one of length zero on the left or right only if the corresponding vertex is the source or target of that path, respectively, in which case it yields the same path.

\begin{defn} Let $Q$ be a $\bb Z^2$-graded quiver. The path algebra $\K Q$ is the unital associative bigraded $\K$ algebra spanned over $\K$ by paths of any length $n\geq 0$, with product defined by
	\[ p \star q = \begin{cases}
		pq & \textup{if the concatenation is defined, and} \\
		0 & \textup{otherwise.}
	\end{cases}\]
\end{defn}

The above conventions for length zero paths imply that
$$\K Q = \otimes^\bullet_S \ \K\langle E_Q\rangle \quad\quad\text{where}\quad\quad S=\bigoplus_{v\in V_Q} \K_v \ ,$$
that is, the path algebra $\K Q$ is the tensor algebra over the semi-simple base ring $S$ of the $S$-bimodule $\K\langle E_Q\rangle$ spanned by the set of edges with left and right module structures determined by concatenation on the left and right with length zero paths.

Let $\K Q_{(n)}\subset \K Q$ be the two-sided ideal spanned as a vector space by paths of length $\geq n$.

\begin{defn}\label{dgquiverdefn} A (graded) \emph{quiver with relations} $(Q,R)$ is a plain (graded) quiver $Q$ together with a two-sided ideal $R\subset\K Q$ such that $R\subset \K Q_{(2)}$. The path algebra $\K Q_R$ of a quiver with relations is defined by $\K Q_R= \K Q/R$.
	
A (graded) \emph{DG quiver} $(Q,d)$ is a cohomologically graded (bigraded) quiver $Q$ together with an $S$-linear derivation $d:\K Q \to \K Q[1]$ such that $d^2=0$. The path algebra $\K Q_d$ of a DG quiver $(Q,d)$ is the DG algebra defined by $(\K Q, d)$.
\end{defn}

It is possible to describe the most general lift of the path algebra of a bigraded quiver to yield a graded DG quiver. In the most concrete terms, we see that a differential $d:\K Q\to \K Q[1]$ is determined on generators by maps
$$ d_n:\K\langle E_Q\rangle \to \otimes^n_S \K\langle E_Q\rangle [1] \ . $$
These maps are equivalently defined in terms of $\bar A =\K\langle E_Q\rangle^\vee[-1]$ by maps
$$ m_n: \otimes^n_S\bar A \to  \bar A [2-n]  \ ,$$
and the condition that $d^2=0$ is equivalent to the condition that the maps $m_n$ satisfy the $A_\infty$ relations. In summary, we obtain that a choice of differential making $\K Q$ into a quasi-free graded DG algebra is equivalent to a graded $A_\infty$ structure on $A=S\oplus \bar A$ compatible with the natural $S$-augmentation.

\begin{rmk}\label{Kosfinrmk} We will sometimes assume that the $S$-bimodule $\K\langle E_Q\rangle$ is graded finite dimensional. Moreover, the differential $d$ is in fact defined only on the completed tensor algebra unless the corresponding $A_\infty$ algebra is finite, in the sense that the structure maps $m_n=0$ vanish for $n$ sufficiently large. These conditions will often hold in our applications of interest.
\end{rmk}

In fact, the preceding construction is a concrete manifestation of Koszul duality for $A_\infty$ algebras, as we now explain:
\begin{defn}
	Let $A$ be an $A_\infty$ algebra over $S$ with an augmentation $\epsilon:A\to S$. The \emph{Koszul dual} of $A$ over $S$ is the DG associative algebra defined by
	$$A^!= \Hom_A(S,S) \ . $$
\end{defn}

For any $A_\infty$ algebra $A$ with augmentation $\epsilon:A\to S$ and augmentation ideal $\bar A=\ker \epsilon$, there is a canonical free resolution of $S$ as an $A$ module, called the \emph{Koszul resolution} given by
$$ \mc K^\bullet:= A\otimes_S \left( \otimes^\bullet_S \bar A[1]\right) = \left[ \cdots \to A\otimes_S \bar A^{\otimes 2}[2] \to   A\otimes_S \bar A[1]  \to A \right] \ ,$$
with $d:\mc K^\bullet \to \mc K^\bullet[1]$ defined on generators by the multiplication map
$$ d|_{A\otimes_S \bar A}=m_2: A\otimes_S \bar A \to A  $$
and its higher arity analogues. We obtain the following computation of the Koszul dual algebra:

\begin{prop}\label{Kozpresprop} The algebra $A^!$ is presented by the quasi-free DG associative algebra
	\[ A^!= \left( \otimes^\bullet_S (  \bar A[1])\right)^\vee \]
with underlying complete associative algebra free on $ \bar A^\vee[-1]$, and differential determined by the $A_\infty$ structure via the equivalence explained above.
\end{prop}
\begin{proof}
	We use the Koszul resolution to compute the underlying cochain complex
	\[ \Hom_A(S,S) \cong \Hom^0_A(\mc K^\bullet, S) \cong \Hom^0_S( \otimes^\bullet_S \bar A[1],S )=\left(\otimes^\bullet_S ( \bar A[1])\right)^\vee \ ,  \]
	as desired, and observe that the left module structure of $A^!$ over itself as a tensor algebra is identified with the right module structure on $\Hom_A(S,S)$ over itself given by precomposition with $\Hom_A(\mc K^\bullet,\mc K^\bullet)$.
\end{proof}

We now recall the primary statements of Koszul duality in the present setting, following the presentation of \cite{LPWZ}:
\begin{theo}\cite{LPWZ}\label{Kozsqtheo} Let $A$ be a strongly locally finite $A_\infty$ algebra. Then $A^!$ is strongly locally finite, and moreover there is a natural isomorphism of algebras $(A^!)^!\cong A$.
\end{theo}

\begin{rmk}\label{koszulquiverrmk} This result suggests the following strategy to describe a given DG algebra over $S$ in terms of the path algebra of a DG quiver, or equivalently, to construct a quasi-free resolution of it: compute the Koszul dual algebra, use homotopy transfer to compute the $A_\infty$ structure on its cohomology, and then construct the corresponding DG quiver $Q$ according to the concrete description of Koszul duality above.
\end{rmk}

One also obtains an equivalence of categories of modules over Koszul dual algebras:

\begin{theo}\cite{LPWZ}\label{Kozeqtheo} Let $A$ be a strongly locally finite, augmented $A_\infty$ algebra over $S$ and $A^!$ its Koszul dual. Then there are mutually inverse equivalences of categories
	\[ \Hom_A(S,\cdot): \xymatrix{\thick(S)   \ar@<.5ex>[r]^{\cong} & \ar@<.5ex>[l] {\textup{D}_\perf}(A^!)}: (\cdot)\otimes_{A^!}S \ ,\]
	where $\thick(S)$ denotes the thick subcategory of $\DD(A)$ generated by $S$ and its graded shifts, and ${\textup{D}_\perf}(A^!)=\thick(A^!)$ that generated by the rank 1 free module $A^!$ and its graded shifts in $\DD(A^!)$, as in the discussion following Definition \ref{thickdefn}.
\end{theo}
	In particular, note that if $A$ is graded finite dimensional then $\thick(S) = \DD^b_\fd(A)$, and thus we obtain an equivalence of the latter with ${\textup{D}_\perf}(A^!)$.
	
	\begin{rmk}\label{Kozequivrmk} In the case that $A$ is an augmented DG algebra, the above result holds under the weaker hypothesis that $A^!$ is locally finite. These results follow from Theorems 5.7 and 5.4 of \emph{loc. cit.}, respectively, and the relevant definitions are recalled in Definition 2.1.
	\end{rmk}

\begin{rmk}\label{cohomshearrmk} Recall that for a plain (non DG) Koszul associative algebra $\Lambda$, the Koszul dual algebra is concentrated in bi-degrees $(k,-k)$ for $k\in \bb N$. Thus, in order to relate the preceding Theorem with more classical accounts of Koszul duality of plain graded algebras, it is necessary to apply a cohomological shearing operation, which we now explain following Section 7.2 of \cite{LPWZ}:

Let $A$ be a graded DG associative algebra with zero differential, and underlying bi-graded vector space
\[ A = \bigoplus_{k,j \in \Z} A^k_j[-k]\langle -j\rangle \ . \]
Then the cohomological shear $A^\sh$ of $A$, defined by
\[ A^\sh = \bigoplus_{k,j \in \Z} (A^\sh)^k_j[-k]\langle -j\rangle \quad\quad \text{with}\quad\quad  (A^\sh)^k_j= A^{k+j}_{-j} \]
is again a graded associative algebra. For example, if $A$ is concentrated in bidegree $(k,-k)$, as we observed for the Koszul dual of a plain graded Koszul associative algebra, then $A^\sh$ defines a plain graded associative algebra, concentrated in cohomological degree $0$.

Moreover, the analogous shear functor gives an equivalence $(\cdot)^\sh:\DD(A)\xrightarrow{\cong} \DD(A^\sh) $ on categories of DG modules, defined by
\[ \quad\quad M=\bigoplus_{i,j \in \Z} M^i_j[-i]\langle - j \rangle \mapsto M^\sh=\bigoplus_{i,j \in \Z} M^{i+j}_{-j}[-i]\langle - j \rangle  \ .\] 
If $A$ is finite dimensional, it also restricts to an equivalence $(\cdot)^\sh:\DD_\fd(A)\xrightarrow{\cong} \DD_\fd(A^\sh) $.
\end{rmk}

\begin{warning} \label{gradingnotationwarning} We will often describe algebras and their DG modules in terms of the cohomologically sheared conventions, but omit the superscript $(\cdot)^\sh$ by abuse of notation. Further, in this setting we will use superscripts and subscripts that differ from those above, to denote the decomposition of a plain (non DG) graded $S$ algebra
\[ A = \bigoplus_{i,j\in I,\ k\in \Z} \ _{i} A_{j}^k\langle - k \rangle \quad\quad \text{where}\quad\quad _i A_j = S_i A S_j \]
denote the components with respect to the idempotents of the semisimple base ring $S=\oplus_i S_i$.
\end{warning}

\subsection{Calabi-Yau structures and quivers with potential}

Let $Q$ be a bigraded quiver as in Definition \ref{quiverdefn} above, $\K Q$ its path algebra, and define
$$ \K Q_\cyc = \K Q /[\K Q, \K Q]  \ . $$

\noindent Suppose that the edge space $\K\langle E_Q \rangle$ of the underlying quiver $Q$ admits a non-degenerate pairing
\begin{equation}\label{Pzerostreqn}
	(\cdot,\cdot):\K\langle E_Q \rangle^{\otimes 2} \to \K[1] 
\end{equation}
such that
\begin{enumerate}
	\item $(a,b)=(-1)^{|a||b|} (b,a)$
	\item $(a,b)=0$ unless $t(a)=s(b)$ and $s(a)=t(b)$.
\end{enumerate}
We also let $\langle\cdot,\cdot\rangle: (\K\langle E_Q \rangle^\vee)^{\otimes 2} \to \K[-1]$ denote the inverse pairing.

\begin{defn}
	A \emph{non-commutative, $(-1)$-shifted symplectic form} on the path algebra $\K Q$ is a pairing $\langle\cdot,\cdot\rangle: (\K\langle E_Q \rangle^\vee)^{\otimes 2} \to \K[-1]$ given by the inverse of a pairing $(\cdot,\cdot):\K\langle E_Q \rangle^{\otimes 2} \to \K[1]$ as in Equation \ref{Pzerostreqn} above.
\end{defn}
For simplicity, we will simply refer to such a pairing as a symplectic form on the underlying $\bb Z$-graded quiver $Q$.

\begin{eg}\label{geoquiveg1}
	In the setting of Section \ref{Geosetupsec} below, let $\Sigma=\Ext^\bullet(F,F)$ and $Q_Y$ the corresponding DG quiver with edge set $\K\langle E_{Q_Y}\rangle = \bar\Sigma^\vee[-1]$ so that $\Lambda= \K Q_Y$. Since $Y$ is a Calabi-Yau threefold, $\bb K Q_Y$ admits a symplectic form by restricting the Serre duality pairing
	\[ \langle \cdot,\cdot \rangle: \Ext^\bullet(F,F)[1] \otimes \Ext^\bullet(F,F)[1] \to \bb K [-1] \]
	noting that the object $F$ is compactly supported.
\end{eg}

For a $\bb Z$-graded quiver $Q$ with symplectic form, there is a canonical non-commutative Poisson bracket $\{ \cdot,\cdot\}: \K Q^{\otimes 2} \to \K Q[1] $ of degree $+1$, defined on generators by the pairing \ref{Pzerostreqn} and extended as a non-commutative biderivation.

\begin{defn}\label{quivpotdefn}
	A \emph{quiver with potential} is a DG quiver $(Q,d)$ in the sense of Definition \ref{dgquiverdefn}, equipped with a symplectic form, and a potential $W\in \K Q_\cyc$ such that \[d=\{W,\cdot\}\ \in \textup{Der}^1(\K Q) \ . \]
\end{defn}
Note that the potential $W\in \K Q_\cyc$ necessarily satisfies the master equation $\{W,W\}=0$ since $d^2=0$ by hypothesis. The following construction is well-known:

\begin{eg}\label{geoquiveg2}
	Let $\Sigma$ and $Q_Y$ be as in Example \ref{geoquiveg1} above. Then there is a canonical potential $W\in \K Q_\cyc$, defined by
	$$ W= \sum_{n\geq 1} \sum_{a_1,...,a_{n+1} \in E_Q} \langle m_n(a_1^\vee,...,a_n^\vee), a_{n+1}^\vee \rangle a_1 \cdot ... \cdot a_{n+1}  \ , $$
	such that $(Q_Y,d,W)$ is a quiver with potential, where $m_n:\Sigma^{\otimes n}\to \Sigma[2-n]$ denote the $\Ainf$ structure maps on $\Sigma$.
\end{eg}

\subsection{$A_\infty$ categories and the twisted objects construction of Kontsevich}\label{twobjsec}

\begin{defn} An $A_\infty$ category $\mc A$ is a set $\textup{ob}(\mc A)$ together with
	\begin{itemize}
		\item graded vector spaces $\mc A(i,j)=\Hom_{\mc A}(i,j)\in \K\Mod_\Z$ for each $i,j\in \ob(\mc A)$, and
		
		\item maps of graded vector spaces
		\[ m_n^{i_0,...,i_n}: \mc A(i_0,i_1)\otimes_{\bb K} \hdots \otimes_{\bb K} \mc A(i_{n-1},i_n) \to \mc A(i_0,i_n)[2-n] \]
		for each finite list of objects $i_0,...,i_n\in \textup{ob}(\mc A)$,
	\end{itemize}
	satisfying the natural generalization of the usual $A_\infty$ relations.
\end{defn}

In particular, the $A_\infty$ category structure maps for $n=1$ endow each of the graded vector spaces $\mc A(i,j)$ with the structure of a cochain complex; $A_\infty$ categories are the natural non-strictly associative homotopical generalization of DG categories, analogous to the way in which $A_\infty$ algebras generalize DG associative algebras.

\begin{eg}\label{AinfcatSeg}
	Let $S=\oplus_{i\in I} S_i = \oplus_{i\in I}\bb K$ be a semi simple base ring and $A$ an $A_\infty$ algebra over $S$. Then $A$ can equivalently be considered as an $A_\infty$ category $\mc A$ with
	\[ \textup{ob}(\mc A) = I \quad\quad \textup{and} \quad\quad \mc A(i,j)=\Hom_{\mc A}(i,j)= \ _i A_j \  .\]
	For example, if $A_M=\Ext^\bullet_\D(M,M)$ for some triangulated category $\mc D$ and $M=\oplus_{i\in I} M_i \in \D$, then the corresponding $A_\infty$ category $\mc A_M$ is given by
	\[ \Hom_{\mc A_M}(i,j) = \Ext_\D^\bullet(M_i,M_j) \ , \]
	together with the natural higher multiplication maps induced by homotopy transfer.
\end{eg}

Similarly, generalizing the usual definition of a map of $A_\infty$ algebras, we have:

\begin{defn} Let $\mc A,\mc B$ be $A_\infty$ categories. A functor $f:\mc A \to \mc B$ is a map $f:\ob(\mc A) \to \ob(\mc B)$ of object sets together with maps of graded vector spaces
	\[ f_n^{i_0,...,i_n}: \mc A(i_0,i_1)\otimes_{\bb K} \hdots \otimes_{\bb K} \mc A(i_{n-1},i_n) \to \mc B(f(i_0), f(i_n))[1-n] \]
	for each finite list of objects $i_0,...,i_n\in \textup{ob}(\mc A)$, satisfying the natural generalization of the usual conditions defining a map of $A_\infty$ algebras.
\end{defn}

Similarly, there is a notion of natural transformations of $A_\infty$ functors, which makes the $A_\infty$ functors $f:\mc A \to \mc B$ into a DG category $\textup{Fun}_\infty(\mc A,\mc B)$. In particular, taking $\mc B$ to be the category $\DGVect $ of cochain complexes, we obtain the DG category of modules $\CC_\infty(\mc A):=\textup{Fun}_\infty(\mc A, \DGVect )$ over the $A_\infty$ category $\mc A$. Concretely, we have:

\begin{defn} Let $\mc A$ be an $A_\infty$ category. An $\Ainf$ module $\mc M\in\CC_\infty(\mc A)$ over $\mc A$ is given by
	\begin{itemize}
		\item graded vector spaces $\mc M(i)\in \bb K\Mod_\Z$ for each $i \in \ob(\mc A)$, and
		\item maps of graded vector spaces
		\begin{equation}\label{Ainfcatmodmapseqn}
			\rho_n^{i_0,...,i_n} : \mc A(i_0,i_1)\otimes_{\bb K} \hdots \otimes_{\bb K} \mc A(i_{n-1},i_n)\otimes \mc M(i_0) \to \mc M(i_n) [1-n] 
		\end{equation}
		for each finite list of objects $i_0,...,i_n\in \textup{ob}(\mc A)$,
	\end{itemize}
	satisfying the natural generalization of the usual conditions defining an $A_\infty$ module.
\end{defn}

\begin{rmk} For simplicity, we will often denote the collection of maps of graded vector spaces of Equation \ref{Ainfcatmodmapseqn} by simply $\mc A^{\otimes n}\otimes \mc M\to \mc M[1-n]$.
\end{rmk}

\begin{eg}
	For each object $j\in \mc A$, there is a module $\mc A_j\in\CC_\infty(\mc A)$ defined concretely by
	\[ \mc A_j(i) = \mc A(i,j) \]
	with module structure maps given by the structure maps for $\mc A$. These modules are the natural generalization of the rank 1 free module over a usual $A_\infty$ algebra, and in the context of Example \ref{AinfcatSeg} of a category $\mc A$ determined by an algebra $A$ over a semi simple base ring $S=\oplus_{i\in I} \bb K$, the module $\mc A_j$ is identified with $A_j$, the rank 1 free module multiplied on the right by the $j^{th}$ idempotent, under the equivalence $\CC_\infty(\mc A)=\CC_\infty(A)$.
	
	This construction defines a functor of $A_\infty$ categories
	\begin{equation}\label{Ainfyoneqn}
		Y: \mc A \to \CC_\infty (\mc A ) \quad\quad j \mapsto \left[ \mc A(\cdot, j): \mc A \to \Vect \right] \ ,
	\end{equation}
	called the Yoneda embedding for $A_\infty$ categories.
\end{eg}

We now describe the twisted objects construction, which was first proposed in \cite{KonHMS}, generalizing the results of \cite{BonKap} to the $\Ainf$ context, and was worked out in detail in \cite{Lef}. The first step is to introduce the category of shifted objects $\Z\mc A$ of $\mc A$:

\begin{defn} Let $\mc A$ be an $\Ainf$ category and define the $\Ainf$ category $\Z\mc A$ by
	\[ \ob(\Z\mc A) = \{(i,n)\ |\  i\in \ob(\mc A), n\in \Z \} \quad\quad \text{and}\quad\quad \Z\mc A( (i,n),(j,m)) = \mc A(i,j)[m-n]\]
	together with the natural extension of the $A_\infty$ structure maps of $\mc A$.
\end{defn}

\begin{warning}
	In the following, we will often omit the integer $n\in \Z$ from the notation and write simply $i\in \Z \mc A$ and $\Z\mc A(i,j)$ with the associated integers left implicit.
\end{warning}

We now define the desired category $\Tw \mc A$ of twisted objects in $\mc A$

\begin{defn}\cite{Lef}\label{twobjdefn} Let $\mc A$ be an $\Ainf$ category. The category $\Tw\mc A$ of twisted objects over $\mc A$ is the $\Ainf$ category defined as follows: an object of the category $\Tw\mc A$ is given by a finite collection of objects
	\[ (i_1,n_1),...,(i_d,n_d) \in \Z\mc A \]
	together with a degree zero element
	\[ \delta=(\delta_{kl}\in \Z\mc A(i_k,i_l)[1]  = \mc A(i_k,i_l)[n_l-n_k+1] )_{k,l=1}^d \in \mf{gl}_d(\Z \mc A) \]
	such that $\delta_{kl}=0 $ for $k\leq l$ and moreover $\delta$ satisfies the Maurer-Cartan equation
\begin{equation}\label{MCeqn}
		 \sum_{t\in \bb N}  m_t^{\mf{gl}_d(\Z\mc A)}(\delta^{\otimes t}) = 0 \ \in \mf{gl}_d(\bb Z \mc A)  \ ; 
\end{equation}
	here $\mf{gl}_d(\Z\mc A)$ denotes the $d\times d$ matrices with coefficients in $\Z \mc A$ and $m_t^{\mf{gl}_d(\Z\mc A)}$ the extension of the $A_\infty$ structure maps on $\mc A$ to $\mf{gl}_d(\Z\mc A)$, given by tensoring with the usual associative product on $\mf{gl}_d$.
	
	The spaces of maps in $\Tw\mc A$ are given by
	\begin{equation}\label{twobjhomspceqn}
		\Tw\mc A ((i_1,...,i_d,\delta), (j_1,...,j_c,\eta))= \bigoplus_{k=1,...,d,\ l = 1,...,c} \Z\mc A ( i_k, j_l) 
	\end{equation}
	and the $A_\infty$ structure maps are defined for each $t\in \bb N$ by
	\begin{equation}\label{twobjstrmapseqn}
		m_t^{\Tw\mc A} = \sum_{b,d\in \bb N } m^{\mf{gl}(\Z\mc A)}_{t+b+d}(\delta^{\otimes b} \otimes \id \otimes \eta^{\otimes d}) \ , 
	\end{equation}
	where $m^{\mf{gl}(\Z\mc A)}_{t+b+d}$ denotes the analogous extension of the $A_\infty$ structure maps on $\Z\mc A$ to matrices with coefficients in $\Z\mc A$, and we identify $\Tw\mc A ((i_1,...,i_d,\delta), (j_1,...,j_c,\eta))$ with the corresponding subspace of the space of $d\times c$ matrices with coefficients in $\Z\mc A$.
\end{defn}

There is a natural, strict functor of $\Ainf$ categories
\[ Y_1:\mc A \to \Tw\mc A \quad\quad\textup{defined by}\quad\quad i \mapsto ((i,0),0) \ , \]
together with the canonical inclusions on $\Hom$ spaces. Moreover, there is a functor of $A_\infty$ categories $Y_2:\Tw\mc A \to \CC_\infty(\mc A)$ defined by
\begin{equation}\label{twobjrealfuneqn}
	(i_1,...,i_r \in \Z\mc A,\ \delta=(\delta_{kl}\in \Z\mc A(i_k,i_l)[1])_{k,l=1}^d) \mapsto  \mc A_{i_1,...,i_r}^\delta :=\left(  \mc A_{i_1,...,i_d}:=\bigoplus_{k=1,...,d} \mc A_{i_k}[n_k] \ , \ (\rho_t)_{t\in \bb N} \right)  
\end{equation}
where
\begin{equation}\label{twobjrealfunstrmapseqn}
	\rho_t = \sum_{k\in \bb N}  (-1)^{\frac{t(t-1)}{2}}\rho^{\mc A_{i_1,...,i_d}}_{t,k}( \id_{\mc A}^{\otimes t}\otimes \id_{\mc A_{i_1,...,i_d}}\otimes \delta^{\otimes k} )\ : \mc A^{\otimes t}\otimes \mc A_{i_1,...,i_d} \to \mc A_{i_1,...,i_d} [1-t]  \ ; 
\end{equation}

\noindent here we let $\mc E_{\mc A}$ denote the $A_\infty$ category with objects and $\Hom$ spaces the same as $\Tw \mc A$ but with $\Ainf$ structure maps given only by those for $\mf{gl}(\bb Z \mc A)$, and note that $\delta \in \mc E_{\mc A}^1((i_1,...,i_d),(i_1,...,i_d))$ and moreover that $\mc A_{i_1,...,i_d}\in \CC^\infty(\mc A,\mc E_{\mc A})$ is naturally an $A_\infty$ bimodule over $(\mc A,\mc E_{\mc A})$ by Yoneda, with structure maps denoted
\begin{equation}\label{bimodstreqn}
	\rho^{\mc A_{i_1,...,i_d}}_{t,k}: \mc A^{\otimes t}\otimes \mc A_{i_1,...,i_d} \otimes \mc E_{\mc A}^{\otimes k} \to  \mc A_{i_1,...,i_d}[1-t-k] . 
\end{equation}

The main result about these constructions is the following:
\begin{theo}\cite{Lef}\label{twobjthm} The $\Ainf$ Yoneda embedding of Equation \ref{Ainfyoneqn} admits a factorization
	\[ \xymatrix{ \mc A \ar[r]^{Y_1} \ar[dr]^{Y} & \Tw\mc A \ar[d]^{Y_2} \\ & \CC_\infty(\mc A) } \]
	inducing an equivalence of triangulated categories
	\begin{equation}\label{twobjequiv}
		H^0(Y_2): H^0(\Tw \mc A) \to \Triang(\mc A_i) 
	\end{equation}
	where $\textup{Triang}(\mc A_i)$ denotes the triangulated subcategory of $\DD(A)$ generated by the objects $\mc A_i$.
\end{theo}

Further, we introduce the following categories

\begin{defn} Let $\mc A$ be an $\Ainf$ category. The $\Ainf$ category $\Tw^0\mc A$ is defined as the full subcategory of $\Tw \mc A$ on objects $((i_1,...,i_d),\delta)\in\Tw\mc A$ such that for each $k=1,...,d$ the corresponding $i_k\in \mc A \subset \bb Z \mc A$, that is, has associated integer $n_k=0$. 
\end{defn}

Concretely, $\Tw^0 \mc A$ is the full subcategory on objects whose image under the functor $Y_2$ has underlying vector space given by a direct sum of the indecomposable summands of the rank 1 free module, that is
\[ \mc A_{i_1,...,i_d}=\bigoplus_{k=1,...,d} \mc A_{i_k}  . \]

\begin{defn}\label{Filtdefn} Let $M_i\in \mc D$ be a collection of objects in a triangulated category $\mc D$. Then $\Filt (M_i)$ is defined as the full subcategory of $\mc D$ on objects admitting a filtration with subquotients in the collection $M_i$.
\end{defn}

Then in addition, we have

\begin{corollary}\cite{Lef}\label{twobjfiltcoro} The equivalence of Equation \ref{twobjequiv} induces an equivalence of full subcategories
	\begin{equation}\label{twobjequivab}
		H^0(\Tw^0 \mc A) \xrightarrow{\cong} \Filt (\mc A_i)
	\end{equation}
\end{corollary}

Applying the Morita theory results of Section \ref{tiltingsec}, we also obtain:

\begin{corollary}\cite{Lef} Let $\mc A_M$ be the $\Ainf$ category determined by the $\Ainf$ algebra $A_M=\Ext^\bullet_\D(M,M)$ for $M=\oplus_{i\in I} M_i \in \D$ a direct sum of objects in a triangulated category $\mc D$, as in Example \ref{AinfcatSeg}. Then the composition of the equivalence of triangulated categories of Equation \ref{twobjequiv} above with that of Equation \ref{dgmoritaeqntriang} gives an equivalence
	\[ H^0(\Tw\mc A_M) \xrightarrow{\cong} \Triang(M_i) \quad\quad \text{and in turn}\quad\quad H^0(\Tw^0\mc A_M) \xrightarrow{\cong} \Filt(M_i) \ . \]
\end{corollary}

We now describe the corresponding graded variant of the twisted objects category and its functor to $\CC_\infty(\mc A)$:

\begin{defn} Let $\mc A$ be a graded $\Ainf$ category and define the graded $\Ainf$ category $\bb Z \mc A$ by
	\[ \ob(\Z\mc A) = \{(i,n,p)\ |\  i\in \ob(\mc A), n,p\in \Z \} \quad\quad \text{and}\quad\quad \bb Z\mc A((i,n,p),(j,m,q))= \mc A (i,j)[m-n]\langle q-p\rangle \]
	together with the natural extension of the graded $A_\infty$ structure maps of $\mc A$.	
\end{defn}

\begin{warning}
	As above,  we will often omit the integers $n,p\in \Z$ from the notation and write simply $i\in \Z \mc A $ and $\Z\mc A (i,j)$ with the associated integers left implicit.
\end{warning}

Following Definition \ref{twobjdefn}, we make the following definition:
\begin{defn}
	Let $\mc A $ be a graded $\Ainf$ category. The category $\tw\mc A $ of twisted objects over $\mc A $ is the graded $\Ainf$ category defined as follows:
	an object of the category $\tw \mc A $ is given by a finite collection of objects \[(i_1,n_1,p_1),...,(i_r,n_d,p_d)\in \bb Z \mc A  \ , \]
	together with a degree zero element
	\[ \delta = (\ \delta_{kl} \in  \bb Z\mc A(i_k,i_l)[1]  =\ \mc A({i_k},{i_l}) [n_l-n_k+1]\langle p_l-p_k\rangle \ )_{k,l=1}^d  \ \in \ \mf{gl}_d(\bb Z \mc A)[1] \]
	such that $\delta_{kl}=0 $ for $k\leq l$ and moreover $\delta$ satisfies the Maurer-Cartan equation
	\[ \sum_{t\in \bb N} m_t^{\mf{gl}_d(\Z\mc A)}(\delta^{\otimes t}) = 0 \ . \]
	The graded vector spaces of homomorphisms and $\Ainf$ category structure maps are defined as in Equations \ref{twobjhomspceqn} and \ref{twobjstrmapseqn}, respectively.
\end{defn}

There is an analogous natural functor of graded $\Ainf$ categories $Y_2:\tw \mc A  \to \CC_\infty(\mc A ) $ to the category of graded $\Ainf$ modules over $\mc A$, defined by
\[  (i_1,...,i_d \in \bb Z \mc A,\ \delta=(\delta_{kl})_{k,l=1}^d) \mapsto  \mc A_{i_1,...,i_d}^\delta :=\left(  \mc A_{i_1,...,i_d}:=\bigoplus_{k=1,...,d} \mc A_{i_k}[n_k]\langle p_k\rangle  \ , \ (\rho_t)_{t\in \bb N} \right)   \ , \]
as in Equation \ref{twobjrealfuneqn}, where
\begin{equation}\label{twobjstrmapgradedeqn}
	\rho_t = \sum_{k\in \bb N}  (-1)^{\frac{t(t-1)}{2}}\rho^{\mc A_{i_1,...,i_d}}_{t,k}( \id_{\mc A }^{\otimes t}\otimes \id_{\mc A_{i_1,...,i_d}}\otimes \delta^{\otimes k} )\ : \mc A^{\otimes t}\otimes \mc A_{i_1,...,i_d} \to \mc A_{i_1,...,i_d} [1-t]  \ , 
\end{equation}
as in Equation \ref{twobjrealfunstrmapseqn}, and $\rho_{t,k}^{\mc A_{i_1,...,i_d}}:\mc A^{\otimes t} \otimes \mc A_{i_1,...,i_d} \otimes \mc E_{\mc A}^{\otimes k} \to \mc A_{i_1,...,i_d}[1-t-k]$ are the $\Ainf$ category bimodule structure maps, as in Equation \ref{bimodstreqn}.

\begin{rmk} All statements in the preceding paragraph are given in terms of the natural, unsheared grading, in the sense of Remark \ref{cohomshearrmk}. We will continue to work in these conventions throughout this section, explicitly adding the superscript $\mc A^\sh$ to denote the cohomologically sheared variants when necessary.
\end{rmk}

We have the following natural analogue of Theorem \ref{twobjthm} in this case:

\begin{corollary}\label{grtwobjequiv} The functor $Y_2:\tw \mc A  \to \CC_\infty(\mc A)$ induces an equivalence of triangulated categories
	\begin{equation}\label{grtwobjequiveqn}
		H^0(Y_2): H^0(\tw \mc A ) \to \triang(\mc A_i)  \ ,
	\end{equation}
	where we recall that $\triang(\mc A_i)$ denotes the minimal triangulated subcategory of $\DD_\ZZ(A)$ containing the objects $\mc A_i\langle k \rangle$ for $i\in I$ and $k\in \bb Z$.
\end{corollary}

\begin{defn}
	Let $\mc A$ be a graded $\Ainf$ category. The graded $\Ainf$ category $\tw^0\mc A$ is defined as the full subcategory on objects $(i_1,...,i_d,\delta)\in \bb Z \mc A$ such that each of the associated integers $n_k=0$ for $k=1,...,d$.
\end{defn}

Concretely, $\tw^0 \mc A$ is the full subcategory on objects whose image under the functor $Y_2$ has underlying graded vector space given by a direct sum of graded shifts of the indecomposable summands of the rank 1 free module, that is
\[ \mc A_{i_1,...,i_d}=\bigoplus_{k=1,...,d} \mc A_{i_k}\langle p_k \rangle\ , \quad\quad \text{or equivalently} \quad\quad   \mc A_{i_1,...,i_d}^\sh=\bigoplus_{k=1,...,d} \mc A_{i_k}^\sh[p_k]\langle -p_k \rangle , \]
in the cohomologically sheared conventions.

\begin{defn}\label{filtdefn}
	Let $M_i\in \mc C$ be a collection of objects in a mixed category $\mc C$. The category $\filt(M_i)$ is defined as the full subcategory of $\mc C$ on objects admitting a filtration with subquotients in the collection of objects $M_i\langle k \rangle \in \mc C$ for $i\in I$ and $k\in\bb Z$.
\end{defn}

We have the following natural analogue of Corollary \ref{twobjfiltcoro}:
\begin{corollary}\label{twobjgrfiltcoro} The equivalence of Equation \ref{grtwobjequiveqn} induces an equivalence of the full subcategories
	\[ H^0(\tw^0\mc A ) \xrightarrow{\cong} \filt(\Sigma_i) \ .\]
\end{corollary}

\section{Perverse coherent sheaves on Calabi-Yau threefolds and quivers with potential}\label{unextsec}

\subsection{Overview of Section \ref{unextsec}}\label{unextovsec} In Section \ref{unextsec}, we explain the proof of a folklore theorem relating coherent sheaves on a certain class of toric Calabi-Yau threefold resolutions $Y\to X$ with representations of an unframed quiver with potential $(Q_Y,W_Y)$, which follows from results of Bridgeland \cite{Bdg1} and Van den Bergh \cite{VdB1} together with some general facts about triangulated categories. This is a precursor to Theorem \ref{Athm} from the introduction, and is simply the special case $M=0$.

In Section \ref{Geosetupsec}, we explain the basic hypotheses on the resolution $Y\to X$ used throughout this paper, and two corresponding descriptions of the derived category of complexes of coherent sheaves induced by two natural collections of generators. In Section \ref{Kozpatsec}, we explain that the equivalence between these two descriptions is an example of Koszul duality analogous to that studied in \cite{BGS}, and deduce the corresponding equivalences of hearts following \emph{loc. cit.} and the results of \cite{Bdg1} and \cite{VdB1}, assuming the tilting algebra is Koszul. In Section \ref{monadsec}, we use this perspective to define a canonical monad presentation of the category of perverse coherent sheaves, and in Section \ref{Ainfkozsec}, we generalize the results of the previous two sections to the case that the tilting algebra is not necessarily Koszul and correspondingly the $\Ext$ algebra of the simple objects is $\Ainf$. In Section \ref{quiversecunfr}, we apply these results to deduce the desired precursor to Theorem \ref{Athm}, and in Section \ref{egunfrsec} we explain the results in several concrete examples. Finally, in Section \ref{beilressec} we explain the relationship of this construction to the well-known arguments using the Beilinson spectral sequence induced by a resolution of the diagonal.

\subsection{Tilting objects on toric Calabi-Yau threefolds}\label{Geosetupsec}

In this subsection, we fix the primary objects of interest and hypotheses required for the main results of the paper.

Let $X$ be an affine, toric Calabi-Yau threefold singularity, $T$ denote the associated torus, and let $Y\xrightarrow{\pi}X$ be a toric resolution of singularities, such that
\begin{enumerate}
	\item  the fibres of $\pi$ are of dimension $\leq 1$, and
	\item $\pi_* \mc O_Y\cong \mc O_X$.
\end{enumerate}
We assume that $X$ has a (unique) $T$-fixed point $x\in X(\C)$, and let $C=Y\times_X \{x\}$ be the scheme theoretic fibre of $\pi$ over $x$. Under these hypotheses, we have that $H^0(C,\mc O_C)=\bb K$, $C$ is Cohen-Macaulay, and $C_\red$ is either a point or an arithmetic genus zero union of projective lines with normal crossings intersections. We let
\[C=\bigcup_{i\in I_+} C_i\]
denote the decomposition of $C$ into irreducible components $C_{i}$, the index set of which we denote $I_+$, and let $\iota:C\to Y$ and $\iota_i:C_i\to Y$ denote the inclusion maps. Although many statements hold without this hypothesis, we also assume
\begin{enumerate}\setcounter{enumi}{2}
	\item the multiplicity of each $C_i$ in $C$ is one.
\end{enumerate}

Let $\hat X = X^\wedge_{x}$ the formal completion of $X$ at $x$, and $\hat Y = Y^\wedge_C$ the formal completion of $Y$ along $C$. Then we have:

\begin{prop}\label{Picprop} There are natural isomorphisms
\[ \xymatrixrowsep{.5pc} \xymatrix{ \Pic(\hat Y) \ar[r]^{\cong} &\Pic(C) \ar[r]^\cong & \bb Z^{I_+} \\
 \mc L \ar@{|->}[r] & \iota^* \mc L \ar@{|->}[r] & (\deg \iota_i^*\mc L)_{i\in I_+} } \ .
 \]	
\end{prop}
\begin{proof}
	This follows from the proof of Lemma 3.4.3 in \cite{VdB1}.
\end{proof}

In particular, there exist line bundles $\hat E_i\in \Pic(\hat Y)$ for each $i\in I_+$ with the property that
\[ \deg \iota_i^* \hat E_j = \delta_{ij}  \ .\]
In addition, we let $\hat E_0=\mc O_{\hat Y}$, $I=I_+\cup \{0\}$ and $\hat E=\oplus_{i\in I} \hat E_i$

We fix once and for all a choice of $T$-equivariant structure on $\hat E_i$ for each $i\in I$, and moreover we fix extensions to global line bundles
\[E_i \in \Pic(Y)^T \quad\quad\text{for each $i\in I$, and} \quad\quad E = \bigoplus_{i\in I} E_i  \ \in \Coh(Y)^T \ , \]
\noindent recalling that every line bundle on a smooth toric variety admits a $T$-equivariant structure.

In Section \ref{NCCRsec}, we recalled the existence of an exotic t-structure on $\DD^b\Coh(Y)$, defined in terms of $\pi:Y\to X$ satisfying a subset of the hypotheses of the present setting. The heart of this t-structure was the abelian category $\Perv(Y)=\Perv(Y/X)$ of perverse coherent sheaves on $Y$ (relative to $f:Y\to X$), in the sense of Definition \ref{PervCohdefn}.

\begin{theo} \cite{VdB1} The object $E$ is a projective generator for $\Perv(Y)$, and similarly on $\hat Y$ the object $\hat E$ is a projective generator for $\Perv(\hat Y)$.
\end{theo}
\begin{proof} This is a special case of Proposition 3.2.5, and Theorem 3.5.5 respectively, in \cite{VdB1}.
\end{proof}

The object $E$ is the vector bundle appearing in the statement of Theorem \ref{derivedVdBequiv}; it is evidently compact and thus defines a classical tilting object, in the sense of Definition \ref{tiltingdefn}, so that letting
\[\Lambda = \Hom_{\DD^b\Coh(Y)}(E,E) \quad\quad \text{and}\quad\quad \hat\Lambda = \Hom_{\DD^b\Coh(\hat Y)}(\hat E,\hat E) ,\]
we have:
\begin{corollary}\label{EMoritacoro} There are triangle equivalences
	\begin{equation*}
		\DD\QC(Y) \xrightarrow{\cong} \DD(\Lambda)  \quad \quad \text{and}\quad \quad  \DD^b\Coh(Y) \xrightarrow{\cong} \DD_\Perf(\Lambda)  \end{equation*}
	intertwining the forgetful functor $\DD(\Lambda) \to \DGV$ with $\Hom_{\DD\QC(Y)}( E,\cdot)$, and analogous equivalences for $\hat Y$ and $\hat \Lambda$.
\end{corollary}
\begin{proof}
	The claim follows from applying the Morita theory from \cite{Kel2} recalled in Theorem \ref{Morita} to the object $E \in \DD\QC(Y)$, and similarly for $\hat E \in \DD\QC(\hat Y)$.
\end{proof}

 The induced $T$-equivariant structure on $E$ makes $\Lambda$ into a trigraded algebra, and we let $\textup{D}_\perf(\Lambda)$ denote the thick subcategory of the derived category $\DD_{\bb Z^3}(\Lambda)$ of trigraded modules generated by $\Lambda$ and its shifts, generalizing the notation introduced in Section \ref{Ainfalgsec}. Then we have:
\begin{corollary}\label{EMoritagrcoro} There are triangle equivalences 
	\begin{equation*}
		\DD\QC(Y)^{T} \xrightarrow{\cong} \DD_{\ZZ^3}(\Lambda) \quad \quad \text{and}\quad \quad \DD^b\Coh(Y)^{T} \xrightarrow{\cong} {\textup{D}_\perf}(\Lambda)    \end{equation*}
	intertwining the forgetful functor $\DD_{\ZZ^3}(\Lambda) \to \DGV$ with $\Hom_{\DD\QC(Y)}( E,\cdot)$.
\end{corollary}
\begin{proof} The claim follows by including the gradings induced by equivariant structures in the equivalences of Corollary \ref{EMoritacoro}.
\end{proof}

We define a new family of objects $F_i\in \Perv(Y)$ for each $i\in I$ given by
\[ F_0 = \iota_* \mc O_C \quad\quad \text{and} \quad\quad F_i= \iota_{i *} \mc O_{C_i}(-1)[1] \ ,\]
for $i\in I_+$. It is straightforward to check these objects lie in the heart of the perverse coherent t-structure, by Corollary \ref{PervCohcoro}.

\begin{prop}\label{simpleobjprop} The objects $E_i$, $F_j$ for $i,j\in I$ satisfy
\[ \Hom_{\DD^b\Coh(Y)}( E_i,F_j) = \begin{cases}
\K & \textup{ if $i=j$, and} \\
 0 & \textup{ if $i\neq j$,}
\end{cases}
\]
similarly for $\hat Y$, and the objects $F_i$ for $i\in I$ are the unique simple objects in $\Perv(\hat Y)$. 
\end{prop}
\begin{proof} This follows from the proof of Proposition 3.5.7 in \cite{VdB1}.
\end{proof}

We fix once and for all a $T$-equivariant structure on each object $F_i$, compatible with those on $E_i$ in the sense that the preceding proposition holds as graded vector spaces, that is, such that the one dimensional $\Hom$ space is in graded degree zero. Further, we define
\begin{equation}\label{Fdefn}
 	F = \oplus_{i\in I} F_i \ \in \DD^b\Coh(Y)^{T} \ .
\end{equation}

Note that by Proposition \ref{simpleobjprop}, the images of the objects $F_i$ and $F\in \DD^b\Coh_\cs(\hat Y)$, and their $T$-equivariant enhancements in $\DD^b\Coh(Y)^{T}$, under the equivalences of Corollaries \ref{EMoritacoro} and \ref{EMoritagrcoro}, respectively, define one dimensional simple modules
$$ S_i= \bb K  \ \in \DD_\Fd(\hat \Lambda) \quad\quad \text{and}\quad\quad S =\oplus_i S_i \ \in \DD_\Fd(\hat \Lambda)   \ , $$
and their analogues in $\DD_\fd(\Lambda)$. This allows us to deduce the following descriptions of the thick subcategories generated by $F$: Let $\Coh_\cs(Y)$ denote the full subcategory of compactly supported coherent sheaves, and $\DD^b\Coh_\cs(Y)$ the derived category of complexes with compactly supported cohomology.

\begin{prop}\label{csequivprop} The equivalences of Corollaries \ref{EMoritacoro} and \ref{EMoritagrcoro} induce equivalences
	\[ \DD^b\Coh_\cs(\hat Y)\cong \Thick(F) \xrightarrow{\cong} \DD_\Fd(\hat \Lambda)\quad \quad \text{and}\quad \quad  \DD^b\Coh_\cs(Y)^{T} \cong \thick(F) \xrightarrow{\cong} \DD_\fd(\Lambda) \ , \]
	respectively, intertwining the forgetful functor $\DD(\Lambda) \to \DGV$ with $\Hom_{\DD\QC(Y)}( E,\cdot)$.
\end{prop}
\begin{proof}
The image under the natural functor $\Hom_{\DD\QC(\hat Y)}( E,\cdot)$ of any coherent sheaf with proper support is finite dimensional, so the image of $\DD^b\Coh_\cs(\hat Y)$ under the equivalence of Corollary \ref{EMoritacoro} is contained in $\DD_\Fd(\hat \Lambda)$, and similarly for the graded case.

The image of the thick subcategory $\Thick(F)$ generated by $F$ is identified with $\Thick(S)\subset \DD_\Fd(\hat \Lambda)$, but this inclusion is an equivalence as every finite dimensional $\hat \Lambda$ module admits a Jordan-Holder filtration with simple subquotients, and the summands $S_i$ of $S$ are precisely the simple modules in $\Lambda\Mod_\textup{Fg}$ by the proof of Proposition \ref{simpleobjprop}. It follows that $\Thick(F)$ is equivalent to $ \DD^b\Coh_\cs(\hat Y)$ and its image is equivalent to $\DD_\Fd(\hat \Lambda)$, and similarly for the graded case.
\end{proof}

The object $F=\oplus_{i\in I} F_i\in \DD^b\Coh(Y)^{T}$ is evidently compact, so that letting
\[ \Sigma = \Hom_{\DD^b\Coh(Y)}(F,F) , \]
equipped with the grading defined by the equivariant structure on $F$, we have:

\begin{corollary}\label{cctmoritacoro} There are equivalences of triangulated categories
\[ \DD^b\Coh_\cs(\hat Y) \xrightarrow{\cong} \DD_\Perf(\Sigma)\quad \quad \text{and}\quad \quad  \DD^b\Coh_\cs(Y)^{T} \xrightarrow{\cong} {\textup{D}_\perf}(\Sigma) \ , \]
intertwining the forgetful functor $\DD(\Sigma) \to \DGV$ with $\Hom_{\DD\QC(Y)}( F,\cdot)$.
\end{corollary}
\begin{proof}
	This follows from the Morita theory from \cite{Kel2} recalled in Theorem \ref{Morita} applied to the object $F$, together with the identifications of $\Thick(F)$ and $\thick(F)$ with the compactly supported derived categories in Proposition \ref{csequivprop}.
\end{proof}

\subsection{Koszul duality patterns in equivariant enumerative geometry}\label{Kozpatsec}

The relationship between the two descriptions of $\DD^b\Coh_\cs(Y)$ given in Proposition \ref{csequivprop} and Corollary \ref{cctmoritacoro}, in terms of modules over the associative algebras $\Lambda$ and $\Sigma$, is an example of Koszul duality. These two algebras are the analogues of the endomorphisms of the projective generator and the $\Ext$ algebra of the direct sum of simple modules in the BGG category $\mc O$, respectively, whose Koszul duality was studied in the seminal paper of Beilinson-Ginzburg-Soergel \cite{BGS}.

Note that the module structure map for the object $S\in \DD_\Fd(\hat \Lambda)$ (and its graded enhancement $S\in \DD_\fd(\Lambda)$) defines a augmentation $\epsilon:\hat \Lambda\to S$ of $\hat\Lambda$ considered as an algebra over the base ring $\hat \Lambda_0=S$ (and similarly a graded augmentation of $\Lambda$). Moreover, we can identify the Koszul dual algebras of $\hat \Lambda$ and $\Lambda$ with $\Sigma$, as we have
$$ \hat \Lambda^! = \Hom_{\DD(\hat \Lambda)}(S,S) \cong \Hom_{\DD^b\Coh(\hat Y)}(F,F) = \hat \Sigma \quad\quad \text{and}\quad\quad  \Lambda^!\cong  \Sigma  $$
similarly, where $\hat \Sigma$ denotes $\Sigma$ as a plain (ungraded) algebra while $\Sigma$ denotes the algebra equipped with its natural trigrading defined by the $T$-equivariant structure on $F$. In terms of this notation, we have:

\begin{prop}\label{SigmaKozprop} There are canonical quasi-isomorphisms
	\begin{equation*}
		\hat\Sigma^! \cong \hat \Lambda \quad\quad\text{and}\quad\quad  \Sigma^! \cong \Lambda \ .
	\end{equation*}
\end{prop}

\begin{warning}
	We will continue to use the notation $\Sigma$ for both the graded and ungraded algebra; the notation for the graded and ungraded derived categories already makes this distinction, and we will explicitly indicate the grading whenever there is an ambiguity.
\end{warning}

\begin{proof} In the graded case, both $\Sigma$ and $\Lambda$ are strongly locally finite, so that by Theorem \ref{Kozsqtheo} we have $\Sigma^! \cong (\Lambda^!)^! \cong \Lambda$.
	
In the ungraded case, by Proposition \ref{Kozpresprop} we have
\[ \Sigma^! \cong ( \otimes^\bullet_S(\bar\Sigma[1]))^\vee \cong \widehat {\otimes}^\bullet_S (\bar \Sigma^\vee[1]) \ , \]	
that is, the bar complex presentation of the Koszul dual of $\Sigma$ (with the trivial grading) is given by the completion of the tensor algebra with respect to the natural augmentation ideal. In the graded case, the augmentation ideal agrees with the ideal $\Lambda_{>0}$ of strictly positively graded elements, and thus the completion agrees with that induced by completion at the unique fixed point $x\in X(\C)$, as desired.
\end{proof}

Finally, we note that the $T$-equivariant structure on $E$ can be chosen so that $\Lambda$ is strongly locally finite, in the sense of Definition 2.1 of \cite{LPWZ}, as a trigraded algebra. We have:

\begin{corollary}\label{Kozequivcoro}
	There are mutually inverse equivalences of categories
	\begin{align} \Hom_{\hat \Lambda}(S,\cdot) &: \xymatrix{\DD_\Fd(\hat \Lambda)  \ar@<.5ex>[r]^{\cong} & \ar@<.5ex>[l] \DD_\Perf(\Sigma)}: (\cdot)\otimes_{\Sigma} S  \quad\quad \ , \label{Kozequivcsungreqn} \\
		\Hom_{\Lambda}(S,\cdot) &: \xymatrix{\DD_\fd(\Lambda)  \ar@<.5ex>[r]^{\cong} & \ar@<.5ex>[l] {\textup{D}_\perf}(\Sigma)}: (\cdot)\otimes_{\Sigma} S  \quad\quad\ \  ,    \label{Kozequivcseqn} \\
  (\cdot)\otimes_{\Lambda} S &: \xymatrix{ \DD_\Perf(\hat \Lambda) \ar@<.5ex>[r]^{\cong} &  \ar@<.5ex>[l]  \DD_\Fd(\Sigma)}: \Hom_{\Sigma}(S,\cdot) \ \ \text{ , and} \label{Kozequivungreqn} \\
   (\cdot)\otimes_{\Lambda} S &: \xymatrix{ {\textup{D}_\perf}(\Lambda) \ar@<.5ex>[r]^{\cong} & \ar@<.5ex>[l]  \DD_\fd(\Sigma)}: \Hom_{\Sigma}(S,\cdot) \quad\  .  \label{Kozequiveqn} 
 \	\end{align}
\end{corollary}
\begin{proof}
In the graded case, both $\Sigma$ and $\Lambda$ are strongly locally finite by construction, so that Theorem \ref{Kozeqtheo} applies with $A=\Lambda$ or $A=\Sigma$ to give the two claimed equivalences of categories of graded modules. The first equivalence in the ungraded case follows from applying the same result with the weaker hypotheses of Remark \ref{Kozequivrmk}, noting that $\Lambda$ is always a plain associative algebra by projectivity of $E$ in $\Perv(Y)$, and that $\Sigma$ is finite dimensional and thus locally finite since $F$ has proper support. The latter ungraded equivalence follows again by applying DG Morita theory to the object $S \in  \DD_\Fd(\Sigma)$ since $\hat \Lambda = \Hom_\Sigma(S,S)$ by Proposition \ref{SigmaKozprop} above.
\end{proof}

Note that we have an inclusion of full subcategories
\[ \DD_\fd(\Lambda) \subset {\textup{D}_\perf}(\Lambda)  \quad\quad \text{corresponding to}\quad\quad \DD^b\Coh_\cs(Y)^{T}\subset \DD^b\Coh(Y)^{T} \ , \]
and similarly ${\textup{D}_\perf}(\Sigma) \subset \DD_\fd(\Sigma)$, since $\Sigma$ is finite dimensional. However, the equivalences of Equations \ref{Kozequivcseqn} and \ref{Kozequiveqn} evidently do not make the natural diagram commute. This minor technical obstruction is resolved by introducing the Nakayama functor:

The Nakayama functor is the t-exact auto-equivalence defined by
\begin{equation}\label{Nakayamaeqn}
	(\cdot)^N:=(\cdot)\otimes_\Sigma \Hom_S(\Sigma,S) : \DD_\fd(\Sigma)\to \DD_\fd(\Sigma) \ , 
\end{equation}
which has the property that it defines an equivalence between the full subcategories of (complexes of) projective and injective objects; see Proposition 7 of \cite{MOS} and Lemma 4.5.6 of \cite{Zim}.

The main result of this subsection, which summarizes several of the results above and their compatibilities, is the following:

\begin{theo}\label{catthmunext} The diagram of triangulated categories
		\begin{equation}\label{Kozgeoextendedungreq}
		\xymatrixcolsep{5pc}
		\xymatrix{	\DD^b\Coh_\cs(\hat Y) \ar@<.5ex>[r]^{\Hom_{\hat Y}(F,\cdot)\otimes_{\Sigma} S}\ar[d]^{\subset} & \ar@<.5ex>[l]^{\Hom_{\hat\Lambda}(S,\cdot)\otimes_{\Sigma} F }  \DD_\Fd(\hat \Lambda)  \ar@<.5ex>[r]^{ \Hom_\Lambda(S,\cdot) } \ar[d]^{\subset}  & \ar@<.5ex>[l]^{(\cdot)\otimes_\Sigma S } \DD_\Perf(\Sigma)  \ar[d]^{(\cdot)^N\circ \subset}	\\
			\DD^b\Coh(\hat Y) \ar@<.5ex>[r]^{\Hom_{\hat Y}(E,\cdot)}  &  \ar@<.5ex>[l]^{E\otimes_{\hat\Lambda}(\cdot) }  \DD_\Perf(\hat \Lambda)    \ar@<.5ex>[r]^{ (\cdot)\otimes_{\Lambda} S}  & \ar@<.5ex>[l]^{\Hom_{\Sigma}(S,\cdot)}  \DD_\Fd(\Sigma) & \textup{,}		 } \end{equation}	
	has horizontal arrows mutually inverse triangle equivalences and vertical arrows inclusions of thick subcategories, and admits canonical commutativity data, and similarly for the diagram
	\begin{equation}\label{Kozgeoextendedeq}
		\xymatrixcolsep{5pc}
	\xymatrix{	\DD^b\Coh_\cs(  Y)^{T} \ar@<.5ex>[r]^{\Hom_{ Y}(F,\cdot)\otimes_{\Sigma} S}\ar[d]^{\subset} & \ar@<.5ex>[l]^{\Hom_{ \Lambda}(S,\cdot)\otimes_{\Sigma} F }  \DD_\fd(  \Lambda)  \ar@<.5ex>[r]^{ \Hom_\Lambda(S,\cdot) } \ar[d]^{\subset}  & \ar@<.5ex>[l]^{(\cdot)\otimes_\Sigma S } {\textup{D}_\perf}(\Sigma)  \ar[d]^{(\cdot)^N\circ \subset}	\\
		\DD^b\Coh(  Y)^{T} \ar@<.5ex>[r]^{\Hom_{Y}(E,\cdot)}  &  \ar@<.5ex>[l]^{E\otimes_{ \Lambda}(\cdot) }  {\textup{D}_\perf}(\Lambda)    \ar@<.5ex>[r]^{ (\cdot)\otimes_{\Lambda} S}  & \ar@<.5ex>[l]^{\Hom_{\Sigma}(S,\cdot)}  \DD_\fd(\Sigma) & \textup{.}		 } \end{equation}			
	
\end{theo}

To complete the proof, we need the following lemma:
 
\begin{lemma} The diagrams of triangulated categories
	\begin{equation}	\xymatrixcolsep{5pc}
		\xymatrixrowsep{2pc} \xymatrix{ \DD_\Fd(\hat \Lambda) \ar[r]^{\Hom_{\hat\Lambda}(S,\cdot) }\ar[d]^{\subset}  & \DD_\Perf(\Sigma) \ar[d]^{(\cdot)^N}  \\  {\textup{D}_\perf}(\Lambda) \ar[r]^{(\cdot)\otimes_\Lambda S} & \DD_\fd(\Sigma) } 
		\quad\quad \text{and}\quad\quad 
		\xymatrixcolsep{5pc}
		\xymatrixrowsep{2pc} \xymatrix{ \DD_\fd(\Lambda) \ar[r]^{\Hom_\Lambda(S,\cdot) }\ar[d]^{\subset}  & {\textup{D}_\perf}(\Sigma) \ar[d]^{(\cdot)^N}  \\  {\textup{D}_\perf}(\Lambda) \ar[r]^{(\cdot)\otimes_\Lambda S} & \DD_\fd(\Sigma) }   \ .
	\end{equation}
	admit canonical commutativity data.
\end{lemma}
\begin{proof} The inverse equivalences commute up to the canonical natural isomorphism
	\begin{align}\label{Nakayamapropeqn} \Hom_\Sigma(S,\cdot) \circ N & =  \Hom_\Sigma(S, (\cdot)\otimes_\Sigma \Hom_S(\Sigma,S))  \nonumber\\
		&   \xrightarrow{\cong}(\cdot)\otimes_\Sigma\Hom_\Sigma(S, \Hom_S(\Sigma,S)) \nonumber \\
		&   \xrightarrow{\cong}(\cdot)\otimes_\Sigma\Hom_S(S\otimes_\Sigma \Sigma,S)  \nonumber \\
		&   \xrightarrow{\cong} (\cdot)\otimes_\Sigma S 	\ . \nonumber
	\end{align}
\end{proof}

\begin{proof}{(of Theorem \ref{catthmunext})} 
	The arrows in the bottom rows define mutually inverse equivalences of categories by Corollaries \ref{EMoritacoro} and \ref{EMoritagrcoro}, as well as Equations \ref{Kozequivungreqn} and \ref{Kozequiveqn} of Corollary \ref{Kozequivcoro}. Commutativity of the squares on the right in each diagram follows from the preceding lemma. This in turn implies commutativity of the squares on the left of each diagram, by Proposition \ref{csequivprop} and Corollary \ref{cctmoritacoro}. It follows the vertical arrows must all be inclusions of thick subcategories.
\end{proof}

Thus, we have obtained three equivalent descriptions of the same triangulated category
\begin{equation}\label{Kozgeoeqiv}	\xymatrixcolsep{3pc}
\xymatrixrowsep{.5pc}
\xymatrix{\DD^b\Coh(Y)^{T} \ar@<.5ex>[r]^{\cong} &   {\textup{D}_\perf}(\Lambda)  \ar@<.5ex>[r]^{\cong} & \DD_\fd(\Sigma) &  \textup{under which} \\
M \ar@{|->}[r] & \Hom(E,M) \ar@{|->}[r] & \Hom(F,M)^N & , \\
E_i \ar@{|->}[r] & P_i \ar@{|->}[r] & S_i &    \text{, and} \\
  F_i \ar@{|->}[r] & S_i \ar@{|->}[r] & I_i & \textup{,} } \end{equation}
where $P_i=\Lambda_i$ is the projective $\Lambda$ module given by the $i^{th}$ direct summand of $\Lambda$ corresponding to $E_i$. By abuse of notation, we have also used $S_i$ to denote the simple $\Sigma$ module corresponding to $P_i$ under the latter equivalence, since both underlying vector spaces are given by $S_i=\K$ and their direct sum $S=\oplus_i S_i$ is the common augmentation module used to define the Koszul duality equivalence. Finally, the object \[I_i=\Sigma_i^N=\Hom_S(\Sigma_i,S)\]
is the injective module given by the linear dual of $\Sigma_i$, where the latter denotes the $i^{th}$ direct summand of $\Sigma$ in the decomposition
$$ \Sigma = \bigoplus_{i\in I} \Sigma_i = \bigoplus_{i\in I}  \Hom_{\Lambda}(S_i,S) \ . $$

It is natural to ask which abelian subcategory of the derived category of $\Sigma$ modules is equivalent to the usual heart of ${\textup{D}_\perf}(\Lambda)$ under the Koszul duality equivalence, and in turn to perverse coherent sheaves on $Y$ relative to $X$ under the equivalence of \cite{Bdg1},\cite{VdB1}, recalled in Theorem \ref{PCohequiv} above. We now describe this category, following the general approach of \cite{BGS,MOS}.

For the initial explanation, we will assume that $\Lambda$ is Koszul, and thus $\Sigma$ is a ($\Z\times \Z^3$-graded) associative algebra, with no non-trivial higher $A_\infty$ multiplications. In the succeeding section, we generalize these results to the case where $\Lambda$ is not necessarily Koszul and correspondingly $\Sigma$ is a general ($\Z^3$-graded) $A_\infty$ algebra. Also, for simplicity of notation, we will often restrict the grading along a cocharacter $\C^\times \to T$ and state the results using the classical language of graded algebras, suppressing the trigrading unless it is necessary.

\begin{warning}\label{LambdaKozwarn} For the remainder of this section, we will assume that $\Lambda$ is Koszul.
\end{warning}

\begin{warning}\label{gradingwarn} Throughout this paper, we will often restrict the trigrading along a cocharacter $\C^\times \to T$ for notational simplicity.
\end{warning}

Note that under this hypotheses, $\Sigma$ is concentrated in bi-degree $(k,-k)$ for $k\in \bb N$, and thus its cohomologically sheared avatar $\Sigma^\sh$ is concentrated in cohomological degree $0$, in keeping with Remark \ref{cohomshearrmk}. The resulting algebra $\Sigma^\sh$ is given by simply interpreting the cohomological degree of the above $\Ext$ algebra as an abstract graded degree, as in the classical setting of \cite{BGS}.

\begin{warning}\label{Sigmashwarning}  In keeping with Remark \ref{cohomshearrmk}, we will often express $\Sigma$ and its DG modules in terms of the cohomologically sheared grading, but sometimes omit the superscript $(\cdot)^\sh$ by abuse of notation.
\end{warning}

First, recall that the grading on $\Sigma$ endows the abelian category of finite dimensional graded $\Sigma$ modules $\Sigma\Mod_\fd$ with the structure of a mixed category, in the following sense:

\begin{defn} A \emph{mixed category} is an artinian category $\mc C$ together with an integer-valued function called the \emph{weight} $w:\textup{Irr}(\mc C)\to \Z$ on the set of equivalence classes $\textup{Irr}(\mc C)$ of simple objects of $\mc C$, such that for any two simple objects $M,N\in \mc C$, we have
	\[ w(M)\leq w(N)\quad\quad \textup{implies} \quad\quad \Ext^1_{\mc C}(M,N)=0\ . \]
\end{defn}
We recall that a category $\mc C$ is called \emph{artinian} if it is abelian and every object $M \in \mc C$ admits a finite filtration with simple subquotients.

\begin{defn}\label{LCPdefn} A complex $P^\bullet=\oplus_i P^i[-i]$ of objects in a mixed category $\mc C$ is \emph{linear} if for each $i$, the simple quotient of each indecomposable summand of $P^i$ is pure of weight $i$.
\end{defn}

\noindent Let $\text{LCP}(\mc C)$ denote the full subcategory of $D^b(\mc C)$ of linear complexes of projective objects. 
In general, $\text{LCP}(\mc C)$ is an abelian subcategory of $\DD(\mc C)$, with simple objects given by indecomposable projectives of $\mc C$ concentrated in a single cohomological degree $j$. Defining the weight of such a simple object to be $j$, the category $\text{LCP}(\mc C)$ is a mixed category, with Tate twist given by $[-1]\langle 1\rangle$.

For $\mc C=A\Mod_\fg$, this is equivalent to the condition that $Q^i=A Q^i_i$ is generated in degree $-i$. Consider the full subcategories $\DD^b_\fg(A)^{\leq 0, g}$ and $\DD^b_\fg(A)^{\geq 0,g}$ of $\DD(A)$ on objects isomorphic to a complex of graded projective modules $P^\bullet$ such that $P^i$ is generated in degree $\leq -i$ and $\geq -i$, respectively. We have:

\begin{theo}\label{Kosheartthem}\cite{BGS,MOS} The pair $\DD^b_\fg(A)^{\leq 0, g},\DD^b_\fg(A)^{\geq 0,g}$ define a t-structure on $\DD^b_\fg(A)$, with heart given by $\text{LCP}(A):=\text{LCP}( A\Mod_\fg)$. Moreover, the Koszul duality functor $\Hom_A(S,\cdot):\DD^b_\fd(A) \to {\textup{D}_\perf}(A^!)$ of Theorem \ref{Kozeqtheo} restricts to an equivalence of mixed categories $K: A\Mod_\fd\to \text{LCP}(A^!)$.
\end{theo}

In particular, note that the grading shift functor $\langle - 1 \rangle$ on $A$ modules corresponds to the simultaneous cohomological and grading shift functor $[-1]\langle 1\rangle$ on linear complexes of projectives.

We now proceed to apply these results in our present context, to describe the desired t-structure on $\DD^b_\fd(\Sigma)$ corresponding to the perverse coherent t-structure on $\DD^b\Coh(Y)^{\C^\times}$.  We define a complex of injective $\Sigma$ modules to be linear if its corresponding dual complex of projective objects is linear, in the sense of Definition \ref{LCPdefn} above, and let $\textup{LCI}(\Sigma):=\text{LCP}( \Sigma)^N$ denote the full subcategory on bounded linear complexes of injective objects.

In fact, in the situation at hand of the finite dimensional algebra $\Sigma$ over the base ring $S=\oplus_{i\in I}\bb K$, we have the following equivalent form of these definitions:

\begin{defn}\cite{MOS} Let $M\in \Sigma\modd_\fd$ be a finite dimensional $\Sigma$ module and $M=\oplus_{k} M_k$ a presentation of $M$ as a direct sum of indecomposable summands $M_k$. We define the category $\textup{LC}_\Sigma(M)$ as the full subcategory of ${\textup{D}_\perf}(\Sigma)$ on complexes $N=(N^\bullet,d)$ of $\Sigma $ modules such that for each $i\in \Z$, every indecomposable summand of $N^i$ is of the form $M_k\langle i \rangle$.
\end{defn}

\noindent In particular, letting $P=\Sigma=\oplus_{i\in I} \Sigma_i$ and $I=\Sigma^\vee = \oplus_{i\in I} \Sigma^\vee_i$, we have	
\[ \textup{LCP}(\Sigma) = \textup{LC}_{\Sigma}(P) \quad\quad \text{and}\quad\quad  \textup{LCI}(\Sigma) = \textup{LC}_\Sigma(I) \ . \]

We now state the main result of this subsection:

\begin{theo}\label{unextheartthm}
Restriction to the hearts of the triangulated categories in the diagram of Equation \ref{Kozgeoextendedeq} induces a commutative diagram of mixed categories, with horizontal arrows equivalences of mixed categories and vertical arrows inclusions of full mixed subcategories
\begin{equation}\label{Kozgeoextendedhearteq}
	\xymatrixcolsep{3pc}
	\xymatrixrowsep{.5pc}
	\xymatrix{ \Perv_\cs(Y)^{T} \ar@{}[d]|-*[@]{\subset} \ar@<.5ex>[r]^{ } & \ar@<.5ex>[l]  \Lambda\Mod_\fd \ar@{}[d]|-*[@]{\subset}  \ar@<.5ex>[r]^{ } & \ar@<.5ex>[l] \textup{LCP}(\Sigma) \\
		\Perv(Y)^{T}  \ar@<.5ex>[r]^{ } & \ar@<.5ex>[l]  \Lambda\Mod_\fg } \ .
\end{equation}
Similarly, restriction to hearts in the diagram of Equation \ref{Kozgeoextendedungreq} induces
\begin{equation}\label{Kozgeoextendedheartungreq}
	\xymatrixcolsep{3pc}
	\xymatrixrowsep{.5pc}
	\xymatrix{ \Perv_\cs(\hat Y) \ar@{}[d]|-*[@]{\subset} \ar@<.5ex>[r]^{ } & \ar@<.5ex>[l]  \hat \Lambda\Mod_\Fd \ar@{}[d]|-*[@]{\subset}  \ar@<.5ex>[r]^{ } & \ar@<.5ex>[l] \Filt(\Sigma) \\
		\Perv(\hat Y)  \ar@<.5ex>[r]^{ } & \ar@<.5ex>[l]  \hat \Lambda\Mod_\Fg } \ ,
\end{equation}
where $\Filt(\Sigma)$ is as in Definition \ref{Filtdefn}.
\end{theo}
\begin{proof}

These equivalences intertwine the heart of the perverse coherent t-structure on $\DD^b\Coh(Y)$ with the usual t-structure on $\DD(\Lambda)$ by the results of \cite{Bdg1,VdB1} recalled in Section \ref{NCCRsec}. In the graded case, this is intertwined with (the image under the derived Nakayama functor of) the category of linear complexes of projectives, by the results of \cite{MOS, BGS} recalled in Theorem \ref{Kosheartthem} above. We recall some of the details of this equivalence, following \emph{loc. cit.}, in Section \ref{monadsec} below. We also remind the reader that this latter equivalence requires the additional assumption that $\Lambda$ is Koszul, in keeping with Warning \ref{LambdaKozwarn}. The extension of this theorem to the general case is given in Theorem \ref{unextAinfheartthm} below.

The ungraded case follows from the proof of the extension of this result to the general, not necessarily Koszul case (again, we recall this is given in Theorem \ref{unextAinfheartthm} below), together with Corollary \ref{twobjfiltcoro}.
\end{proof}

\subsection{Monad presentations from Koszul duality}\label{monadsec}

As explained in Section \ref{quiversec}, the description of $\Lambda$ as the path algebra of the quiver $Q_Y$ is equivalent to its description as the Koszul dual of the algebra $\Sigma$. Thus, towards understanding various moduli stacks of coherent sheaves on $Y$ in terms of stacks of representations of quivers, as outlined in the introductory section \ref{agintrosec} above, we now describe more explicitly the correspondence between $\Sigma$ modules, quiver representations, and coherent sheaves.

The compositions of the equivalences of Equations \ref{Kozgeoextendedhearteq} and \ref{Kozgeoextendedheartungreq} define functors
\begin{align}
\hat K(\cdot):= \Hom_\Sigma(S,\cdot)\otimes_{\Lambda} E  & : \DD_\Fd(\Sigma) \to \DD^b\Coh(\hat Y)  & \textup{and} \label{concreteungrKozeqn} \\
K(\cdot):= \Hom_\Sigma(S,\cdot)\otimes_{\Lambda} E  & : \DD_\fd(\Sigma) \to \DD^b\Coh(Y)^{T}  & . \label{concreteKozeqn}
\end{align}
The complex of coherent sheaves $K(M)$ corresponding to a $\Sigma$ module $M$ under this equivalence can be computed explicitly, as we now explain.

\begin{warning}
	For concreteness, we give most of the exposition in this section in the graded case, noting the ungraded variant can be recovered by completing and forgetting the grading. Further, for notation simplicity we will restrict the trigrading along a cocharacter $\C^\times \to T$, in keeping with Warning \ref{gradingwarn}.
\end{warning}

To begin, recall that using the Koszul resolution
\[\mc K^\bullet:=\otimes_S^\bullet( \bar\Sigma[1]) \otimes_S \Sigma\xrightarrow{\cong} S \ ,\]
we can compute the underlying bigraded vector space
\[ \Hom_\Sigma(S,M)  \cong  \Hom_\Sigma(\otimes_S^\bullet( \bar\Sigma[1]) \otimes_S \Sigma,M)  \cong \Hom_S( \otimes_S^\bullet( \bar\Sigma[1]),M) \cong M\otimes_S \Lambda \ , \]
and thus the underlying cohomologically graded equivariant coherent sheaf as
\begin{equation}\label{Kozshfeqn}
	K(M)=\Hom_\Sigma(S,M)  \otimes_\Lambda E \cong M\otimes_S \Lambda \otimes_\Lambda E \cong M \otimes_S E \ . 
\end{equation}

Moreover, the differential $d: M\otimes_S E \to M\otimes_S E[1]$ is given by
\begin{equation}\label{Kozdifleqn}
	 d(m\otimes e) = d_M m \otimes e + \sum_n \rho_n^\vee(m) \cdot e \quad\quad \textup{where} \quad\quad \rho_n^\vee: M \to M\otimes_S \bar\Sigma^\vee [-1]^{\otimes n} \subset M\otimes_S \Lambda 
\end{equation}
is the dual of the $A_\infty$ module structure map $\rho_n:\Sigma^{\otimes n}\otimes M \to M[1-n]$ and $\rho_n^\vee(m) \cdot e $ denotes the action of the second tensor factor $\Lambda=\Hom_{\DD^b\Coh(Y)}(E,E)$ on $E$.

For example, if $\Sigma$ is a strict, Koszul associative algebra and $\rho=\rho_1:\Sigma\otimes M \to M$ is the structure map for a strict $\Sigma$ module $M$, the differential is defined by
\[ d(m\otimes e) = d_M m \otimes e + \sum_\alpha v_\alpha m \otimes v_\alpha^\vee e \]
where $\{v_\alpha\}$ are a basis for the degree 1 component $\Sigma^1$ of $\Sigma$ and $v_\alpha^\vee$ denote the dual basis for $(\Sigma^1)^\vee=\Lambda^1$, which give generators for the quadratic dual algebra $\Lambda$.

\begin{rmk}
	
	In terms of the cohomological sheared grading of Remark \ref{cohomshearrmk}, the image of the module $M\in \DD(\Sigma^\sh)$ under the Koszul duality equivalence is given by
	\[ \Hom_\Sigma(S,M) \cong M\otimes_S \Lambda =  \bigoplus_{k,j\in \Z} M^k_j\otimes_S \Lambda_l [-k-j] \langle - l +  j\rangle \]
	as in the proof of Theorem 2.12.1 in \cite{BGS} (where this object is denoted $F(M)$), so that in particular for a plain graded $\Sigma$ module $M \in \Sigma^\sh\Mod_\fg$ given by
	\[ M = \bigoplus_{j\in \Z} M_j \langle -j\rangle \quad\quad\text{we obtain}\quad\quad K(M) = \bigoplus_{j\in \Z} M_j \otimes_S E [ - j] \]
	as the underlying cohomologically graded coherent sheaf.
\end{rmk}

\begin{warning}\label{cohomshearusewarning} We will use the cohomologically sheared grading on $\Sigma$ modules implicitly throughout the remainder of this section, omitting the superscript $(\cdot)^\sh$ by abuse of notation, in keeping with Warning \ref{gradingnotationwarning}. Similarly, the use of indices in subscripts and superscripts describing direct summands of modules will also be as in \emph{loc. cit.}, which in particular disagrees with their use in the preceding remark.
\end{warning}

It is clear that the simple modules $S_i$ over $\Sigma$ correspond to the projective $\Lambda$ modules $P_i$ and in turn the summands $E_i$ of the tilting object, as necessary:
\[ K(S_i)=\Hom_\Sigma(S,S_i) \otimes_\Lambda E \cong S_i \otimes_S \Lambda \otimes_\Lambda E \cong P_i \otimes_\Lambda E \cong E_i  \ . \]
More generally, any finite dimensional $\Sigma$ module $M$ admits a semi-simple composition series, that is, a filtration with semi-simple subquotients:
\[ 0=M_{m+1}\subset M_m\subset \cdots \subset M_1 \subset M_0=M \quad\quad \text{such that} \quad\quad  M_k/M_{k+1}= \bigoplus_{i_k\in I_k} S_{i_k}\langle -k \rangle \]
for some finite sets $I_k$ of elements $i_k \in I$ including repetitions, where we have assumed for notational simplicity $M$ is non-negatively graded; we summarize this situation by writing
\[ M = \left[ \bigoplus_{i_m\in I_m} S_{i_m}\langle - m \rangle  < \cdots< \bigoplus_{i_1\in I_1} S_{i_1}\langle - 1\rangle < \bigoplus_{i_0\in I_0} S_{i_0} \right] \ .\]
Following \cite{BGS}, the image of such a module $M$ under the Koszul duality equivalence is given by the corresponding complex of projective objects
\begin{equation}\label{Kosreseqn}
	 K(M)\cong  \left[ \bigoplus_{i_0\in I_0} E_{i_0} \to \bigoplus_{i_1\in I_1} E_{i_1}[-1] \to \cdots \to \bigoplus_{i_m\in I_{m}} E_{i_m}[-m] \right] \ .
\end{equation}
where each summand $E_{i_k}\to E_{i_{k+1}}$ of the differential is determined by the class of the extension defining $M$ in $\Ext^1_\Sigma(S_{i_k}\langle  -k \rangle, S_{i_{k+1}}\langle -k-1\rangle)$; the identification of $\Lambda$ with the Koszul dual algebra of $\Sigma$ gives an isomorphism
\[ \Hom_{\Sigma}(S,S) \cong\Lambda= \Hom_{\DD^b\Coh(Y)} (E,E) \]
under which the above extension class determines a map $E_{i_k} \to E_{i_{k+1}}$.

These maps can be understood concretely as follows: the failure of the extension to split is detected by an element of $\Sigma$ sending $S_{i_k}$ to $S_{i_{k+1}}\langle -1\rangle$, or equivalently a non-trivial restriction of the module structure map $\Sigma\otimes S_{i_k}\to S_{i_{k+1}}\langle -1\rangle $. This determines the desired differential by inducing the dual map $S_{i_k} \to \Lambda^1\otimes S_{i_{k+1}}$ (after inverting the cohomological shearing of Remark \ref{cohomshearrmk}) along the inclusion $S\to \Lambda$ to give
\[ P_{i_k} \to \Lambda^1 \otimes P_{i_{k+1}}\xrightarrow{m} P_{i_{k+1}} \quad\quad \text{or equivalently}\quad\quad  E_{i_k} \to E_{i_{k+1}} \]
under the equivalence ${\textup{D}_\perf}(\Lambda)\cong\DD^b\Coh(Y)^{\C^\times}$, where $m: \Lambda^1\otimes P_{i_{k+1}} \to P_{i_{k+1}}$ is the $\Lambda$ module structure map. Similarly, the higher arity structure maps of the $A_\infty$ module determine terms in the differential given by the action of higher degree components $\Lambda^n$ of $\Lambda$.

In particular, the composition series of the injective $\Sigma$ modules $I_i=\Sigma^\vee_i$, given by
\begin{equation}\label{injcompeqn}
	I_i = [ I_i^0 < I_i^{-1}\langle 1 \rangle < \cdots < I_i^{-m_i}\langle m_i \rangle ]   
\end{equation}
in the cohomologically sheared notation (see Warning \ref{cohomshearusewarning}), determine canonical Koszul-type resolutions of the simple objects $F_i$ in terms of the distinguished projective objects $E_{i_j}$:
\begin{equation}  \label{simplereseqn}
	 K(I_i)  = \left[   I_i^{-m_i} \otimes_S E[m_i] \to \cdots \to I_i^{-1}\otimes_S E [1]  \to E_i \right] \xrightarrow{\cong} F_i   \ . 
\end{equation}

More generally, we obtain a canonical Koszul-type projective resolution of an arbitrary compactly supported, $\C^\times$-equivariant, perverse coherent sheaf $H\in \Perv_\cs(Y)^{\C^\times}$ in terms of the distinguished projective objects $E_i\in\Coh(Y)^{\C^\times}$. Again, we begin with the simplifying assumption that $\Lambda$ is Koszul, and correspondingly $\Sigma$ has no non-trivial higher $\Ainf$ multiplication maps.

\begin{warning} Throughout the remainder of this section we assume that $\Lambda$ is Koszul, as in Warning \ref{LambdaKozwarn}. The main result of this section, Proposition \ref{monadunextprop} below, is proved in the general case in Proposition \ref{monadunextAinfprop}.
\end{warning}

Under this assumption, the desired resolution of $H\in \Perv_\cs(Y)^{\C^\times}$ is simply defined by the image of the corresponding linear complex of injectives $L \in \textup{LCP}(\Sigma)$ under the explicit model of the Koszul duality equivalence $K:\DD^b_\fd(\Sigma) \xrightarrow{\cong} \DD^b\Coh(Y)^{\C^\times}$ explained above.

Note that each term of the differential in a linear complex of injectives decomposes as a direct sum of maps
$I_i \to I_j\langle 1 \rangle $ between the distinguished injectives, of abstract graded degree $1$ by the linearity assumption. By Theorem \ref{catthmunext}, we have identifications
\begin{equation}\label{extideqn} 
	\ _i \Sigma^k_j \cong \Hom^0_{\Sigma}(I_i,I_j\langle k \rangle) \cong \Hom^0_{\DD^b\Coh(Y)}(K(I_i), K(I_j)[k]) 
\end{equation}
between the space $ _i\Sigma^k_j$ dual to the degree $1-k$ edges of the DG quiver $Q_Y$ corresponding to $Y$, the space of graded degree $k$ maps between the injective $\Sigma$ modules $I_i$, and the space of cohomological degree $k$ maps between the canonical Koszul resolutions of the simple coherent sheaves $F_i$ from Equation \ref{simplereseqn}. Concretely, an element $\delta_{ij} \in\   _i\Sigma_{j}^k$ determines by multiplication a map $d_{\delta_{ij}} :\Sigma_i\to \Sigma_j\langle k \rangle$ between the distinguished projective $\Sigma$ modules, and in turn a dual map
\[ d_{\delta_{ij}}^N :I_i \to I_j\langle k \rangle \quad\quad \text{or equivalently} \quad\quad (d_{\delta_{ij}}^N)^{l}: I_i^{l} \to I_j^{l+k}  \quad\quad \text{for $l\in \bb Z$}\]
in graded components. Finally, this determines a map $K( d_{\delta_{ij}}^N):K(I_i)\to K(I_j)[k]\langle -k\rangle $ of the corresponding complexes of projectives of Equation \ref{simplereseqn}, defined explicitly as
\[\xymatrix{\cdots&  0  &   I_i^{-m_i} \otimes_S E \ar[r]\ar[d]^{(d_{\delta_{ij}}^N)^{-n}\otimes \id_E} &  \cdots \ar[r] \ar@{}[d]^{\cdots}&  I_i^{-k} \otimes_S E \ar[d]^{(d_{\delta_{ij}}^N)^{-k}\otimes \id_E} \ar[r]& \cdots \ar[r] & E_i \\
	 I_j^{-m_j} \otimes_S E \ar[r] & \cdots \ar[r] &  I_j^{-(m_i-k)}\otimes_S E \ar[r] &  \cdots \ar[r] & E_j   & 0 & \cdots  } \ . \]

Thus, we can explicitly describe the canonical Koszul resolution of an arbitrary extension
$$ 0 \to F_j \to H \to F_i \to 0 $$
of two simple, compactly supported perverse coherent sheaves $F_i$: it is the image of the linear complex of injectives $I_i \xrightarrow{ d_{\delta_{ij}}^N} I_j\langle 1\rangle $ under the explicit description of the Koszul duality functor above, which gives the totalization of the double complex
\[\xymatrix{ \cdots & 0 &   I_i^{-m_i} \otimes_S E \ar[r]\ar[d] &  \cdots \ar[r] \ar@{}[d]^\hdots&  I_i^{-2}\otimes_S E \ar[r] \ar[d]&  I_i^{-1}\otimes_S E \ar[d] \ar[r]  & E_i \\
 I_j^{-m_j}\otimes_S E  \ar[r] & \cdots \ar[r] &   I_j^{-(m_i-1)} \otimes_S E \ar[r] &\cdots \ar[r] &  I_j^{-1}\otimes_S E \ar[r] &  E_j   & 0 &    } \ .\]
More generally, for an arbitrary iterated extension of $d$ simple objects
\[H = \left[ F_{i_d}\langle -d\rangle < \cdots < F_{i_1}\langle -1\rangle  < F_{i_0} \right]  \quad \in \Perv_\cs(Y)^{\C^\times} \ , \]
the image of the corresponding linear complex of injectives
\[ L = \left[I_{i_o} \to I_{i_1}[-1]\langle 1 \rangle \to \cdots \to  I_{i_{d-1}} [-(d-1)]\langle {d-1}\rangle \to  I_{i_d} [-d]\langle d\rangle   \right]  \quad \in \textup{LCI}(\Sigma)\]
under the Koszul duality equivalence is given by the totalization of the double complex
\[\xymatrix{ \cdots & 0 &   I_{i_0}^{-m_{i_0}} \otimes_S E \ar[r]\ar[d] &  \cdots \ar[r] \ar@{}[d]^\hdots&  I_{i_0}^{-2}\otimes_S E \ar[r] \ar[d]&  I_{i_0}^{-1}\otimes_S E \ar[d] \ar[r]  & E_{i_0} \\
	I_{i_{1}}^{-m_{i_{1}}}\otimes_S E  \ar[r]\ar[d] & \cdots \ar[d]\ar[r] &   I_{i_{1}}^{-(m_{i_0}-1)} \otimes_S E \ar[r]\ar[d] &\cdots \ar[r]\ar[d] &  I_{i_{1}}^{-1}\otimes_S E \ar[r]\ar[d] &  E_{i_{1}}  & 0 \\
\vdots\ar[d] & \vdots \ar[d]&\vdots \ar[d]&\vdots \ar[d]&\vdots &0  & \vdots \\
  \cdots \ar[r]\ar[d] &  I_{i_{d-1}}^{-2}\otimes_S E \ar[r]\ar[d] &  I_{i_{d-1}}^{-1}\otimes_S E  \ar[r]\ar[d]  & E_{i_{d-1}} & 0 &\vdots &  \\
  \cdots \ar[r] &  I_{i_d}^{-1}\otimes_S E  \ar[r]  & E_{i_d} & 0 & \vdots &  & 
} \ .\]

The collection of all such iterated extensions $H$ of the simple objects $F_i\in \Perv_\cs(Y)^{\C^\times}$ is parameterized by the space of finite dimensional graded $\Lambda$ modules, or equivalently representations of the graded DG quiver $Q_Y$ with path algebra $\Lambda=\K Q_Y$, as we have seen in Theorem \ref{Kozgeoextendedhearteq}. The canonical resolutions constructed in the preceding discussion can thus also be understood in terms of quiver representations, which we now explain:

The K-theory class of an object $H\in \Perv_\cs(Y)^{\C^\times}$ is determined by the multiplicities $d_i\in \bb N$ of the simple factors $F_i$ (and their graded shifts), and we can collect multiplicities in the object $K(L)$ with respect to the variable $i \in I$ to express the complex as
\begin{equation}\label{colmulteqn}
	 K(L) = \left ( \bigoplus_{i\in I} K(I_i)\otimes V_i  \ , \  K(d_\delta^N) \right) \quad\quad \text{with}\quad\quad K(d_\delta^N) := \sum_{k=1}^{d} K(d_{\delta_{i_{k-1} i_k}}^N) \ ,  
\end{equation}
where $V_i$ is a (graded) vector space with $\dim V_i = d_i$ for each $i\in I$ and $\sum_{i\in I} d_i=d+1$.

These multiplicities correspond to those of the simple $\Lambda$ module factors $S_i$ occurring in the corresponding quiver representation, so that the quiver representation corresponding to $H$ is determined by a $\dd=(d_i)_{i\in I}$ dimensional graded $S$ module
\begin{equation}\label{monadkoszuleqn}
	V = \bigoplus_{i\in V_{Q_Y}} V_i \quad\quad \textup{together with} \quad\quad \rho:\K\langle E_{Q_Y}\rangle^0=(\Sigma^1)^\vee \to \Endi(V)
\end{equation}
a map of graded $S$ bimodules from the space of cohomological degree $0$ edges of the quiver to the endomorphism algebra of $V$, such that the induced map from the cohomological degree zero component of the quasi-free resolution of the path algebra $\Lambda^0=\otimes_S^\bullet \K\langle E_Q\rangle^0$ to $\Endi_{S\Bimod}(V)$ takes the ideal of relations in the graded quiver $Q_Y$ (or equivalently the image of the differential in the quasi-free resolution) to zero in $\Endi_{S\Bimod}(V)$.

The map $\rho$ of Equation \ref{monadkoszuleqn} can equivalently be interpreted as an element
\begin{equation}\label{deltaeqn}
	\delta = \sum_{i,j\in I} b_{ij} \otimes B_{ij} \quad  \in \quad \Sigma^1 \otimes_S\Endi_{S\Bimod}(V)= \bigoplus_{i,j\in I} \Ext^1(F_i,F_j)\otimes \Hom(V_i,V_j)  \ .
\end{equation}

\begin{warning}\label{puretenswarning} We use the abuse of notation $b_{ij}\otimes B_{ij} \in \Ext^1(F_i,F_j)\otimes \Hom(V_i,V_j)$ to denote a not necessarily pure tensor
	\[ b_{ij}\otimes B_{ij} := \sum_{\alpha} b_{ij}^\alpha \otimes B_{ij}^\alpha \quad\quad\textup{with}\quad\quad  b_{ij}^\alpha \in \Ext^1(F_i,F_j), \ B_{ij}^\alpha\in \Hom(V_i,V_j)  \]
	where the sum is over $\alpha\in \mc B_{ij}$ the parameterizing set for a basis $\{b_{ij}^\alpha\}$ for $\Ext^1(F_i,F_j)$. We will often use the shorthand $B_{ij}=(B_{ij}^\alpha)$ and $b_{ij}=(b_{ij}^\alpha)$ to denote these collections of data.
\end{warning}

Further, the element $\delta$ determines a cohomological degree +1 map $d_B:\tilde H \to \tilde H[1]$ of the equivariant perverse coherent complex
\begin{equation}\label{monadvseqn}
	 \tilde H:= \bigoplus_{i\in I} K(I_i)\otimes V_i  \ \in \Perv_\cs(Y)^{\bb C^\times}  \quad\quad \text{ or }\quad\quad  \tilde H:=\bigoplus_{i\in I} \hat K(I_i)\otimes V_i  \ \in \Perv_\cs(\hat Y)   
\end{equation}
in the ungraded case, where $K(I_i)$ is the complex of Equation \ref{simplereseqn}, via the isomorphism
\[ \Ext^1(F_i,F_j)\cong \Hom( K(I_i), K(I_j)[1]) \]
of Equation \ref{extideqn}. Moreover, the resulting differential is precisely $K(d_{\delta}^N)$ of Equation \ref{colmulteqn} above, so that we have the identification
\begin{equation}\label{monadkozdifeqn}
	 K(d_\delta^N)=d_B:=\sum_{i,j \in I} K(b_{ij}) \otimes B_{ij} \quad\quad\text{with}\quad\quad   K(b_{ij}) \otimes B_{ij}: K(I_i)\otimes V_i \to  K(I_j)\otimes V_j[1] \ .
\end{equation}

In fact, the relations required for the map $\rho$ of Equation \ref{monadkoszuleqn} to define a quiver representation are equivalent to the condition that $ d_B$ defines a differential deforming the complex $\tilde H$ of Equation \ref{monadvseqn} to a projective resolution of the object $H$:

\begin{prop}\label{monadunextprop} Let $\tilde H \in \Perv_\cs(Y)^{T}$ be as in Equation \ref{monadvseqn} and fix
\[	B:=(B_{ij})_{i,j\in I}\quad\quad\text{with}\quad\quad B_{ij}=(B_{ij}^\alpha \in \Hom(V_i,V_j))_{ \alpha \in \mc B_{ij}} \ . \]
	 The following are equivalent: 
\begin{itemize}
	\item the induced map $d_B:\tilde H \to \tilde H[1]$ of Equation \ref{monadkozdifeqn} satisfies $d_B^2=0$, and
	\item the induced map $\rho:(\Sigma^1)^\vee\to \Endi_{S\Bimod}(V)$ of Equation \ref{monadkoszuleqn} defines a representation of the graded DG quiver $Q_Y$.
\end{itemize}
Further, when these conditions hold, the resulting complex $(\tilde H,d_B)$ is a projective resolution of the object $H \in \Perv_\cs(Y)^{T}$ corresponding to the quiver representation $V\in \Lambda\Mod_\fd$.

\smallskip

Similarly, for $\tilde H \in \Perv_\cs(\hat Y)$ as in Equation \ref{monadvseqn}, the following are equivalent:
\begin{itemize}
	\item the induced map $d_B:\tilde H \to \tilde H[1]$ of Equation \ref{monadkozdifeqn} satisfies $d_B^2=0$, and
	\item the induced map $\rho:(\Sigma^1)^\vee\to \Endi_{S\Bimod}(V)$ of Equation \ref{monadkoszuleqn} defines a nilpotent representation of the DG quiver $Q_Y$.
\end{itemize}
Further, when these conditions hold, the resulting complex $(\tilde H,d_B)$ is a projective resolution of the object $H \in \Perv_\cs(\hat Y)$ corresponding to the quiver representation $V\in \Lambda\Mod_\Fd$.
\end{prop}
\begin{proof} This result is a special case of the more general result proved in Proposition \ref{monadunextAinfprop} without the assumption that $\Lambda$ is Koszul, in keeping with Proposition \ref{Kozcaseprop}.
\end{proof}

\subsection{Generalization to the $A_\infty$ case via twisted objects}\label{Ainfkozsec}

In this section, we use the twisted objects construction recalled in Section \ref{twobjsec} to generalize the results of Theorem \ref{unextheartthm}, which describes the induced equivalences between the hearts of the triangulated categories of Theorem \ref{catthmunext}, to the case where $\Lambda$ is not Koszul and in turn $\Sigma$ is a general $A_\infty$ algebra with non-trivial higher multiplication maps. Similarly, we generalize the monad presentation given in Proposition \ref{monadunextprop} to the case where $\Lambda$ is not Koszul.

Let $\mc D = \DD^b\Coh(Y)^{\bb C^\times}$, $F=\oplus_{i\in I} F_i$ and consider the graded variant of the $\Ainf$ category $\mc A_F$ defined in Example \ref{AinfcatSeg}, with objects $i\in I$ and $\Hom$ spaces given by
\[ \mc A_F(i,j) = \Ext^\bullet(F_i,F_j) =\ _i\Sigma_j \ .  \]

In this setting, an object of the category $\tw \mc A_F$ is given by a finite collection of objects \[(i_1,n_1,p_1),...,(i_d,n_d,p_d)\in \bb Z \mc A_F \ , \]
together with a degree zero element
\[ \delta = (\ \delta_{kl} \in  \ _{i_k}\Sigma_{i_l} [n_l-n_k+1]\langle p_l-p_k\rangle \ )_{k,l=1}^d  \ \in \ \mf{gl}_d(\bb Z \Sigma)[1] \]
such that $\delta_{kl}=0 $ for $k\leq l$ and moreover $\delta$ satisfies the Maurer-Cartan equation
\[ \sum_{t\in \bb N}  m_t^{\mf{gl}_d(\Z\Sigma)}(\delta^{\otimes t}) = 0 \ . \]

By Corollary \ref{grtwobjequiv}, we obtain that the functor $Y_2:\tw \mc A_F \to \CC_\infty(\Sigma)$ induces an equivalence of triangulated categories
\begin{equation}\label{grtwobjequivgreqn}
	H^0(Y_2): H^0(\tw \mc A_F) \to \triang(\Sigma_i)  \ ,
\end{equation}
where we recall that $\triang(\Sigma_i)$ denotes the minimal triangulated subcategory of $\DD_\Z(\Sigma)$ containing the objects $\Sigma_i\langle k \rangle$ for $i\in I$ and $k\in \bb Z$.

Similarly, we recall that $\tw^0\mc A_F$ is the full subcategory on objects $(i_1,...,i_r,\delta)$ such that each of the associated integers $n_k=0$ for $k=1,...,d$, so that their image under the functor $Y_2$ has underlying graded vector space
\[ \Sigma_{i_1,...,i_d}=\bigoplus_{k=1,...,d} \Sigma_{i_k}\langle p_k \rangle\ , \quad\quad \text{or equivalently} \quad\quad   \Sigma_{i_1,...,i_d}^\sh=\bigoplus_{k=1,...,d} \Sigma_{i_k}^\sh[p_k]\langle -p_k \rangle , \]
in the cohomologically sheared conventions.

We can also define the corresponding ungraded $\Ainf$ category $\mc A_F$, as before, and have the corresponding $\Ainf$ category $\Tw \mc A_F$ with objects given by a finite collection of objects \[(i_1,n_1),...,(i_r,n_d)\in \bb Z \mc A_F \ , \]
together with a degree zero element
\[ \delta = (\ \delta_{kl} \in  \ _{i_k}\Sigma_{i_l} [n_l-n_k+1] \ )_{k,l=1}^d  \ \in \ \mf{gl}_d(\bb Z \Sigma)[1] \]
satisfying the analogous conditions. Similarly, we have the full category $\Tw^0 \mc A_F$ on objects $(i_1,...,i_d,\delta)$ such that each of the associated integers $n_k=0$ for $k=1,...,d$, so that their image under the functor $Y_2$ has underlying vector space
\[ \Sigma_{i_1,...,i_d}=\bigoplus_{k=1,...,d} \Sigma_{i_k} \ . \]

We now establish the generalization of Theorem \ref{unextheartthm} to the case where $\Lambda$ is not necessarily Koszul:

\begin{theo}\label{unextAinfheartthm} Restriction to the hearts of the triangulated categories in the diagram of Equation \ref{Kozgeoextendedeq} induces mutually inverse equivalences of mixed categories
	\[	\xymatrixcolsep{3pc}
	\xymatrixrowsep{.5pc}
	\xymatrix{ \Perv_\cs(Y)^{T}  \ar@<.5ex>[r]^{} & \ar@<.5ex>[l]  \Lambda\Mod_\fd  \ar@<.5ex>[r]^{} & \ar@<.5ex>[l] H^0(\tw^0 \mc A_F)} \ .  \]
	
	Similarly, restriction to hearts in the diagram of Equation \ref{Kozgeoextendedungreq} induces
	\[	\xymatrixcolsep{3pc}
	\xymatrixrowsep{.5pc}
	\xymatrix{ \Perv_\cs(\hat Y) \ar@<.5ex>[r]^{} & \ar@<.5ex>[l]  \hat \Lambda\Mod_\Fd  \ar@<.5ex>[r]^{} & \ar@<.5ex>[l] H^0(\Tw^0 \mc A_F)} \ .  \]
\end{theo}
\begin{proof}
The mutually inverse equivalences on the left of each diagram have already been established in Theorem \ref{unextheartthm}, following Bridgeland \cite{Bdg1} and Van den Bergh \cite{VdB1}. The equivalences on the right follow from the identifications
\[ \Lambda\Mod_\fd \cong \filt(S_i) \xrightarrow{\cong} \filt(\Sigma_i) \cong H^0(\tw^0 \mc A_F)  \ , \]
and similarly in the ungraded case
\[ \Lambda\Mod_\Fd \cong \Filt(S_i) \xrightarrow{\cong} \Filt(\Sigma_i) \cong H^0(\Tw^0 \mc A_F)  \ , \]
where the first isomorphism in each line follows from the existence of Jordan-Holder filtrations, as in the proof of Proposition \ref{csequivprop} for example, and the final isomorphism in each line follows from Corollaries \ref{twobjgrfiltcoro} and \ref{twobjfiltcoro}, respectively. Finally, the middle equivalence in each line follows from the fact that triangulated functors preserve categories of iterated extensions, and we have seen that the given functor maps $S_i$ to $\Sigma_i$ (and, in the graded case, preserves grading shifts in the unsheared conventions).
\end{proof}

We note that this equivalence is indeed a generalization of Theorem \ref{unextheartthm}, in the sense that we have the following proposition:

\begin{prop}\label{Kozcaseprop} Suppose $\Lambda$ is a Koszul algebra. Then the cohomological grading shear equivalence of Remark \ref{cohomshearrmk} restricts to an equivalence of mixed categories
	\[  H^0(\tw^0 \mc A_F) \xrightarrow{\cong} \textup{LCP}(\Sigma)  \ . \]
\end{prop}
\begin{proof} For $\Lambda$ a Koszul algebra, the corresponding Koszul dual algebra $\Sigma$ is a plain, Koszul associative algebra, and thus the $\Ainf$ bimodule structure maps
\[	\rho^{\Sigma_{i_1,...,i_d}}_{t,k}: \Sigma^{\otimes t}\otimes \Sigma_{i_1,...,i_d} \otimes \mc E_{\Sigma}^{\otimes k} \to  \Sigma_{i_1,...,i_d}[1-t-k] \]
of Equation \ref{bimodstreqn} necessarily vanish for $t\geq 2$, $k\geq 2$ and for $t=k=1$ and $t=k=0$. Thus, the only non-trivial structure maps for objects in $ H^0(\tw^0 \mc A_F)$ are $\rho_{1,0}$ and $\rho_{0,1}$. The former is simply the usual $\Sigma$ module structure map on the direct sum of projective modules $\Sigma_{i_1,...,i_d}$, and the latter gives precisely the usual allowed linear differentials defining a linear complex of projectives.
\end{proof}

Next, we describe the analogous monad formalism arising from Theorem \ref{unextAinfheartthm}, generalizing Proposition \ref{monadunextprop} to the $A_\infty$ case. Given an object $(i_1,...,i_d,\delta)\in \tw^0\mc A_F$, the corresponding object of $\DD_\fd(\Sigma)$ has underlying bigraded vector space given by
\[ \Sigma_{i_1,...,i_d} = \bigoplus_{k=1,...,d} \Sigma_{i_k}\langle p_k\rangle \cong \bigoplus_{i\in I} \Sigma_i \otimes V_i  \]
where each $V_i$ is a graded vector space
\[ V_i=\bigoplus_{p\in \bb Z} V_{i,-p}\langle p \rangle \quad\quad \text{with}\quad\quad \dim V_{i,p} = \# \{ k\in \{1,...,d\} \ | \ i_k=i\textup{ and } p_k=p \} \ . \]
\noindent Note in particular that $\dim \oplus_{i\in I} V_i = d$.

Now, the degree zero element $\delta\in \mf{gl}_r(\bb Z\mc A_F)[1]$ is in fact an element of the subspace
\[ \delta \in \Hom_{\DD(\Sigma)}(\bigoplus_{i\in I}  \Sigma_i\otimes V_i, \bigoplus_{j\in I} \Sigma_j \otimes V_j [1] ) \subset  \mf{gl}_d(\Sigma)[1] \cong \mf{gl}_d(\mc A_F)[1]  \]
and thus admits a decomposition
\begin{equation}\label{deltadecompeqn}
	\delta = \sum_{i,j\in I} b_{ij}\otimes B_{ij}  \quad \in \quad \bigoplus_{i,j\in I}  \ _i\Sigma_j\otimes \Hom(V_i,V_j)[1] \ ,
\end{equation}
as in Equation \ref{deltaeqn}, where we again use the abuse of notation $ B_{ij} \otimes b_{ij}$ to denote the not necessarily pure tensor, as in Warning \ref{puretenswarning}.

The $\Ainf$ module structure maps for $\Sigma_{i_1,...,i_d}^\delta$ are corrected from those of $\Sigma_{i_1,...,i_d}$, the underlying direct sum of rank 1 free modules, by the higher $(\mc A_F,\mc E_{\mc A_F})$ bimodule structure maps evaluated on the tensor powers of $\delta$, according to Equation \ref{twobjstrmapgradedeqn}. In particular, the underlying cochain complex differential is given by
\[d_\delta = \sum_{k\in \Z} \rho_k^{\Sigma}(\cdot, \delta^{\otimes k-1}) \quad \in\quad \Hom_{\DD(\Sigma)}( \Sigma_{i_1,...,i_d}, \Sigma_{i_1,...,i_d}[1])  \]
where $\rho_k^\Sigma: \Sigma_{i_1,...,i_d}\otimes\Sigma^{\otimes k-1} \to \Sigma_{i_1,...,i_d}[1-k]$ denote the $\Ainf$ module structure maps for $\Sigma_{i_1,...,i_d}$ over $\mc E_{\mc A_F}$. Decomposing according to Equation \ref{deltadecompeqn}, we have the analogous decomposition
\begin{equation}\label{ddeltacompseqn}
	d_\delta = \sum_{k\in \bb Z,\  i,i_2,...,i_{k-1},j \in I} m_k^\Sigma(\cdot, b_{i i_2},...,b_{i_{k-1}  j})\otimes (B_{i i_2}  ...  B_{i_{k-1}  j} )   \quad \in \quad \bigoplus_{i,j\in I}  \Hom(\Sigma_i, \Sigma_j[1]) \otimes \Hom(V_i,V_j)\ 
\end{equation} 
of the differential on $\Sigma_{i_1,...,i_d}^\delta$. Thus, the image of the Nakayama dual of $\Sigma_{i_1,...,i_r}^{\delta}\in {\textup{D}_\perf}(\Sigma)$ under the Koszul duality equivalence of Equation \ref{concreteKozeqn} is given by
\begin{equation}\label{monadvsAinfeqn}
	K(\Sigma_{i_1,...,i_d}^{\delta,N})= \left(K(\Sigma_{i_1,...,i_d}^N) , K(d_\delta^N) \right) \quad\quad \text{where} \quad\quad K(\Sigma_{i_1,...,i_d}^N)=\bigoplus_{i\in I} K(I_i)\otimes V_i \ ,
\end{equation}
and where the differential $K(d_\delta^N):K(\Sigma_{i_1,...,i_d}^N) \to K(\Sigma_{i_1,...,i_d}^N)[1]$ is given by
\begin{equation}\hspace*{-1cm}\label{Kdeltadefneqn}
	K(d_\delta^N) :=\sum K (m_k^{\Sigma}(\cdot, b_{i,i_2},...,b_{i_{k-1}, j})^N)\otimes  ( B_{i,i_2}  ...  B_{i_{k-2}, j} )   \quad \in \quad \bigoplus_{i,j\in I} \Hom(K(I_i), K(I_j)[1] )\otimes \Hom(V_i,V_j) \ ,
\end{equation}
where the sum is over the same index set as in Equation \ref{ddeltacompseqn} above.

We also introduce the notation $d_B:\tilde H \to \tilde H[1]$, where $\tilde H \in \Perv_\cs(Y)^{\bb C^\times}$ is as in Equation \ref{monadvseqn} for the map
\begin{equation}\label{monadkozdifAinfeqn}
	d_B= \sum_{k\in \bb Z,\  i,i_2,...,i_{k-1},j \in I}  K (m_k^{\Sigma}(\cdot, b_{i,i_2},...,b_{i_{k-1}, j})^N)\otimes  ( B_{i,i_2}  ...  B_{i_{k-2}, j} )  \ . 
\end{equation}
This is essentially the same as the map in Equation \ref{Kdeltadefneqn} above, but emphasizes that a differential of this functional form is determined uniquely by the choice of linear maps $B=(B_{i,j}^\alpha)$, generalizing the definition of $d_B$ in Equation \ref{monadkozdifeqn}.

We can now state the desired generalization of Proposition \ref{monadunextprop}:

\begin{prop}\label{monadunextAinfprop} Let $\tilde H \in \Perv_\cs(Y)^{T}$ be as in Equation \ref{monadvseqn} and fix
	\[	B:=(B_{ij})_{i,j\in I}\quad\quad\text{with}\quad\quad B_{ij}=(B_{ij}^\alpha \in \Hom(V_i,V_j))_{ \alpha \in \mc B_{ij}} \ . \]
	The following are equivalent: 
	\begin{itemize}
		\item the induced map $d_B:\tilde H \to \tilde H[1]$ of Equation \ref{monadkozdifAinfeqn} satisfies $d_B^2=0$, and
		\item the induced map $\rho:(\Sigma^1)^\vee\to \Endi_{S\Bimod}(V)$ of Equation \ref{monadkoszuleqn} defines a representation of the graded DG quiver $Q_Y$.
	\end{itemize}
	Further, when these conditions hold, the resulting complex $(\tilde H,d_B)$ is a projective resolution of the object $H \in \Perv_\cs(Y)^{T}$ corresponding to the quiver representation $V\in \Lambda\Mod_\fd$.
	
	\smallskip
	
	Similarly, for $\tilde H \in \Perv_\cs(\hat Y)$ as in Equation \ref{monadvseqn}, the following are equivalent:
	\begin{itemize}
		\item the induced map $d_B:\tilde H \to \tilde H[1]$ of Equation \ref{monadkozdifAinfeqn} satisfies $d_B^2=0$, and
		\item the induced map $\rho:(\Sigma^1)^\vee\to \Endi_{S\Bimod}(V)$ of Equation \ref{monadkoszuleqn} defines a nilpotent representation of the DG quiver $Q_Y$.
	\end{itemize}
	Further, when these conditions hold, the resulting complex $(\tilde H,d_B)$ is a projective resolution of the object $H \in \Perv_\cs(\hat Y)$ corresponding to the quiver representation $V\in \Lambda\Mod_\Fd$.
\end{prop}
\begin{proof}
	
	By definition, the putative differential $d_B$ is uniquely determined by an element $\delta$ as in Equation \ref{deltadecompeqn} by the formula for $d_\delta$ given in Equation \ref{ddeltacompseqn}, and by the standard argument for $\Ainf$ algebras the condition $d_\delta^2=0$ is equivalent to the Maurer-Cartan equation for $\delta$ of Equation \ref{MCeqn}.
	
	The Maurer-Cartan equation can be restated as a system of equations on the linear algebraic data $B_{ij}\in \Hom(V_i,V_j)$ with coefficients determined by the higher products of the corresponding elements $b_{ij}\in \Sigma[1]$. Indeed, for $i,j\in I$, the corresponding component of the Maurer-Cartan equation on $\delta \in \mf{gl}_d(\mc A_F)$ is given by
\begin{equation}\label{MCcompeqn}
		0 = \left( \sum_{t\in \bb N} m_t(\delta^{\otimes t}) \right)_{ij} = \sum_{t\in \bb N} \sum_{i_2,...,i_{t-1}\in I}  m_t(b_{i i_2} ,..., b_{i_{t-1} j} )\otimes (B_{i i_2} \cdots  B_{i_{t-1} j}) \ , 
\end{equation}
	understood as a polynomial equation on the linear maps $B_{ij}$. Thus, we need to show that these equations are equivalent to the requirement that the collection of linear maps $B_{ij}$ define a representation of the quiver $Q_Y$, or equivalently a module over the path algebra $\Lambda=\Sigma^!$.
	
	The element $\delta$ is equivalent to the map of $S$ bimodules
	\[\rho:(\Sigma^1)^\vee \to \Endi(V) \]
	as in Equation \ref{monadkoszuleqn}, and extends by the free-forgetful adjunction to define a map of algebras internal to $S\Bimod$
	\[ r_\rho: \Lambda^0 = \otimes_S^\bullet(\Sigma^1)^\vee \to \Endi(V) \ .  \]
	
	\noindent The map $r_\rho$ extends to define a module structure on $V$ over the path algebra $\Lambda$ of the DG quiver $Q_Y$ if and only if the composition
	\[ \Lambda^{-1} \xrightarrow{d_\Lambda} \Lambda^0 \xrightarrow{r_\rho} \Endi(V) \]
	vanishes. The differential $d_\Lambda$ is defined on the generators
	\[ (\Sigma^2)^\vee \subset  \Lambda^{-1}  = (\Sigma^2)^\vee \otimes_S \otimes_S^\bullet(\Sigma^1)^\vee  \]
	of $\Lambda^{-1}$ over $\Lambda^0$, by the formula
	\[ d_\Lambda = \sum_{t\in \bb N} m_t^\vee : (\Sigma^2)^\vee \to \otimes^\bullet_S (\Sigma^1)^\vee \ , \]
	that is, as the sum of the duals of the $\Ainf$ multiplication maps
	\[ m_t:  \otimes^t_S \Sigma \to \Sigma[2-n] \ . \]
	Thus, vanishing of the composition
\[	r_\rho \circ d_\Lambda  = r_\rho \circ  \sum_{t\in \bb N} m_t^\vee = \sum_{t\in \bb N}  \rho^{\otimes t} \circ m_t^\vee  \ : (\Sigma^2)^\vee \to \Endi(V)  \]
	is equivalent under Hom-tensor adjunction to the vanishing of the element
	\[  \sum_{t\in \bb N} m_t(\delta^{\otimes t}) = \sum_{t\in \bb N} \sum_{i_1,...,i_t \in I } m_n(b_{i_1 i_2}, ... , b_{i_{t-1} i_t}) \otimes ( B_{i_1 i_2} \cdots  B_{i_{t-1} i_t}) \ \in \Sigma^2\otimes_S \Endi(V)  \ ,\]
	which is evidently equivalent to the vanishing of the Maurer-Cartan equation given in Equation \ref{MCcompeqn}, as desired.
\end{proof}

\subsection{Moduli spaces of perverse coherent sheaves and quivers with potential}\label{quiversecunfr}

In this section, we formulate and prove the first variant of Theorem \ref{Athm} from the introduction. The categorical results discussed above determine a description of the moduli stacks of objects in the categories we have introduced above in terms of linear algebraic data. We begin by recalling the geometric description of the moduli stack of quiver representations and then state the main result establishing its equivalence with the moduli stack of perverse coherent sheaves.

Let $Q$ be a graded quiver with edge set $E_Q$, vertex set $V_Q$, and path algebra
\[ \K Q = \otimes^\bullet_S \K\langle E_Q\rangle \quad\quad \text{where}\quad\quad S=\oplus_{i\in V_Q} S_i = \oplus_{i\in V_Q} \K \]
as in Section \ref{quiversec}. For each dimension vector $\dd=(d_i)\in \bb N^{V_Q}$, define
\[ X_\dd(Q) = \bigoplus_{e\in E_Q} \Hom( \K^{d_{s(e)}}, \K^{d_{t(e)}} )  \quad\quad \textup{and} \quad\quad G_\dd(Q) = \prod_{i\in V_Q} \Gl_{d_i}(\K) \ . \]
The stack of representations $\mf M(Q)$ of $Q$ is defined as the disjoint union of the quotient stacks
\[ \mf{M}(Q) = \bigsqcup_{\dd \in \bb N^{V_Q}} \mf{M}_\dd(Q) \quad\quad\text{where}\quad\quad  \mf{M}_\dd(Q) = \left[ X_\dd(Q) / G_\dd(Q) \right] \ . \]
Let $\Endi(\K^\dd):=\Hom(\K^\dd,\K^\dd)\in S\Bimod$ denote the matrix algebra on $\K^\dd=\oplus_{i\in V_Q} \K^{d_i}$ with its natural $S$ bimodule structure, and note that there are canonical identifications
\[ X_\dd(Q) \cong \Hom_{S\Bimod}( \K\langle E_Q\rangle, \Endi(\K^\dd)) \cong \Hom_{\Alg_\Ass(S\Bimod)}( \K Q, \Endi(\K^\dd) )   \ ,\]
so that we have
\[ \mf{M}_\dd(Q)(\K) \cong \{ V\in S\Mod_\fd,\ \varphi\in \Hom_{\Alg_\Ass(S\Bimod)}( \K Q, \Endi(V) )\ | \ \dim V = \dd \  \} \ ,  \]
that is, the groupoid of geometric points of $\mf{M}_\dd$ is the maximal subgroupoid of modules over the free path algebra $\K Q$ of dimension $\dd$ over $S$.

Now, suppose $(Q,R)$ is a quiver with relations $R\subset \K Q$, as in Definition \ref{dgquiverdefn}, and define the closed subvariety $Z_\dd(Q,R) \subset X_\dd(Q)$ by
\[ Z_\dd(Q,R) = \{ \varphi \in \Hom_{\Alg_\Ass(S\Bimod)}( \K Q, \Endi(\K^\dd) ) \ |\ \varphi(R) = \{0\} \ \subset \Endi(\K^\dd) \}  \ . \]
Note that $Z_\dd(Q,R)$ is $G_\dd(Q)$-invariant, as the condition that the corresponding map to $\Endi_{S\Bimod}(V)$ satisfies $\varphi(R)=0$ is well-defined, independent of the choice of isomorphism $V\cong \K^\dd$. The stack of representations $\mf M(Q,R)$ of $(Q,R)$ is defined analogously as the disjoint union of the quotient stacks
\[ \mf{M}(Q,R) = \bigsqcup_{\dd \in \bb N^{V_Q}} \mf{M}_\dd(Q,R)  \quad\quad\textup{where} \quad\quad \mf{M}_\dd(Q,R) = \left[ Z_\dd(Q,R) / G_\dd(Q) \right]\]
and we have the analogous description of the geometric points
\[ \mf{M}_\dd(Q,R)(\K) = \{ V\in S\Bimod,\ \varphi\in \Hom_{\Alg_\Ass(S\Bimod)}( \K Q/R , \Endi(V) )\ | \ \dim V = \dd \   \}  \]
as parameterizing the modules over the path algebra $\K Q_R=\K Q/R$ of dimension $\dd$ over $S$.

Finally, suppose $(Q,W)$ is a quiver with potential $W\in \K Q_\cyc= \K Q/ [ \K Q, \K Q]$, as in Definition \ref{quivpotdefn}. Then we obtain a canonical function
\[ \Tr(W)_\dd: X_\dd(Q) \to \K \quad\quad\textup{defined by}\quad\quad \varphi \mapsto  \Tr_{\K^\dd}(\varphi(W))  \]
where we identify $\varphi\in X_\dd(Q)\cong \Hom_{\Alg_\Ass(S\Bimod)}( \K Q, \Endi(\K^\dd))$, and we define the closed subvariety $Z_\dd(Q,W)\subset X_\dd(Q)$ by
\[ Z_\dd(Q,W) : = \Crit(\Tr(W)_\dd) = \Gamma_{d\Tr(W)_\dd} \times_{T^\vee X_\dd} X_\dd  \ . \]
Again, $Z_\dd(Q,W)$ is $G_\dd$ invariant, since the function $\Tr(W)_\dd$ is itself, and the stack of representations of $(Q,W)$ is defined as the disjoint union of the quotient stacks
\[ \mf{M}_\dd(Q,W) = \bigsqcup_{\dd \in \bb N^{V_Q}} \mf{M}_\dd(Q,W)  \quad\quad \textup{where}\quad\quad \mf{M}_\dd(Q,W) = \left[ Z_\dd(Q,W) / G_\dd(Q) \right]  \ . \]
There are also canonical derived enhancements of these stacks: letting
\[ Z^h_\dd(Q,W) : = \textup{dCrit}(\Tr(W)_\dd) = \Gamma_{d\Tr(W)_\dd} \times^h_{T^\vee X_\dd} X_\dd  \ . \]
the derived stack of representations of the quiver with potential $(Q,W)$ is defined by
\[ \mf{M}^h_\dd(Q,W) = \bigsqcup_{\dd \in \bb N^{V_Q}} \mf{M}^h_\dd(Q,W)  \quad\quad \textup{where}\quad\quad \mf{M}^h_\dd(Q,W) = \left[ Z^h_\dd(Q,W) / G_\dd(Q) \right] \ . \]

There are subspaces $X_\dd^\nil(Q)\subset X_\dd(Q)$ defined by
\[ X_\dd^\nil(Q) = \{ \varphi \in \Hom_{\Alg_\Ass(S\Bimod)}( \K Q, \Endi(\K^\dd) ) \ |\ \varphi(\K Q_{(n)} ) =  \{0\} \ \subset  \Endi(\K^\dd)  \ \textup{for $n \gg 0$} \}  \ , \]
and in turn a substack $\mf{M}^\nil(Q) \subset \mf M(Q)$  defined by
\[ \mf{M}^\nil(Q) = \bigsqcup_{\dd \in \bb N^{V_Q}} \mf{M}^\nil_\dd(Q) \quad\quad \textup{where}\quad\quad  \mf{M}_\dd^\nil(Q) = \left[ X_\dd^\nil(Q) / G_\dd(Q) \right] \ . \]

Similarly, there are also subvarieties  $Z_\dd^\nil(Q,R)\subset Z_\dd(Q,R)$ and $ Z_\dd^\nil(Q,W) \subset Z_\dd(Q,W)$ defined by
\[ Z_\dd^\nil(Q,R)= Z_\dd(Q,R)\times_{X_\dd(Q)} X_\dd^\nil(Q) \quad\quad \text{and} \quad\quad Z_\dd^\nil(Q,W)=Z_\dd(Q,W)\times_{X_\dd(Q)}  X_\dd^\nil(Q) \ , \]
and in turn closed substacks $\mf M^\nil(Q,R)$ and $\mf M^\nil(Q,W) \subset \mf M^\nil(Q)$ defined by
\[ \mf{M}^\nil(Q,R) = \bigsqcup_{\dd \in \bb N^{V_Q}} \mf{M}^\nil_\dd(Q,R) \quad\quad\textup{where}\quad\quad  \mf{M}_\dd^\nil(Q,R) = \left[ Z_\dd^\nil(Q,R) / G_\dd(Q) \right] \ , \]
and similarly
\[ \mf{M}^\nil(Q,W) = \bigsqcup_{\dd \in \bb N^{V_Q}} \mf{M}^\nil_\dd(Q,W) \quad\quad\textup{where}\quad\quad  \mf{M}_\dd^\nil(Q,W) = \left[ Z_\dd^\nil(Q,W) / G_\dd(Q) \right] \ .\]

We now recall a general construction of the moduli of objects in an abelian category, following for example Definition 7.8 of \cite{AHH}.  Throughout, we let $\mc C$ be a locally Noetherian, cocomplete, $\bb K$-linear abelian category.

\begin{defn}\label{basecatdefn} Let $R$ be a commutative $\bb K$ algebra. The base change category $\mc C_R$ of $\mc C$ is defined to be the category of $R$-module objects internal to $\mc C$, that is, pairs $(C,\xi_C)$ of an object $C\in \mc C$ together with a map of $\bb K$ algebras $\xi_R:R \to \End_{\mc C}(C)$.
\end{defn}

\begin{defn} The moduli stack $\mf M_{\mc C}$ of objects of $\mc C$ is defined by
	\[ \mf M_{\mc C} ( R) = \{ E \in  \mc C_R  \ \big| \  \textup{$E$ is flat over $R$ and compact} \ \} \ ,\]
for each commutative $\bb K$ algebra $R$.
\end{defn}

It is shown in \emph{loc. cit.} that $\mf M_{\mc C}$ indeed defines a stack in the big fppf topology on $\bb K$ schemes. More generally, if $\mc C$ is a non-cocomplete, locally Noetherian, $\bb K$-linear abelian category, we write $\mf M_{\mc C}:= \mf M_{\textup{Ind}(\mc C)}$, and we recall that the finitely presented (or equivalently, compact) objects of $\textup{Ind}(\mc C)$ are given by the objects of $\mc C$, noting that $\mc C$ is abelian and thus idempotent-complete.

We can now state the main result of this subsection. For simplicity, we introduce the notation
\[ \mf M(Y) := \mf M_{\Perv_\cs(\hat Y)} \ . \]
Then we have:

\begin{theo}\label{stackthmunext} Let $Y\xrightarrow{\pi} X$ be a toric resolution of singularities of an affine, toric, Calabi-Yau threefold, such that the fibres of $\pi$ are of dimension $\leq 1$ and $\pi_* \mc O_Y \cong \mc O_X$, and let $(Q_Y,W_Y)$ be the associated quiver with potential. There is an equivalence of algebraic stacks
	\begin{equation}\label{modulistackequivunexteqn} 
		\mf M^\nil(Q_Y, W_Y)  \xrightarrow{\cong} \mf M(Y)
	\end{equation}
	where the induced equivalence of groupoids of $\bb K$ points is defined on objects by
	\begin{equation*}\hspace*{-1cm}
		(V_i, B_{ij})   \mapsto \left( \tilde H:=  \bigoplus_{i\in I} K(I_i)\otimes V_i\ , \  d_B :=  \sum_{k\in \bb Z,\  i,i_2,...,i_{k-1},j \in I}  K (m_k^{\Sigma}(\cdot, b_{i,i_2},...,b_{i_{k-1}, j})^N)\otimes  ( B_{i,i_2}  ...  B_{i_{k-2}, j} )  \right)  \ ,
	\end{equation*}
	in the notation of Section \ref{monadsec}.
	
	Similarly, this induces an equivalence of the homotopy $T$ fixed points
	\begin{equation*}
		\mf M(Q_Y, W_Y)^{T}   \xrightarrow{\cong} \mf M_{\Perv_\cs( Y)^T}  \ .
	\end{equation*}
\end{theo}
The main remaining ingredient of the proof is the following standard fact:

\begin{lemma}\label{quivermodulilemma} Let $\Lambda=\bb K\langle Q\rangle_R$ be the path algebra of a quiver with relations $(Q,R)$ and $\hat \Lambda$ its completion with respect to the filtration by path length. There are equivalences of algebraic stacks
	\[ \mf M(Q,R)  \xrightarrow{\cong} \mf M_{\Lambda\Mod_\Fd} \quad\quad \text{and} \quad\quad \mf M^\nil(Q,R) \xrightarrow{\cong}\mf M_{\hat \Lambda\Mod_\Fd}   \ . \]
\end{lemma}
\begin{proof}
	This is a standard result, but we give a proof a generalization of it in Lemma \ref{quivermodulilemma}; taking $\dd=(\dd_0,0)\in \bb N^{V_{Q}}\times \{0\} \subset \bb N^{V_{Q_M}}$ in \emph{loc. cit.} gives precisely this proposition, by commutativity of the diagram in Equation \ref{KozgeoFcompeqn}.
\end{proof}

\begin{proof}{(of Theorem \ref{stackthmunext}}
	The definition of the moduli of objects functor is evidently natural with respect to equivalences of categories, so that we obtain a canonical equivalence of algebraic stacks 
	\[\mf M_{\hat \Lambda\Mod_\Fr}  \xrightarrow{\cong} \mf M_{\Perv_\cs(\hat Y)}  , \]
	by the equivalence of Theorem \ref{unextheartthm}. Composing with the equivalence of algebraic stacks of Proposition \ref{quivermodulilemma} yields the desired equivalence of Equation \ref{modulistackequivunexteqn}, and the induced equivalence on groupoids of $\bb K$ points is defined on objects by the claimed formula, by Proposition \ref{monadunextAinfprop}. Finally, we note that the quiver with relations determined by $\Sigma$ is in fact a quiver with potential, by Example \ref{geoquiveg2} above.
\end{proof}

\subsection{Examples}\label{egunfrsec}

In this section, we recall computations of the quivers with potential corresponding to Calabi-Yau threefolds in several concrete examples.

\begin{eg}\label{a3eg}
	Let $Y=X=\C^3$, so that we have a single projective object $E=\mc O_{\C^3}$ and corresponding simple object $ F_0 = \iota_*\mc O_\pt$. The algebra $\Sigma$ is given by
	$$ \Sigma \cong \Sym^\bullet( \K^3[-1]) = \K \oplus \K^3[-1] \oplus \K^3[-2] \oplus \K[-3] $$
	the symmetric algebra on 3 generators of degree $1$, and the corresponding quiver with potential $(Q_Y,W_Y)$ is given by
	\[ Q_Y=
	\begin{tikzcd}
		\mathcircled{F_0}\arrow[out=340,in=20,loop,swap,"B_3"]
		\arrow[out=220,in=260,loop,swap,"B_2"]
		\arrow[out=100,in=140,loop,swap,"\normalsize{B_1}"]
	\end{tikzcd}  \quad\quad\text{and}\quad\quad   W_Y = B_1[B_2,B_3]   \ .
	\]
	Thus the path algebra $\Lambda$ is given by the commutative algebra $\K[x,y,z]=\mc O(\C^3)$ of global functions on $\C^3$.
	
The single compactly supported simple object $F=\iota_*\mc O_\pt$ corresponding to the one node of the quiver has projective resolution determined by the injective $\Sigma$ module $I=\Sigma^\vee$ with composition series
	\[ I = \left[  \K < \K^3[1] < \K^3[2] < \K[3]   \right] \ . \]
	There are three independent extension classes of the simple $\Sigma$ module $S=\bb K$ by itself, which correspond to multiplication by the coordinate functions $x,y,z$, under the identification
	\[ \Ext^1_\Sigma(S,S)\cong \Lambda^1= \C^3_{x,y,z} \subset \Lambda = \Hom(\mc O, \mc O) = \mc O(\bb A^3) = \K[x,y,z]  \ , \]
	so that the resolution determined by the composition series of $I$ as in Equation \ref{simplereseqn} is given in this example by
	\begin{equation}
		K(I)=	[ \mc{O}\xrightarrow{\scriptsize{\begin{pmatrix}-x\\ y\\ -z \end{pmatrix}}}\mc{O}^3\xrightarrow{\scriptsize{\begin{pmatrix}0&-z&-y\\ -z&0&x\\ y& x& 0 \end{pmatrix}}}\mc{O}^3\xrightarrow{\scriptsize{\begin{pmatrix}x&y&z \end{pmatrix}}}\mc{O}]\xrightarrow{\cong} \iota_* \mc{O}_{\textup{pt}}=F \ , 
	\end{equation}
	which is simply the usual Koszul resolution of $\iota_* \mc O_\pt$.
	
	The generators $b_i\in \Ext^1(F,F)$ for $i=1,2,3$ correspond to the maps of complexes
	\begin{eqnarray*}
		b_1\quad :\quad 
		\xymatrixcolsep{3.5pc}
		\xymatrixrowsep{3pc}
		\xymatrix{ & \mc{O}\ar[d]^{\scriptsize{\begin{pmatrix}-1\\0\\0\end{pmatrix}}} \ar[r] & \mc{O}^3\ar[d]^{\scriptsize{\begin{pmatrix}0&0 & 0\\0&0&-1\\0&-1&0\end{pmatrix}}} \ar[r] &  \mc{O}^3\ar[d]^{\scriptsize{\begin{pmatrix}1&0&0\end{pmatrix}}} \ar[r] & \mc{O}  \\
			\mc{O} \ar[r] & \mc{O}^3 \ar[r] &  \mc{O}^3 \ar[r] & \mc{O}&}\\
		\\ \nonumber
		b_2\quad :\quad
		\xymatrixcolsep{3pc}
		\xymatrixrowsep{3pc}
		\xymatrix{ & \mc{O}\ar[d]^{\scriptsize{\begin{pmatrix}0\\1\\0\end{pmatrix}}} \ar[r] & \mc{O}^3\ar[d]^{\scriptsize{\begin{pmatrix}0&0 & 1\\0&0&0\\-1&0&0\end{pmatrix}}} \ar[r] &  \mc{O}^3\ar[d]^{\scriptsize{\begin{pmatrix}0&1&0\end{pmatrix}}} \ar[r] & \mc{O}  \\
			\mc{O} \ar[r] & \mc{O}^3 \ar[r] &  \mc{O}^3 \ar[r] & \mc{O}&}\\
		\\ \nonumber
		b_3 \quad :\quad
		\xymatrixcolsep{3pc}
		\xymatrixrowsep{3pc}
		\xymatrix{ & \mc{O}\ar[d]^{\scriptsize{\begin{pmatrix}0\\0\\-1\end{pmatrix}}} \ar[r] & \mc{O}^3\ar[d]^{\scriptsize{\begin{pmatrix}0&1 & 0\\1&0&0\\0&0&0\end{pmatrix}}} \ar[r] &  \mc{O}^3\ar[d]^{\scriptsize{\begin{pmatrix}0&0&1\end{pmatrix}}} \ar[r] & \mc{O}  \\
			\mc{O} \ar[r] & \mc{O}^3 \ar[r] &  \mc{O}^3 \ar[r] & \mc{O}&}
	\end{eqnarray*}
	so that the monad formalism of Proposition \ref{monadunextprop}, which determines the resolution $(\tilde H,d)$ of an object $H\in \Coh_\cs(\bb C^3)$ in terms of the corresponding quiver representation $V$, is given by
	\begin{equation}
		\mc{O}\otimes V \xrightarrow{\scriptsize{\begin{pmatrix}B_1-x\\ y-B_2\\ B_3-z \end{pmatrix}}}\mc{O}^3\otimes  V\xrightarrow{\scriptsize{\begin{pmatrix}0&B_3-z&B_2-y \\ B_3-z &0&x-B_1\\ y-B_2& x-B_1& 0 \end{pmatrix}}}\mc{O}^3 \otimes V\xrightarrow{\scriptsize{\begin{pmatrix}x-B_1&y-B_2&z-B_3 \end{pmatrix}}}\mc{O}\otimes V\ ,
	\end{equation}
	with $\dim V=n$ given by the coefficient of the $K$ theory class of $H$ in terms of the simple object $[H]=n[F]$, or equivalently the length of the corresponding pure sheaf of dimension $0$.
\end{eg}

\begin{eg}\label{o2eg}
	Let $Y=|\mc O_{\bb P^1}\oplus \mc O_{\bb P^1}(-2)| \to X= Y^\textup{aff}\cong  \Spec \C[x_1,x_2,x_3,x_4]/(x_1x_2-x_3^2)$ so that the simple objects in $\Perv_\cs(Y)$ are given by $F_0=\iota_* \mc O_{\bb P^1}$ and $F_1= \iota_*\mc O_{\bb P^1}(-1)[1]$.
	
	The algebra $\Sigma$ can be computed explicitly as follows: let $\tilde F_0=\mc O_{\bb P^1}$ and $\tilde F_1=\mc O_{\bb P^1}(-1)[1]$, and for simplicity we drop the dependence on $\bb P^1$ from the notation. Then we have
	\[ \Sigma =  \bigoplus_{i,j=0,1}  \ _j \Sigma_i \cong \bigoplus_{i,j=0,1} \Ext^\bullet( F_i, F_j\otimes \Sym^\bullet((\mc O \oplus \mc O(-2)) [-1]))  \ ,  \]
	so that we have
	\begin{align*} 
		_0 \Sigma_0 & := \Ext^\bullet( \mc O(-1)[1], \mc O(-1)[1]\otimes \Sym^\bullet((\mc O \oplus \mc O(-2)) [-1])) \\
		& = \Ext^0(\mc O, \mc O) \oplus \Ext^0(\mc O,\mc O)[-1]\oplus \Ext^1(\mc O,\mc O(-2))[-2] \oplus \Ext^1(\mc O, \mc O(-2))[-3] \\
		& \cong \K \oplus \K[-1] \oplus \K[-2] \oplus \K[-3] \\
		_1\Sigma_1 & :=  \Ext^\bullet( \mc O, \Sym^\bullet((\mc O \oplus \mc O(-2)) [-1]))  \\
		& = \Ext^0(\mc O, \mc O) \oplus \Ext^0(\mc O,\mc O)[-1]\oplus \Ext^1(\mc O,\mc O(-2))[-2] \oplus \Ext^1(\mc O, \mc O(-2))[-3] \\
		& \cong \K \oplus \K[-1] \oplus \K[-2] \oplus \K[-3] \\
		_1 \Sigma_0 & :=  \Ext^\bullet( \mc O(-1)[1], \Sym^\bullet((\mc O \oplus \mc O(-2)) [-1])) \\
		& \cong \Ext^0(\mc O, \mc O(1))[-1] \oplus \Ext^0(\mc O,\mc O(1))[-2] \\
		& \cong \K^2[-1] \oplus \K^2[-2] \\
		_0 \Sigma_1 & := \Ext^\bullet( \mc O, \mc O(-1)[1]\otimes \Sym^\bullet((\mc O \oplus \mc O(-2)) [-1])) \\
		& \cong \Ext^1(\mc O, \mc O(-3))[-1] \oplus \Ext^1(\mc O, \mc O(-3))[-2] \\
		& \cong \K^2[-1] \oplus \K^2[-2]  & .
	\end{align*}
and the corresponding quiver with potential $(Q_Y,W_Y)$ is given by
	\[	Q_Y=
	\begin{tikzcd}
	 \arrow[out=160,in=200,loop,swap,"E"] \mathcircled{F_0} \arrow[r, bend left=25 ] \arrow[r, bend left=40 ,  "A\ C"]  & \arrow[l, bend left=25 ] \arrow[l, bend left=40 ,  "B\ D"] \mathcircled{F_1}\arrow[out=340,in=20,loop,swap,"F"]
	\end{tikzcd}
 \quad\quad \text{and}\quad\quad W_Y =E(BC-DA)+ F(AD-CB) \ .
	\] 
	
	The compactly supported objects $F_0=\iota_*\mc O_{\bb P^1}$ and $F_1=\iota_* \mc O_{\bb P^1}(-1)[1]$ corresponding to the two nodes of the quiver have projective resolutions determined by the injective $\Sigma$ modules
	\begin{align*}
		I_0 &  = [ S_0 < S_0\oplus S_1^{\oplus 2}[1] < S_0\oplus S_1^{\oplus 2}[2] < S_0[3] ] \\
		I_1 & = [ S_1 < S_1\oplus S_0^{\oplus 2}[1] < S_1\oplus S_0^{\oplus 2}[2] < S_1[3]  ]  \ .
	\end{align*}
	
	There is a single independent self extension class for each of the simple objects which both correspond to the map of coherent sheaves $y$, as well as two independent extension classes in each of the groups $\Ext^1_\Sigma(S_1,S_2)$ and $\Ext^1_\Sigma(S_2,S_1)$, which correspond to the maps of coherent sheaves $1,z$ and $x,y$, respectively, under the identifications
	\begin{align*}
		\Ext^1_\Sigma(S_0,S_0) & \cong(_0\Lambda_0)^1 \cong \Hom_{\bb P^1}(\mc O_{\bb P^1}, \mc O_{\bb P^1}) = \C_y  & \subset \   _0\Lambda_0 = \Hom_Y(\mc O, \mc O)   \\	
		\Ext^1_\Sigma(S_1,S_1) & \cong(_1\Lambda_1)^1 \cong \Hom_{\bb P^1}(\mc O_{\bb P^1}, \mc O_{\bb P^1}) = \C_y  & \subset \   _1\Lambda_1 = \Hom_Y(\mc O, \mc O)   \\
		\Ext^1_\Sigma(S_0,S_1)  & \cong(_0\Lambda_1)^1 \cong \Hom_{\bb P^1}(\mc O_{\bb P^1}, \mc O_{\bb P^1}(1) ) = \C^2_{1,z} & \subset \   _0\Lambda_1 = \Hom_Y(\mc O, \mc O(1))   \\
		\Ext^1_\Sigma(S_1,S_0) & \cong ( _1 \Lambda_0)^1 \cong \Hom_{\bb P^1} (\mc O_{\bb P^1}, \mc O_{\bb P^1}(1))= \bb C_{x,xz}^2 & \subset \   _1\Lambda_0 = \Hom_Y(\mc O(1), \mc O) \end{align*}
	where $z$ denotes some choice of local coordinate on the base $\bb P^1$ and $x,y$ denote the linear coordinates on the fibres of $\mc O(-2)$ and $\mc O$, respectively. Thus, the resolutions of the simple objects $F_i$ are given by
	\begin{align*}
		K(I_0) & = [ \mc O \xrightarrow{\scriptsize{\begin{pmatrix}y \\ 1 \\ z \end{pmatrix}}} \mc O \oplus \mc O(1)^2 \xrightarrow{\scriptsize{\begin{pmatrix}0 & xz & -x \\ -z & 0  & y \\ 1 & -y & 0  \end{pmatrix}}} \mc O \oplus \mc O(1)^2 \xrightarrow{\scriptsize{\begin{pmatrix}y & x & xz \end{pmatrix}}} \mc O]   \xrightarrow{\cong} s_* \mc O_{\bb P^1}=F_0 & \textup{ and}	\\
		K(I_1) & = [ \mc O(1) \xrightarrow{\scriptsize{\begin{pmatrix}y \\ x \\ xz \end{pmatrix}}}\mc O(1)\oplus \mc O^2 \xrightarrow{\scriptsize{\begin{pmatrix}0 & z & -1 \\ -xz & 0 & y \\ x & -y & 0  \end{pmatrix}}} \mc O(1)\oplus \mc O^2  \xrightarrow{\scriptsize{\begin{pmatrix} y & 1 & z \end{pmatrix}}} \mc O(1)]   \xrightarrow{\cong} s_* \mc O_{\bb P^1}(-1)[1] = F_1 & \ .
	\end{align*}
The maps between these resolutions corresponding to the generators of $\Sigma^1$ are given by:
\begin{align}\nonumber
	e :\quad 
	\xymatrixcolsep{3pc}
	\xymatrixrowsep{3pc}
	\xymatrix{ & \mc{O}\ar[d]^{\scriptsize{\begin{pmatrix}1\\0\\0\end{pmatrix}}} \ar[r] & \mc{O}\oplus \mc O(1)^2\ar[d]^{\scriptsize{\begin{pmatrix}0&0 & 0\\0&0&-1\\0&1&0\end{pmatrix}}} \ar[r] &  \mc{O}\oplus\mc O(1)^2\ar[d]^{\scriptsize{\begin{pmatrix}1& 0 & 0\end{pmatrix}}} \ar[r] & \mc{O}  \\
		\mc{O} \ar[r] & \mc O\oplus \mc{O}(1)^2 \ar[r] &  \mc{O}\oplus \mc{O}(1)^2 \ar[r] & \mc{O} & }\\ \nonumber
	f: \quad 
	\xymatrixcolsep{3pc}
	\xymatrixrowsep{3pc}
	\xymatrix{ & \mc{O}(1)\ar[d]^{\scriptsize{\begin{pmatrix}1\\0\\0\end{pmatrix}}} \ar[r] & \mc{O}(1)\oplus \mc O^2\ar[d]^{\scriptsize{\begin{pmatrix}0&0 & 0\\0&0&-1\\0&1&0\end{pmatrix}}} \ar[r] &  \mc{O}(1)\oplus\mc O^2\ar[d]^{\scriptsize{\begin{pmatrix}1& 0 & 0\end{pmatrix}}} \ar[r] & \mc{O}(1)  \\
		\mc{O}(1) \ar[r] & \mc O(1)\oplus \mc{O}^2 \ar[r] &  \mc{O}(1)\oplus \mc{O}^2 \ar[r] & \mc{O}(1) &}\\ \nonumber
	a :\quad 
	\xymatrixcolsep{3pc}
	\xymatrixrowsep{3pc}
	\xymatrix{ & \mc{O}\ar[d]^{\scriptsize{\begin{pmatrix}0\\1\\0\end{pmatrix}}} \ar[r] & \mc{O}\oplus \mc O(1)^2\ar[d]^{\scriptsize{\begin{pmatrix}0&0 & 1\\0&0&0\\-1&0&0\end{pmatrix}}} \ar[r] &  \mc{O}\oplus\mc O(1)^2\ar[d]^{\scriptsize{\begin{pmatrix}0& 1 & 0\end{pmatrix}}} \ar[r] & \mc{O}  \\
		\mc{O}(1) \ar[r] & \mc O(1)\oplus \mc{O}^2 \ar[r] &  \mc{O}(1)\oplus \mc{O}^2 \ar[r] & \mc{O}(1) & }\\ \nonumber
	c: \quad 
	\xymatrixcolsep{3pc}
	\xymatrixrowsep{3pc}
	\xymatrix{ & \mc{O}\ar[d]^{\scriptsize{\begin{pmatrix}0\\0\\1\end{pmatrix}}} \ar[r] & \mc{O}\oplus \mc O(1)^2\ar[d]^{\scriptsize{\begin{pmatrix}0&-1 & 0\\1&0&0\\0&0&0\end{pmatrix}}} \ar[r] &  \mc{O}\oplus\mc O(1)^2\ar[d]^{\scriptsize{\begin{pmatrix}0&0&1\end{pmatrix}}} \ar[r] & \mc{O}  \\
		\mc{O}(1) \ar[r] & \mc O(1)\oplus \mc{O}^2 \ar[r] &  \mc{O}(1)\oplus \mc{O}^2 \ar[r] & \mc{O}(1) &}\\\nonumber
	\\ \nonumber
	b: \quad 
	\xymatrixcolsep{3pc}
	\xymatrixrowsep{3pc}
	\xymatrix{ & \mc{O}(1)\ar[d]^{\scriptsize{\begin{pmatrix}0\\1\\0\end{pmatrix}}} \ar[r] & \mc{O}(1)\oplus \mc O^2\ar[d]^{\scriptsize{\begin{pmatrix}0&0 & 1\\0&0&0\\-1&0&0\end{pmatrix}}} \ar[r] &  \mc{O}(1)\oplus\mc O^2\ar[d]^{\scriptsize{\begin{pmatrix}0& 1 & 0\end{pmatrix}}} \ar[r] & \mc{O}(1)  \\
		\mc{O} \ar[r] & \mc O\oplus \mc{O}(1)^2 \ar[r] &  \mc{O}\oplus \mc{O}(1)^2 \ar[r] & \mc{O} &} \\\nonumber
	d: \quad 
	\xymatrixcolsep{3pc}
	\xymatrixrowsep{3pc}
	\xymatrix{ & \mc{O}(1)\ar[d]^{\scriptsize{\begin{pmatrix}0\\0\\1\end{pmatrix}}} \ar[r] & \mc{O}(1)\oplus \mc O^2\ar[d]^{\scriptsize{\begin{pmatrix}0&-1 & 0\\1&0&0\\0&0&0\end{pmatrix}}} \ar[r] &  \mc{O}(1)\oplus\mc O^2\ar[d]^{\scriptsize{\begin{pmatrix}0&0&1\end{pmatrix}}} \ar[r] & \mc{O}(1)  \\
		\mc{O} \ar[r] & \mc O\oplus \mc{O}(1)^2 \ar[r] &  \mc{O}\oplus \mc{O}(1)^2 \ar[r] & \mc{O} &}
\end{align}
so that the monad formalism of Proposition \ref{monadunextprop} in this example determines the resolution $(\tilde H,d)$ of an object $H\in\Perv_\cs(Y)$ in terms of the quiver representation $(V_0,V_1)$ by
\begin{equation*}
	\xymatrix @R=-.3pc @C=-2.2pc{ & {\scriptsize{\begin{pmatrix} y-E & 0 \\ 1 & B \\ z & D \\ 0 & y-F \\ A & x \\ C & xz\end{pmatrix}}} & && {\scriptsize{\begin{pmatrix} 0 & xz & -x & 0 & D & -B \\ -z & 0 & y-E & -D & 0 & 0 \\ 1 & E-y & 0 & B & 0 & 0 \\ 0 & C & -A & 0 & z & -1 \\ -C & 0 & 0 & -xz & 0 & y-F \\ A & 0 & 0 & x & F-y& 0  \end{pmatrix}}} && && \scriptsize{\begin{pmatrix} y-E & x & xz & 0 & B & D \\ 0 & A & C & y-F & 1 & z \end{pmatrix}}&\\
 \mc O \otimes V_0   & &( \mc O \oplus \mc O(1)^2) \otimes V_0  && && (\mc O \oplus \mc O(1)^2) \otimes V_0 &&  &&\mc O  \otimes V_0 	\\
 \oplus&  \longrightarrow & \oplus && \longrightarrow && \oplus && \longrightarrow && \oplus \\
 \\
 \mc O(1) \otimes V_1  & & (\mc O(1)\oplus \mc O^2) \otimes V_1 &&& &   (\mc O(1)\oplus \mc O^2) \otimes V_1 &&&& \mc O(1) \otimes V_1} \ .
\end{equation*}
\end{eg}

\begin{eg}\label{o1eg} Let $Y=|\mc O_{\bb P^1}(-1)\oplus \mc O_{\bb P^1}(-1)|  \to X = Y^\textup{aff}\cong \Spec \C[x_1,x_2,x_3,x_4]/(x_1x_2-x_3x_4)$ so that the simple objects in $\Perv_\cs(Y)$ are given by $F_0=\iota_* \mc O_{\bb P^1}$ and $F_1= \iota_*\mc O_{\bb P^1}(-1)[1]$, as in the previous example.
		
The algebra $\Sigma$ can be computed explicitly, as in the previous example, by
\[ \Sigma=\bigoplus_{i,j=0,1} \ _i\Sigma_j  \cong \bigoplus_{i,j=0,1} \Ext^\bullet( F_i, F_j\otimes \Sym^\bullet((\mc O(-1) \oplus \mc O(-1)) [-1])) \]
where we have
		\begin{align*} 
			_0 \Sigma_0 & := \Ext^\bullet(\mc O(-1)[1],\mc O(-1)[1]\otimes \Sym^\bullet((\mc O(-1) \oplus \mc O(-1)) [-1])) \\ 
			& \cong \Ext^0(\mc O,\mc O) \oplus \Ext^1(\mc O,\mc O(-2))[-3] \\
			& \cong \K \oplus \K[-3] \\
			_1 \Sigma _1 & := \Ext^\bullet(\mc O,\Sym^\bullet((\mc O(-1) \oplus \mc O(-1)) [-1])) \\ 
			& \cong \Ext^0(\mc O,\mc O) \oplus \Ext^1(\mc O,\mc O(-2))[-3] \\
			& \cong \K \oplus \K[-3] \\
			_0 \Sigma _1 & : = \Ext^\bullet(\mc O,\mc O(-1)[1]\otimes \Sym^\bullet((\mc O(-1) \oplus \mc O(-1)) [-1])) \\
			& \cong  \Ext^1(\mc O, \mc O(-2))^{\oplus 2}[-1] \oplus \Ext^1(\mc O,\mc O(-3))[-2] \\
			& \cong \K^2[-1] \oplus \K^2[-2] \\
			_1 \Sigma _0 & : = \Ext^\bullet(\mc O(-1)[1], \Sym^\bullet((\mc O(-1) \oplus \mc O(-1)) [-1])) \\		
			& \cong	\Ext^0(\mc O, \mc O(1))[-1] \Ext^0(\mc O, \mc O)^{\oplus 2}[-2]\\
			& \cong \K^2[-1] \oplus \K^2[-2] 
		\end{align*}
Thus, the corresponding quiver with potential $(Q_Y,W_Y)$ is given by
\[Q_Y=
\begin{tikzcd}
	\mathcircled{F_0} \arrow[r, bend left=25 ] \arrow[r, bend left=40 ,  "A\ C"]  & \arrow[l, bend left=25 ] \arrow[l, bend left=40 ,  "B\ D"]  \mathcircled{F_1}
\end{tikzcd}
\quad\quad \text{and}\quad\quad
W_Y= ABCD-ADCB \ .
\]

	The compactly supported objects $F_0=s_*\mc O_{\bb P^1}$ and $F_1=s_* \mc O_{\bb P^1}(-1)[1]$ corresponding to the two nodes of the quiver have projective resolutions determined by the injective $\Sigma$ modules
	\begin{align*}
		I_1 &  = [ S_1 < S_2^{\oplus 2}[1] < S_2^{\oplus 2}[2] < S_1[3] ] \\
		I_2 & = [ S_2 < S_1^{\oplus 2}[1] < S_1^{\oplus 2}[2] < S_2[3]  ]  \ .
	\end{align*}
	
	There are no non-trivial self extensions of the simple objects, while there are two independent extension classes in each of the groups $\Ext^1_\Sigma(S_1,S_2)$ and $\Ext^1_\Sigma(S_2,S_1)$, which correspond to the maps of coherent sheaves $1,z$ and $x,y$, respectively, under the identifications
	\begin{align*}
		\Ext^1_\Sigma(S_1,S_2)  & \cong(_1\Lambda_2)^1 \cong \Hom_{\bb P^1}(\mc O_{\bb P^1}, \mc O_{\bb P^1}(1) ) = \C^2_{1,z} & \subset \   _1\Lambda_2 = \Hom_Y(\mc O, \mc O(1))   \\
		\Ext^1_\Sigma(S_2,S_1) & \cong ( _2 \Lambda_1)^1 \cong \Hom_{\bb P^1} (\mc O_{\bb P^1}, \mc O_{\bb P^1}^{\oplus 2})= \bb C_{x,y}^2 & \subset \   _2\Lambda_1 = \Hom_Y(\mc O(1), \mc O) \end{align*}
	where $z$ denotes some choice of local coordinate on the base $\bb P^1$ and $x,y$ denote the linear coordinates on the fibres of the rank two bundle. Thus, the resolutions of the simple objects $F_i$ are given by
	\begin{align*}
		K(I_1) & = [ \mc O \xrightarrow{\scriptsize{\begin{pmatrix}1 \\ z \end{pmatrix}}} \mc O(1)^2 \xrightarrow{\scriptsize{\begin{pmatrix}zy & -y \\ -zx & x  \end{pmatrix}}} \mc O(1)^2 \xrightarrow{\scriptsize{\begin{pmatrix}x & y \end{pmatrix}}} \mc O]   \xrightarrow{\cong} s_* \mc O_{\bb P^1}=F_1	\\
		K(I_2) & = [ \mc O(1) \xrightarrow{\scriptsize{\begin{pmatrix} x \\ y \end{pmatrix}}} \mc O^2 \xrightarrow{\scriptsize{\begin{pmatrix}yz & -xz \\ -y & x  \end{pmatrix}}} \mc O^2  \xrightarrow{\scriptsize{\begin{pmatrix}1 & z \end{pmatrix}}} \mc O(1)]   \xrightarrow{\cong} s_* \mc O_{\bb P^1}(-1)[1] = F_2
	\end{align*}
	The maps between these resolutions corresponding to the generators of $\Sigma^1$ are given by:
	\begin{align}\nonumber
		a :\quad 
		\xymatrixcolsep{3pc}
		\xymatrixrowsep{3pc}
		\xymatrix{ & \mc{O}\ar[d]^{\scriptsize{\begin{pmatrix}-1 \\ 0 \end{pmatrix}}} \ar[r] & \mc O(1)^2\ar[d]^{\scriptsize{\begin{pmatrix}0&-y \\ y &0 \end{pmatrix}}} \ar[r] & \mc O(1)^2\ar[d]^{\scriptsize{\begin{pmatrix} 1 & 0\end{pmatrix}}} \ar[r] & \mc{O}  \\
			\mc{O}(1) \ar[r] & \mc{O}^2 \ar[r] &  \mc{O}^2 \ar[r] & \mc{O}(1) & }\\ \nonumber
		c: \quad 
		\xymatrixcolsep{3pc}
		\xymatrixrowsep{3pc}
		\xymatrix{ & \mc{O}\ar[d]^{\scriptsize{\begin{pmatrix}0\\-1\end{pmatrix}}} \ar[r] & \mc O(1)^2\ar[d]^{\scriptsize{\begin{pmatrix}0&x \\-x&0 \end{pmatrix}}} \ar[r] & \mc O(1)^2\ar[d]^{\scriptsize{\begin{pmatrix}0& 1 \end{pmatrix}}} \ar[r] & \mc{O}  \\
			\mc{O}(1) \ar[r] & \mc{O}^2 \ar[r] &  \mc{O}^2 \ar[r] & \mc{O}(1) & }\\ \nonumber
		b: \quad 
		\xymatrixcolsep{3pc}
		\xymatrixrowsep{3pc}
		\xymatrix{ & \mc{O}(1)\ar[d]^{\scriptsize{\begin{pmatrix}-1\\0\end{pmatrix}}} \ar[r] & \mc O^2\ar[d]^{\scriptsize{\begin{pmatrix}0 & -z\\ z&0\end{pmatrix}}} \ar[r] &  \mc O^2\ar[d]^{\scriptsize{\begin{pmatrix} 1 & 0\end{pmatrix}}} \ar[r] & \mc{O}(1)  \\
			\mc{O} \ar[r] & \mc{O}(1)^2 \ar[r] & \mc{O}(1)^2 \ar[r] & \mc{O} &} \\ \nonumber
		d: \quad 
		\xymatrixcolsep{3pc}
		\xymatrixrowsep{3pc}
		\xymatrix{ & \mc{O}(1)\ar[d]^{\scriptsize{\begin{pmatrix}0 \\-1 \end{pmatrix}}} \ar[r] & \mc O^2\ar[d]^{\scriptsize{\begin{pmatrix}0& 1\\-1&0 \end{pmatrix}}} \ar[r] &  \mc O^2\ar[d]^{\scriptsize{\begin{pmatrix}0& 1 \end{pmatrix}}} \ar[r] & \mc{O}(1)  \\
			\mc{O} \ar[r] & \mc{O}(1)^2 \ar[r] & \mc{O}(1)^2 \ar[r] & \mc{O} &} \
	\end{align}
	so that the monad formalism of Proposition \ref{monadunextAinfprop} in this example determines the resolution $(\tilde H,d)$ of an object $H\in\Perv_\cs(Y)$ in terms of the quiver representation $(V_0,V_1)$ by
	\begin{equation*}
		\xymatrix @R=-.3pc @C=-1.5pc{ & {\scriptsize{\begin{pmatrix}  1 & -B \\ z & -D \\ -A & x \\ - C & y \end{pmatrix}}} & && {\scriptsize{\begin{pmatrix}zy-DC & CB-y & 0 & zB-D \\ D A - zx & x-BA & D-zB & 0 \\ 0 & yA-xC & yz - CD & AD-xz \\ xC-yA & 0 & CB-y & x-AB \end{pmatrix}}} && && \scriptsize{\begin{pmatrix} x & y & B & D \\ A & C & 1 & z \end{pmatrix}}&\\
			\mc O \otimes V_0   & & \mc O(1)^2 \otimes V_0  && &&\mc O(1)^2 \otimes V_0 &&  &&\mc O  \otimes V_0 	\\
			\oplus&  \longrightarrow & \oplus && \longrightarrow && \oplus && \longrightarrow && \oplus \\
			\\
			\mc O(1) \otimes V_1  & & \mc O^2 \otimes V_1 &&& &    \mc O^2 \otimes V_1 &&&& \mc O(1) \otimes V_1} \ .
	\end{equation*}

\end{eg}

\subsection{Koszul resolutions from Beilinson spectral sequences}\label{beilressec} In this section, we explain the existence of the canonical Koszul resolutions of perverse coherent sheaves described in Equation \ref{Kosreseqn}, in terms of the Beilinson spectral sequence induced by a natural resolution of the diagonal in terms of the distinguished projective generators. Similar arguments feature in the proofs of some of the main results of \cite{Bdg1}, \cite{VdB1},\cite{VdB2}, and \cite{BrKR}, though in this section we follow most closely the explicit discussion in \cite{King}.

\begin{warning} The results of this section will not be used in the remainder of the text, so the proofs are omitted for brevity. We recommend \cite{King} for statements of the results closest to the presentation here.
\end{warning}

Throughout this section, let $Y$ be a smooth variety, $E\in \DD^b\Coh(Y)$ a classical tilting object, and $\Lambda=\Hom_{\DD^b\Coh(Y)}(E,E)$ the (classical) associative algebra of endomorphisms of $E$. Then as above, we obtain inverse equivalences of triangulated categories 
\[ \Hom_{\DD^b\Coh(Y)}(E,\cdot):\DD^b\Coh(Y) \xymatrix{  \ar@<.5ex>[r]^{\cong} & \ar@<.5ex>[l]} {\textup{D}_\perf}(\Lambda) : (\cdot)\otimes_{\Lambda}E  \ , \] 
so that in particular we obtain a natural isomorphism
\begin{equation}\label{tilttransfeqn}
	\Hom_{\DD^b\Coh(Y)}(E,H)\otimes_\Lambda E \xrightarrow{\cong} H
\end{equation} 
for each object $H\in \DD^b\Coh(Y)$.

The notions of $\Lambda$ modules in $\DD^b\Coh(Y)$ and their tensor products with objects in $\DD(\Lambda)$ occuring in the preceding expressions are defined because $\DD^b\Coh(Y)$ is tensored over $\DGV$. In general, the tensor product of two objects $H_1\in \Lambda\Mod(\DD^b\Coh(Y)), H_2 \in \Lambda\Mod(\DD^b\Coh(Y))$ can be defined in terms of the induced algebra $\Lambda_{\mc O_Y} = \Lambda\otimes_\K \mc O_Y \in \Alg_\Ass(\DD^b\Coh(Y))$ by
\[ H_1\otimes_\Lambda H_2= H_1 \otimes_{ \Lambda_{\mc O_Y} } H_2 \]
where on the right hand side the objects $H_i$ are interpreted in the usual sense of module objects over an algebra object internal to the same category $\DD^b\Coh(Y)$. Similarly, one defines
\[ H_1 \boxtimes_\Lambda H_2 = \pi_1^* H_1 \otimes_{\Lambda_{\mc O_{Y\times Y}}} \pi_2^*H_2  \quad \in \DD^b\Coh(Y\times Y)  \ ,\]
which satisfies the usual identification
\[ \Delta^* (H_1 \boxtimes_\Lambda H_2) =  H_1\otimes_\Lambda H_2 \ . \]

Now, taking $H=\mc O_Y$, the isomorphism of Equation \ref{tilttransfeqn} is given by
\[ E^\vee \otimes_\Lambda E \xrightarrow{\cong} \mc O_Y \quad\quad \text{inducing a map}\quad\quad E^\vee \boxtimes_\Lambda E \to \Delta_* \mc O_Y  \] 
under the $(\Delta^*,\Delta_*)$ adjunction. Moreover, we have:

\begin{prop}\cite{King} Let $E\in \Coh(Y)$ a classical tilting object which is in addition a vector bundle. Then the natural map
	\[ E^\vee \boxtimes_\Lambda E \xrightarrow{\cong} \Delta_* \mc O_Y \]
	is an isomorphism.
\end{prop}

Concretely, the resulting resolution of the diagonal can be computed explicitly as follows: given a projective resolution of $\Lambda$ as a $(\Lambda,\Lambda)$ bimodule
\[ Q^\bullet = \left[ \cdots \to Q^i \to \cdots \to Q^{-1} \to Q^0 \right] \xrightarrow{\cong} \Lambda \ , \]
the exterior tensor product is computed by the complex\begin{equation}\label{diagreseqn}
	E^\vee \boxtimes^\bullet_\Lambda E := \left[  \pi_1^* E^\vee \otimes_{\Lambda_{\mc O_{Y\times Y}}} Q^\bullet_{\mc O_{Y\times Y}} \otimes_{\Lambda_{\mc O_{Y\times Y}}} \pi_2^* E  \right] \xrightarrow{\cong}  E^\vee \boxtimes_\Lambda E   \ . 
\end{equation} 

Now, we return to the setting of the preceding section, and suppose $Y$ is a quiver-accessible variety with distinguished projective objects $E_i$ and simple objects $F_i$ in $\DD^b\Coh(Y)$ with corresponding $\Lambda$ modules $P_i$ and $S_i$, so that $\Lambda$ is equivalent to the path algebra of a DG quiver determined by $\Sigma=\Ext_{\DD(\Lambda)}^\bullet(S,S)$. In particular, we also assume that the objects $E_i$ are vector bundles.

In this setting, there is a natural projective resolution of $\Lambda$ as a $(\Lambda,\Lambda)$ bimodule with
\[ Q^k = \bigoplus_{i,j\in V_Q} \Lambda S_i \otimes_\K (S_i \Sigma^k S_j)^\vee \otimes_\K S_j \Lambda \ , \]
where $i,j$ runs over $V_Q$ the index set of the distinguished projective objects, which is by definition the vertex set of the corresponding quiver $Q$; the resolution is determined by the identifications
\[ \Lambda = \Lambda\otimes_S S \otimes_\Lambda S \otimes_S \Lambda \cong \Lambda \otimes_S ((\otimes^\bullet_S \bar\Lambda[1]) \otimes_S\Lambda) \otimes_\Lambda S \otimes_S \Lambda \cong \Lambda \otimes_S \Sigma^\vee \otimes_S \Lambda   \]
where we have used the Koszul resolution of $\Lambda$ together with the identification $(\otimes^\bullet_S \bar\Lambda[1]) \cong \Sigma^\vee$. The resulting resolution of the diagonal in Equation \ref{diagreseqn} above is thus
\[  E^\vee \boxtimes^\bullet_\Lambda E= \left[ \bigoplus_{i,j\in V_Q}  \pi_1^* E_i^\vee \otimes (S_i \Sigma^\bullet S_j)_{\mc O_{Y^{\times 2}}}^\vee \otimes \pi_2^* E_j \right] \xrightarrow{\cong} \Delta_* \mc O_Y \ .  \]

Now, given an arbitrary object $H\in \Perv_\cs(Y)^{\C^\times}$, we obtain a canonical resolution of $H$ via the Beilinson spectral sequence associated to the above resolution of the diagonal: following the standard argument, we have a canonical isomorphism
\[ H \cong \pi_{2 *} \left( \pi_1^* H \otimes \Delta_* \mc O_Y \right) . \]
and the Grothendieck spectral sequence for the composition of derived functors, computed using the resolution of the diagonal constructed above, yields the canonical resolution
\begin{equation}\label{BeilKozres}
	\left[ \bigoplus_{i,j\in V_Q} \Hom_{\DD^b\Coh(Y)}(E_i,H) \otimes_\K (S_i\Sigma^\bullet S_j)^\vee \otimes_\K E_j  \right] \xrightarrow{\cong} H \ , 
\end{equation}
so that by construction, we have:
\begin{prop}\label{BeilinsonssKozprop} Under the equivalence $\Lambda\Mod\xrightarrow{\cong} \Perv(Y)^{\C^\times}$, the resolutions of objects in $\Lambda\Mod$ constructed in terms of Koszul duality, as in Equation \ref{Kosreseqn}, are canonically identified with those induced by the Beilinson spectral sequence, as in Equation \ref{BeilKozres}.
\end{prop}

\section{Perverse coherent extensions on Calabi-Yau threefolds and extended quivers}\label{extsec}

\subsection{Overview of Section \ref{extsec}}\label{Extoverviewsec}

Let $\pi:X\to Y$ be as in Section \ref{Geosetupsec}. In this section, we extend the results of Theorems \ref{catthmunext}, \ref{unextAinfheartthm}, and \ref{stackthmunext} to describe certain categories generated by iterated extensions of compactly supported perverse coherent sheaves on $Y$ together with an auxiliary object $M\in \DD^b\Coh(Y)^{T}$ which does not necessarily have compact support. These descriptions are given in terms of finite rank representations of an \emph{extended} quiver, in a sense we will explain below. We recommend the reader consult the description of the motivation for these constructions given in Section \ref{agintrosec} of the introduction prior to reading this section.

Throughout this section, we suppose that $M\in\DD^b\Coh(Y)^{T}$ satisfies the following hypotheses:
\begin{enumerate}
	\item The algebra $\Ext^0(M,M)$ is commutative,
	\item $\Ext^i(F,M)=0$ and $\Ext^i(M,F)=0$ for all $i\leq 0$ and $F\in \Perv_\cs(Y)$, and
	\item $M\in \Perv^\p(Y)$ lies in the heart of an admissible Bridgeland-Deligne perverse coherent t-structure on $Y$, in the sense of Definitions \ref{BDpervcohdef} and \ref{BDadmissibledef}, and in particular $\Ext^i(M,M)=0$ for all $i<0$.
\end{enumerate}

In Section \ref{tstrsec}, we introduce the Bridgeland-Deligne perverse coherent t-structures used in the above hypotheses. In Sections \ref{pervextcatsec}, \ref{monadextsec}, and \ref{quiversecfr}  we explain the analogues of the results of Sections \ref{Kozpatsec}, \ref{monadsec} and \ref{quiversec} in the setting of perverse coherent extensions. In Section \ref{framingsec}, we define the notion of framing structures and prove Theorem \ref{Athm} from the introduction. In Section \ref{framedexamplessec}, we explain many examples of these constructions relevant to the applications to representation theory in Section \ref{funsec}.

\subsection{Bridgeland-Deligne t-structures}\label{tstrsec}

To begin, we recall the classical theory of tilting $t$-structures by torsion theories.

\begin{defn} Let $\mc B$ be an abelian category. A torsion theory on $\mc B$ is a pair $(\mc T,\mc F)$ of full subcategories of $\mc B$ such that
	\begin{enumerate}
		\item $\Hom(T,F)=0$ for each $T\in \mc T$ and $F \in \mc F$, and
		\item for each $E\in \mc B$ there is a short exact sequence
		\[ 0 \to T \to E \to F \to 0 \]
		with $T\in \mc T$ and $F\in \mc F$.
	\end{enumerate}
\end{defn}

Let $\mc D$ be a triangulated category with $t$-structure $(\D^{\leq 0},\D^{\geq 0})$ and heart $\mc B= \D^{\leq 0}\cap \D^{\geq 0}$. Then, given a torsion theory $(\mc T,\mc F)$ on $\mc B$, define
\[ \D_t^{\leq 0 }=\{ E \in \D^{\leq 0} \ | \ H^0E \in \mc T \}\quad\quad \text{and}\quad\quad \D_t^{\geq 0} = \{ E \in \D^{\geq -1 } \ |\ H^{-1}E \in \mc F\}  \ .\]

\begin{prop} The pair $(\D_t^{\leq 0},\D_t^{\geq 0})$ defines a $t$-structure on $\D$.
\end{prop}
\begin{proof}
	See for example Proposition 2.1 of \cite{HRS}.
\end{proof}

The heart of this $t$-structure is given by
\[ \mc B_t:= \{ E\in \D^{[-1,0]} \ | \ H^0E \in \mc T, \ H^{-1}E\in \mc F\} \ , \]
or more concretely, if $\D$ is the derived category of the abelian category $\mc B$ with the standard $t$ structure, then $\mc B_t$ is the full subcategory of two-term complexes of the form
\[ \mc B_t = \{ E=\left[ E^{-1}[1]  \xrightarrow{\varphi} E^0\right] \ \bigg| \ E^i\in \mc B,\  \coker\varphi \in \mc T,\  \ker \varphi \in \mc F   \} \ .\]

In particular, we have:
\begin{corollary}  $\mc B_t$ is an abelian category.
\end{corollary}

\begin{eg} Let $\mc B=\Coh(X)$ be the abelian category of coherent sheaves on $X$ an irreducible algebraic variety of dimension $d$, $\mc T$ the full subcategory on objects with support of dimension $\leq d-1$, and $\mc F$ the full subcategory on torsion free coherent sheaves on $X$. There are evidently no possible maps $\mc T$ to $\mc F$, and for each $E\in \Coh(X)$, the torsion filtration
	\[	 0 \to T_{d-1}E \to E \to E/T_{d-1}E \to 0 \]
gives the required short exact sequence, where $T_{d-1}E$ denotes the maximal subsheaf of $E$ with support of dimension $\leq d-1$.

The resulting tilted category $\mc B_t$ is given by
\[ \mc B_t = \{ E=\left[ E^{-1}[1]  \xrightarrow{\varphi} E^0\right] \ \bigg| \ E^i\in \Coh(X),\ \dim\textup{supp}H^0E\leq d-1, \ H^{-1}E \textup{ is torsion free} \ \} \ . \] 
\end{eg}

More generally, we have:
\begin{eg}\label{dimsupptilteg} Let $\mc B=\Coh(X)$ as in the previous example, $\mc T=\Coh_{\leq k}(X)$ the full subcategory on objects with support of dimension $\leq k$, and
\[\mc F=\Coh_{\geq k+1}(X):=\mc T^\perp  \ , \]
the right orthogonal to $\mc T$. Since $\mc B$ is Noetherian and $\mc T$ is closed under extensions and quotients, we have that $(\mc T, \mc F)$ is a torsion pair. The resulting tilted category $\mc B_t$ is given by
\[ \mc B_t = \{ E=\left[ E^{-1}[1]  \xrightarrow{\varphi} E^0\right] \ \bigg| \ E^i\in \Coh(X),\ \dim\textup{supp}\ H^0E\leq k, \ H^{-1}E\in \Coh_{\geq k+1}(X) \ \} \ . \]
\end{eg}

Another variant of this construction is to use coherent sheaves supported along a particular subvariety:

\begin{eg} Let $\mc B=\Coh(X)$ as in the previous example, and fix a closed subvariety $Z\subset X$. Then let $\mc T=\Coh_Z(X)$ denote the full subcategory on objects supported on $Z$ and $\mc F = \mc T^\perp$. Since $\mc B$ is Noetherian and $\mc T$ is closed under extensions and quotients, we have that $(\mc T, \mc F)$ is a torsion pair. The resulting tilted category $\mc B_t$ is given by
\[ \mc B_t = \{ E=\left[ E^{-1}[1]  \xrightarrow{\varphi} E^0\right] \ \bigg| \ E^i\in \Coh(X),\ \textup{supp}\ H^0E\subset Z, \ H^{-1}E\in \Coh_Z(X)^\perp \ \} \ . \]
\end{eg}

In fact, the previous construction can be iterated to define a family of t-structures on the derived category of coherent sheaves; this construction is recorded in \cite{ArBez}, following unpublished results of Deligne. Similar results appeared in \cite{Gab} and \cite{Kash}, and we recommend the latter for a more explicit explanation of the geometric interpretation of the perversity function. We now recall the general definition:

\begin{defn} Let $X$ be a variety and $X^\textup{top}$ denote the Zariski topological space underlying $X$. A (monotone and comonotone) \emph{perversity function} $\p:X^\textup{top} \to \bb Z$ is a function such that
	\[ \p(y) \geq \p(x) \quad\quad \text{and} \quad\quad  \p(y) \leq \p(x) + \dim(x)-\dim(y)  \quad\quad \text{for any} \quad y \in \overline{\{x\}}  \ . \]
Given a perversity function $\p:X^\textup{top} \to \bb Z$, the \emph{Deligne perverse coherent t-structure} associated to $\p$ is defined by
\begin{align*}
	\DD^b\Coh(X)^{\p,\leq 0} & := \{ E \in \DD^b\Coh(Y) \ |\  \iota_x^*E \in \DD^{\leq \p(x)}(\mc O_x)\ \textup{for any $x\in X^\textup{top}$} \}  &  ,\\
	\DD^b\Coh(X)^{\p,\geq 0} &:=\{ E \in \DD^b\Coh(Y) \  |  \  \iota_x^! E \in \DD^{\geq \p(x)}(\mc O_x)\ \textup{for any $x\in X^\textup{top}$} \}  &  .
\end{align*}
\end{defn}

Indeed, we have the following result:

\begin{theo}\label{PervCohthm}\cite{ArBez} The pair $(\DD^b\Coh(X)^{\p,\leq 0},\DD^b\Coh(X)^{\p,\geq 0})$ defines a t-structure on $\DD^b\Coh(X)$.
\end{theo}
\begin{proof} See Theorem in \cite{ArBez}.
\end{proof}

We introduce the notation $\Perv^\p(X)$ to denote the category of Deligne perverse coherent sheaves on $X$ associated to perversity function $\p$

These t-structures are often called simply perverse coherent t-structures. We use the prefix Deligne to avoid confusion with the notion of perverse coherent sheaf defined by Bridgeland, recalled in Definition \ref{pervcohdef}. In fact, the category of perverse coherent sheaves (in the sense of Bridgeland) can also be constructed by tilting the abelian category of coherent sheaves with respect to a torsion theory:

\begin{eg} Let $f:Y\to X$ be as in Section \ref{NCCRsec} and let $\mc C^\heartsuit$ denote the full subcategory of $\Coh(Y)$ on objects $C$ such that $f_*C=0$. Define full subcategories $\mc T$ and $\mc F$ of $\Coh(Y)$ by
\begin{align*}
 \mc T  & = \{ E \in \Coh(Y) \ | \ \bb R^1f_*T=0, \ \Hom(T,C)=0 \ \text{for any $C\in \mc C^\heartsuit$}\ \} \\
 \mc F  & = \{ E \in \Coh(Y) \ | \ \bb R^0f_*T=0 \ \} \ .
\end{align*}
From the first exact triangle of Equation \ref{Bdgrecseqeqn}, one can check that the pair $(\mc T,\mc F)$ defines a torsion theory on $\Coh(Y)$, and the resulting tilted category is equivalent to the category of perverse coherent sheaves,
\[ \Coh(Y)_t = \Perv(Y) \ , \]
by Corollary \ref{PervCohcoro}.
\end{eg}

We can now introduce the desired $t$-structures, which will be used to define the categories of perverse coherent extensions.

\begin{defn}\label{BDpervcohdef}  Let $f:Y\to X$ be as in Section \ref{NCCRsec}, and let $\p:X^\textup{top} \to \bb Z$ be a perversity function on $X$. The \emph{Bridgeland-Deligne perverse coherent t-structure} associated to $\p$ is defined by
\begin{align}
\DD^b\Coh(Y)^{p,\leq 0} & := \{ E \in \DD^b\Coh(Y) \ |\  f_*E \in \DD^b\Coh(X)^{\p,\leq 0}\text{ and } \iota^L E \in \mc C^{-1,\leq 0} \} &  ,\\
\DD^b\Coh(Y)^{p,\geq 0} &:=\{ E \in \DD^b\Coh(Y) \  |  \  f_*E \in \DD^b\Coh(X)^{\p,\geq 0}\text{ and } \iota^R E \in \mc C^{-1,\geq 0} \} & .
\end{align}
The category of Bridgeland-Deligne perverse coherent sheaves on $f:Y\to X$ with respect to the perversity function $\p$ is the abelian category defined by
\[ \Perv^\p(Y/X) = \DD^b\Coh(Y)^{\p,\leq 0} \cap \DD^b\Coh(Y)^{\p,\geq 0}  \ .\]
\end{defn}

\noindent Note that this indeed defines a t-structure, by Proposition \ref{Bdgrecprop} and Theorem \ref{BBD}, combined with Theorem \ref{PervCohthm} above. In particular, we make the following definition. As for the usual perverse coherent sheaves in the sense of Bridgeland, recalled in Definition \ref{pervcohdef}, we will often drop the dependence on $X$ from the notation and write simply \[ \Perv^\p(Y):=\Perv^\p(Y/X) \ .\]

For our primary application of interest, it will be necessary to require the following additional hypothesis:

\begin{defn}\label{BDadmissibledef} A Bridgeland-Deligne perverse coherent t-structure is called \emph{admissible} if the corresponding heart satisfies
\[ \Perv_\cs(Y) \subset  \Perv^\p(Y) \ ,\]
that is, it contains the category of compactly supported perverse coherent sheaves.
\end{defn}

\begin{eg}\label{mainteg} Let $X$ be an affine variety and define the perversity function $\p:X^\textup{top}\to \bb Z$ by $\p(x)=0$ for each closed point $x\in X(\bb K)$ and $\p(y)=-1$ otherwise. Then $\Perv^\p(X) = \Coh(X)_t$, the tilt of $\Coh(X)$ with respect to the torsion theory $(\mc T, \mc F)$ defined by $\mc T=\Coh_{\leq 0}(X)$ and $\mc F = \mc T^\perp$, as in Example \ref{dimsupptilteg}.
	
Moreover, for $f:Y\to X$ as in Section \ref{NCCRsec}, the Bridgeland-Deligne perverse coherent t-structure determined by $\p$ is admissible, and for any subvariety $Z\subset Y$ of dimension $\geq 1$ the object $M=\mc O_Z[1]\in \Perv^\p(Y)$.
\end{eg}

\begin{eg}\label{semidynamicalteg} Let $X$ be an affine variety and define the perversity function $\p:X^\textup{top}\to \bb Z$ by $\p(x)=0$ for each closed point $x\in X(\bb K)$ or one dimensional subvariety, and $\p(y)=-1$ otherwise. Then $\Perv^\p(X) = \Coh(X)_t$, the tilt of $\Coh(X)$ with respect to the torsion theory $(\mc T, \mc F)$ defined by $\mc T=\Coh_{\leq 1}(X)$ and $\mc F = \mc T^\perp$, as in Example \ref{dimsupptilteg}.  We remark that a similar t-structure appeared in the proof of the main theorem of \cite{Toda}.

Moreover, for $f:Y\to X$ as in Section \ref{NCCRsec}, the Bridgeland-Deligne perverse coherent t-structure determined by $\p$ is evidently admissible, and for any subvariety $Z\subset Y$ of pure dimension $\geq 2$ the object $M=\mc O_Z[1]\in \Perv^\p(Y)$, and similarly for any subvariety $Z$ of pure dimension $\leq 1$ the object $\mc O_Z\in \Perv^\p(Y)$.
\end{eg}

\subsection{Categories of perverse coherent extensions}\label{pervextcatsec}

Let $\Thick(F\oplus M) $ denote the thick subcategory of $\DD^b\Coh(\hat Y)$ generated by the compactly supported simple objects $F_i$ together with $M$, and similarly for the equivariant analogues. There are equivalent corresponding thick subcategories of $\DD_\perf(\hat \Lambda)$ and $\DD_\Fd(\Sigma)$, and their graded analogues, and we will describe them (as well as their hearts with respect to certain t-structures) in terms of explicit extensions of canonical resolutions determined by complexes of $\Sigma$, generalizing the descriptions given in Sections \ref{Geosetupsec},\ref{Kozpatsec} and \ref{monadsec} of the images of the objects in $\DD^b\Coh_\cs(Y)$ (and in $\Perv_\cs(Y)$) under these equivalences.

In order to analogously parameterize this larger class of extensions, and ultimately give descriptions of the moduli stacks of objects of these categories, it will be necessary to first understand the category in terms of Morita theory for the object $F\oplus M$, which leads naturally to the extended quiver, as we now explain.

To begin, in analogy with Corollary \ref{cctmoritacoro}, letting
\[ \Sigma_M =\Hom(F\oplus M, F\oplus M) \quad\quad \text{and}\quad\quad \Sigma_\infty=\Hom(M,M) \ \]
denote the (graded) DG associative algebras of endomorphisms of $F\oplus M$ and $M$, respectively, we have:

\begin{corollary}\label{extMoritacoro} There are equivalences of triangulated categories
	\[ \Thick(F\oplus M)  \xrightarrow{\cong} \DD_\Perf(\Sigma_M)\quad \quad \text{and}\quad \quad  \thick(F\oplus M) \xrightarrow{\cong} {\textup{D}_\perf}(\Sigma_M) \ , \]
	intertwining the forgetful functor $\DD(\Sigma_M) \to \DGV$ with $\Hom_{\DD\QC(Y)}( F\oplus M,\cdot)$, and similarly
		\[ \Thick(M) \xrightarrow{\cong} \DD_\Perf(\Sigma_\infty)\quad \quad \text{and}\quad \quad  \thick(M) \xrightarrow{\cong} {\textup{D}_\perf}(\Sigma_\infty) \ , \]
	intertwining the forgetful functor $\DD(\Sigma_\infty) \to \DGV$ with $\Hom_{\DD\QC(Y)}( M, \cdot)$.
\end{corollary}
\begin{proof} As in the analogous statement for $F$, we apply the DG Morita theory of Keller \cite{Kel1} recalled in Theorem \ref{Morita} to the thick subcategory generated by the objects $F\oplus M$ and $M$, respectively.
\end{proof}

Further, we let $I_M=I\cup \{\infty \}$ and define
\[ S_\infty = \Ext^0(M,M) \ \in \DD_\Fd(\Sigma_M) \quad\quad \textup{and}\quad\quad S_M= \bigoplus_{i\in I_M} S_i = S\oplus S_\infty  \ \in \DD_\Fd(\Sigma_M)  \ ,\]
as well as their natural graded enhancements. Note that the module structure on $S_\infty$ factors through the natural projection $\Sigma_M\to \Sigma_\infty$, or equivalently is pulled back from a module $S_\infty \in \DD_\Fd(\Sigma_\infty)$, for which we use the same notation by abuse, and similarly for the graded variants.

Further, as in Section \ref{Kozpatsec}, we can introduce corresponding Koszul dual graded algebras
\[ \Lambda_\infty = \Hom_{\Sigma_\infty}(S_\infty, S_\infty) \quad\quad\text{and}\quad\quad \Lambda_M = \Hom_{\Sigma_M}(S_M,S_M) \ ,\]
as well as their I-adically complete, ungraded variants $\hat \Lambda_\infty$ and $\hat\Lambda_M$. As in Equation \ref{}, we have induced presentations of $\hat \Lambda_\infty$ and $\hat \Lambda_M$ (and their graded variants) as quasi-free complete (or graded) DG associative algebras, with underlying cohomologically (and abstractly bi)graded associative algebras given by
\[  \ . \]

We again let $S_M\in \hat\Lambda_M \Mod $ and similarly $S_M \in \Lambda_M\Mod_\Z$ denote the corresponding augmentation modules, and we make the following definition:

\begin{defn}\label{frdefn} The derived category $\DD_\Fr(\hat \Lambda_M)$ of finite rank (over $S_M$) $\hat \Lambda_M$ modules is defined as the thick subcategory $\DD_\Fr(\hat \Lambda_M) = \Thick(S_M)$ generated by $S_M$, and similarly for the category of plain finite rank modules $\hat \Lambda_M\Mod_\Fr=\Filt(S_M)$. Their graded variants are defined similarly, as $\DD_\fr(\Lambda_M)=\thick(S_M)$ and $\Lambda_M\Mod_\fr=\filt(S_M)$.
\end{defn}

We again have induced Koszul duality equivalences, which are the analogues for $\Sigma_M$ and $\Sigma_\infty$ of the equivalences relating $\Sigma$ modules and $\Lambda$ modules given in Corollary \ref{Kozequivcoro}:

\begin{corollary}\label{Kozequivextcoro} There are mutually inverse equivalences of triangulated categories
	\begin{align} \Hom_{\Lambda_M}(S_M,\cdot) &: \xymatrix{\DD_\Fr(\hat \Lambda_M)  \ar@<.5ex>[r]^{\cong} & \ar@<.5ex>[l] \DD_\Perf(\Sigma_M)}: (\cdot)\otimes_{\Sigma_M} S_M  \quad\quad \ ,\\
		\Hom_{\Lambda_M}(S_M,\cdot) &: \xymatrix{\DD_\fr(\Lambda_M)  \ar@<.5ex>[r]^{\cong} & \ar@<.5ex>[l] {\textup{D}_\perf}(\Sigma_M)}: (\cdot)\otimes_{\Sigma_M} S_M  \quad\quad\ \  ,    \\
		\Hom_{\Lambda_\infty}(S_\infty,\cdot) &: \xymatrix{\DD_\Fr(\hat \Lambda_\infty)  \ar@<.5ex>[r]^{\cong} & \ar@<.5ex>[l] \DD_\Perf(\Sigma_\infty)}: (\cdot)\otimes_{\Sigma_\infty} S_\infty  \quad\quad \ \text{, and}\\
		\Hom_{\Lambda_\infty}(S_\infty,\cdot) &: \xymatrix{\DD_\fr(\Lambda_\infty)  \ar@<.5ex>[r]^{\cong} & \ar@<.5ex>[l] {\textup{D}_\perf}(\Sigma_\infty)}: (\cdot)\otimes_{\Sigma_\infty} S_\infty  \quad\quad\ \  .
	\end{align}	
\end{corollary}
\begin{proof} Follows from the proof of Corollary \ref{Kozequivcoro}, \emph{mutatis mutandis}.
\end{proof}

We now proceed to prove the generalization of Theorem \ref{catthmunext} to the thick subcategory $\Thick(F\oplus M)$, and its graded variant. To begin, we note
\begin{prop}
	There exist functors
	\begin{align*}
		(\cdot)\otimes_{S_M} S & : \DD_\Perf(\Sigma_M) \to \DD_\Fd(\Sigma) &(\cdot)\otimes_{S_M} S & : {\textup{D}_\perf}(\Sigma_M) \to \DD_\fd(\Sigma) \\
		(\cdot)\otimes_{\hat\Lambda_M} \hat \Lambda & : \DD_\Fr(\hat \Lambda_M) \to \DD_\Perf(\hat\Lambda) & (\cdot)\otimes_{\Lambda_M}  \Lambda & : \DD_\fr( \Lambda_M) \to {\textup{D}_\perf}(\Lambda)  & .
	\end{align*}
\end{prop}
\begin{proof}
 The codomain of the functors in the first line are as claimed, since $\Sigma_M\otimes_{S_M} S=\Hom(F,F\oplus M)$ is finite dimensional over the base field, as $F$ has compact support. The codomain of the functors on the second line are as claimed since we have $\DD_\Fr(\hat \Lambda_M) \subset \DD_\Perf(\hat\Lambda_M)$, and similarly for the graded variant, and the functors are clearly defined on the latter categories and preserve perfect complexes.
\end{proof}

In terms of these functors, the desired generalization of Theorem \ref{catthmunext} is the following:

\begin{theo}\label{catthm} The diagram of triangulated categories
\begin{equation}\label{Kozgeoungreqn}
		\xymatrixcolsep{6pc}
 \xymatrix{ \Thick(F,M)  \ar[d]^{\subset}  \ar@<.5ex>[r]^{ \Hom_{\hat Y}(F\oplus M, \cdot)\otimes_{\Sigma_M} S_M} & \ar@<.5ex>[l]^{\Hom_{\hat \Lambda_M}(S_M,\cdot)\otimes_{\Sigma_M} (F\oplus M)}  \DD_\Fr(\hat \Lambda_M) \ar@<.5ex>[r]^{\Hom_{\hat \Lambda_M}(S_M,\cdot)} \ar[d]^{(\cdot)\otimes_{\hat\Lambda_M} \hat \Lambda} & \ar@<.5ex>[l]^{(\cdot)\otimes_{\Sigma_M} S_M } \DD_\Perf(\Sigma_M)  \ar[d]^{(\cdot)^N\circ (\cdot)\otimes_{S_M} S } \\
	\DD^b\Coh(\hat Y) \ar@<.5ex>[r]^{\Hom_{\hat Y}(E,\cdot)}  &  \ar@<.5ex>[l]^{E\otimes_{\hat\Lambda}(\cdot) }  \DD_\Perf(\hat \Lambda)    \ar@<.5ex>[r]^{ (\cdot)^N\circ (\cdot)\otimes_{\Lambda} S}  & \ar@<.5ex>[l]^{\Hom_{\Sigma}(S,\cdot)}  \DD_\Fd(\Sigma) & }
\end{equation}
has horizontal arrows mutually inverse triangle equivalences and vertical arrows inclusions of thick subcategories, and admits canonical commutativity data, and similarly for the diagram
\begin{equation}\label{Kozgeoeqn}
		\xymatrixcolsep{6pc}
\xymatrix{ \thick(F,M)  \ar[d]^{\subset}  \ar@<.5ex>[r]^{ \Hom_{ Y}(F\oplus M, \cdot)\otimes_{\Sigma_M} S_M} & \ar@<.5ex>[l]^{\Hom_{  \Lambda_M}(S_M,\cdot)\otimes_{\Sigma_M} (F\oplus M)}  \DD_\fr(  \Lambda_M) \ar@<.5ex>[r]^{\Hom_{  \Lambda_M}(S_M,\cdot)} \ar[d]^{(\cdot)\otimes_{ \Lambda_M}   \Lambda} & \ar@<.5ex>[l]^{(\cdot)\otimes_{\Sigma_M} S_M } {\textup{D}_\perf}(\Sigma_M)  \ar[d]^{(\cdot)^N\circ(\cdot)\otimes_{S_M} S } \\
	\DD^b\Coh(  Y)^{T} \ar@<.5ex>[r]^{\Hom_{  Y}(E,\cdot)}  &  \ar@<.5ex>[l]^{E\otimes_{ \Lambda}(\cdot) }  {\textup{D}_\perf}(  \Lambda)    \ar@<.5ex>[r]^{ (\cdot)\otimes_{\Lambda} S}  & \ar@<.5ex>[l]^{\Hom_{\Sigma}(S,\cdot)}  \DD_\fd(\Sigma) & \textup{.} }
\end{equation}

\end{theo}

To complete the proof, we will need the following lemma:

\begin{lemma} The diagrams of triangulated categories
\[ \xymatrix{ \DD_\Perf(\Sigma_M) \ar[r]^{(\cdot)^{N_{\Sigma_M}\circ \subset}} \ar[d]^{(\cdot)\otimes_{S_M} S} & \DD_{\Fr}(\Sigma_M) \ar[d]^{(\cdot)\otimes_{S_M}S } \\
	\DD_\Fd(\Sigma) \ar[r]^{(\cdot)^{N_\Sigma}} & \DD_\Fd(\Sigma) }
\quad\quad \text{and}\quad\quad 
\xymatrix{ {\textup{D}_\perf}(\Sigma_M) \ar[r]^{(\cdot)^{N_{\Sigma_M}\circ \subset}} \ar[d]^{(\cdot)\otimes_{S_M} S} & \DD_{\fr}(\Sigma_M) \ar[d]^{(\cdot)\otimes_{S_M}S } \\
	\DD_\fd(\Sigma) \ar[r]^{(\cdot)^{N_\Sigma}} & \DD_\fd(\Sigma) }
\]
admit canonical commutativity data., where $(\cdot)^{N_{\Sigma_M}}:\DD_\Fr(\Sigma_M)\to \DD_\Fr(\Sigma_M)$ denotes the Nakayama functor for $\Sigma_M$ and similarly for $N_\Sigma$ and their graded variants.
\end{lemma}
\begin{proof} First, note that we have the natural isomorphism
\begin{align*}
	S\otimes_{S_M} \Sigma_M^{N_{\Sigma_M}} \otimes_{S_M} S & = S\otimes_{S_M} \Hom_{S_M} ( \Sigma_M , S_M) \otimes_{S_M} S   \\
	& \cong  \Hom_{S_M} ( \Sigma_M , S) \otimes_{S_M} S \\
	& \cong  \Hom_{S_M} ( \Hom(F,F\oplus M) , S) \otimes_{S_M} S \\
	& \cong  \Hom_{S_M} ( \Hom(F,F) , S)  \\
	& \cong \Sigma^{N_\Sigma} & .
\end{align*}
Now, it suffices to check the result on the generator $\Sigma_M$ of $\DD_\Perf(\Sigma_M)$; we have 
\begin{align*}  \Sigma_M \otimes_{S_M} S \otimes_{\Sigma} \Sigma^{N_\Sigma} 	
	& = \Sigma_M \otimes_{S_M} S \otimes_{\Sigma} \left(S\otimes_{S_M} \Sigma_M^{N_{\Sigma_M}} \otimes_{S_M} S \right) \\
	& \cong S\otimes_{S_M}\Sigma_M \otimes_{S_M} S \otimes_{\Sigma} \left( \Sigma_M^{N_{\Sigma_M}} \otimes_{S_M} S \right) \\
	& \cong \Sigma_M^{N_{\Sigma_M}} \otimes_{S_M} S & ,
\end{align*}
where the first equality follows from the preceding natural isomorphism.
\end{proof}

\begin{proof}{(of Theorem \ref{catthm})}
	Note that we have already established the existence of the mutually inverse triangle equivalences in the bottom rows of the diagram in Theorem \ref{catthmunext}, and those in the top rows in Corollaries \ref{extMoritacoro} and \ref{Kozequivextcoro}. Thus, it remains to check commutativity, which we now prove for the rightmost and outer squares, as this is sufficient:
	
	For the rightmost squares, applying the preceding lemma, we have the natural isomorphism
	\begin{align*} \Hom_\Sigma(S, (\Hom_{\hat\Lambda_M}(S_M, \cdot ) \otimes_{S_M} S )^{N_\Sigma} ) & \cong  \Hom_\Sigma(S, (\Hom_{\hat\Lambda_M}(S_M, \cdot )^{N_{\Sigma_M}} \otimes_{S_M} S ) ) \\
		& \cong  \Hom_\Sigma(S, (\cdot)\otimes_{\hat \Lambda_M} S_M \otimes_{S_M} S ) \\
		& \cong  \Hom_\Sigma(S, (\cdot)\otimes_{\hat \Lambda_M} \hat \Lambda \otimes_{\hat\Lambda} S ) \\
		& \cong (\cdot)\otimes_{\hat \Lambda_M} \hat \Lambda & ,
	\end{align*}
as desired, and similarly for the graded case. Commutativity of the outer squares follows from the proof of Theorem \ref{catthmunext}, noting we have the natural isomorphism
\[ \Hom_{\hat Y} (F\oplus M , \cdot) \otimes_{S_M} S \cong \Hom_{\hat Y}(F,\cdot)  \ ,  \]
and similarly for the graded case.
\end{proof}

We now describe the induced equivalences of the natural hearts of the categories in the top lines of the diagrams in Theorem \ref{catthm}, generalizing the descriptions of Theorems \ref{unextheartthm} and \ref{unextAinfheartthm} to the extended case. Recall that by hypothesis
\[M\in \Perv^\p(Y)=\Perv^\p(Y/X)\]
is an object in the heart of $\DD^b\Coh(Y)$ with respect to an admissible Bridgeland-Deligne perverse coherent t-structure, in the sense of Definition \ref{BDpervcohdef}. Further, recall that the compactly supported perverse coherent sheaves $F\in \Perv_\cs(Y)\subset \Perv^\p(Y)$ are objects of this heart by the admissible assumption. Thus, we can define the strictly full subcategories
\begin{align}
\Perv^\p_M( Y)  := \Filt(F\oplus M)  &\subset \Perv^\p( Y)   & \text{and}\\
\Perv^\p_M(Y)^{T} : = \filt(F\oplus M) & \subset \Perv^\p(Y)^{\C^\times} & ,
\end{align}
on objects admitting a filtration with subquotients given by the direct summands of the object $F\oplus M$ (and its graded shifts, in the graded case), as in Definitions \ref{Filtdefn} and \ref{filtdefn}, respectively.

Further, let $\mc D = \DD^b\Coh(Y)^{T}$, $F\oplus M=\oplus_{i\in I} F_i \oplus M$ and consider the graded variant of the $\Ainf$ category $\mc A_{F\oplus M}$ defined in Example \ref{AinfcatSeg}, with objects $i\in I_M=I\cup \{\infty\}$ and $\Hom$ spaces given by
\[ \mc A_{F\oplus M}(i,j) =\ _i(\Sigma_M)_j \quad\quad \text{for }\quad\quad i,j \in I_M .  \]
The analogue of Theorem \ref{unextAinfheartthm} is given by the following:

\begin{theo}\label{heartthm} Restriction to the hearts of the triangulated categories in the diagram of Equation \ref{Kozgeoeqn} induces mutually inverse equivalences of mixed categories
	\[	\xymatrixcolsep{3pc}
	\xymatrixrowsep{.5pc}
	\xymatrix{ \Perv_M^\p(Y)^{T}  \ar@<.5ex>[r]^{} & \ar@<.5ex>[l]  \Lambda_M\Mod_\fr  \ar@<.5ex>[r]^{} & \ar@<.5ex>[l] H^0(\tw^0 \mc A_{F\oplus M})} \ .  \]
	
	Similarly, restriction to hearts in the diagram of Equation \ref{Kozgeoungreqn} induces
	\[	\xymatrixcolsep{3pc}
	\xymatrixrowsep{.5pc}
	\xymatrix{ \Perv_M^\p(\hat Y) \ar@<.5ex>[r]^{} & \ar@<.5ex>[l]  \hat \Lambda_M\Mod_\Fr \ar@<.5ex>[r]^{} & \ar@<.5ex>[l] H^0(\Tw^0 \mc A_{F\oplus M})} \ .  \]
\end{theo}
\begin{proof}
The equivalences on the right follow from the identifications
\[ \Lambda_M\Mod_\fr = \filt(S_M) \xrightarrow{\cong} \filt(\Sigma_M) \cong H^0(\tw^0 \mc A_{F\oplus M})  \ , \]
and similarly in the ungraded case
\[ \hat\Lambda_M\Mod_\Fr = \Filt(S_M) \xrightarrow{\cong} \Filt(\Sigma_M) \cong H^0(\Tw^0 \mc A_{F\oplus M})  \ , \]
where the latter isomorphism in each line follow from Corollaries \ref{twobjgrfiltcoro} and \ref{twobjfiltcoro}, respectively, and the middle equivalence in each line follows from the fact that triangulated functors preserve categories of iterated extensions, and we have seen that the given functor maps $S_M$ to $\Sigma_M$ (and, in the graded case, preserves grading shifts in the unsheared conventions).

Similarly, the triangulated equivalence given by the composition of the two horizontal equivalences in the top lines of Equations \ref{Kozgeoeqn} and \ref{Kozgeoungreqn} are given by the functor $\Hom(F\oplus M, \cdot)$, which evidently induces equivalences
\[  \Perv^\p_M(Y)^{T} : = \filt(F\oplus M)  \xrightarrow{\cong} \filt(\Sigma_M)\cong H^0(\tw^0 \mc A_{F\oplus M}) \ , \]
and similarly in the ungraded case 
\[  	\Perv^\p_M(\hat Y) = \Filt(F\oplus M)\xrightarrow{\cong} \Filt(\Sigma_M) \cong H^0(\Tw^0 \mc A_{F\oplus M})  \ , \]
which again follows from the fact that triangulated functors preserve categories of iterated extensions, and $\Hom(F\oplus M,\cdot)$ evidently maps $F\oplus M$ to $\Sigma_M$.
\end{proof}

We now recall several results about the compatibility between the equivalences of the preceding Theorem \ref{catthm} and those of Corollary \ref{cctmoritacoro} and Theorem \ref{catthmunext} describing $\Thick(F)$, as well as the analogous descriptions of $\Thick(M)$, and their graded variants.

To begin, note there are natural (graded) bimodules
\begin{align}
	S\otimes_{S_M} \Sigma_M &=\Hom(F\oplus M, F) \ \in\  (\Sigma,\Sigma_M)\Bimod   & \text{and}\\
	 S_\infty \otimes_{S_M} \Sigma_M&=\Hom(F\oplus M, M) \ \in (\Sigma_\infty,\Sigma_M)\Bimod    & ,  
\end{align}
inducing functors
\begin{align*}
 ( \cdot) \otimes_{\Sigma} (S\otimes_{S_M} \Sigma_M)  & : \DD_\Perf(\Sigma) \to \DD_\Perf(\Sigma_M)  & \text{and}\\
  ( \cdot) \otimes_{\Sigma_\infty}( S_\infty\otimes_{S_M} \Sigma_M)  & : \DD_\Perf(\Sigma_\infty) \to \DD_\Perf(\Sigma_M) & ,
\end{align*}
and similarly for the graded variants. Further, note there are natural maps of algebras
\[ \varpi:\Lambda_M \to \Lambda \quad\quad \text{and}\quad\quad \varpi_\infty:\Lambda_M\to \Lambda_\infty \]
and induced restriction functors on their module categories
\[ \varpi^*:\DD_\Fd(\hat \Lambda) \to \DD_\Fr(\hat\Lambda_M) \quad\quad \text{and}\quad\quad \varpi_\infty^*:\DD_\Fr(\hat \Lambda_\infty ) \to \DD_\Fr(\hat\Lambda_M) \ , \]
and similarly for the graded variants. The main compatibility results are the following:

\begin{prop} The diagram of triangulated categories
	
\begin{equation}\label{KozgeoFcompeqn}
		\xymatrixcolsep{6pc}
	\xymatrix{
 \Thick(F) \ar@<.5ex>[r]^{\Hom_{\hat Y}(F,\cdot)\otimes_{\Sigma} S}\ar[d]^{\subset} & \ar@<.5ex>[l]^{\Hom_{\hat\Lambda}(S,\cdot)\otimes_{\Sigma} F }  \DD_\Fd(\hat \Lambda)  \ar@<.5ex>[r]^{ \Hom_\Lambda(S,\cdot) } \ar[d]^{\varpi^*}  & \ar@<.5ex>[l]^{(\cdot)\otimes_\Sigma S } \DD_\Perf(\Sigma)  \ar[d]^{(\cdot)\otimes_{\Sigma}(S\otimes_{S_M} \Sigma_M) }	\\
		\Thick(F,M) \ar@<.5ex>[r]^{ \Hom_{\hat Y}(F\oplus M, \cdot)\otimes_{\Sigma_M} S_M} & \ar@<.5ex>[l]^{\Hom_{\hat \Lambda_M}(S_M,\cdot)\otimes_{\Sigma_M} (F\oplus M)}  \DD_\Fr(\hat \Lambda_M) \ar@<.5ex>[r]^{\Hom_{\hat \Lambda_M}(S_M,\cdot)} 
		   & \ar@<.5ex>[l]^{(\cdot)\otimes_{\Sigma_M} S_M } \DD_\Perf(\Sigma_M) }
\end{equation}

\noindent has horizontal arrows mutually inverse triangle equivalences and vertical arrows inclusions of thick subcategories, and admits canonical commutativity data, and similarly for the diagram
\begin{equation}\label{Kozgeoinfcompeqn}
		\xymatrixcolsep{6pc}
\xymatrix{
	\Thick(M) \ar@<.5ex>[r]^{\Hom_{\hat Y}(M,\cdot)\otimes_{\Sigma_\infty} S_\infty}\ar[d]^{\subset} & \ar@<.5ex>[l]^{\Hom_{\hat\Lambda_\infty}(S_\infty,\cdot)\otimes_{\Sigma_\infty} M }  \DD_\Fr(\hat \Lambda_\infty)  \ar@<.5ex>[r]^{ \Hom_{\Lambda_\infty}(S_\infty,\cdot) } \ar[d]^{\varpi_\infty^*}  & \ar@<.5ex>[l]^{(\cdot)\otimes_{\Sigma_\infty} S_\infty } \DD_\Perf(\Sigma_\infty)  \ar[d]^{(\cdot)\otimes_{\Sigma_\infty}(S_\infty\otimes_{S_M} \Sigma_M) }	\\
	\Thick(F,M) \ar@<.5ex>[r]^{ \Hom_{\hat Y}(F\oplus M, \cdot)\otimes_{\Sigma_M} S_M} & \ar@<.5ex>[l]^{\Hom_{\hat \Lambda_M}(S_M,\cdot)\otimes_{\Sigma_M} (F\oplus M)}  \DD_\Fr(\hat \Lambda_M) \ar@<.5ex>[r]^{\Hom_{\hat \Lambda_M}(S_M,\cdot)} 
	& \ar@<.5ex>[l]^{(\cdot)\otimes_{\Sigma_M} S_M } \DD_\Perf(\Sigma_M) }
\end{equation}
as well as their graded variants
\begin{equation}\label{KozgeoFcompgreqn}
		\xymatrixcolsep{6pc}
\xymatrix{
	\thick(F) \ar@<.5ex>[r]^{\Hom_{  Y}(F,\cdot)\otimes_{\Sigma} S}\ar[d]^{\subset} & \ar@<.5ex>[l]^{\Hom_{ \Lambda}(S,\cdot)\otimes_{\Sigma} F }  \DD_\fd(  \Lambda)  \ar@<.5ex>[r]^{ \Hom_\Lambda(S,\cdot) } \ar[d]^{\varpi^*}  & \ar@<.5ex>[l]^{(\cdot)\otimes_\Sigma S } {\textup{D}_\perf}(\Sigma)  \ar[d]^{(\cdot)\otimes_{\Sigma}(S\otimes_{S_M} \Sigma_M) }	\\
	\thick(F,M) \ar@<.5ex>[r]^{ \Hom_{  Y}(F\oplus M, \cdot)\otimes_{\Sigma_M} S_M} & \ar@<.5ex>[l]^{\Hom_{  \Lambda_M}(S_M,\cdot)\otimes_{\Sigma_M} (F\oplus M)}  \DD_\fr(  \Lambda_M) \ar@<.5ex>[r]^{\Hom_{  \Lambda_M}(S_M,\cdot)} 
	& \ar@<.5ex>[l]^{(\cdot)\otimes_{\Sigma_M} S_M } {\textup{D}_\perf}(\Sigma_M) }
\end{equation}
and
\begin{equation}\label{Kozgeoinfcompgreqn}
		\xymatrixcolsep{6pc}
\xymatrix{
	\thick(M) \ar@<.5ex>[r]^{\Hom_{  Y}(M,\cdot)\otimes_{\Sigma_\infty} S_\infty}\ar[d]^{\subset} & \ar@<.5ex>[l]^{\Hom_{ \Lambda_\infty}(S_\infty,\cdot)\otimes_{\Sigma_\infty} M }  \DD_\fr(  \Lambda_\infty)  \ar@<.5ex>[r]^{ \Hom_{\Lambda_\infty}(S_\infty,\cdot) } \ar[d]^{\varpi_\infty^*}  & \ar@<.5ex>[l]^{(\cdot)\otimes_{\Sigma_\infty} S_\infty } {\textup{D}_\perf}(\Sigma_\infty)  \ar[d]^{(\cdot)\otimes_{\Sigma_\infty}(S_\infty\otimes_{S_M} \Sigma_M) }	\\
	\thick(F,M) \ar@<.5ex>[r]^{ \Hom_{  Y}(F\oplus M, \cdot)\otimes_{\Sigma_M} S_M} & \ar@<.5ex>[l]^{\Hom_{  \Lambda_M}(S_M,\cdot)\otimes_{\Sigma_M} (F\oplus M)}  \DD_\fr(  \Lambda_M) \ar@<.5ex>[r]^{\Hom_{  \Lambda_M}(S_M,\cdot)} 
	& \ar@<.5ex>[l]^{(\cdot)\otimes_{\Sigma_M} S_M } {\textup{D}_\perf}(\Sigma_M) } \ .
\end{equation}
\end{prop}
\begin{proof}
\end{proof}

Moreover, we have the following induced equivalences of the natural hearts of these categories:

\begin{corollary} \label{heartinfthm} Restriction to the hearts of the triangulated categories in the diagram of Equation \ref{Kozgeoinfcompgreqn} induces mutually inverse equivalences of mixed categories
	\[	\xymatrixcolsep{3pc}
	\xymatrixrowsep{.5pc}
	\xymatrix{ \filt(M) \ar@<.5ex>[r]^{} & \ar@<.5ex>[l]  \Lambda_\infty\Mod_\fr  \ar@<.5ex>[r]^{} & \ar@<.5ex>[l] H^0(\tw^0 \mc A_{M})} \ .  \]
	
	Similarly, restriction to hearts in the diagram of Equation \ref{Kozgeoinfcompeqn} induces
	\[	\xymatrixcolsep{3pc}
	\xymatrixrowsep{.5pc}
	\xymatrix{ \Filt(M) \ar@<.5ex>[r]^{} & \ar@<.5ex>[l]  \hat \Lambda_\infty\Mod_\Fr \ar@<.5ex>[r]^{} & \ar@<.5ex>[l] H^0(\Tw^0 \mc A_{M})} \ .  \]
\end{corollary}

\subsection{Monad presentations of perverse coherent extensions}\label{monadextsec} In this section, we generalize the construction of Section \ref{monadsec}, and its generalization in Section \ref{Ainfkozsec}, describing certain canonical resolutions of compactly supported perverse coherent sheaves in terms of representations of quivers, to describe objects of the category $\Filt(F\oplus M)$ in terms of representations of extended quivers.

The compositions of the equivalences of Equations \ref{Kozgeoungreqn} and \ref{Kozgeoeqn} define triangle equivalences
\[	\hat K_M(\cdot): \DD_\Perf(\Sigma_M) \xrightarrow{\cong} \Thick(F\oplus M) \quad\quad \text{and}\quad\quad    K_M(\cdot): \DD_\perf(\Sigma_M) \xrightarrow{\cong}  \thick(F\oplus M) \ , \]
analogous to the functors of Equations \ref{concreteungrKozeqn} and \ref{concreteKozeqn}. Our goal is to give an explicit presentation of the subcategory $\Perv_M(\hat Y)$ as the image of the corresponding subcategory $H^0(\Tw^0 \mc A_{F\oplus M})$, and thus in terms of representations of the extended quiver $Q_M$.

One of the primary implications of commutativity in the statement of Theorem \ref{catthm} is the following:

\begin{corollary}
	There are canonical natural isomorphisms
	\begin{align} \hat K_M \cong \hat K \circ (\cdot)^N \circ (\cdot) \otimes_{S_M} S  & : \DD_\Perf(\Sigma_M) \to \DD^b\Coh(\hat Y)  & \textup{and}\\
		 K_M \cong  K \circ (\cdot)^N \circ (\cdot) \otimes_{S_M} S  & : \DD_\perf(\Sigma_M) \to \DD^b\Coh( Y)^{T}  & ,
	\end{align}
where $\hat K$ and $K$ are the functors of Equations \ref{concreteungrKozeqn} and \ref{concreteKozeqn}.
\end{corollary}
\begin{proof} 
These natural isomorphisms are simply the commutativity data for the outer squares of Equations \ref{Kozgeoungreqn} and \ref{Kozgeoeqn}.
\end{proof}

Further, the preceding corollary induces the following equivalent presentations of these categories:
\begin{corollary} The image of $\DD_\perf(\Sigma_M)$ under $(\cdot) \otimes_{S_M} S: \DD_\fr(\Sigma_M) \to \DD_\fd(\Sigma)$ is given by
	\[  \DD_\perf(\Sigma_M) \xrightarrow{\cong} \thick(\Sigma\oplus \Hom(F,M)) \subset  \DD_\fd(\Sigma)  \ , \]
	inducing an equivalence of mixed categories between the corresponding subcategories
	\[ H^0(\tw^0(\mc A_{F\oplus M})) \xrightarrow{\cong} \filt_\Sigma(\Sigma\oplus \Hom(F,M)) \ . \]

	Similarly, image of $\DD_\Perf(\Sigma_M)$ under $(\cdot) \otimes_{S_M} S: \DD_\Fr(\Sigma_M) \to \DD_\Fd(\Sigma)$ is given by
	\[  \DD_\Perf(\Sigma_M) \xrightarrow{\cong} \Thick(\Sigma\oplus \Hom(F,M)) \subset  \DD_\Fd(\Sigma)  \ , \]
	inducing an equivalence
	\[ H^0(\Tw^0(\mc A_{F\oplus M})) \xrightarrow{\cong} \Filt_\Sigma(\Sigma\oplus \Hom(F,M))   \ .\]
\end{corollary}
\begin{proof} This follows immediately, noting that the image of $\Sigma_M$ is given by
	\[ \Sigma_M \otimes_{S_M} S = \Hom( F\oplus M, F\oplus M) \otimes_{S_M} S \cong \Hom(F,F\oplus M) = \Sigma \oplus \Hom(F,M)  \ .\]
\end{proof}

Thus, it remains to understand the image of (the Nakayama dual of) $\filt_\Sigma(\Sigma\oplus \Hom(F,M))$ under the functor $K:\DD_\perf(\Sigma) \to \DD^b\Coh(Y)^{T}$ studied in Section \ref{monadsec}, generalizing the calculation of the image of $\filt_\Sigma(\Sigma)$ in \emph{loc. cit.} (as well as its I-adically complete, ungraded variant).

\begin{warning}
	For concreteness, we give most of the exposition in this section in the graded case, noting the ungraded variant can be recovered by completing and forgetting the grading. Further, for notation simplicity we will restrict the trigrading along a cocharacter $\C^\times \to T$, in keeping with Warning \ref{gradingwarn}.
\end{warning}

To begin, we let
\[\Sigma_\infty= \Hom(F,M), \ I_\infty = \Hom(F,M)^N \ \in \DD_\fd(\Sigma) \]
and similarly for $\Sigma_\infty,I_\infty\in \DD_\Fd(\Sigma)$, and suppose $I_\infty$ has Jordan-Holder series
\begin{equation}\label{injcompeinftyqn}
	I_\infty = [ I_\infty^0 < I_\infty^{-1}\langle 1 \rangle < \cdots < I_\infty^{-m_\infty}\langle m_\infty \rangle]   \ ,
\end{equation}
in the cohomologically sheared notation, in analogy with Equation \ref{injcompeqn}. Then the image of $I_\infty$ under the explicit presentation of the Koszul duality functor established in Section \ref{monadsec} is given by
\[ K(I_\infty) = \left[  I_i^{-m_\infty} \otimes_{S} E [m_\infty] \to \cdots \to I_\infty^{-1} \otimes_S E[1] \to I_\infty^0\otimes_S E \right] \xrightarrow{\cong} M \ , \]
as in Equation \ref{simplereseqn}. Further, we have canonical identifications
\begin{align}
\ _\infty(\Sigma_M)_i^1  &\cong \Hom_\Sigma(I_\infty, I_i\langle 1 \rangle ) \cong \Hom^0_{\DD^b\Coh(Y)}( K(I_\infty), K(I_i)[1])  \\
	\ _i(\Sigma_M)_\infty^1  & \cong \Hom_\Sigma(I_i, I_\infty\langle 1 \rangle ) \cong \Hom^0_{\DD^b\Coh(Y)}( K(I_i), K(I_\infty)[1])  \\
\ _\infty(\Sigma_M)_\infty^1 & \cong \Hom_\Sigma(I_\infty, I_\infty\langle 1 \rangle )  \cong \Hom^0_{\DD^b\Coh(Y)}( K(I_\infty), K(I_\infty)[1]) 
\end{align}
as in Equation \ref{extideqn}, though we note that $I_\infty \in \DD_\fd(\Sigma)$ is not necessarily injective so the $\Hom$ functors over $\Sigma$ appearing in the preceding equations are not necessarily exact, in contrast to those in \emph{loc. cit.}.

Given an object $(i_1,...,i_{d},\delta)\in \tw^0\mc A_{F\oplus M}$, where we recall that each $i_k \in I_M = I \cup \{ \infty\}$, the corresponding object of $\DD_\fd(\Sigma)$ has underlying bigraded vector space given by
\[ \Sigma_{i_1,...,i_d} = \bigoplus_{k=1,...,d} \Sigma_{i_k}\langle p_k\rangle \cong \bigoplus_{i\in I_M} \Sigma_i \otimes V_i  \cong  \left(\bigoplus_{i\in I} \Sigma_i \otimes V_i\right)\oplus  ( \Sigma_\infty \otimes V_\infty) \]
where each $V_i$ is a graded vector space
\[ V_i=\bigoplus_{p\in \bb Z} V_{i,-p}\langle p \rangle \quad\quad \text{with}\quad\quad \dim V_{i,p} = \# \{ k\in \{1,...,d\} \ | \ i_k=i\textup{ and } p_k=p \} \ . \]
\noindent Note in particular that $\dim \oplus_{i\in I_M} V_i = d$.

Now, the degree zero element $\delta\in \mf{gl}_d(\bb Z\mc A_{F\oplus M})[1]$ is in fact an element of the subspace
\[ \delta \in \Hom_{\DD(\Sigma_M)}(\bigoplus_{i\in I_M}  (\Sigma_M)_i\otimes V_i, \bigoplus_{j\in  I_M} (\Sigma_M)_j \otimes V_j [1] ) \subset  \mf{gl}_d(\Sigma_M)[1] \cong \mf{gl}_d(\mc A_{F\oplus M})[1]  \]
and thus admits a decomposition
\begin{equation}\label{deltadecompexteqn}
	\delta =  \sum_{i,j\in I_M} b_{ij}\otimes B_{ij}  \quad \in \quad \bigoplus_{i,j\in I_M}  \ _i(\Sigma_M)_j\otimes \Hom(V_i,V_j)[1] \ ,
\end{equation}
as in Equation \ref{deltadecompeqn}, where we again use the abuse of notation $ B_{ij} \otimes b_{ij}$ to denote a not necessarily pure tensor, as in Warning \ref{puretenswarning}.

The underlying cochain complex differential is given by
\[d_\delta = \sum_{k\in \Z} \rho_k^{\Sigma\oplus \Sigma_\infty}(\cdot, \delta^{\otimes k-1}) \quad \in\quad \Hom_{\DD(\Sigma)}( \Sigma_{i_1,...,i_d}, \Sigma_{i_1,...,i_d}[1])  \]
where $\rho_k^{\Sigma\oplus \Sigma_\infty}: \Sigma_{i_1,...,i_d}\otimes\Sigma_M^{\otimes k-1} \to \Sigma_{i_1,...,i_d}[1-k]$ denote the image under $(\cdot)\otimes_{S_M} S$ of the $\Ainf$ module structure maps for $(\Sigma_M)_{i_1,...,i_s}$ over $\mc E_{\mc A_{F\oplus M}}$. Decomposing according to Equation \ref{deltadecompexteqn}, we have the analogous decomposition
\begin{equation}\label{ddeltacompsexteqn} 
	d_\delta = \sum_{k\in \bb Z,\  i,i_2,...,i_{k-1},j \in I_M}\rho_k^{\Sigma\oplus \Sigma_\infty}(\cdot, b_{i i_2},...,b_{i_{k-1}  j})\otimes (B_{i i_2}  ...  B_{i_{k-1}  j} )   \quad \in \quad \bigoplus_{i,j\in I_M}  \Hom(\Sigma_i, \Sigma_j[1]) \otimes \Hom(V_i,V_j)\ 
\end{equation} 
of the differential on $\Sigma_{i_1,...,i_d}^\delta$.

Thus, the image of the Nakayama dual of $\Sigma_{i_1,...,i_d}^{\delta}\in {\textup{D}_\perf}(\Sigma)$ under the Koszul duality equivalence of Equation \ref{concreteKozeqn} is given by
\begin{equation}\label{}
	K(\Sigma_{i_1,...,i_d}^{\delta,N})= \left(K(\Sigma_{i_1,...,i_d}^N) , K(d_\delta^N) \right) \quad\quad \text{where} \quad\quad K(\Sigma_{i_1,...,i_d}^N)=\bigoplus_{i\in I_M} K(I_i)\otimes V_i \ ,
\end{equation}
and where the differential $K(d_\delta^N):K(\Sigma_{i_1,...,i_d}^N) \to K(\Sigma_{i_1,...,i_d}^N)[1]$ is given by
\begin{equation}\hspace*{-1cm}\label{} 
	K(d_\delta^N) :=\sum K (\rho_k^{\Sigma\oplus \Sigma_\infty}(\cdot, b_{i,i_2},...,b_{i_{k-1}, j})^N)\otimes  ( B_{i,i_2}  ...  B_{i_{k-1}, j} )   \quad \in \quad \bigoplus_{i,j\in I_M} \Hom(K(I_i), K(I_j)[1] )\otimes \Hom(V_i,V_j) \ ,
\end{equation}
where the sum is over the same index set as in Equation \ref{ddeltacompsexteqn} above.

We now give a concrete description of the equivalence of this data with that of a representation of the extended quiver $Q_M$, which is by definition given by a (finite rank, in the sense of Definition \ref{frdefn}) module over the path algebra $\Lambda_M=( \otimes^\bullet_{S_M} \bar\Sigma_M[1])^\vee$, as ensured by the equivalence of Theorem \ref{heartthm} between $\Perv_M^\p(Y)^{T}$ and $\Lambda_M\Mod_\fr$.

The K theory class of an object $H\in \Perv_M^\p(Y)^{T}$ is determined by the multiplicites $d_i\in \bb N$ for $i\in I$ of the simple factors $F_i$ together with the multiplicity $d_\infty\in \bb N$ of the auxiliary object $M$. These multiplicities correspond to those of the finite dimensional simple $\Lambda_M$ modules $S_i$ for $i\in I$ together with that of the remaining finite rank generator $S_\infty$. The extended quiver representation corresponding to $H$ is then determined by a $\dd=(d_i)_{i\in I_M}$ dimensional $S_M$ module
\[ V= \bigoplus_{i\in V_{Q_M}} V_i = V_\infty \oplus \bigoplus_{i \in V_Q} V_i   \]
 together with a map
\begin{equation}\label{quiverrepexteqn}
	\K\langle E_{Q_M}\rangle^0 = (\Sigma_M^1)^\vee \to \Endi_{S_M\Bimod}(V) 
\end{equation}
of $S_M$ bimodules, such that the induced map from the cohomological degree zero component of the quasi-free resolution of the path algebra $\Lambda_M^0=\otimes_{S_M}^\bullet \K\langle E_{Q_M}\rangle^0$ to $\Endi_{S\Bimod}(V)$ maps the ideal of relations in the quiver $Q_M$ to zero.

The map of Equation \ref{quiverrepexteqn} can equivalently be interpreted as an element
\begin{equation*}
	\delta= \sum_{i,j \in I_M} b_{ij}\otimes B_{ij} \quad \in \quad \Sigma_M^1\otimes \Endi(V)  =  \bigoplus_{i,j\in I_M} \ _i(\Sigma_M)_j \otimes \Hom(V_i,V_j)  \ ,
\end{equation*}
which determines a cohomological degree $+1$ endomorphism $d_B:\tilde H\to \tilde H[1]$ of the perverse coherent complex
\begin{equation}\label{monadextvseqn}
	\tilde H:=  \bigoplus_{i\in I_M}K(I_i)\otimes_{S_i}  V_i\   \in  \Perv_M^\p(Y)^{\bb C^\times}  \quad\quad \text{, or}\quad\quad 	\tilde H:=  \bigoplus_{i\in I_M}\hat K(I_i)\otimes_{S_i}  V_i\   \in  \Perv_M^\p(\hat Y)
\end{equation}
in the ungraded case, by the formula
\begin{equation}\label{monadkozdifAinfexteqn} %
	d_B= \sum_{k\in \bb Z,\  i,i_2,...,i_{k-1},j \in I_M}  K (\rho_k^{\Sigma\oplus \Sigma_\infty}(\cdot, b_{i,i_2},...,b_{i_{k-1}, j})^N)\otimes  ( B_{i,i_2}  ...  B_{i_{k-1}, j} ) \ . 
\end{equation}

We can now state the desired generalization of Propositions \ref{monadunextprop} and \ref{monadunextAinfprop}:

\begin{prop}\label{monadprop} 
	Let $\tilde H \in \Perv_M^\p(Y)^{T}$ be as in Equation \ref{monadextvseqn} and fix
	\[	B:=(B_{ij})_{i,j\in I_M}\quad\quad\text{with}\quad\quad B_{ij}=(B_{ij}^\alpha \in \Hom(V_i,V_j))_{ \alpha \in \mc B_{ij}} \ . \]
	The following are equivalent: 
	\begin{itemize}
		\item the induced map $d_B:\tilde H \to \tilde H[1]$ of Equation \ref{monadkozdifAinfexteqn} satisfies $d_B^2=0$, and
		\item the induced map $\rho:(\Sigma^1_M)^\vee\to \Endi(V)$ of Equation \ref{quiverrepexteqn} defines a representation of the graded DG quiver $Q_M$.
	\end{itemize}
	Further, when these conditions hold, the resulting complex $(\tilde H,d_B)$ is a projective resolution of the object $H \in \Perv_M^\p(Y)^{T}$ corresponding to the extended quiver representation $V\in \Lambda_M\Mod_\fr$.
	
	\smallskip
	
	Similarly, for $\tilde H \in \Perv_M^\p(\hat Y)$ as in Equation \ref{monadextvseqn}, the following are equivalent:
	\begin{itemize}
		\item the induced map $d_B:\tilde H \to \tilde H[1]$ of Equation \ref{monadkozdifAinfexteqn} satisfies $d_B^2=0$, and
		\item the induced map $\rho:(\Sigma^1_M)^\vee\to \Endi_{S\Bimod}(V)$ of Equation \ref{quiverrepexteqn} defines a nilpotent representation of the DG quiver $Q_M$.
	\end{itemize}
	Further, when these conditions hold, the resulting complex $(\tilde H,d_B)$ is a projective resolution of the object $H \in \Perv_M(\hat Y)$ corresponding to the extended quiver representation $V\in \Lambda_M\Mod_\Fr$.
\end{prop}
\begin{proof}
The proof is the same as that of Proposition \ref{monadunextAinfprop}, \emph{mutatis mutandis}.
\end{proof}

\subsection{Moduli spaces of perverse coherent extensions and extended quivers}\label{quiversecfr}

Let $Q_\M$ be an extended quiver, with vertex set $V_{Q_\M}=V_Q\cup \{\infty\}$ and edge set $E_{Q_\M}=E_Q\cup E_\infty$, where $E_\infty$ denotes all edges of the extended quiver with $\infty$ as at least one of its source and target vertices. The free path algebra of the extended quiver is given by
\[ \bb K Q_\M = \otimes^\bullet_{S_\M} \bb K\langle E_{Q_\M}\rangle \quad\quad \text{where} \quad\quad S_\M=\oplus_{i\in V_Q} S_i \oplus S_\infty = \oplus_{I\in I} \bb K \oplus S_\infty \ . \]

For each dimension vector $\dd=(\dd_0,d_\infty)\in \bb N^{V_{Q_M}}=\bb N^{V_Q}\times \bb N$ for the extended quiver, define
\[ X_{\dd}(Q_\M) =  \bigoplus_{i,j \in I_\M}  \ _i \Sigma^1_j \otimes \Hom(\bb K^{d_i}, \bb K^{d_j}) \quad\quad\text{and}\quad\quad G_{\dd}(Q_\M) = \prod_{i\in V_{Q_\M}} \Gl_{d_i}(S_i)  \ . \]
The stack of representations $\mf M(Q_M)$ of the free $S_M$-algebra on the extended quiver $Q_\M$ is defined as the disjoint union of the quotient stacks
\[ \mf M(Q_\M) = \bigsqcup_{\dd\in \bb N^{V_{Q_\M}}} \mf{M}_{\dd}(Q_\M) \quad\quad\text{where}\quad\quad  \mf{M}_{\dd}(Q_\M) = \left[ X_\dd(Q_\M) / G_\dd(Q_\M) \right] \ . \]

\begin{rmk}\label{trivframermk}
		Note that for $\dd=(\dd_0,0)\in  \bb N^{V_{Q_M}}=\bb N^{V_Q} \times \bb N$, we have tautological identifications
	\[\bigoplus_{i,j \in I}  \ _i \Sigma^1_j \otimes \Hom(\bb K^{d_i}, \bb K^{d_j}) = \bigoplus_{e\in E_Q} \Hom( \K^{d_{s(e)}}, \K^{d_{t(e)}} )\quad\quad \text{and}\quad\quad \Gl_{d_i}(S_i) = \Gl_{d_i}(\bb K) \   \]
for $i\in V_Q$,	so that in this case the definitions of $X_\dd(Q_M),\ G_\dd(Q_M)$ and $\mf M_\dd(Q_M)$ reduce to the standard definitions of $X_\dd(Q),\  G_\dd(Q)$ and $\mf M_\dd(Q)$ given in Section \ref{quiversec}.
\end{rmk}

Let $\Endi(S_M^\dd):=\Hom(S_M^\dd,S_M^\dd)\in S_M\Bimod$ denote the matrix algebra on
\[S_M^\dd:=\bigoplus_{i\in V_{Q_M}} S_i^{d_i}=\bb K^{\dd_0} \oplus S_\infty^{d_\infty}\]
with its natural $S_M$ bimodule structure, and note that there are canonical identifications
\[ X_\dd(Q_M) \cong \Hom_{S_M\Bimod}( \K\langle E_{Q_M}\rangle, \Endi(S_M^\dd)) \cong \Hom_{\Alg_\Ass(S_M\Bimod)}( \K{Q_M}, \Endi(S_M^\dd) )   \ ,\]
so that we have
\[ \mf{M}_\dd(Q_M)(\K) \cong \{ V\in S_M\Mod_\Fr,\ \varphi\in \Hom_{\Alg_\Ass(S\Bimod)}( \K Q_M, \Endi(V) )\ | \ \dim V = \dd \  \} \ ,  \]
that is, the groupoid of geometric points of $\mf{M}_\dd$ is the maximal subgroupoid of the category of modules over the free path algebra $\bb K Q_M$ with underlying $S_M$ module a direct sum of free $S_i$ modules of finite rank $\dd$.

Let $R_M\subset \bb K Q_M$ be the ideal generated by the image of the Koszul differential
\[d_{\Lambda_M}^{-1}:\Lambda_M^{-1} \to \Lambda_M^0= \bb K Q_M \ ,\]
and for each $\dd\in \bb N^{V_{Q_M}}$ define the closed subvariety $Z_\dd(Q_\M,R_\M) \subset X_\dd(Q_M)$ by
\begin{align*}
	 Z_\dd(Q_\M,I_\M) &  =\{ \varphi \in  \Hom_{\Alg_\Ass(S_M\Bimod)}( \K{Q_M}, \Endi(S_M^\dd) ) \ | \ \varphi(I_M) = \{0\}\  \subset  \Endi(S_M^\dd)  \}  \\ 
	 & = \Hom_{\Alg_\Ass(S_M\Bimod)}( \Lambda_M , \Endi(S_M^\dd) )  &  .
\end{align*}

\noindent Note that $Z_\dd(Q_M,I_M)$ is $G_\dd(Q_M)$ invariant, as the condition that the corresponding map to $ \Endi(V) $ satisfies $\varphi(I_M)=\{0\}$ is well defined, independent of the choice of isomorphism $V\cong S_M^\dd$. The stack of representations $\mf M(Q_M,R_M)$ of the extended quiver $(Q_M,R_M)$ is defined analogously as the disjoint union of the quotient stacks
\[ \mf{M}(Q_M,R_M) = \bigsqcup_{\dd \in \bb N^{V_Q}} \mf{M}_\dd(Q_M,R_M)  \quad\quad \text{where}\quad\quad \mf{M}_\dd(Q_M,R_M) = \left[ Z_\dd(Q_M,R_M) / G_\dd(Q_M) \right]\]
and we have the analogous description of the geometric points
\[ \mf{M}_\dd(Q_M,R_M)(\K) = \{ V\in S_M\Bimod,\ \varphi\in \Hom_{\Alg_\Ass(S_M\Bimod)}( \K Q_M/R_M , \Endi(V) )\ | \ \dim_{S_M} V = \dd \   \}  \]
as parameterizing the modules over the path algebra $\K Q_M/R_M= \Lambda_M$ with underlying $S_M$ module a direct sum of free $S_i$ modules of finite rank $\dd$.

There are subspaces $X_\dd^\nil(Q_M)\subset X_\dd(Q_M)$ defined by
\[ X_\dd^\nil(Q_M) = \{ \varphi \in \Hom_{\Alg_\Ass(S_M\Bimod)}( \K Q_M, \Endi(S_M^\dd) ) \ |\ \varphi( (\K Q_M)_{(n)} ) =  \{0\} \ \subset  \Endi(S_M^\dd)  \ \textup{for $n \gg 0$} \}  \ , \]
and in turn a substack $\mf M^\nil(Q_M) \subset \mf M(Q_M)$ defined by
\[ \mf{M}^\nil(Q_M) = \bigsqcup_{\dd \in \bb N^{V_{Q_M}}} \mf{M}^\nil_\dd(Q_M) \quad\quad\text{where}\quad\quad  \mf{M}_\dd^\nil(Q_M) = \left[ X_\dd^\nil(Q_M) / G_\dd(Q_M) \right] \ . \]

\noindent Similarly, there are closed subvarieties $Z_\dd^\nil(Q_M,R_M)\subset X_\dd^\nil(Q_M)$ defined by
\[Z_\dd^\nil(Q_M,R_M)= Z_\dd(Q_M,R_M)\times_{X_\dd(Q_M)} X_\dd^\nil(Q_M)\]
and in turn a closed substack $\mf M^\nil_\dd(Q_M,R_M) \subset \mf M^\nil(Q_M)$ defined by
\[ \mf{M}^\nil(Q_M,R_M) = \bigsqcup_{\dd \in \bb N^{V_{Q_M}}} \mf{M}^\nil_\dd(Q_M,R_M) \quad\quad\text{where}\quad\quad  \mf{M}_\dd^\nil(Q_M,R_M) = \left[ Z_\dd^\nil(Q_M, R_M) / G_\dd(Q) \right] \ . \]

We now state the main result of this section, which is the full statement of Theorem A from the introduction. As in Section \ref{quiversecunfr}, we introduce the notation
\[ \mf M(Y,M) :=  \mf M_{\Perv_M^\p(\hat Y)} \ , \]
and we have:

\begin{theo}\label{stackthm} Let $M\in \Perv^\p(Y)^{T}$ be an object satisfying the hypotheses in Section \ref{Extoverviewsec}, and let $Q_M$ be the associated extended quiver. There is an equivalence of algebraic stacks
	\begin{equation}\label{modulistackequiveqn}
		\mf M^\nil(Q_M,R_\M)  \xrightarrow{\cong}  \mf M(Y,M) 
	\end{equation}
	where the induced equivalence of groupoids of $\bb K$ points is defined on objects by
	\begin{equation*}\hspace*{-1cm}
		(V_i, B_{ij})   \mapsto \left(	\tilde H:=  \bigoplus_{i\in I_M}K(I_i)\otimes_{S_i}  V_i \ ,\ d_B:= \sum_{k\in \bb Z,\  i,i_2,...,i_{k-1},j \in I_M}  K (\rho_k^{\Sigma\oplus \Sigma_\infty}(\cdot, b_{i,i_2},...,b_{i_{k-1}, j})^N)\otimes  ( B_{i,i_2}  ...  B_{i_{k-1}, j} )  \right)  \ ,
	\end{equation*}
	in the notation of Section \ref{monadextsec}.
	
	Similarly, this induces an equivalence of the $T$ fixed points
	\begin{equation}\label{modulistackfpeqn}
		\mf M(Q_M,R_M)^{T}   \xrightarrow{\cong} \mf M_{\Perv_M^\p( Y)^T}  \ .
	\end{equation}
\end{theo}

The main remaining ingredient of the proof is the following lemma:

\begin{lemma}\label{quivermoduliextlemma} Let $\Lambda_M=\bb K\langle Q_M\rangle / R_M$ be the path algebra of the extended quiver $Q_M$ and $\hat \Lambda_M$ be its completion. There is an equivalence of algebraic stacks
	\[ \mf M^\nil(Q_\M,R_\M) \xrightarrow{\cong}\mf M_{\hat \Lambda_\M\Mod_\Fr}   \ . \]
\end{lemma}

\begin{proof}
	Let $R$ be a commutative ring with $T=\Spec R$, recall that $W = \Spf S_\infty \subset \hat X$ denotes the formal subscheme of $\hat X$ on which $\pi_* M$ is supported, and fix $\dd \in \bb N^{V_{Q_M}}$. The groupoid of $R$ points of the component $\mf M^\nil_\dd(Q_\M,I_\M)$ of the left hand side is given by
	\begin{align*}
		\mf M^\nil_\dd(Q_\M,I_\M)(R) & =  [Z_\dd(Q_M,I_M) / G_\dd(Q_M)](R) \\
		& = \{ P \in \Bun_{G_\dd(Q_M)}(T) \ ,\  \varphi \in \Gamma(T, P\times_{G_\dd(Q_M)} X_\dd^\nil(Q_M) ) \ | \  \varphi (I_M\otimes R) = \{0\} \  \} \ .
		\ \end{align*}
	Note that the underlying $G_\dd(Q_M) = \Gl_{\dd_0}(\bb K) \times \Gl_{d_\infty}(S_\infty)$ bundle $P$ over $R$ is equivalent to a rank $d_i$ vector bundle $V_i$ over $T$ for each $i\in V_Q$ together with a rank $d_\infty$ vector bundle $V_\infty$ over $T\times W$, trivializable over $\{t\}\times W$ for each $\bb K$ point $t\in T(\bb K)$. Under this identification, the section $\varphi$ is equivalent to a collection of sections of the associated endomorphism bundle
	\begin{equation}\label{quiverrepfamilyeqn}
		\varphi \in \bigoplus_{i,j\in I_M} \ _i\Sigma_j^1 \otimes_{S_i\otimes S_j^\op}  \Gamma(T, \Homi(V_i,V_j) \  )  
	\end{equation}
	satisfying the relations generating the ideal $I_M$, such that for each $\bb K$ point $t\in T(\bb K)$, the induced morphism of associative algebras $\Lambda_M \to \Endi(V_t)$ has nilpotent image, where $V_t$ denotes the fibre of $V=\oplus_{i\in I_M} V_i$ at $t$.
	
	Alternatively, the groupoid of $R$ points of $\mf M_{\hat \Lambda_\M\Mod_\Fr}$ is by definition given by
	\[ \mf M_{\hat \Lambda_\M\Mod_\Fr}(R) = \{ V \in \textup{Ind}(\hat\Lambda_M\Mod_\Fr)_R \ |\ \textup{$V$ is flat over $R$ and compact } \} \ , \]
	where we recall that $\mc C_R$ denotes the category of $R$ module objects internal to $\mc C$, as in Definition \ref{basecatdefn}.
	
	The compact objects of $\textup{Ind}(\hat \Lambda_M\Mod_\Fr)_R$ are $\hat\Lambda_M\otimes_\K R$ modules with underlying $S_M$ module given as a colimit of free modules of finite rank over $S_M$, which are finitely generated as $\hat \Lambda_M\otimes_\K R$ modules. These conditions together imply that the underlying $S_M\otimes_\K R$ module is finitely generated.
	
	Given an object $V\in \textup{Ind}(\hat\Lambda_M\Mod_\Fr)_R$, let $V=\oplus_{i\in V_Q} V_i \oplus V_\infty$ denote the natural decomposition as a module over $S_M= S\oplus S_\infty$.
	If $V$ is compact, then each $V_i$ defines a finitely generated $R$ module, and $V_\infty$ a finitely generated $S_\infty \otimes_\K R$ module such that $(V_\infty)_t=V_\infty\otimes_R \bb K$ is projective for each $\bb K$ point $t\in T(\bb K)$; the latter follows from the fact that the category of compact objects of the ind completion of a category $\mc C$ is given by the Karoubi completion of $\mc C$, and the  Karoubi completion of the category of free $S_\infty$ modules is the category of projective $S_\infty$ modules. Moreover, since $S_\infty$ is a local ring, projective and finitely generated implies free, so that $(V_\infty)_t$ is free of finite rank over $S_\infty$ for each $\bb K$ point $t\in T(\bb K)$.
	
	Finally, for $V\in \textup{Ind}(\hat\Lambda_M\Mod_\Fr)_R^\textup{c}$ compact, the condition that $V$ be flat over $R$ corresponds under the above identifications to the condition that each $V_i$ is flat as an $R$ module for $i\in V_{Q_M}$. Thus, for each $i\in V_Q$, $V_i$ defines a finitely generated, projective $R$ module or equivalently a vector bundle on $T$ of finite rank $d_i$. Similarly, $V_\infty$ defines a vector bundle on $T\times W$ of finite rank $d_\infty$ such that $(V_\infty)_t$ is trivializable for each $t\in T(\bb K)$.
	
	Thus, we have identified the $S_M\otimes_\K R$ module underlying an object $V\in \mf M_{\hat \Lambda_\M\Mod_\Fr}(R)$ with the collection of vector bundles underlying an object in the groupoid of $R$ points $\mf M^\nil_\dd(Q_\M,I_\M)(R)$. The additional data required to define the $\hat \Lambda_M\otimes R$ module structure on $V$ is evidently equivalent to that of $\varphi$ in Equation \ref{quiverrepfamilyeqn}; this completes the proof of the lemma.
\end{proof}

We now conclude the proof of Theorem \ref{stackthm} using the lemma and explain some implications of the result:

\begin{proof}{(of Theorem \ref{stackthm})}
The definition of the moduli of objects functor is evidently natural with respect to equivalences of categories, so that we obtain a canonical equivalence of algebraic stacks 
\[\mf M_{\hat \Lambda_\M\Mod_\Fr}  \xrightarrow{\cong} \mf M_{\Perv_M^\p(\hat Y)}  , \]
by the equivalence of Theorem \ref{heartthm}. Composing with the equivalence of algebraic stacks of Lemma \ref{quivermoduliextlemma} yields the desired equivalence of Equation \ref{modulistackequiveqn}, and the induced equivalence on groupoids of $\bb K$ points is defined on objects by the claimed formula, by Proposition \ref{monadprop}.
\end{proof}

\begin{corollary} Let $\dd=(\dd_0,0) \in \bb N^{V_{Q_M}}=\bb N^{V_Q} \times \bb N$. There is a commutative diagram of equivalences of algebraic stacks
\begin{equation*}	\xymatrix{ \mf M_{\dd_0}^\nil(Q_Y,W_Y) \ar[r]^\cong\ar[d]^\cong & \mf M_{\Perv_\cs(\hat Y), \dd_0} \ar[d]^\cong \\ 
		\mf M^\nil_\dd(Q_M,R_M) \ar[r]^\cong & \mf M_{\Perv_M^\p(\hat Y),\dd} } \ .
\end{equation*}
where the horizontal arrows are given by the equivalences of Theorems \ref{stackthmunext} and \ref{stackthm}, and the vertical arrows are given by the tautological identifications of Remark \ref{trivframermk} and of the image of $\Filt(F_i)\subset \Filt(F_i,M)$ as the full subcategory on objects with $0$ factors of $M$ in their composition series.
\end{corollary}
\begin{proof}
The result follows from commutativity of the diagram in Equation \ref{KozgeoFcompeqn}, together with the induced equivalences of hearts given in Theorems \ref{unextheartthm} and \ref{heartthm}.
\end{proof}

For $Q_M$ the extended quiver corresponding to the object $M\in \Perv^\p(Y)^{T}$, let $Q_\infty$ denote the full extended subquiver of $Q_M$ on the single extended vertex $\infty \in V_{Q_M}$, and similarly define
\[ X_{d_\infty}(Q_\infty) = \ _\infty \Sigma^1_\infty \otimes \End(\bb K^{d_\infty}) \quad\quad \text{and}\quad\quad \mf M_{d_\infty}(Q_\infty) = [ X_{d_\infty}(Q_\infty) / \Gl_{d_\infty}(S_\infty)]  \ , \]
as well as closed subvarieties $Z_{d_\infty}(Q_\infty, I_\infty) \subset  X_{d_\infty}(Q_\infty)$ and corresponding closed substacks $\mf M_{d_\infty}(Q_\infty,I_\infty) \subset \mf M_{d_\infty}(Q_\infty)$ by
\[ \hspace*{-1cm} Z_{d_\infty}(Q_\infty, I_\infty)  = X_{d_\infty}(Q_\infty) \times_{X_{\dd}(Q_M)} Z_{\dd}(Q_M,R_M) \quad\quad \text{and}\quad\quad \mf M_{d_\infty}(Q_\infty,I_\infty)=[ Z_{d_\infty}(Q_\infty, I_\infty)/ \Gl_{d_\infty}(S_\infty) ] \ ,  \]
where $\dd=(\textbf{0},d_\infty) \in  \bb N^{V_{Q_M}}=\bb N^{V_Q} \times \bb N$.

\begin{rmk}\label{trivunframermk}
	Note that for $\dd=(\textbf{0},d_\infty)\in  \bb N^{V_{Q_M}}=\bb N^{V_Q} \times \bb N$, we have tautological identifications
	\[ X_\dd(Q_M) =  X_{d_\infty}(Q_\infty) \quad\quad \text{and}\quad\quad G_\dd(Q_M) = \Gl_{d_\infty}(S_\infty)  \]
 so that in this case the definitions of $X_\dd(Q_M,R_M)$ and $\mf M_\dd(Q_M,R_M)$ reduce to the definitions of $X_{d_\infty}(Q_\infty,I_\infty)$ and $\mf M_{d_\infty}(Q_\infty,I_\infty)$ given in the preceding paragraph.
\end{rmk}

\begin{corollary}\label{framingstackcompcoro}Let $\dd=(\textbf{0},d_\infty)\in  \bb N^{V_{Q_M}}=\bb N^{V_Q} \times \bb N$. There is a commutative diagram of equivalences of algebraic stacks
	\begin{equation*}	\xymatrix{ \mf M_{d_\infty}^\nil(Q_\infty,I_\infty) \ar[r]^\cong\ar[d]^\cong & \mf M_{\Filt(M),d_\infty} \ar[d]^\cong \\ 
			\mf M^\nil_\dd(Q_M,R_M) \ar[r]^\cong & \mf M_{\Perv_M^\p(\hat Y),\dd} } \ .
	\end{equation*}
where the horizontal arrows are given by the equivalence of Theorem \ref{stackthm}, and the vertical arrows are given by the tautological identifications of Remark \ref{trivunframermk} and of the image of $\Filt(M) \subset \Filt(F_i,M)$ as the full subcategory on objects with $0$ factors of $F_i$ in their composition series for each $i\in V_Q$. 
\end{corollary}
\begin{proof}
	The result follows from commutativity of the diagram in Equation \ref{Kozgeoinfcompeqn},together with the induced equivalences of hearts given in Theorems \ref{heartinfthm} and \ref{heartthm}.
\end{proof}

\begin{eg}\label{Higgseg} Let $W\subset \hat Y$ denote a smooth, toric subvariety and $M=\mc O_W$, so that $S_\infty=H^0(W,\mc O_W)$. Then using the Koszul resolution, we can compute
\[ _\infty \Sigma^1_\infty \cong \Ext^1(\mc O_W, \mc O_W) \cong H^0(W, N_{W/Y})  \ , \]
the space of global sections of the normal bundle $N_{W/Y}$ to $W$ in $Y$. Moreover, the space
\[ Z_{d_\infty}(Q_\infty, I_\infty) \subset  X_{d_\infty}(Q_\infty) \cong H^0(W, N_{W/Y}) \otimes \End(\bb K^{d_\infty}) \]
identifies with the space of global $N_{W/Y}$-twisted Higgs fields on the trivial bundle $\mc O_W^{\oplus d_\infty}$. Thus, the stack of nilpotent representations of the quiver with relations $(Q_\infty,I_\infty)$ is given by
\[ \mf M^\nil_{d_\infty}(Q_\infty, I_\infty) =\left[Z_{d_\infty}(Q_\infty, I_\infty) / \Gl_{d_\infty}(S_\infty) \right] \cong \textup{Higgs}^\nil_{d_\infty}(W, N_{W/Y})_0 \ , \]
the stack of nilpotent $N_{W/Y}$-twisted Higgs bundles on $W$ of rank $d_\infty$ with underlying vector bundle isomorphic to the trivial bundle, though without such an isomorphism fixed.

Thus, in this example, we find that Corollary \ref{framingstackcompcoro} reduces to the relevant variant of the spectral correspondence, giving an equivalence between the above-described stack of Higgs bundles and the stack of coherent sheaves on $\hat Y$ isomorphic to an iterated extension of the structure sheaf $\mc O_W$.
\end{eg}

\subsection{Framing structures}\label{framingsec}

For an arbitrary element
\begin{equation*}
	\delta= \sum_{i,j \in R_M} b_{ij}\otimes B_{ij} \quad \in \quad \Sigma_M^1\otimes \Endi(V)  =  \bigoplus_{i,j\in R_M} \ _i\Sigma_j \otimes \Hom(V_i,V_j)  \ ,
\end{equation*}
consider the decomposition $V=V_\textbf{0}\oplus V_\infty := \left(\oplus_{i\in I} V_i \right) \oplus V_\infty$ and the induced decomposition
\[ \delta =  \ _0 \delta_0 + \  _\infty\delta_0 +\ _0\delta_\infty +\  _\infty\delta_\infty \quad \in \quad \Sigma^1\otimes \Endi(V)  =  \bigoplus_{i,j\in \{\textbf{0},\infty\} } \ _i\Sigma_j \otimes \Hom(V_i,V_j)   \ . \]
We now formulate an additional hypothesis on the object $M\in \DD^b\Coh(Y)^T$, which will be necessary for the results of this section: we require that the $i=j=\infty$ component of the Maurer-Cartan equations on $\delta$ determined by $\Sigma_M$, in the sense of Equation \ref{MCeqn}, satisfies
\begin{equation}\label{framingconditioneqn}
	 \ _\infty \left(\sum_{t\in \bb N } m_t^{\mf{gl}_{\dd}(\Sigma_M)} (\delta^{\otimes t} ) \right)\ \!\!\!  _\infty \  = \sum_{t\in \bb N}  m_t^{\mf{gl}_{d_\infty}( \!\!\ _\infty\Sigma_\infty)}(( _\infty \delta_\infty)^{\otimes t})   \ \in \mf{gl}_{d_\infty}(\!\!\ _\infty\Sigma_\infty)  \ ,
\end{equation}
that is, the $(\infty,\infty)$ component of the Maurer-Cartan operator for $\Sigma_M$ evaluated on $\delta$ is equal to the Maurer-Cartan operator for $\Sigma_\infty$ evaluated on the $(\infty,\infty)$ component of $\delta$.

\begin{prop}\label{framingmapprop} Let $M\in \DD^b\Coh(Y)^T$ satisfy the condition of Equation \ref{framingconditioneqn}, in addition to the usual hypotheses introduced in Section \ref{Extoverviewsec}. Then there exists a canonical map of stacks
	\begin{equation*}
		\mf M(Q_M,R_M)  \to \mf M (Q_\infty, I_\infty) \quad\quad \text{defined by}\quad\quad  	(V_i\ , B_{ij})_{i,j\in R_M} \mapsto (V_\infty \ , B_{\infty\infty})  \ .
	\end{equation*}
\end{prop}
\begin{proof}
The projection map to the direct summand corresponding to $i=j=\infty$ defines a canonical map
\[ X_\dd(Q_M) = \bigoplus_{i,j \in I_\M}  \ _i \Sigma^1_j \otimes \Hom(\bb K^{d_i}, \bb K^{d_j})  \to  \ _\infty\Sigma_\infty^1\otimes \End(\bb K^{d_\infty} ) = X_{d_\infty}(Q_\infty)  \]
which evidently induces a map on quotient stacks
\[ \mf M_\dd(Q_M) \to \mf M_{d_\infty}(Q_\infty) \ .\]
It remains to check that the restriction of the projection map to $Z_\dd(Q_M,R_M) \subset X_\dd(Q_M)$ defines a map
\[ Z_\dd(Q_M,R_M)  \to Z_{d_\infty}(Q_\infty, I_\infty)  \quad\quad\text{and thus a map} \quad\quad  \mf M_\dd(Q_M,R_M) \to \mf M_{d_\infty}(Q_\infty, I_\infty) \ , \]
as desired, on the corresponding closed substacks of the above quotient stacks. The left hand side of Equation \ref{framingconditioneqn} is one of the generators of the defining ideal of the subvariety $Z_\dd(Q_M,R_M) \subset X_\dd(Q_M)$, and thus the equality with the right hand side, which generates the ideal defining $ Z_{d_\infty}(Q_\infty, I_\infty) \subset X_{d_\infty}(Q_\infty)$, implies the result.
\end{proof}

Under the hypotheses of the preceding proposition, we also have:
\begin{corollary} There exists a canonical map of stacks
	\begin{equation*}
		\mf M^\nil(Q_M,R_M) \to \mf M^\nil(Q_\infty, I_\infty) \quad\quad \text{defined by}\quad\quad  	(V_i\ , B_{ij})_{i,j\in I_M} \mapsto (V_\infty \ , B_{\infty\infty})  \ .
	\end{equation*}
\end{corollary}
\begin{proof} The projection map $ X_\dd(Q_M) \to  X_{d_\infty}(Q_\infty)$ evidently maps $ X^\nil_\dd(Q_M) $ to $X^\nil_{d_\infty}(Q_\infty)$, and the result follows readily from the proof of Proposition \ref{framingmapprop}.
\end{proof}

We have the following equivalent geometric form of the corollary:
\begin{corollary}\label{geoframingmapcoro} There exists a canonical map of stacks
	\begin{equation*}
		 \mf M(Y,M) \to \mf M_{\Filt(M)}
	\end{equation*}
defined by forgetting the compactly supported composition factors and remembering only the underlying iterated extension of $M$.
\end{corollary}
\begin{proof}
This is equivalent to the statement of the preceding corollary, under the identifications of Theorem \ref{stackthm} and Corollary \ref{framingstackcompcoro}. Note that while the operation on complexes of coherent sheaves defining the map is evidently not well-defined in general, the geometric content of Proposition \ref{framingmapprop} above is that this operation is well defined under the given hypotheses.
\end{proof}

\begin{defn} A \emph{framing structure} for $M$ of rank $d_\infty$ is a $\bb K$-point
	\[ \f \in Z_{d_\infty}(Q_\infty, I_\infty)(\bb K) \ . \]
\end{defn}

Given a framing structure $\f$ for $M$ of rank $d_\infty$, we define the stack $\mc F_\f$ over $\mf M_{d_\infty}(Q_\infty, I_\infty)$ by
\[ \mc F_\f = \left[ \{ \f\} /  G_\f \right]\cong \left[ \bb O_\f / \Gl_{d_\infty}(S_\infty)\right] \ \to \  \mf M_{d_\infty}(Q_\infty, I_\infty) \ , \]
where $ G_\f = \textup{Stab}_{\Gl_{d_\infty}(S_\infty)}(\f)$ denotes the subgroup of $\Gl_{d_\infty}(S_\infty)$ stabilizing the point $\f$, and $\bb O_\f\subset Z_{d_\infty}(Q_\infty,I_\infty)$ denotes the orbit of $\f$ under $\Gl_{d_\infty}(S_\infty)$. Similarly, we define the stack $\mc F^\cn_\f$ over $\mf M_{d_\infty}(Q_\infty, I_\infty)$ by
\[ \mc F^\cn_\f = \left[ \{ \f\} /  G^\cn_\f \right]\cong \left[ \bb O^\cn_\f / \Gl_{d_\infty}(\bb K)\right] \ \to \  \mf M_{d_\infty}(Q_\infty, I_\infty) \ , \]
where $ G^\cn_\f = \textup{Stab}_{\Gl_{d_\infty}(\bb K)}(\f)$ denotes the subgroup of $\Gl_{d_\infty}(\bb K)$ stabilizing the point $\f$, and $\bb O^\cn_\f\subset Z_\dd(Q_M,R_M)$ denotes the orbit of $\f$ under $\Gl_{d_\infty}(\bb K)$.

Moreover, for any $\dd=(d_i)\in \bb N^{V_Q}$, we define the stack of $\mc F_\f$-framed representations of $Q_M$ of dimension $\dd$ by
\begin{equation*}
\mf M^{\mc F_\f}_\dd(Q_M, R_M) = \mf M_{(\dd,d_\infty)}(Q_M,R_M) \times_{ \mf M_{d_\infty}(Q_\infty,I_\infty)} \mc F_\f
\end{equation*}
and define the full moduli stack of $\mc F_\f$-framed representations of $Q_M$ by
\[  \mf M^{\mc F_\f}(Q_M, R_M)  = \bigsqcup_{\dd\in \bb N^{V_Q}} \mf M^{\mc F_\f}_\dd (Q_M,R_M)  \ , \]
as usual. Similarly, we define the stack of $\mc F_\f^\cn$-framed representations of $Q_M$ by
\begin{equation*}\hspace*{-1cm}
 \mf M^{\mc F^\cn_\f}(Q_M, R_M)  = \bigsqcup_{\dd\in \bb N^{V_Q}} \mf M^{\mc F^\cn_\f}_\dd (Q_M,R_M) \quad\quad \text{where}\quad\quad  \mf M^{\mc F^\cn_\f}_\dd(Q_M, R_M) = \mf M_{(\dd,d_\infty)}(Q_M,R_M) \times_{ \mf M_{d_\infty}(Q_\infty,I_\infty)} \mc F^\cn_\f \ .
\end{equation*}

 More explicitly, we introduce the space
\begin{align*}
	 X^\f_\dd(Q_M) & = X_\dd(Q_M) \times_{X_{d_\infty}(Q_\infty)} \{\f\}  \\ \nonumber
	 &  \cong \bigoplus_{i,j \in V_Q}  \ _i \Sigma^1_j \otimes \Hom(\bb K^{d_i}, \bb K^{d_j})  \oplus \bigoplus_{i \in V_Q} \ _i\Sigma^1_\infty \otimes \Hom(\bb K^{d_i},\bb K^{d_\infty})  \oplus  \bigoplus_{j \in V_Q} \ _\infty\Sigma^1_j \otimes \Hom(\bb K^{d_\infty}, \bb K^{d_j}) 
\end{align*}
as well as the closed subvariety $Z^\f_\dd(Q_M,R_M) \subset X^\f_\dd(Q_M)$, defined by
\[ Z^\f_\dd(Q_M,R_M) = Z_\dd(Q_M,R_M) \times_{Z_{d_\infty}(Q_\infty,I_\infty)} \{\f\} \ . \]
Then the moduli stacks of $\mc F_{\f}$-framed and $\mc F_f^\cn$-framed representations of $Q_M$ of dimension $\dd$ are given by
\[ \hspace*{-1cm} \mf M^{\mc F_\f}_\dd(Q_M, R_M) =  \left[ Z^\f_\dd(Q_M,R_M) / (G_\dd(Q_Y) \times G_\f ) \right] \quad\quad \text{and}\quad\quad  \mf M^{\mc F^\cn_\f}_\dd(Q_M, R_M) =  \left[ Z^\f_\dd(Q_M,R_M) / (G_\dd(Q_Y) \times G_\f^\cn ) \right]   \ , \]
respectively.

Finally, we define the stack of $\f$-framed representations of $Q_M$ by
\begin{equation*}
\mf M^{\f}(Q_M, R_M)  = \bigsqcup_{\dd\in \bb N^{V_Q}} \mf M^{\f}_\dd (Q_M,R_M) \quad\quad \text{where}\quad\quad	\mf M^{\f}_\dd(Q_M, R_M) = \mf M_{(\dd,d_\infty)}(Q_M,R_M) \times_{ \mf M_{d_\infty}(Q_\infty,I_\infty)} \{ \f \}
\end{equation*} and the analogous concrete description is simply given by
\[	\mf M^{\f}_\dd(Q_M, R_M)=  \left[ Z^\f_\dd(Q_M,R_M) / G_\dd(Q_Y)  \right] \ . \]

We can also define the corresponding equivalent geometric moduli stacks
\begin{align}
\mf M^{\f}(Y,M) &  = \bigsqcup_{\dd \in \bb N^{V_Q}} \mf M_\dd^{\f}(Y,M) &\textup{where}\quad \mf M_\dd^{\f}(Y,M) & = \mf M_{(\dd,d_\infty)}(Y,M) \times_{\mf M_{\Filt(M),d_\infty}}  \{ \f \}   &  \textup{,} \label{fframedstackeqn}\\
\mf M^{\mc F^\cn_\f}(Y,M) & = \bigsqcup_{\dd \in \bb N^{V_Q}} \mf M_\dd^{\mc F^\cn_\f}(Y,M) &\textup{where}\quad \mf M_\dd^{\mc F^\cn_\f}(Y,M) & = \mf M_{(\dd,d_\infty)}(Y,M) \times_{\mf M_{\Filt(M),d_\infty}} \mc F^\cn_\f  & \quad\quad\textup{,} \label{Fcframedstackeqn}\\
\mf M^{\mc F_\f}(Y,M) & = \bigsqcup_{\dd \in \bb N^{V_Q}}  \mf M_\dd^{\mc F_\f}(Y,M) &\textup{where}\quad \mf M_\dd^{\mc F_\f}(Y,M) & = \mf M_{(\dd,d_\infty)}(Y,M) \times_{\mf M_{\Filt(M),d_\infty}} \mc F_\f  & \textup{.} \label{Fframedstackeqn}
\end{align}
We can now give the framed analogue of Theorem \ref{stackthm}, which establishes a special case of Conjecture 10.3.5 a) in \cite{KS7}:
 
\begin{theo}\label{stackpotthm} Let $M\in \Perv^\p(Y)^T$ as in Theorem \ref{stackthm} and let $\f$ be a framing structure for $M$. There is a canonical framed quiver with potential $(Q_M^\f,W_M^\f)$ and an equivalence of algebraic stacks
\begin{equation}
\mf M(Q_M^\f,W_M^\f) \xrightarrow{\cong} \mf M^{\f}(Y,M) \ ,
\end{equation}
where the induced equivalence of groupoids of $\bb K$ points is defined on objects by
	\begin{equation*}\hspace*{-1cm}
	(V_i, B_{ij})   \mapsto \left(	\tilde H:=  \bigoplus_{i\in I_M}K(I_i)\otimes_{S_i}  V_i \ ,\ d_B:= \sum_{k\in \bb Z,\  i,i_2,...,i_{k-1},j \in I_M}  K (\rho_k^{\Sigma\oplus \Sigma_\infty}(\cdot, b_{i,i_2},...,b_{i_{k-1}, j})^N)\otimes  ( B_{i,i_2}  ...  B_{i_{k-1}, j} )  \right)  \ ,
\end{equation*}
in the notation of Section \ref{monadextsec}.
\end{theo}
 \begin{proof} By Theorem \ref{stackthm} and Corollary \ref{framingstackcompcoro}, the desired result is equivalent to an equivalence of algebraic $\bb K$-stacks
$ \mf M(Q_M^\f,W_M^\f) \xrightarrow{\cong} \mf M^{\f}(Q_M, R_M)$
 which we now construct.
 
 The underlying quiver $Q_M^\f$ is essentially the same as the extended quiver $Q_M$, but with the extended node corresponding to the object $M$ considered as a framing node, and the edges from the extended node to itself excluded from the edge set; more formally, we define the set of internal vertices by $V_{Q_M^\f}=V_{Q_Y}$, and define the $\bb K$-linear span of the edge set to be
 \[	\bb K\langle E_{Q_M^\f}\rangle :=  \bigoplus_{i,j \in V_Q}  \ _i \Sigma_j^\vee[-1]  \oplus \bigoplus_{i \in V_Q}\  _i\Sigma_\infty^\vee[-1]  \oplus  \bigoplus_{j \in V_Q} \ _\infty\Sigma_j^\vee[-1]  \ .\]
 
 Note that we evidently have an identification of vector spaces
 \[ X_\dd(Q_M^\f):=\Hom_{S\Bimod}(\bb K\langle E_{Q_M^\f}\rangle, \Endi(\bb K^\dd) ) \cong  X^\f_\dd(Q_M) \ . \]
 
 \noindent Thus, it remains to check that the closed subvariety $Z^\f_\dd(Q_M,R_M) \subset X^\f_\dd(Q_M)$ is given by the critical locus of a potential. In fact, the potential is given by a slight variant of the formula given in Example \ref{geoquiveg2}, as follows: we fix a decomposition
 \[ \f = \sum_{\alpha}  f_\alpha \otimes F_\alpha \ \in\   _\infty \Sigma^1_\infty \otimes \End(\bb K^{d_\infty}) \]
 in terms of which the potential (which is evidently independent of this choice of decomposition) is given by
 \[	 W_M^\f= \sum_{n\geq 1} \sum_{a_1,...,a_{n+1} \in E_{Q_M^\f}\cup \{f_\alpha \} } \langle m_n^{\Sigma_M}(a_1^\vee,...,a_n^\vee), a_{n+1}^\vee \rangle a_1 \cdot ... \cdot a_{n+1}  \ , \]
 where we note that $E_{Q_M^\f}\cup \{f_\alpha\}\subset E_{Q_M}$, so that their duals define elements on which the $\Ainf$ multiplication $m_n^{\Sigma_M}:\Sigma_M^{\otimes n} \to \Sigma_M[2-n]$ can be evaluated. We then interpret the variables $a_i$ corresponding to edges $a_i\in E_{Q_M^\f}\subset E_{Q_M}$ as the usual defining variables in the potential, while the variables $a_i=f_\alpha\in\  _\infty \Sigma^1_\infty$ are defined to evaluate to the fixed linear maps $F_\alpha$ when defining the induced potential $W^\f_{M,\dd}:X^\f_\dd(Q_M)\to \bb K$ for any $\dd\in \bb N^{V_Q}$. In analogy with Example \ref{geoquiveg2}, the critical locus equations induce the Maurer-Cartan equations for $\Sigma_M$, with the modification of the variables corresponding to $\{\f_\alpha\}$ inducing the specialization at $\f \in Z_{d_\infty}(Q_\infty, I_\infty)$, so that we have that \[Z^\f_\dd(Q_M,R_M)=\textup{crit}(W^\f_{M,\dd}) \subset X^\f_\dd(Q_M) \ ,\] as desired.
 \end{proof}

Finally, we give a brief explanation of the geometric interpretation of these stacks. By Corollary \ref{geoframingmapcoro}, we have the following modular descriptions of the spaces of $\bb K$ points:
\begin{itemize}
	\item For $\mf M^{\f}(Y,M)$, a perverse coherent extension equipped with an isomorphism from the underlying iterated extension of $M$ to that determined by $\f$.
	\item For $\mf M_\dd^{\mc F^\cn_\f}(Y,M)$, a perverse coherent extension equipped with an isomorphism from the underlying iterated extension of $M$ to a complex of sheaves of the form $(K(I_\infty)\otimes_\bb K W,\  d)$ where $W$ is a $\bb K$ vector space of dimension $d_\infty$ and $d$ an auxiliary differential which is conjugate under $\Gl(W)$ to the differential determined by $f$.
	\item For $\mf M_\dd^{\mc F_\f}(Y,M)$, a perverse coherent extension such that the underlying iterated extension of $M$ is isomorphic to that determined by $\f$.
\end{itemize}

We now explain the distinction between these variants in the simplest examples, though we note that several more involved examples are given in Section \ref{framedexamplessec} below. We fix the dimension vector to be $\dd= \bf 0 \in \bb N^{V_Q}$, so that by Corollary \ref{framingstackcompcoro}, we have
\[ \mf M_{(\textbf{0}, d_\infty)}(Y,M) \cong \mf M_{\Filt(M),d_\infty}\ , \]
that is, the (unframed) moduli of perverse coherent extensions is given by the moduli stack of objects in the full subcategory on objects admitting filtration with subquotients isomorphic to $M$ of total multiplicity $d_\infty$. In turn, this implies that the framed moduli spaces are simply given by
\[ \mf M_\textbf{0}^{\f}(Y,M) \cong \{ \f \}  \ , \quad\quad \mf M_\textbf{0}^{\mc F^\cn_\f}(Y,M) \cong \mc F^\cn_\f  \ , \text{ and}  \quad\quad \mf M_\textbf{0}^{\mc F_\f}(Y,M) \cong  \mc F_\f  \ .  \]

We now explain these isomorphisms explicitly in some special cases:

\begin{eg}\label{DTframingeg} Let the framing object be $M=\mc O_{\hat Y}$ the structure sheaf of $\hat Y$. Since the structure sheaf does not admit any non-trivial extensions with itself, the only possible framing structure is given by $\f=0$, and $\mf M_{(\bf{0}, d_\infty)}(Y,M)$ is equivalent to the moduli stack of objects which are isomorphic to $\mc O_{\hat Y}^{\oplus d_\infty}$. Thus, in this case, we have that $\mf M^{\mc F_\f}_{(\bf{0}, d_\infty)}(Y,M)\cong \mf M_{(\bf{0}, d_\infty)}(Y,M)$ is given by the moduli stack of sheaves which are isomorphic to $\mc O_{\hat Y}^{\oplus d_\infty}$, which is evidently equivalent to $\mc F_\f= \left[ \pt / \Gl_{d_\infty}(\mc O_{\hat Y})\right]$.
	
Similarly, we have that $\mf M_\dd^{\mc F^\cn_\f}(Y,M)$ is the moduli stack of sheaves equipped with an isomorphism to $\mc O_{\hat Y}\otimes_\bb K W$ for $W$ a $\bb K$ vector space of dimension $d_\infty$, which is evidently equivalent to \[ \mc F_\f^\cn=\left[ \pt/ \Gl_{d_\infty}(\bb K)\right]  \ . \]

Finally, we have that $\mf M_\dd^{\f}(Y,M)$ is the moduli stack of sheaves equipped with an isomorphism to $\mc O_{\hat Y}^{\oplus d}$, which is evidently equivalent to just the single point $\{ \mc O_{\hat Y}^{\oplus d_\infty}\}$ corresponding to $\f=\{0\}$.
\end{eg}

\subsection{Examples}\label{framedexamplessec} In this section, we explain several examples of the construction of Theorem \ref{stackpotthm} and their relationships with previously known constructions of moduli spaces of coherent sheaves in terms of quivers.

\begin{eg}\label{pervcohsyseg} The most basic family of examples of perverse coherent extensions is the case when the extending object is given by $M=\mc O_Y[1]$, the structure sheaf of the threefold shifted down in cohomological degree by one. As we explain in Section \ref{DTsec} below, in this case the definition of a perverse coherent extension is essentially equivalent to that of a perverse coherent system in the sense of Nagao-Nakajima \cite{NN}, and the resulting quivers generalize those studied in \emph{loc. cit.} in the case of the resolved conifold following \cite{Sz1}.
	
Indeed, for $M=\mc O_Y[1]$, we have
\[ _\infty\Sigma_\infty^1=0 \quad\quad \ _\infty \Sigma^1_j=0 \quad\quad \ _i \Sigma^1_\infty = \begin{cases} \bb K & \textup{if $i=0$} \\ \{0\} & \textup{otherwise} \end{cases} \]
independent of the threefold $Y$. Thus, the extended quivers are given by adding a single node and a single arrow from that node to the zeroth node of the quiver $Q_Y$. In the cases of Examples \ref{a3eg}, \ref{o2eg}, and \ref{o1eg}, the extended quivers $Q_M$ are thus given by
\[
\begin{tikzcd}
	\mathcircled{M}\arrow[r, "I"] &
	\mathcircled{F_0}\arrow[out=340,in=20,loop,swap,"B_3"]
	\arrow[out=220,in=260,loop,swap,"B_2"]
	\arrow[out=100,in=140,loop,swap,"\normalsize{B_1}"]
\end{tikzcd}
\quad\quad
	\begin{tikzcd}
		\mathcircled{M}\arrow[d, "I"] \\ 
	\arrow[out=160,in=200,loop,swap,"E"] \mathcircled{F_0} \arrow[r, bend left=25 ] \arrow[r, bend left=40 ,  "A\ C"]  & \arrow[l, bend left=25 ] \arrow[l, bend left=40 ,  "B\ D"] \mathcircled{F_1}\arrow[out=340,in=20,loop,swap,"F"]
\end{tikzcd}
\quad\quad
\begin{tikzcd}
			\mathcircled{M}\arrow[d, "I"] \\ 
	\mathcircled{F_0} \arrow[r, bend left=25 ] \arrow[r, bend left=40 ,  "A\ C"]  & \arrow[l, bend left=25 ] \arrow[l, bend left=40 ,  "B\ D"]  \mathcircled{F_1}
\end{tikzcd}
\ ,
\]
respectively. Since $_\infty\Sigma_\infty^1=0$, for each $d_\infty\in \bb N$ is only a single framing stucture $\f=0\in Z_{d_\infty}(Q_\infty,I_\infty)=\{ 0 \}$, and the corresponding framed quivers with potential $(Q_M^\f,W_M^\f)$ are given by

\begin{equation}\label{DTquiverseqn}
	\begin{tikzcd}
	\boxed{V_\infty}\arrow[r, "I"] &
	\mathcircled{V}\arrow[out=340,in=20,loop,swap,"B_3"]
	\arrow[out=220,in=260,loop,swap,"B_2"]
	\arrow[out=100,in=140,loop,swap,"\normalsize{B_1}"]
\end{tikzcd}
\quad\quad
\begin{tikzcd}
	\boxed{V_\infty}\arrow[d, "I"] \\ 
	\arrow[out=160,in=200,loop,swap,"E"] \mathcircled{V_0} \arrow[r, bend left=25 ] \arrow[r, bend left=40 ,  "A\ C"]  & \arrow[l, bend left=25 ] \arrow[l, bend left=40 ,  "B\ D"] \mathcircled{V_1}\arrow[out=340,in=20,loop,swap,"F"]
\end{tikzcd}
\quad\quad
\begin{tikzcd}
	\boxed{V_\infty}\arrow[d, "I"] \\ 
	\mathcircled{V_0} \arrow[r, bend left=25 ] \arrow[r, bend left=40 ,  "A\ C"]  & \arrow[l, bend left=25 ] \arrow[l, bend left=40 ,  "B\ D"]  \mathcircled{V_1}
\end{tikzcd}
\ ,
\end{equation}

\noindent together with the same potentials as in Examples \ref{a3eg}, \ref{o2eg}, and \ref{o1eg}.

The $\Sigma$-module $I_\infty=S_0$ is given by a single copy of the simple module $S_0$ in degree $1$, and thus the corresponding projective resolution is the trivial one $K(I_\infty)=\mc O_Y[1]=M$. The representative $i:K(I_\infty)\to K(I_0)$ of $\ _0 \Sigma^1_\infty=\bb K$ is induced by the constant map $\mc O\to \mc O$; for example, in the case of the resolved conifold, following Example \ref{o1eg}, it is given by
\[\xymatrix{ & & & \mc O \ar[d]^{1} \\ \mc{O} \ar[r] & \mc{O}(1)^2 \ar[r] & \mc{O}(1)^2 \ar[r] & \mc{O} &} \ . \]
Thus, the monad presentation for perverse coherent systems induced by Proposition \ref{monadprop} is given by modifying that for compactly supported perverse coherent sheaves by a single summand of the form $\mc O\otimes V_\infty$ in cohomological degree $-1$ and a single additional auxiliary differential from degree $-1$ to degree $0$ given by $I$; for example, in the case of the resolved conifold, following Example \ref{o1eg}, it is given by
	\begin{equation}\label{conifoldmonadeqn}
	\xymatrix @R=-.3pc @C=-1.5pc{ & {\scriptsize{\begin{pmatrix}  1 & -B \\ z & -D \\ -A & x \\ - C & y \end{pmatrix}}} & && {\scriptsize{\begin{pmatrix}zy-DC & CB-y & 0 & zB-D \\ D A - zx & x-BA & D-zB & 0 \\ 0 & yA-xC & yz - CD & AD-xz \\ xC-yA & 0 & CB-y & x-AB  \\ 0 & 0 & 0 & 0\end{pmatrix}}} && && \scriptsize{\begin{pmatrix} x & y & B & D & I \\ A & C & 1 & z & 0 \end{pmatrix}}&\\
		\mc O \otimes V_0   & & \mc O(1)^2 \otimes V_0  && &&\mc O(1)^2 \otimes V_0 &&  &&\mc O  \otimes V_0 	\\
		\oplus&  \longrightarrow & \oplus && \longrightarrow && \oplus && \longrightarrow && \oplus \\
		\\
		\mc O(1) \otimes V_1  & & \mc O^2 \otimes V_1 &&& &    \mc O^2 \otimes V_1 &&&& \mc O(1) \otimes V_1 \\ 
				&   &  &&  && \oplus &&  &&  \\
	& & && && \mc O \otimes V_\infty } \ .
\end{equation}
\end{eg}

\begin{eg}\label{ADHMeg} Let $Y=X=\bb C^3=\Spec \C[x,y,z]$ and let $M=\mc O_{\bb C^2}[1]$ the shifted structure sheaf of the coordinate subspace $\bb C^2_{xy} \subset \bb C^3$. Then we have
\[ _\infty \Sigma^1_0=\bb K \quad\quad \ _0 \Sigma^1_\infty =\bb K \quad\quad _\infty\Sigma_\infty^1=S_\infty \]
so that the extended quiver $Q_M$ is given by
\[ Q_M = \quad 
\begin{tikzcd}
	\mathcircled{M} 	\arrow[out=160,in=200,loop,swap, "A"]  \arrow[r, shift left=0.5ex, "I"] & \arrow[l, shift left=0.5ex, "J"]
	\mathcircled{F_0}\arrow[out=340,in=20,loop,swap,"B_3"]
	\arrow[out=220,in=260,loop,swap,"B_2"]
	\arrow[out=100,in=140,loop,swap,"\normalsize{B_1}"]
\end{tikzcd} \ . \]
The equivalence classes of constant framing structures $\f$ on $M$ of rank $d_\infty$ are given by conjugacy classes of matrices in $A_\f\in \mf{gl}_{d_\infty}$, and the corresponding framed quiver with potential is given by
\begin{equation}\label{ADHM3deqn}
	 Q_M^\f =	\quad
\begin{tikzcd}\boxed{V_\infty} 	\arrow[out=160,in=200,loop,swap,"\bullet" marking, "A_\f"]  \arrow[r, shift left=0.5ex, "I"] & \arrow[l, shift left=0.5ex, "J"]
\mathcircled{V}\arrow[out=340,in=20,loop,swap,"B_3"]
\arrow[out=220,in=260,loop,swap,"B_2"]
\arrow[out=100,in=140,loop,swap,"\normalsize{B_1}"] \end{tikzcd}
 \quad\quad\text{and}\quad\quad   W_M^\f = B_1[B_2,B_3] + B_3I J - IA_\f J   \ .  
\end{equation}
In particular, taking $A_\f=0$, we recover the standard quiver with potential giving the `tripled' or `three dimensional' variant of the ADHM construction, which corresponds under dimensional reduction to the usual ADHM quiver. More generally, $(Q_M^\f,W_M^\f)$ gives the quiver with potential studied in \cite{CCDS}. These relationships are discussed further in Section \ref{} below.

The $\Sigma$ module $I_\infty$ corresponding to $\mc O_{\bb C^2}[1]$ is given by
\[ I_\infty = \left[ S_0\langle 1 \rangle < S_0\langle 2 \rangle \right] \quad\quad \text{so that}\quad\quad K(I_\infty) = \left[ \mc O[2] \xrightarrow{z} \mc O[1] \right ] \xrightarrow{\cong} \mc O_{\bb C^2}[1] \ , \]
and we can choose representatives of the auxiliary elements of $\Sigma^1$ according to
\begin{eqnarray*}
	i\in  \ _0 \Sigma^1_\infty \quad \mapsto\quad & 
			\xymatrixcolsep{3pc}
\xymatrixrowsep{3pc}
\xymatrix{ &  & \mc{O}\ar[d]^{\scriptsize{\begin{pmatrix}0\\0\\ 1\end{pmatrix}}} \ar[r] &  \mc{O}\ar[d]^{\scriptsize{1}} \\
	\mc{O} \ar[r] & \mc{O}^3 \ar[r] &  \mc{O}^3 \ar[r] & \mc{O}} \\
	j \in \ _\infty \Sigma^1_0 \quad \mapsto\quad &
\xymatrixcolsep{3pc}
\xymatrixrowsep{3pc}
\xymatrix{ \mc{O} \ar[d]^{-1} \ar[r] & \mc{O}^3 \ar[d]^{\scriptsize{\begin{pmatrix}0&0&1 \end{pmatrix}}} \ar[r] &  \mc{O}^3\ar[r] & \mc{O}  \\
	\mc{O} \ar[r] & \mc{O}   && } \\
	a\in \ _\infty\Sigma^1_\infty \quad \mapsto\quad &
\xymatrixcolsep{3pc}
\xymatrixrowsep{3pc}
\xymatrix{ &  & \mc{O}\ar[d]^{\scriptsize{1}} \ar[r] &  \mc{O} \\
	& \mc{O} \ar[r] &  \mc{O} & }
\end{eqnarray*}
so that the monad presentation induced by Proposition \ref{monadprop} is given by
	\begin{equation}
	\xymatrixcolsep{-1pc}
\xymatrixrowsep{-.3pc}\xymatrix{
	&{\scriptsize{\begin{pmatrix}B_1-x\\ y-B_2\\ B_3-z \\ -J\end{pmatrix}}} &&& {\scriptsize{\begin{pmatrix}0&B_3-z&B_2-y & 0 \\ B_3-z &0&x-B_1 & 0\\ y-B_2& x-B_1& 0 & I \\ 0 & 0 & J & A_\f \end{pmatrix}}} &&& {\scriptsize{\begin{pmatrix}x-B_1&y-B_2&z-B_3 & I \end{pmatrix}}}\\
	&&& \mc{O}^3\otimes  V &&& \mc{O}^3 \otimes V &  \\ 
	\mc{O}\otimes V  &\longrightarrow & & \oplus  &\longrightarrow&&   \oplus  &\longrightarrow&&&& \mc{O}\otimes V \\ &&& \mc O\otimes V_\infty &&& \mc O\otimes V_\infty } \ .
\end{equation}
\end{eg}

\begin{eg}\label{spikedeg} Generalizing the previous example, let $Y=X=\bb C^3=\Spec \C[x,y,z]$, with $M=\mc O_{\bb C^2_{xy}}\oplus \mc O_{\bb C^2_{xz}}\oplus \mc O_{\bb C^2_{yz}}[1]$, and consider the simplest framing structure $\f=0$. Then the framed quiver with potential $(Q_M^\f,W_M^\f)$ is given by
\[ Q_M^\f =	\quad
\begin{tikzcd}
	& & \boxed{V_\infty^2} \ar[dl, shift left=0.5ex, "I_2"] \\
	\boxed{V_\infty^3} \arrow[r, shift left=0.5ex, "I_3"] & \arrow[l, shift left=0.5ex, "J_3"]
	\mathcircled{V}\arrow[out=340,in=20,loop,swap,"B_3"]
	\arrow[out=220,in=260,loop,swap,"B_2"]
	\arrow[out=100,in=140,loop,swap,"\normalsize{B_1}"]  \ar[ur, shift left=0.5ex, "J_2"] \ar[dr, shift left=0.5ex, "J_1"] \\
	& & \boxed{V_\infty^1} \ar[ul, shift left=0.5ex, "I_1"]
\end{tikzcd}
\quad\quad\text{and}\quad\quad   W_M^\f = B_1[B_2,B_3] + \sum_{k=1}^3 B_k I_k J_k \ , \]
the quiver introduced in \cite{NekP} to describe the moduli space of spiked instantons. Thus, the moduli stack of perverse coherent extensions provides a model for this moduli space defined purely in terms of algebraic geometry.

There is also an interesting space of framing structures in this family of examples. With $V_\infty^1=\{0\}$ for simplicity, a framing structure $\f$ is determined by a tuple of linear maps $A_\f=(A_2,A_3,A_2^3,A_3^2)$, and the corresponding framed quiver with potential is given by
\begin{equation}\label{spikedfrmeqn}
	 Q_M^\f = \quad 
\begin{tikzcd}
	\boxed{V_\infty^3} \arrow[out=115,in=155,loop, "\bullet" marking, swap,"A_3"]  \arrow[rd, bend left, "I_3"]  \arrow[rr,  "\bullet" marking, bend left, "A_2^3"]   & & \arrow["\bullet" marking]{ll}[swap]{A_3^2}  \arrow[ld, bend left,  "I_{2}"] \arrow[out=25,in=65,loop, "\bullet" marking, swap,"A_{2}"]   \boxed{V_\infty^2}  \\
	& \arrow[lu, bend left, "J_3"] \arrow[ru, bend left, "J_2"]  \mathcircled{V_0}\arrow[out=310,in=350,loop,swap,"B_3"]
	\arrow[out=190,in=230,loop,swap,"B_1"]
	\arrow[out=250,in=290,loop,swap,"\normalsize{B_2}"]&
\end{tikzcd} \quad\quad \text{and}
\end{equation}
\[ W_M^\f =B_1[B_2,B_3] + B_3I_{3}J_3  + B_2 I_{2} J_2 - I_{3}A_{3}J_3 - I_{2}A_{2}J_2 + I_{2}A_{2}^{3} J_{3}  + I_{3}A_{3}^{2} J_{2}  \ .\]
The monad presentations for these generalizations can be computed analogously.
\end{eg}

\begin{eg}\label{KNeg}	Let $Y=|\mc O_{\bb P^1}\oplus \mc O_{\bb P^1}(-2)| \to X= \Spec \C[x_1,x_2,x_3,x_4]/(x_1x_2-x_3^2)$, following Example \ref{o2eg}, and let $M=\mc O_{W}[1]$ for $W=|\mc O_{\bb P^1}(-2)|$. Then we have
	\[ _0 \Sigma^1_\infty= \bb K \quad\quad  \ _\infty\Sigma^1_0=\bb K \quad\quad _1 \Sigma^1_\infty= \{0\} \quad\quad  \ _\infty\Sigma^1_1=\{0\} \quad\quad \ _\infty\Sigma^1_\infty = S_\infty  \]
so that the extended quiver $Q_M$ is given by
\[Q_M = \quad \begin{tikzcd}
\mathcircled{M}\arrow[d,shift left=0.5ex, "I"] \arrow[out=160,in=200,loop,swap,"G"] \\ 
\arrow[out=160,in=200,loop,swap,"E"] \ar[u,shift left=0.5ex, "J"] \mathcircled{F_0} \arrow[r, bend left=25 ] \arrow[r, bend left=40 ,  "A\ C"]  & \arrow[l, bend left=25 ] \arrow[l, bend left=40 ,  "B\ D"] \mathcircled{F_1}\arrow[out=340,in=20,loop,swap,"F"]
\end{tikzcd} \ . \]
Again, a constant framing structure $\f$ is given by an endomorphism $G_\f$, and the resulting framed quiver with potential $(Q_M^\f,W_M^\f)$ is given by
\[ \hspace*{-1cm} Q_M^\f = \quad \begin{tikzcd}
	\boxed{V_\infty}\arrow[d,shift left=0.5ex, "I"] \arrow[out=160,in=200,loop, "\bullet" marking,swap,"G_\f"] \\ 
	\arrow[out=160,in=200,loop,swap,"E"] \ar[u,shift left=0.5ex, "J"] \mathcircled{V_0} \arrow[r, bend left=25 ] \arrow[r, bend left=40 ,  "A\ C"]  & \arrow[l, bend left=25 ] \arrow[l, bend left=40 ,  "B\ D"] \mathcircled{V_1}\arrow[out=340,in=20,loop,swap,"F"]
\end{tikzcd} \quad\quad \text{and}\quad\quad W_M^\f =E(BC-DA)+ F(AD-CB) + EIJ - IG_\f J  \ . \]
In particular, for $G_\f=0$ the resulting quiver is the framed and tripled, in the sense of Ginzburg \cite{Ginz}, affine type $\textup{A}_1$ Dynkin quiver, and corresponds under dimensional reduction to the framed and doubled, in the sense of Nakajima \cite{Nak1}, affine type $\textup{A}_1$ Dynkin quiver
\[ \begin{tikzcd}
	\boxed{V_\infty}\arrow[d,shift left=0.5ex, "I"] \\ 
	\ar[u,shift left=0.5ex, "J"] \mathcircled{V_0} \arrow[r, bend left=25 ] \arrow[r, bend left=40 ,  "A\ C"]  & \arrow[l, bend left=25 ] \arrow[l, bend left=40 ,  "B\ D"] \mathcircled{V_1}
\end{tikzcd} \quad\quad \text{with relations}\quad\quad \begin{cases} BC-DA + IJ & =0 \\ AD-CB &  = 0 \end{cases} \ , \]
which describes the moduli space of instantons on $|\mc O_{\bb P^1}(-2)|$, the resolved $\textup{A}_1$ singularity, by the results of Kronheimer-Nakajima \cite{KrNak}.

The $\Sigma$ module $I_\infty$ corresponding to $M=\mc O_{|\mc O_{\bb P^1}(-2)|}[1]$ is given by
\[ I_\infty = \left[ S_0\langle 1 \rangle < S_0\langle 2 \rangle \right] \quad\quad \text{so that}\quad\quad K(I_\infty)=\left[ \mc O[2] \xrightarrow{y} \mc O[1] \right] \xrightarrow{\cong} \mc O_{|\mc O_{\bb P^1}(-2)|}[1] \]
and we can choose representatives of the auxiliary elements of $\Sigma^1$ according to
\begin{eqnarray*}
	i\in  \ _0 \Sigma^1_\infty \quad \mapsto\quad & 
	\xymatrixcolsep{3pc}
	\xymatrixrowsep{3pc}
	\xymatrix{ &  & \mc{O}\ar[d]^{\scriptsize{\begin{pmatrix}1\\0\\ 0\end{pmatrix}}} \ar[r] &  \mc{O}\ar[d]^{\scriptsize{1}} \\
			\mc{O} \ar[r] & \mc O\oplus \mc{O}(1)^2 \ar[r] &  \mc{O}\oplus \mc{O}(1)^2 \ar[r] & \mc{O} &} \\
	j \in \ _\infty \Sigma^1_0 \quad \mapsto\quad &
	\xymatrixcolsep{3pc}
	\xymatrixrowsep{3pc}
	\xymatrix{ 	\mc{O} \ar[r] \ar[d]^{1} & \mc O\oplus \mc{O}(1)^2 \ar[r] \ar[d]^{\scriptsize{\begin{pmatrix}1&0&0 \end{pmatrix}}} &  \mc{O}\oplus \mc{O}(1)^2 \ar[r] & \mc{O} \\	\mc{O} \ar[r] & \mc{O}   && } \\
	g\in \ _\infty\Sigma^1_\infty \quad \mapsto\quad &
	\xymatrixcolsep{3pc}
	\xymatrixrowsep{3pc}
	\xymatrix{ &  & \mc{O}\ar[d]^{\scriptsize{1}} \ar[r] &  \mc{O} \\
		& \mc{O} \ar[r] &  \mc{O} & }
\end{eqnarray*}
so that the monad presentation induced by Proposition \ref{monadprop} is given by
\begin{equation*}
\hspace*{-1cm}	\xymatrix @R=-.3pc @C=-2.2pc{ & {\scriptsize{\begin{pmatrix} y-E & 0 \\ 1 & B \\ z & D \\ 0 & y-F \\ A & x \\ C & xz \\ J & 0\end{pmatrix}}} & && {\scriptsize{\begin{pmatrix} 0 & xz & -x & 0 & D & -B & I \\ -z & 0 & y-E & -D & 0 & 0 & 0 \\ 1 & E-y & 0 & B & 0 & 0 & 0 \\ 0 & C & -A & 0 & z & -1 & 0 \\ -C & 0 & 0 & -xz & 0 & y-F & 0 \\ A & 0 & 0 & x & F-y& 0 & 0 \\  J & 0 & 0 & 0 & 0 & 0 & G_\f\end{pmatrix}}} && && \scriptsize{\begin{pmatrix} y-E & x & xz & 0 & B & D & I\\ 0 & A & C & y-F & 1 & z & 0 \end{pmatrix}}&\\
		  & &( \mc O \oplus \mc O(1)^2) \otimes V_0  && && (\mc O \oplus \mc O(1)^2) \otimes V_0 &&  &&	\\
		\mc O \otimes V_0 &   & \oplus &&  && \oplus &&  && \mc O  \otimes V_0 \\
		\\
		\oplus &\longrightarrow & (\mc O(1)\oplus \mc O^2) \otimes V_1 && \longrightarrow & &   (\mc O(1)\oplus \mc O^2) \otimes V_1 && \longrightarrow &&\oplus \\
\mc O(1) \otimes V_1 	&   & \oplus &&  && \oplus &&  &&  \mc O(1) \otimes V_1  \\
	& & \mc O\otimes V_\infty && && \mc O\otimes V_\infty
}  \ .
\end{equation*}
\end{eg}

\begin{eg}\label{Beilinsoneg} Let $Y=|\mc O_{\bb P^1}\oplus \mc O_{\bb P^1}(-2)| \to X= \Spec \C[x_1,x_2,x_3,x_4]/(x_1x_2-x_3^2)$, following Example \ref{o2eg}, and let $M=\mc O_{W}[1]$ for $W=|\mc O_{\bb P^1}|$. Then we have
	\[ _0 \Sigma^1_\infty= \bb K \quad\quad  \ _\infty\Sigma^1_0=\{0\} \quad\quad _1 \Sigma^1_\infty= \{0\} \quad\quad  \ _\infty\Sigma^1_1=\bb K^2 \quad\quad \ _\infty\Sigma^1_\infty = \{0\}  \]
	so that the extended quiver $Q_M$ is given by
	\[Q_M = \quad \begin{tikzcd}
		\mathcircled{M}\arrow[d,shift left=0.5ex, "I"] \\ 
		\arrow[out=160,in=200,loop,swap,"E"] \mathcircled{F_0} \arrow[r, bend left=25 ] \arrow[r, bend left=40 ,  "A\ C"]  & \arrow[l, bend left=25 ] \arrow[l, bend left=40 ,  "B\ D"] \mathcircled{F_1}\arrow[out=340,in=20,loop,swap,"F"] \ar[ul,shift left=1ex, bend right=40]\ar[ul, bend right=40, swap,"J_1 \ J_2"]    
	\end{tikzcd} \ . \]
The only framing structure is the trivial one $\f=0$ since $ \ _\infty\Sigma^1_\infty = \{0\}$, and the corresponding framed quiver with potential $(Q_M^\f,W_M^\f)$ is given by
	\[ \hspace*{-1cm}Q_M^\f = \quad \begin{tikzcd}
	\boxed{V_\infty}\arrow[d,shift left=0.5ex, "I"] \\ 
	\arrow[out=160,in=200,loop,swap,"E"] \mathcircled{V_0} \arrow[r, bend left=25 ] \arrow[r, bend left=40 ,  "A\ C"]  & \arrow[l, bend left=25 ] \arrow[l, bend left=40 ,  "B\ D"] \mathcircled{V_1}\arrow[out=340,in=20,loop,swap,"F"] \ar[ul,shift left=1ex, bend right=40]\ar[ul, bend right=40, swap,"J_1 \ J_2"]    
\end{tikzcd}  \quad\quad \textup{and}\quad\quad W_M^\f =E(BC-DA)+ F(AD-CB) + IJ_1A + IJ_2C  \ . \]
This quiver with potential corresponds under dimensional reduction to a framed variant of the product of the standard Beilinson quiver for $\bb P^1$ with the Jordan quiver, given by
	\[  \begin{tikzcd}
	\boxed{V_\infty}\arrow[d,shift left=0.5ex, "I"] \\ 
	\arrow[out=160,in=200,loop,swap,"E"] \mathcircled{V_0}  & \arrow[l,swap, "B\ D"] \arrow[l,shift left=1ex] \mathcircled{V_1}\arrow[out=340,in=20,loop,swap,"F"] \ar[ul,shift left=1ex, bend right=40]\ar[ul, bend right=40, swap,"J_1 \ J_2"]  
\end{tikzcd} \quad\quad \text{with relations}\quad\quad \begin{cases} EB-BF + I J_2 & =0 \\ DF-ED + I J_1 &  = 0 \end{cases} \ . \]

The $\Sigma$ module $I_\infty$ corresponding to $M=\mc O_{|\mc O_{\bb P^1}|}[1]$ is given by
\[\hspace*{-1cm} I_\infty = \left[ S_0\langle 1 \rangle< S_1^{\oplus 2} \langle 2 \rangle < S_0\langle 3 \rangle \right] \quad\quad \text{so that}\quad\quad K(I_\infty)=\left[ \mc O[3] \xrightarrow{\scriptsize{\begin{pmatrix}-z \\ 1 \end{pmatrix}}} \mc O(1)^2[2] \xrightarrow{\scriptsize{\begin{pmatrix}x& xz\end{pmatrix}}} \mc O[1] \right] \xrightarrow{\cong} \mc O_{|\mc O_{\bb P^1}(-2)|}[1] \]
and we can choose representatives of the auxiliary elements of $\Sigma^1$ according to
\begin{eqnarray*}
	i\in  \ _0 \Sigma^1_\infty \quad \mapsto\quad & 
	\xymatrixcolsep{3pc}
	\xymatrixrowsep{3pc}
	\xymatrix{ &   \mc O \ar[r]\ar[d]^{\scriptsize{\begin{pmatrix}1\\0\\ 0\end{pmatrix}}} & \mc{O}(1)^2 \ar[r] \ar[d]^{\scriptsize{\begin{pmatrix}0 & 0\\ 1 & 0\\ 0 & 1\end{pmatrix}}}&  \mc{O}\ar[d]^{\scriptsize{1}} \\
		\mc{O} \ar[r] & \mc O\oplus \mc{O}(1)^2 \ar[r] &  \mc{O}\oplus \mc{O}(1)^2 \ar[r] & \mc{O} &} \\
	j_1 \in \ _\infty \Sigma^1_1 \quad \mapsto\quad &
	\xymatrixcolsep{3pc}
	\xymatrixrowsep{3pc}
	\xymatrix{ &  \mc{O}(1) \ar[r] \ar[d]^{\scriptsize{\begin{pmatrix}1 \\ 0 \end{pmatrix}}} & \mc O(1)\oplus \mc{O}^2 \ar[r]  \ar[d]^{\scriptsize{\begin{pmatrix}0&1&0 \end{pmatrix}}} &  \mc{O}(1)\oplus \mc{O}^2 \ar[r] & \mc{O}(1) & \\
	\mc O \ar[r] & \mc O(1)^2 \ar[r] & \mc O 		 } \\
	j_2 \in \ _\infty \Sigma^1_1 \quad \mapsto\quad &
\xymatrixcolsep{3pc}
\xymatrixrowsep{3pc}
\xymatrix{ &  \mc{O}(1) \ar[r] \ar[d]^{\scriptsize{\begin{pmatrix}0 \\ 1 \end{pmatrix}}} & \mc O(1)\oplus \mc{O}^2 \ar[r]  \ar[d]^{\scriptsize{\begin{pmatrix}0&0&1 \end{pmatrix}}} &  \mc{O}(1)\oplus \mc{O}^2 \ar[r] & \mc{O}(1) & \\
	\mc O \ar[r] & \mc O(1)^2 \ar[r] & \mc O 		 }
\end{eqnarray*}
so that the monad presentation induced by Proposition \ref{monadprop} is given by
\begin{equation*}
	\hspace*{-1cm}	\xymatrix @R=-.3pc @C=-2.2pc{ & {\scriptsize{\begin{pmatrix} y-E & 0 & I\\ 1 & B & 0 \\ z & D & 0 \\ 0 & y-F & 0 \\ A & x & 0 \\ C & xz & 0 \\ 0 & J_1 & 0 \\ 0 & J_2 & 0 \end{pmatrix}}} & && {\scriptsize{\begin{pmatrix} 0 & xz & -x & 0 & D & -B & 0 & 0 \\ -z & 0 & y-E & -D & 0 & 0 & I & 0 \\ 1 & E-y & 0 & B & 0 & 0 & 0 & I \\ 0 & C & -A & 0 & z & -1 & 0 & 0 \\ -C & 0 & 0 & -xz & 0 & y-F & 0 & 0 \\ A & 0 & 0 & x & F-y& 0 & 0  & 0\\  0 & 0 & 0 & 0 & J_1 & J_2 & 0 & 0\end{pmatrix}}} && && \scriptsize{\begin{pmatrix} y-E & x & xz & 0 & B & D & I\\ 0 & A & C & y-F & 1 & z & 0 \end{pmatrix}}&\\
	\mc O \otimes V_0	& &( \mc O \oplus \mc O(1)^2) \otimes V_0  && && (\mc O \oplus \mc O(1)^2) \otimes V_0 &&  &&	\\
		 \oplus&   & \oplus &&  && \oplus &&  && \mc O  \otimes V_0 \\
		\\
	\mc O(1) \otimes V_1 	 &\longrightarrow & (\mc O(1)\oplus \mc O^2) \otimes V_1 && \longrightarrow & &   (\mc O(1)\oplus \mc O^2) \otimes V_1 && \longrightarrow &&\oplus \\
		\oplus	&   & \oplus &&  && \oplus &&  &&  \mc O(1) \otimes V_1  \\
	\mc O\otimes V_\infty	& & \mc O(1)^2\otimes V_\infty  && && \mc O\otimes V_\infty
	}  \ .
\end{equation*}
\end{eg}

\begin{eg}\label{prechainsaweg}
	Let $Y=|\mc O_{\bb P^1}\oplus \mc O_{\bb P^1}(-2)| \to X= \Spec \C[x_1,x_2,x_3,x_4]/(x_1x_2-x_3^2)$, following Example \ref{o2eg}, and let $M=\mc O_{W}[1]$ for $W=\mc O_{\bb P^1}|_{z=0}\oplus \mc O_{\bb P^1}(-2)|_{z=0}\cong \bb A^2$. Then we have	
	\[ _0 \Sigma^1_\infty= 0 \quad\quad  \ _\infty\Sigma^1_0=\bb K \quad\quad _1 \Sigma^1_\infty= \bb K \quad\quad  \ _\infty\Sigma^1_1=\{0\} \quad\quad \ _\infty\Sigma^1_\infty = S_\infty  \]
	so that the extended quiver $Q_M$ is given by
\[Q_M = \quad \begin{tikzcd}
	\arrow[out=160,in=200,loop,swap,"E"] \mathcircled{F_0} \ar[d,"J"] \arrow[r, bend left=25 ] \arrow[r, bend left=40 ,  "A\ C"]  & \arrow[l, bend left=25 ] \arrow[l, bend left=40 ,  "B\ D"] \mathcircled{F_1}\arrow[out=340,in=20,loop,swap,"F"]   \\
	\mathcircled{M}\arrow[ur, bend right=40, "I"] 	\arrow[out=160,in=200,loop,swap,"G"] 
\end{tikzcd} \ . \]

\noindent For the trivial framing structure $G_\f=0$, the corresponding framed quiver with potential is given by
\[Q_M^\f = \quad \begin{tikzcd}
	\arrow[out=160,in=200,loop,swap,"E"] \mathcircled{V_0} \ar[d,"J"] \arrow[r, bend left=25 ] \arrow[r, bend left=40 ,  "A\ C"]  & \arrow[l, bend left=25 ] \arrow[l, bend left=40 ,  "B\ D"] \mathcircled{V_1}\arrow[out=340,in=20,loop,swap,"F"]   \\
	\boxed{V_\infty}\arrow[ur, bend right=40, "I"] 
\end{tikzcd} \quad\quad \textup{and}\quad\quad W_M^\f =E(BC-DA)+ F(AD-CB) + IJD  \ . \]
This quiver with potential corresponds under dimensional reduction to the principal chainsaw quiver of \cite{KanT} for $\rho=\rho_\textup{prin}$ a principal nilpotent, given by
\[  \begin{tikzcd}
 \arrow[out=160,in=200,loop,swap,"E"] \mathcircled{V_0} \ar[d,"J"] \arrow[r, bend left=25, "A" ] & \arrow[l, bend left=25, "B" ]  \mathcircled{V_1}\arrow[out=340,in=20,loop,swap,"F"]   \\	
\boxed{V_\infty}\arrow[ur, bend right=40, "I"] 
\end{tikzcd}  \quad\quad \text{with relations}\quad\quad \begin{cases}  FA-AE + IJ &  = 0 \\ EB-BF & =0  \end{cases} \ . \]

The $\Sigma$ module $I_\infty$ corresponding to $M=\mc O_{W}[1]$ is given by
\[\hspace*{-1cm} I_\infty = \left[ S_0\langle 1 \rangle< S_1 \langle 2 \rangle \right] \quad\quad \text{so that}\quad\quad K(I_\infty)=\left[ \mc O[2] \xrightarrow{\scriptsize{z}} \mc O(1)[1] \right] \xrightarrow{\cong} \mc O_{W}[1] \]
and we can choose representatives of the auxiliary elements of $\Sigma^1$ according to
\begin{eqnarray*}
	i\in  \ _1 \Sigma^1_\infty \quad \mapsto\quad & 
	\xymatrixcolsep{3pc}
	\xymatrixrowsep{3pc}
	\xymatrix{ &   & \mc{O}\ar[r] \ar[d]^{\scriptsize{\begin{pmatrix}0 \\0 \\ 1\end{pmatrix}}}&  \mc{O}(1) \ar[d]^{\scriptsize{1}} \\
		\mc{O}(1) \ar[r] & \mc O(1)\oplus \mc{O}^2 \ar[r] &  \mc{O}(1)\oplus \mc{O}^2 \ar[r] & \mc{O}(1)  &} \\
	j \in \ _\infty \Sigma^1_0 \quad \mapsto\quad &
	\xymatrixcolsep{3pc}
	\xymatrixrowsep{3pc}
	\xymatrix{   \mc{O} \ar[r] \ar[d]^{\scriptsize{1}} & \mc O\oplus \mc{O}(1)^2 \ar[r]  \ar[d]^{\scriptsize{\begin{pmatrix}0&0&1 \end{pmatrix}}} &  \mc{O}\oplus \mc{O}(1)^2 \ar[r] & \mc{O}& \\
	 \mc O\ar[r] & \mc O(1)  } 
\end{eqnarray*}
so that the monad presentation induced by Proposition \ref{monadprop} is given by
\begin{equation*}
	\hspace*{-1cm}	\xymatrix @R=-.3pc @C=-2.2pc{ & {\scriptsize{\begin{pmatrix} y-E & 0 \\ 1 & B \\ z & D \\ 0 & y-F \\ A & x \\ C & xz \\ J & 0\end{pmatrix}}} & && {\scriptsize{\begin{pmatrix} 0 & xz & -x & 0 & D & -B & 0 \\ -z & 0 & y-E & -D & 0 & 0 & 0 \\ 1 & E-y & 0 & B & 0 & 0 & 0 \\ 0 & C & -A & 0 & z & -1 & 0 \\ -C & 0 & 0 & -xz & 0 & y-F & 0 \\ A & 0 & 0 & x & F-y& 0 & I \\  0 & 0 & J & 0 & 0 & 0 & 0 \end{pmatrix}}} && && \scriptsize{\begin{pmatrix} y-E & x & xz & 0 & B & D & 0\\ 0 & A & C & y-F & 1 & z & I\end{pmatrix}}&\\
		& &( \mc O \oplus \mc O(1)^2) \otimes V_0  && && (\mc O \oplus \mc O(1)^2) \otimes V_0 &&  &&	\\
		\mc O \otimes V_0 &   & \oplus &&  && \oplus &&  && \mc O  \otimes V_0 \\
		\\
		\oplus &\longrightarrow & (\mc O(1)\oplus \mc O^2) \otimes V_1 && \longrightarrow & &   (\mc O(1)\oplus \mc O^2) \otimes V_1 && \longrightarrow &&\oplus \\
		\mc O(1) \otimes V_1 	&   & \oplus &&  && \oplus &&  &&  \mc O(1) \otimes V_1  \\
		& & \mc O\otimes V_\infty && && \mc O(1)\otimes V_\infty
	}  \ .
\end{equation*}
\end{eg}

\begin{eg}\label{chainsaweg}	Let $Y=|\mc O_{\bb P^1}\oplus \mc O_{\bb P^1}(-2)| \to X= \Spec \C[x_1,x_2,x_3,x_4]/(x_1x_2-x_3^2)$, following Example \ref{o2eg}, and let $M=\mc O_{W_0}[1]\oplus \mc O_{W_1}[1]$ for $W_0=\mc O_{\bb P^1}|_{z=0}\oplus \mc O_{\bb P^1}(-2)|_{z=0}\cong \bb A^2$ and $W_1=|\mc O_{\bb P^1}|$. Then generalizing the previous two examples, extended quiver is given by
\[Q_M = \quad \begin{tikzcd}
		\mathcircled{M_1}\arrow[d,shift left=0.5ex, "I"] \ar[dd,shift left=-2.7ex, out=220,in=140,swap, "K" ] \\ 
		\arrow[out=160,in=200,loop,shift right =-1ex,looseness=3,swap,"E"] \mathcircled{F_0}\ar[d,"J_0"]  \arrow[r, bend left=25 ] \arrow[r, bend left=40 ,  "A\ C"]  & \arrow[l, bend left=25 ] \arrow[l, bend left=40 ,  "B\ D"] \mathcircled{F_1}\arrow[out=340,in=20,loop,swap,"F"] \ar[ul,shift left=1ex, bend right=40]\ar[ul, bend right=40, swap,"J_1 \ J_2"]   \\
			\mathcircled{M_0}\arrow[ur, bend right=40,swap, "I_0"] 	\arrow[out=160,in=200,loop,swap,"G"]  \ar[uu,shift left=1.7ex, out=140,in=220, swap, "H" ]
	\end{tikzcd} \ . \]

In particular, consider the framing structure $\f$ given by $G=H=0$ but with $K_\f$ potentially non-trivial. Then the corresponding framed quiver with potential is given by
 \begin{equation}\label{chainsawwpotsl2eqn}
 	 Q_M^\f = \quad \begin{tikzcd}
	\boxed{V_{\infty_1}}\arrow[d,shift left=0.5ex, "I"] \ar[dd,shift left=-1.7ex, out=220,in=140,swap, "\bullet" marking, "K" ] \\ 
	\arrow[out=160,in=200,loop,shift right =-1ex,looseness=3,swap,"E"] \mathcircled{V_0}\ar[d,"J_0"]  \arrow[r, bend left=25 ] \arrow[r, bend left=40 ,  "A\ C"]  & \arrow[l, bend left=25 ] \arrow[l, bend left=40 ,  "B\ D"] \mathcircled{V_1}\arrow[out=340,in=20,loop,swap,"F"] \ar[ul,shift left=1ex, bend right=40]\ar[ul, bend right=40, swap,"J_1 \ J_2"]   \\
	\boxed{V_{\infty_0}}\arrow[ur, bend right=40, swap,"I_0"] 
\end{tikzcd} \quad\quad \textup{and}\quad\quad \begin{array}{l} W_M^\f = E(BC-DA)+ F(AD-CB)  \\ \quad\quad\quad\quad\quad\quad + IJ_1A + IJ_2C  + I_0J_0D+I_0 K J_1 \end{array} \ . 
 \end{equation}
This quiver with potential corresponds under dimensional reduction to the quiver
\[ \hspace*{-1cm} \begin{tikzcd}
	\boxed{V_{\infty_1}}\arrow[d,shift left=0.5ex, "I"] \ar[dd,shift left=-1.7ex, out=220,in=140,swap, "\bullet" marking, "K" ] \\ 
	\arrow[out=160,in=200,loop,shift right =-1ex,looseness=3,swap,"E"] \mathcircled{V_0}\ar[d,"J_0"]   \arrow[r, bend left=25 , "A"]  & \arrow[l, bend left=25 ,  "B"] \mathcircled{V_1}\arrow[out=340,in=20,loop,swap,"F"] \ar[ul, bend right=40, swap,"J_2"]   \\
	\boxed{V_{\infty_0}}\arrow[ur, bend right=40, swap,"I_0"] 
\end{tikzcd} \quad\quad \textup{with relations}\quad\quad \begin{cases}  FA-AE + I_0J_0 &  = 0 \\ EB-BF + I J_2& =0 \\
AI+I_0 K & = 0 \end{cases}  \ . \]
In particular, if $d_{\infty_1}\leq d_{\infty_0}$ and $K$ is injective, we conjecture that for appropriate choice of stability conditions, the third relation gives an equivalence between the space of stable representations of this quiver with relations and that of the quiver
\begin{equation}\label{chainsawsl2eq}
	  \begin{tikzcd}
  	\boxed{V_{\infty_1}}\arrow[d,shift left=0.5ex, "I"] \\ 
  	\arrow[out=160,in=200,loop,swap,"E"] \mathcircled{V_0}\ar[d," \tilde J_0"]   \arrow[r, bend left=25 , "A"]  & \arrow[l, bend left=25 ,  "B"] \mathcircled{V_1}\arrow[out=340,in=20,loop,swap,"F"] \ar[ul, bend right=40, swap,"J_2"]   \\
  	\boxed{\tilde V_{\infty_0}}\arrow[ur, bend right=40, swap,"\tilde I_0"] 
  \end{tikzcd} \quad\quad \textup{with relations}\quad\quad \begin{cases}  FA-AE + \tilde I_0 \tilde J_0 &  = 0 \\ EB-BF + I J_2& =0 \end{cases}  \ , 
\end{equation}
  where $\tilde d_{\infty_0}=\dim \tilde V_{\infty_0}= d_{\infty_0}-d_{\infty_1}$; the latter quiver with relations is precisely the chainsaw quiver of \cite{FinR} for $\rho$ a nilpotent in $\gl_n$ for $n=d_{\infty_1}+\tilde d_{\infty_0}$ with Jordan blocks of size $d_{\infty_1}$ and $\tilde d_{\infty_0}$, and this is a special case of Conjecture \ref{chainsawconj}. The monad presentation in this example is a straightforward generalization of those in the preceding two examples.
\end{eg}

\begin{eg}\label{NYeg}
	Let $Y=|\mc O_{\bb P^1}(-1)\oplus \mc O_{\bb P^1}(-1)|  \to X =\Spec \C[x_1,x_2,x_3,x_4]/(x_1x_2-x_3x_4)$, following Example \ref{o1eg}, and let $M=\mc O_{|\mc O_{\bb P^1}(-1)|}[1]$. Then we have
	\[ _0 \Sigma^1_\infty= \bb K \quad\quad  \ _\infty\Sigma^1_0=\{0\} \quad\quad _1 \Sigma^1_\infty= \{0\} \quad\quad  \ _\infty\Sigma^1_1=\bb K \quad\quad \ _\infty\Sigma^1_\infty =\{0\} \]
	so that the extended quiver $Q_M$ is given by
		\[Q_M = \quad \begin{tikzcd}
		\mathcircled{M}\arrow[d,shift left=0.5ex, "I"] \\ 
	\mathcircled{F_0} \arrow[r, bend left=25 ] \arrow[r, bend left=40 ,  "A\ C"]  & \arrow[l, bend left=25 ] \arrow[l, bend left=40 ,  "B\ D"] \mathcircled{F_1}\ar[ul, bend right=40, swap,"J"]    
	\end{tikzcd} \ . \]
The only framing structure is the trivial one $\f=0$ since $ \ _\infty\Sigma^1_\infty = \{0\}$, and the corresponding framed quiver with potential $(Q_M^\f,W_M^\f)$ is given by
	\[Q_M^\f = \quad \begin{tikzcd}
	\boxed{V_\infty}\arrow[d,shift left=0.5ex, "I"] \\ 
	\mathcircled{V_0} \arrow[r, bend left=25 ] \arrow[r, bend left=40 ,  "A\ C"]  & \arrow[l, bend left=25 ] \arrow[l, bend left=40 ,  "B\ D"] \mathcircled{V_1}\ar[ul, bend right=40, swap,"J"]    
\end{tikzcd} \quad\quad \text{and}\quad\quad   W_M^\f = ABCD-ADBC + IJC \ . \]
This quiver corresponds under dimensional reduction to the Nakajima-Yoshioka quiver
	\[\begin{tikzcd}
	\boxed{V_\infty}\arrow[d,shift left=0.5ex, "I"] \\ 
	\mathcircled{V_0} \arrow[r,  "d"]  & \arrow[l, bend left=25 ] \arrow[l, bend left=40 ,  "B_1\ B_2"] \mathcircled{V_1}\ar[ul, bend right=40, swap,"J"] \end{tikzcd} \quad\quad \text{with relations}\quad\quad B_1dB_2-B_2dB_1+IJ=0 \ , \]
which describes the moduli space of instantons on the blowup $\widehat{\C}^2 =|\mc O_{\bb P^1}(-1)|$, as explained in \cite{NY1}; here we have renamed the variables by $A\mapsto d$, $B\mapsto B_2$ and $D\mapsto B_1$ to agree with the notation of \emph{loc. cit.}.

The $\Sigma$ module $I_\infty$ corresponding to $M=\mc O_{|\mc O_{\bb P^1}(-1)|}[1]$ is given by
\[ I_\infty = \left[ S_0\langle 1 \rangle< S_1 \langle 2 \rangle\right] \quad\quad \text{so that}\quad\quad K(I_\infty)=\left[ \mc O(1)[2] \xrightarrow{y} \mc O[1] \right] \xrightarrow{\cong} \mc O_{|\mc O_{\bb P^1}(-1)|}[1] \]
and we can choose representatives of the auxiliary elements of $\Sigma^1$ according to
\begin{eqnarray}
	i \in \ _0 \Sigma^1_\infty \quad \mapsto\quad &
\xymatrixcolsep{3pc}
\xymatrixrowsep{3pc}
\xymatrix{  & & \mc O(1) \ar[r]\ar[d]^{\scriptsize{\begin{pmatrix}0 \\ 1\end{pmatrix}}}  & \mc O \ar[d]^{1} \\  \mc{O} \ar[r] & \mc O(1)^2 \ar[r] & \mc O(1)^2\ar[r] & \mc{O}  }\\ \nonumber
	j \in \ _\infty \Sigma^1_1 \quad \mapsto\quad &
\xymatrixcolsep{3pc}
\xymatrixrowsep{3pc}
\xymatrix{  \mc{O}(1)\ar[d]^{1} \ar[r] & \mc O^2\ar[d]^{\scriptsize{\begin{pmatrix}0 & 1\end{pmatrix}}} \ar[r] &  \mc O^2 \ar[r] & \mc{O}(1)  \\
	\mc{O}(1) \ar[r] & \mc{O}}
\end{eqnarray}
so that the monad presentation induced by Proposition \ref{monadprop} is given by
	\begin{equation*}
	\xymatrix @R=-.3pc @C=-1.5pc{ & {\scriptsize{\begin{pmatrix}  1 & -B \\ z & -D \\ -A & x \\ - C & y \\  0 & J \end{pmatrix}}} & && {\scriptsize{\begin{pmatrix}zy-DC & CB-y & 0 & zB-D & 0\\ D A - zx & x-BA & D-zB & 0 & I \\ 0 & yA-xC & yz - CD & AD-xz & 0 \\ xC-yA & 0 & CB-y & x-AB & 0 \\ 0 & 0 & 0 & J & 0\end{pmatrix}}} && && \scriptsize{\begin{pmatrix} x & y & B & D & I \\ A & C & 1 & z & 0 \end{pmatrix}}&\\
		\mc O \otimes V_0   & & \mc O(1)^2 \otimes V_0  && &&\mc O(1)^2 \otimes V_0 &&  &&\mc O  \otimes V_0 	\\
		\oplus&  \longrightarrow & \oplus && \longrightarrow && \oplus && \longrightarrow && \oplus \\
		\\
		\mc O(1) \otimes V_1  & & \mc O^2 \otimes V_1 &&& &    \mc O^2 \otimes V_1 &&&& \mc O(1) \otimes V_1 \\ 
		&   & \oplus &&  && \oplus &&  &&  \\
		& &\mc O(1) \otimes V_\infty && && \mc O \otimes V_\infty } \ .
\end{equation*}

\end{eg}

\begin{eg} Let $Y=X=\bb C^3$ and let $M=\mc O_\C[1]$ the shifted structure sheaf of the coordinate subspace $\C_x\subset \bb C^3$. Then we have
	\[ \ _\infty \Sigma ^1_0 = \bb K^2 \quad\quad  \ _0\Sigma^1_\infty = \bb K \quad\quad \ _\infty \Sigma^1_\infty = S_\infty^2 \]
so that the extended quiver $Q_M$ is given by
\[
\begin{tikzcd}
	\mathcircled{M} \arrow[r, bend left, shift left=0.4ex, "I"]  \arrow[out=115,in=155,loop,swap, "A"]  \arrow[out=205,in=245,loop,swap, "A"]& \arrow[l,   "J_2"]   \arrow[l, shift left=0.4ex, bend left,   "J_3"]   \mathcircled{F_0}\arrow[out=340,in=20,loop,swap,"B_3"]
	\arrow[out=70,in=110,loop,swap,"B_1"]
	\arrow[out=250,in=290,loop,swap,"\normalsize{B_2}"]
\end{tikzcd} \ .
\]
In particular, for the trivial framing structure $\f=0$, the corresponding framed quiver with potential $(Q_M^\f,W_M^\f)$ is given by
\[ Q_M^\f = \quad 
\begin{tikzcd}
	\boxed{V_\infty} \arrow[r, bend left, shift left=0.4ex, "I"] & \arrow[l,   "J"]   \arrow[l, shift left=0.4ex, bend left,   "J"]   \mathcircled{V_0}\arrow[out=340,in=20,loop,swap,"B_3"]
	\arrow[out=70,in=110,loop,swap,"B_1"]
	\arrow[out=250,in=290,loop,swap,"\normalsize{B_2}"]
\end{tikzcd} \quad\quad \text{and}\quad\quad W = B_1[B_2,B_3] + B_2IJ_2 + B_3 I J_3 \ .
\]
This quiver with potential is precisely the one introduced by Davison-Ricolfi \cite{DavRic} in their study of the local DT-PT correspondence.

There are also generalizations analogous to those of Example \ref{spikedeg}, with corresponding quivers given by
\[
\begin{tikzcd}
	\boxed{V_\infty^3} \arrow[out=70,in=110,loop, "\bullet" marking, swap,"C_1"]  \arrow[out=160,in=200,loop, "\bullet" marking, swap,"C_2"] \arrow[rd, bend left, "I_3"]  \arrow[rr,  "\bullet" marking, bend left, "A_2^3"]   & & \arrow["\bullet" marking]{ll}[swap]{A_3^2}  \arrow[ld, bend left,  "I_{2}"]   \boxed{V_\infty^2} \arrow[out=70,in=110,loop, "\bullet" marking, swap,"D_1"] \arrow[out=340,in=20,loop, "\bullet" marking, swap,"D_3"] \\
	& \arrow[lu, bend left, "J_3^1"] \arrow[ru, bend left, "J_2^1"]  \mathcircled{V_0} \ar[ul, "J_3^2"] \ar[ur,"J_2^3"] \arrow[out=310,in=350,loop,swap,"B_3"]
	\arrow[out=190,in=230,loop,swap,"B_1"]
	\arrow[out=250,in=290,loop,swap,"\normalsize{B_2}"]&
\end{tikzcd}
\]
where we have for simplicity only given a representative example, analogous to that of Equation \ref{spikedfrmeqn}, and we have omitted the calculation of the potential.
\end{eg}

\begin{eg}\label{DTmoduleeg} We can also consider the case where \[M=\mc O_{\bb C^3}[1]\oplus \mc O_{\bb C_x}\oplus \mc O_{\C_y} \oplus \mc O_{\C_z} \ , \] using the $t$-structure of Example \ref{semidynamicalteg}. It is a straightforward generalization of the above computations to check that the resulting extended quiver is essentially that of \cite{Jaf}, which was apparently computed to describe the DT topological vertex. In particular, the framing vertex has four components
\[ V_\infty = V_\infty^0 \oplus V_\infty^x\oplus V_\infty^y \oplus V_\infty^z  \,\]
corresponding to the four summands of $M$ above.

In this example the moduli stack of framed perverse coherent extensions, rather than parameterizing maps from $\mc O_{\C^3}$ to an iterated extension of $\mc O_\pt$, parameterizes maps from $\mc O_{\C^3}$ to an iterated extension of $\mc O_\pt$ together with a fixed underlying iterated extension of $\mc O_{\C_x}$, $\mc O_{\C_y}$, and $\mc O_{\C_z} $.

Indeed, we conjecture that for appropriate stability conditions and framing structure, the fixed points of the resulting moduli space are indexed by the set of plane partitions with fixed asymptotics $\alpha,\beta,\gamma$, where each of the partitions is determined by the framing structure at the framing vertex $i$, given by the pair $A^i_1,A^i_2$ of fixed endomorphisms of $V_\infty^i$, for $i=x,y,z$ respectively. As we will see, this is closely related to their occurrence in the construction of \cite{FJMM}.
\end{eg}

\begin{eg}\label{VWmoduleeg} Similarly, we can also consider the analogous computation in the setting of Example \ref{ADHMeg}, taking \[M=\mc O_{\bb C_{xy}^2}[1]\oplus\mc O_{\bb C_{xz}^2}[1]\oplus \mc O_{\bb C_{yz}^2}[1]\oplus  \mc O_{\bb C_x}\oplus \mc O_{\C_y} \oplus \mc O_{\C_z} \ , \] and again using the $t$-structure of Example \ref{semidynamicalteg}. In particular, the framing vertex has six components
\[ V_\infty = V_\infty^{xy} \oplus V_\infty^{xz} \oplus V_\infty^{yz} \oplus V_\infty^x\oplus V_\infty^y \oplus V_\infty^z  \,\]
corresponding to the summands of $M$ above. This gives rise to the following apparently new framed quiver with potential
\begin{equation}\label{VWmodquivereqn}
	 Q_M^\f =	
\begin{tikzcd}
	& \boxed{V_\infty^x} \ar[dl,shift right = .5ex, "J_x"] \ar[d]  \arrow[d, shift left=0.5ex, "I^x_1 I^x_2"]	\arrow[out=160,in=200,loop,swap,"\bullet" marking, "A^x_1 A^x_2"] \arrow[out=160,in=200,loop,swap,"\bullet" marking,looseness=5, shift right = 1ex]  \\
	\boxed{V_\infty^{xy}} \ar[ur, shift left = 1ex,"I_x"] \ar[dr, shift left = - 1ex,"I_y"]	\arrow[out=160,in=200,loop,swap,"\bullet" marking, "A^{xy}"]  \arrow[r, shift left=0.5ex, "I"]  &  \mathcircled{V} \ar[u, bend left, swap, shift left = 1ex,"J^x" ]  \ar[d, bend right,shift right = 1ex, "J^y" ]  \arrow[l, shift left=0.5ex,"J"]   \arrow[out=340,in=20,loop,swap, "B_1 B_2 B_3"] \arrow[out=340,in=20,loop,swap,looseness=5, shift right = 1ex]  \arrow[out=340,in=20,loop,swap,looseness=3, shift right = .5ex]  \\ 
	&	\boxed{V_\infty^y}\ar[u] \ar[ul,shift left = .5ex, "J_y"]\arrow[u, shift left=0.5ex, swap,"I^x_1 I^x_2"]	\arrow[out=160,in=200,loop,swap,"\bullet" marking, "A^y_1 A^y_2"] 	\arrow[out=160,in=200,loop,swap,"\bullet" marking,looseness=5, shift right = 1ex]
\end{tikzcd}
\quad \quad\quad\text{and}\quad\quad 
\end{equation}
which we have drawn in the case $V_\infty^{xz} =V_\infty^{xz}=V_\infty^z=0$, for simplicity.

Again, in this example the moduli stack of framed perverse coherent extensions parameterizes, rather than just maps from a fixed iterated extension of $\mc O_{\C^2}$ to an iterated extension of $\mc O_\pt$, maps from the same source to an iterated extension of $\mc O_\pt$ together with a fixed underlying iterated extension of $\mc O_{\C_x}$, $\mc O_{\C_y}$, and $\mc O_{\C_z} $, determined by the framing structure.

Indeed, generalizing the previous example, we conjecture that for appropriate stability conditions, the fixed points of the resulting moduli space are indexed by the set of plane partitions with fixed asymptotics $\alpha,\beta,\gamma$, and a pit at location $(m,n)$, in the sense of \cite{BerFM}, where the partitions $\alpha,\beta,\gamma$ are those of length $\dim V_\infty^i$ determined by the pair $A^i_1,A^i_2$ of endomorphisms of $V_\infty^i$ given by the framing structure at the framing vertex $i$, for $i=x,y,z$ respectively, $\dim V_\infty^{xz}=m$, $\dim V_\infty^{yz}=n$, and $V_\infty^{xy}=0$, and $A^{xz}$ and $A^{yz}$ are given by principal nilpotents in $\gl_m$ and $\gl_n$, respectively. For $n=0$, these appear to be the cohomological variants of the modules constructed in \cite{FFJMM}.
\end{eg}

\section{Representations of cohomological Hall algebras from perverse coherent extensions}\label{cohareptotalsec}

\subsection{Overview of Section \ref{cohareptotalsec}}\label{coharepintrosec} In Section \ref{cohareptotalsec}, we recall the construction of the Kontsevich-Soibelman cohomological Hall algebra from \cite{KS1}, and prove Theorem \ref{Bthm} from the introduction, following the approach of \cite{Soi}. In Section \ref{vanishingbasicsec} we recall some necessary facts about sheaves of vanishing cycles, and in Section \ref{cohasec} we recall the construction of the cohomological Hall algebra. In Section \ref{coharepsec} we prove Theorem \ref{Bthm}, and in Section \ref{invtsec} we explain the interpretation of the vector spaces underlying these representations as cohomological enumerative invariants. In Section \ref{quiveryangsec}, we outline an extension of the construction of the representation of Theorem \ref{Bthm} to a representation of a larger algebra which appears to agree with the shifted quiver Yangian of \cite{GLY}.

\subsection{Vanishing cycles sheaves and cohomology}\label{vanishingbasicsec} In this section, we review the formalism of nearby and vanishing cycles, and recall some of their key properties that we use in section.

Let $X$ be a smooth variety, $W:X\to \bb A^1$ a regular function, and $X_0=X\times_{\bb A^1} \{0\}$, $U=X\times_{\bb A^1}\bb G_m$, $\tilde{\bb G}_m$ the universal cover of $\bb G_m$, and $\tilde U = U\times_{\bb G_m} \tilde{\bb G}_m$, so we have that
\begin{equation}
	\xymatrix{ X_0 \ar[r]^{\iota_0} \ar[d] & X \ar[d]^{W} & U  \ar[l]_{j}\ar[d] & \ar[l]_{\tilde \pi} \tilde U \ar[d] \\
		\{0\} \ar[r] & \bb A^1 & \ar[l] \bb G_m & \ar[l] \tilde{\bb G}_m & , } 
\end{equation}
is commutative, with all squares Cartesian. The nearby cycles functor is defined as
\begin{equation*}
	\psi_W:= \iota_0^* j_* \tilde\pi_*\tilde \pi^* : \DC(U) \to \DC(X_0) \ .
\end{equation*}

\noindent Note that the unit of the $(\pi^*,\pi_*)$ adjunction for $\pi=j \circ \tilde \pi: \tilde U \to X$ defines a natural transformation
\[ \iota^*A \to \psi_W  j^*A \] 
the cone on which extends to define the vanishing cycles functor
\[ \phi_{W,0}:\DC(X) \to \DC(X_0) \quad\quad\quad\quad  A \mapsto \phi_{W,0}(A)\cong \textup{Cone}\left[\iota^*A \to \psi_W j^*A\right]  \ . \]
In fact, it is convenient to consider the cohomological shifts $\psi_W^p=\psi_W[-1], \phi_{W,0}^p=\phi_{W,0}[-1]$ as these both preserve the hearts of the respective perverse t-structures, that is, they restrict to functors
\[ \psi_W^{p,0}:\Per(U) \to \Per(X_0) \quad\quad \text{and}\quad\quad \phi_{W,0}^{p,0}: \Per(X) \to \Per(X_0) \ .\]

Let $Z=\textup{Crit}(W):=\Gamma_{dW}\times_{T^\vee X} X  \to X$ denote the (classical) critical locus of $W$, that is, the (classical) vanishing locus of $dW\in \Omega^1(X)$. Note $W$ is constant on each of the finitely many irreducible components of $Z$, so that $f(Z)\subset \bb A^1$ is a finite set of points, and we have
\[ Z=\bigsqcup_{c\in f(Z)} Z_c \quad\quad \text{where} \quad\quad Z_c = Z\times_{\bb A^1} \{c\}  \] 
so that we have the following commutative diagram, with all squares Cartesian:
\[  \xymatrix{ Z_c \ar[r]^{i_c } \ar[d] & X_c \ar[r]\ar[d]^{\iota_c} & \{c\} \ar[d] \\ 
 Z \ar[r]^{i} & X \ar[r]^{W} & \bb A^1  & . }   \]

For each $c \in W(Z)\subset \bb A^1$, we consider the corresponding shifted vanishing cycles functor for the function $W-c$, denoted $\phi_{W,c}^p:= \phi_{W-c,0}^p: \DD(X)\to \DD(X_c) $, and define
\[\phi_{W,c}^{p,0}: \Per(X) \to \Per(X_c)  \quad\quad \text{and}\quad\quad \varphi_{W,c} = \phi_{W,c}^{p,0}( \underline{\bb Q}_X[\dim X]) \ \in \Per(X_c) \ . \]
In fact, the support of $ \varphi_{W,c}$ is contained in $Z_c \subset X_c$, and thus we also define the internal variant of the $c$ component of the vanishing cycle sheaf
\[ \tilde \varphi_{W,c} = i_c^*  \varphi_{W,c} \in \Per(Z_c) \]
as the image of $\varphi_{W,c}$ under the equivalence $\Per(X_c)_{Z_c}\cong \Per(Z_c)$. Finally, we define the total intrinsic variant of the vanishing cycle sheaf
\[ \tilde \varphi_W := \bigoplus_{c\in W(Z)} \tilde \varphi_{W,c} \quad \in\quad \Per(Z)=\bigoplus_{c\in W(Z)} \Per(Z_c)  \ , \]
and similarly its extrinsic analogue, which will be the primary object of interest
\[ \varphi_W:= \iota_* \tilde \varphi_W \ \in \Perv(X) \ .\]
More generally, we define the total extrinsic vanishing cycles functor
\[ \phi_W = \bigoplus_{c\in W(Z)} \iota_{c *} \circ \phi^p_{W,c}: \DD(X) \to \DD(X) \ , \]
which in particular reproduces the preceding object as $\varphi_W = \phi_W( \underline{\bb Q}_X [\dim X])$.

We define the compactly supported vanishing cycles cohomology of $X$, and its perverse shifted analogue, by
\[ H^\bullet_c(X,\phi_W\underline{\bb Q}_X ) = p_! \phi_W p^* \bb Q  \quad\quad\quad\quad H^\bullet_c(X,\varphi_W) = p_! \phi_W p^* \bb Q[\dim X]     \ ,  \]
where $p:X\to \pt$. It satisfies the following functoriality properties:

\begin{prop} Let $W:Y\to \bb A^1$ and $f:X\to Y$ a proper map of smooth varieties. Then $f$ induces a pullback map
\begin{equation}\label{pullbackeqn}
	 f^*: H_c^\bullet(Y,\phi_W \underline{\bb Q}_Y) \to H^\bullet_c(X,\phi_{W\circ f} \underline{\bb Q}_X)  \ . 
\end{equation}
	Similarly, a smooth map $f:X\to Y$ of smooth varieties induces a pushforward map
\begin{equation}\label{pushfwdeqn}
	 f_*:   H^\bullet_c(X,\phi_{W\circ f} \underline{\bb Q}_X) \to H_c^\bullet(Y,\phi_W \underline{\bb Q}_Y)[-2d]    \ , 
\end{equation}
	where $d=\dim X-\dim Y$ denotes the relative dimension of $f$.
\end{prop}
\begin{proof} Recall that for any map of smooth varieties, there is a natural transformation
	\[ \phi_W \xrightarrow{} f_* \phi_{W\circ f} f^* \ ;  \]
 see for example Equation 20 in \cite{Dav1}. This induces the desired map as
	\[ p_{Y!} \phi_W p_Y^*\bb Q \to  p_{Y !} f_* \phi_{W\circ f} f^* p_{Y}^* \bb Q \cong p_{X !}\phi_{W\circ f}p_X^*\bb Q \ ,   \]
where we have used properness of $f$ to identify $f_*=f_!$.

Similarly, for $f$ a smooth map of smooth varieties, the induced natural transformation
\[   f^* \phi_W \xrightarrow{\cong}\phi_{W\circ f} f^* \]
is an isomorphism; see for example Equation 23 in \emph{loc. cit.}. This induces the desired map as
\[ p_{X!} \phi_{W\circ f}p_X^* \bb Q \cong p_{Y !} f_! \phi_{W\circ f} f^* p_Y^* \bb Q \xrightarrow{\cong}  p_{Y !}f_!f^*  \phi_W p_Y^* \bb Q  \to p_{Y !} \phi_W p_Y^* \bb Q[-2d] \]
where we have used the counit $f_!f^! \to \id$ of the $(f_!,f^!)$ adjunction, as well as the identification $f^*= f^![-2d]$ which follows from smoothness of $f$.
\end{proof}

In particular, we also obtain:

\begin{corollary} For $f:X\to Y$ an affine fibration of relative dimension $d$, the induced pushforward map of Equation \ref{pushfwdeqn} is an isomorphism, so that there is an inverse isomorphism
\begin{equation}\label{pushfwdinveqn}
	 f_*^{-1}: H_c^\bullet(Y,\phi_W \underline{\bb Q}_Y) \xrightarrow{\cong} H^\bullet_c(X,\phi_{W\circ f}\underline{\bb Q}_X) [2d]  \ . 
\end{equation}
\end{corollary}
\begin{proof} If $f$ is an affine fibration, the counit of the $(f_!,f^!)$ adjunction is an isomorphism, and thus by the proof of the preceding proposition the induced map is an isomorphism.
\end{proof}

\subsection{Cohomological Hall algebras of quivers with potential and Calabi-Yau threefolds}\label{cohasec}

Throughout this section, let $(Q,W)$ be a quiver with potential, and recall from Section \ref{quiversecunfr} the moduli stacks of representations of the underlying quiver
\[ \mf M_\dd(Q) = [X_\dd(Q)/ G_\dd(Q)]  \quad\quad \text{where}\quad\quad  X_\dd(Q)= \prod_{e \in E_Q} \Hom( \K^{d_{s(e)}}, \K^{d_{t(e)}})  \quad\quad G_\dd(Q)=  \prod_{i\in V_Q} \Gl_{d_i}(\K) \]
for each dimension vector $\dd\in \bb N^{V_Q}$; for simplicity, we will often omit $Q$ from the notation when there is only a single quiver under consideration, as in the remainder of this subsection.

Given $\aaa, \bbb \in \bb N^{V_Q}$, we define
\begin{equation*}
\mf M_{\aaa,\bbb} = [X_{\aaa,\bbb}/G_{\aaa,\bbb}] \quad\quad \text{where}\quad\quad \begin{cases} X_{\aaa,\bbb} = \big\{ \varphi\in X_{\aaa+\bbb} \ \big|\ \varphi\big(\bigoplus_{e\in E_Q} \bb K^{a_{s(e)} }\big) \subset \bigoplus_{e\in E_Q} \bb K^{a_{t(e)} } \big\} \\
	G_{\aaa,\bbb} = \big\{ g \in G_{\aaa+\bbb} \ \big| \ g\big(\bigoplus_{i\in V_Q} \K^{a_i} \big) \subset \bigoplus_{i\in V_Q} \K^{a_i+b_i}  \big\} \ .
\end{cases}
\end{equation*}
Equivalently, these spaces are defined component-wise as
\[X_{\aaa,\bbb} = \prod_{e\in E_Q} \Hom_{a_{s(e)},a_{t(e)}}(\K^{a_{s(e)}+b_{s(e)}} ,\K^{a_{t(e)}+b_{t(e)}} ) \quad\quad \textup{and}\quad\quad G_{\aaa,\bbb} = \prod_{i\in V_Q} \Gl_{a_i,b_i}(\K) \]
where $\Hom_{m_1,m_2}(\K^{m_1+n_1}, \K^{m_2+n_2}) $ denotes the space of block upper triangular linear maps, that is, satisfying $\varphi(\K^{m_1}) \subset \K^{m_2}$, and $ \Gl_{m,n}(\K) \subset \Gl_{m+n}(\K)$ is the parabolic subgroup preserving the flag $\{0\} \subset \K^{m} \subset \K^{m+n}$.

Geometrically, $\mf M_{\aaa,\bbb}$ is the moduli stack of short exact sequences of representations of the quiver $Q$ with sub object and quotient object of dimension $\aaa$ and $\bbb$, respectively. In particular, note that we have a closed embedding $X_{\aaa,\bbb}\to X_{\aaa+\bbb}$ by definition, as well as a smooth map $X_{\aaa,\bbb}\to X_\aaa\times X_\bbb$ given by forgetting the extension data. These maps define a correspondence of algebraic $\bb K$-stacks
\[ \xymatrix{ & \mf M_{\aaa,\bbb} \ar[dl] \ar[dr] \\ \mf M_{\aaa}\times \mf M_\bbb & & \mf M_{\aaa+\bbb} } \ , \]
inducing an associative product on the Borel-Moore homology of $\mf M$ with coefficients in the sheaf of vanishing cycles determined by $W$, as we now explain.

We recall the definition of the cohomological Hall algebra associated to the quiver with potential $(Q,W)$ following \cite{KS1}. The underlying graded vector space is given by
\[ \mc H(Q,W)=\bigoplus_{\dd\in \bb N^{V_Q}} \mc H_\dd(Q,W) \quad\quad \mc H_\dd(Q,W) = H_\bullet^{G_\dd(Q)}( X_\dd(Q),\varphi_{W_\dd })  \]
where $H_\bullet^G(X,\varphi):=H^\bullet_{G,c}(X,\varphi)^\vee$ denotes the equivariant Borel-Moore homology with coefficients in the sheaf of vanishing cycles, defined as the dual of the compactly supported cohomology with coefficients in the sheaf of vanishing cycles, and we recall that the natural perverse grading shift in the equivariant case is given by
\[ H^\bullet_{G,c}(X,\varphi)= H^\bullet_{G,c}(X, \phi_W \underline{\bb Q}_X ) [\dim X - \dim G]  \ . \]
Again, for simplicity, we will often omit $(Q,W)$ from the notation when there is only a single quiver with potential under consideration.

The cohomological Hall algebra is not a graded algebra in the usual sense: rather than being an associative algebra object internal to the category $\Vect_{\bb N^{V_Q}}$ of $\bb N^{V_Q}$-graded complexes with its usual monoidal structure, it is an associative algebra in this category with respect to the usual monoidal structure twisted by a cohomological degree shift, defined by
\[( \bigoplus_{\aaa \in \bb N^{V_Q}} V_\aaa )\otimes^\tw (\bigoplus_{\bbb \in \bb N^{V_Q}} W_\bbb) := \bigoplus_{\dd\in \bb N^{V_Q}} \bigoplus_{\aaa+\bbb=\dd} V_\aaa\otimes W_\bbb [\chi(\bbb,\aaa)-\chi(\aaa,\bbb)]  \ ,\]
as explained in Section 2.7 of \cite{KS1} or 3.1 of \cite{Dav1}, where we have used the bilinear form $\chi:\bb Z^{V_Q}\times \bb Z^{V_Q}\to \Z$ defined by
\[ \chi(\aaa,\bbb) = \sum_{i\in V_Q} a_i b_i - \sum_{e\in E_Q} a_{s(e)}b_{t(e)} \ . \]
In particular, note that $\chi(\aaa,\aaa)=\dim G_\aaa-\dim X_\aaa$ and similarly
\begin{equation}\label{cohaproddegeqn}
	 \chi(\aaa,\bbb)=\dim G_{\aaa,\bbb}-\dim G_\aaa\times G_\bbb-\dim X_{\aaa,\bbb} + \dim X_{\aaa}\times X_\bbb \ . 
\end{equation}

We now recall the construction of the cohomological Hall algebra of a quiver with potential introduced by Kontsevich-Soibelman in \cite{KS1}.

\begin{theo}\cite{KS1}\label{cohaprodthm} Let $(Q,W)$ be a quiver with potential. The Kontsevich-Soibelman cohomological Hall algebra \[\mc H(Q,W)\in \Alg_\Ass(\Vect_{\bb N^{V_Q}}^{\otimes^\tw}) \ ,\] that is, $\mc H(Q,W)$ admits a canonical (twisted $\bb N^{ V_Q}$-graded) associative algebra structure.
\end{theo}
\begin{proof}
	
The associative product map $m:\mc H\otimes^\tw \mc H \to \mc H$ on the cohomological Hall algebra $\mc H=\mc H(Q,W)$ is defined in components by
\[  m_{\aaa,\bbb}:\mc H_\aaa\otimes \mc H_\bbb \to \mc H_{\aaa+\bbb}[\chi(\aaa,\bbb)-\chi(\bbb,\aaa) ]  \ .\]
The components $m_{\aaa,\bbb}$ will be constructed in terms of the equivalent dual maps
\[ m_{\aaa,\bbb}^\vee:  H^\bullet_{G_{\aaa+\bbb},c}(X_{\aaa+\bbb},\varphi_{W_{\aaa+\bbb}} )  \to H^\bullet_{G_\aaa,c}(X_{\aaa},\varphi_{W_\aaa})\otimes H^\bullet_{G_\bbb,c}(X_{\bbb},\varphi_{W_\bbb}) [\chi(\aaa,\bbb)-\chi(\bbb,\aaa)] \ , \]
or after accounting explicitly for the perverse shifts,
\begin{equation}\label{cohaprodmapeqn}
m_{\aaa,\bbb}^\vee: H^\bullet_{G_{\aaa+\bbb},c}(X_{\aaa+\bbb},\phi_{W_{\aaa+\bbb}} \underline{\bb Q}_{X_{\aaa+\bbb}} ) \to H^\bullet_{G_\aaa,c}(X_{\aaa},\phi_{W_{\aaa}} \underline{\bb Q}_{X_\aaa})\otimes H^\bullet_{G_\bbb,c}(X_{\bbb},\phi_{W_{\bbb}} \underline{\bb Q}_{X_\bbb})[  2\chi(\aaa,\bbb) ] \ ,
\end{equation}
where the claimed cohomological degree shift follows from the equality
\[ \chi(\aaa+\bbb,\aaa+\bbb) - \chi(\aaa,\aaa) - \chi(\bbb,\bbb) +\chi(\aaa,\bbb)-\chi(\bbb,\aaa)  = 2\chi(\aaa,\bbb)  \ .\]
We now proceed with the construction of the desired maps:

There is a restriction of equivariance map
\[ H^\bullet_{G_{\aaa+\bbb},c}(X_{\aaa+\bbb},\phi_{W_{\aaa+\bbb}} \underline{\bb Q}_{X_{\aaa+\bbb}} ) \to  H^\bullet_{G_{\aaa,\bbb},c}(X_{\aaa+\bbb},\phi_{W_{\aaa+\bbb}} \underline{\bb Q}_{X_{\aaa+\bbb}} ) \]
given by the proper pullback map of Equation \ref{pullbackeqn} induced by the map
\[ X_{\aaa+\bbb}/ G_{\aaa,\bbb} \to  X_{\aaa+\bbb} / G_{\aaa+\bbb} \]
which is proper and schematic, as its fibres are isomorphic to the partial flag variety $G_{\aaa+\bbb} /G_{\aaa,\bbb} $. Similarly, pullback along the closed embedding $X_{\aaa,\bbb}\to X_{\aaa+\bbb}$ gives a map
\[H^\bullet_{G_{\aaa,\bbb},c}(X_{\aaa+\bbb},\phi_{W_{\aaa+\bbb}} \underline{\bb Q}_{X_{\aaa+\bbb}} ) \to H^\bullet_{G_{\aaa,\bbb},c}(X_{\aaa,\bbb},\phi_{W_{\aaa,\bbb}} \underline{\bb Q}_{X_{\aaa,\bbb}} ) \ . \]

Next, there is a further restriction of equivariance map
\[ H^\bullet_{G_{\aaa,\bbb},c}(X_{\aaa,\bbb},\phi_{W_{\aaa,\bbb}} \underline{\bb Q}_{X_{\aaa,\bbb}} ) \to H^\bullet_{G_{\aaa}\times G_\bbb,c}(X_{\aaa,\bbb},\phi_{W_{\aaa,\bbb}} \underline{\bb Q}_{X_{\aaa,\bbb}} )[2(\dim G_{\aaa,\bbb} - \dim G_\aaa\times G_\bbb ) ]  \]
given by the inverse pushforward map, as in Equation \ref{pushfwdinveqn}, induced by the affine fibration
\[ X_{\aaa,\bbb}/(G_\aaa\times G_\bbb) \to X_{\aaa,\bbb}/ G_{\aaa,\bbb}   \]
with fibres given by the unipotent group $G_{\aaa,\bbb}/(G_\aaa\times G_\bbb)$. There is also a pushforward map
\[ H^\bullet_{G_{\aaa}\times G_\bbb,c}(X_{\aaa,\bbb},\phi_{W_{\aaa,\bbb}} \underline{\bb Q}_{X_{\aaa,\bbb}} ) \to H^\bullet_{G_{\aaa}\times G_\bbb,c}(X_{\aaa} \times X_\bbb, \phi_{W_{\aaa}\boxtimes W_\bbb} \underline{\bb Q}_{X_{\aaa} \times X_\bbb} )[-2(\dim X_{\aaa,\bbb} - \dim X_\aaa\times X_\bbb ) ]  \]
induced as in Equation \ref{pushfwdeqn} by the smooth map
 \[ X_{\aaa,\bbb} / (G_\aaa \times G_\bbb) \to (X_{\aaa}\times X_\bbb) / (G_\aaa\times G_\bbb ) \ .\]
Composing the above sequence of maps with the Thom-Sebastiani isomorphism
\[H^\bullet_{G_{\aaa}\times G_\bbb,c}(X_{\aaa} \times X_\bbb, \phi_{W_{\aaa}\boxtimes W_\bbb} \underline{\bb Q}_{X_{\aaa} \times X_\bbb} ) \cong H^\bullet_{G_\aaa,c}(X_{\aaa},\phi_{W_{\aaa}} \underline{\bb Q}_{X_\aaa})\otimes H^\bullet_{G_\bbb,c}(X_{\bbb},\phi_{W_{\bbb}} \underline{\bb Q}_{X_\bbb}) \ ,  \]
gives the desired map of Equation \ref{cohaprodmapeqn}, where the correct shift follows from the equality of Equation \ref{cohaproddegeqn}. Associativity of the above product follows from a standard argument; see for example Section 2.3 of \cite{KS1}.

\end{proof}

In particular, we make the following definition:

\begin{defn} Let $Y\xrightarrow{\pi} X$ be a resolution of singularities of a toric Calabi-Yau threefold satisfying the hypotheses of Section \ref{Geosetupsec}. The cohomological Hall algebra of $Y$ is defined by 
\[ \mc H(Y) := \mc H(Q_Y,W_Y)  \  \in \Alg_\Ass(\Vect_{\bb N^{V_Q}}^{\otimes^\tw}) \ , \]
as constructed in the preceding Proposition, where $(Q_Y,W_Y)$ is the quiver with potential associated to $Y$ as in Theorem \ref{stackthmunext}.
\end{defn}

\subsection{Representations of cohomological Hall algebras from perverse coherent extensions}\label{coharepsec}

In this section, we prove Theorem \ref{Bthm} from Section \ref{rtintrosec} of the introduction, following the results outlined in Section 4 of \cite{Soi}.

Let $M\in \Perv^\p(Y)^T$ satisfy the hypotheses of Section \ref{Extoverviewsec} with $Q_M$ be the associated extended quiver, and let $\f$ be a framing structure for $M$ of rank $d_\infty$ with $(Q_M^\f,W_M^\f)$ the associated framed quiver with potential, as in Theorem \ref{stackpotthm}. Recall that by the construction in \emph{loc. cit.} we have
\[  X_\dd(Q_M^\f) = \bigoplus_{i,j \in V_Q}  \ _i \Sigma^1_j \otimes \Hom(\bb K^{d_i}, \bb K^{d_j})  \oplus \bigoplus_{i \in V_Q} \ _i\Sigma^1_\infty \otimes \Hom(\bb K^{d_i},\bb K^{d_\infty})  \oplus  \bigoplus_{j \in V_Q} \ _\infty\Sigma^1_j \otimes \Hom(\bb K^{d_\infty}, \bb K^{d_j}) \]
for each $\dd \in \bb N^{V_{Q_Y}}$, as well as the closed subvariety
\[ Z_\dd(Q_M^\f) = \Crit(W^\f_{M,\dd}) \subset  X_\dd(Q_M^\f) \quad\quad \text{invariant under} \quad\quad G_\dd(Q_Y)= \prod_{i\in V_{Q_Y}} \Gl_{d_i}(\bb K)  \ .\]
Throughout this section, we will drop the dependence on the choice of quiver $Q_Y$ and extended quiver $Q_M$ from the notation, and write simply
\[ X^\f_\dd=X_\dd(Q_M^\f) \quad\quad X_\dd = X_\dd(Q_Y) \quad\quad\textup{and} \quad\quad  G_\dd=G_\dd(Q_Y) \ . \]
Further, we fix a choice of stability condition $\zeta \in \bb R^{V_{Q}}$ for the framed quiver with potential $(Q_M^\f, W_M^\f)$ and let $X^{\f,\zeta}_\dd \subset X^\f_\dd$ denote the open subvariety of $\zeta$-stable points.

As in the preceding Section \ref{cohasec}, for $\aaa, \bbb \in \bb N^{V_Q}$, we let $\aaa^0=(\aaa,0), \bbb^\f = (\bbb,d_\infty) \in \bb N^{V_{Q_M^\f}}$ and define
\begin{equation*}
	\mf M^\f_{\aaa,\bbb} = [X^\f_{\aaa,\bbb}/G_{\aaa,\bbb}] \quad\quad \text{where}\quad\quad \begin{cases} X_{\aaa,\bbb}^\f = \big\{ \varphi\in X^\f_{\aaa+\bbb} \ \big|\ \varphi\big(\bigoplus_{e\in E_{Q_M^\f}} \bb K^{a^0_{s(e)} }\big) \subset \bigoplus_{e\in E_{Q_M^\f}} \bb K^{a^0_{t(e)} } \big\} \\
		G_{\aaa,\bbb} = \big\{ g \in G_{\aaa+\bbb} \ \big| \ g\big(\bigoplus_{i\in V_{Q_Y}} \K^{a_i} \big) \subset \bigoplus_{i\in V_{Q_Y}} \K^{a_i}  \big\} 
	\end{cases}\ .
\end{equation*}
Equivalently, these spaces are defined component-wise as
\[X^\f_{\aaa,\bbb} = \prod_{e\in E_{Q_M^\f}} \Hom_{a^0_{s(e)},a^0_{t(e)}}(\K^{a^0_{s(e)}+b^\f_{s(e)}} ,\K^{a^0_{t(e)}+b^\f_{t(e)}} ) \quad\quad \textup{and}\quad\quad G_{\aaa,\bbb} = \prod_{i\in V_{Q_Y}} \Gl_{a_i,b_i}(\K) \]
where $\Hom_{m_1,m_2}(\K^{m_1+n_1}, \K^{m_2+n_2}) $ denotes the space of block upper triangular linear maps, that is, satisfying $\varphi(\K^{m_1}) \subset \K^{m_2}$, and $ \Gl_{m,n}(\K) \subset \Gl_{m+n}(\K)$ is the parabolic subgroup preserving the flag $\{0\} \subset \K^{m} \subset \K^{m+n}$.

Geometrically, $\mf M^\f_{\aaa,\bbb}$ is the moduli stack of short exact sequences of representations of the framed quiver $Q_M^\f$ with sub object and quotient object of dimension $\aaa^0$ and $\bbb^\f$, respectively. Note that since the sub object has framing dimension $0$ and thus is equivalent to a representation of the unframed quiver $Q_Y$. In particular, note that we again have a closed embedding $X^\f_{\aaa,\bbb}\to X^\f_{\aaa+\bbb}$ by definition, as well as a smooth map $X^\f_{\aaa,\bbb}\to X_\aaa\times X^\f_\bbb$ given by forgetting the extension data. These maps define a correspondence of algebraic $\bb K$-stacks
\[ \xymatrix{ & \mf M^\f_{\aaa,\bbb} \ar[dl] \ar[dr] \\ \mf M_{\aaa}\times \mf M^\f_\bbb & & \mf M^\f_{\aaa+\bbb} } \ . \]
Analogously to the previous section, this correspondence will induce the desired module structure on the Borel-Moore homology of $\mf M^{\f,\zeta}$ with coefficients in the sheaf of vanishing cycles determined by $W^\f$, but we must modify the correspondence to ensure it preserves the moduli spaces of $\zeta$-stable representations.

We define subschemes $X_{\aaa,\bbb}^{\f,\zeta_\bbb} , X_{\aaa,\bbb}^{\f,\zeta_{\aaa+\bbb}} , X_{\aaa,\bbb}^{\f,\zeta} \subset X_{\aaa,\bbb}^\f$ by requiring that the diagrams
\[\xymatrix{X_{\aaa,\bbb}^{\f,\zeta_\bbb} \ar[r] \ar[d] & X_\bbb^{\f,\zeta } \ar[d] \\ X_{\aaa,\bbb}^\f \ar[r]  & X^\f_\bbb  } \quad\quad \xymatrix{X_{\aaa,\bbb}^{\f,\zeta_{\aaa+\bbb}} \ar[r] \ar[d] & X_{\aaa+\bbb}^{\f,\zeta } \ar[d] \\ X_{\aaa,\bbb}^\f \ar[r]  & X^\f_{\aaa+\bbb} } \quad\quad
\xymatrix{X_{\aaa,\bbb}^{\f,\zeta} \ar[r] \ar[d] & X_\bbb^{\f,\zeta }  \times X_{\aaa+\bbb}^{\f,\zeta }\ar[d] \\ X_{\aaa,\bbb}^\f \ar[r]  &  X^\f_\bbb \times X^\f_{\aaa+\bbb} } 
   \]
are Cartesian. Note that the subschemes $X_{\aaa,\bbb}^{\f,\zeta_\bbb} , X_{\aaa,\bbb}^{\f,\zeta_{\aaa+\bbb}} , X_{\aaa,\bbb}^{\f,\zeta} \subset X_{\aaa,\bbb}^\f$ are all open, by base change, and we have that $X_{\aaa,\bbb}^{\f,\zeta} = X_{\aaa,\bbb}^{\f,\zeta_{\aaa+\bbb}}\times _{ X_{\aaa,\bbb}^\f }X_{\aaa,\bbb}^{\f,\zeta_\bbb}$. Similarly, the maps
\[ X_{\aaa,\bbb}^{\f,\zeta_\bbb} \to X_\aaa\times X_\bbb^{\f,\zeta} \quad\quad \text{and}\quad\quad  X_{\aaa,\bbb}^{\f,\zeta_{\aaa+\bbb}}  \to  X_{\aaa + \bbb}^{\f,\zeta} \]
are smooth and a closed embedding, respectively, by base change. In order to complete the construction, we will need to make an additional hypothesis on the stability condition $\zeta$:

\begin{defn} A stability condition $\zeta$ is called \emph{compatible} with $M$ if the composition
\[[ X_{\aaa,\bbb}^{\f,\zeta}/G_{\aaa,\bbb}] \to [X_{\aaa,\bbb}^{\f,\zeta_{\aaa+\bbb} }/G_{\aaa,\bbb}] \to [X_{\aaa + \bbb}^{\f,\zeta}/G_{\aaa,\bbb}]   \]
is proper for each $\aaa,\bbb \in \bb N^{V_{Q_Y}}$.
\end{defn}

We now construct the desired representations $\V=\V^{\f,\zeta}(M)$ defined by an object $M\in \DD^b\Coh(Y)^T$ together with a choice of framing structure $\f$ and a compatible stability condition $\zeta$. The underlying graded vector space is given by
\[ \V^{\f,\zeta}(M) = \bigoplus_{\dd\in \bb N^{V_Q}} \V^{\f,\zeta}_\dd(M)  \quad\quad \text{with}\quad\quad \V^{\f,\zeta}_\dd(M) =H_\bullet^{G_\dd(Q)}( X_\dd^{\zeta}(Q_M^\f), \varphi_{W_{M,\dd}^{\f}} )  \ , \]
where we recall that $H_\bullet^G(X,\varphi):=H^\bullet_{G,c}(X,\varphi)^\vee$ denotes the equivariant Borel-Moore homology with coefficients in the sheaf of vanishing cycles, defined as the dual of the compactly supported cohomology with coefficients in the sheaf of vanishing cycles, as in Section \ref{cohasec}. The main result of this section is the following, which closely follows the proof of Theorem \ref{cohaprodthm} above from \cite{KS1}:

Let $Y\to X$ be as in Section \ref{Geosetupsec}, $M\in \DD^b\Coh(Y)^T$ satisfy the hypotheses of Sections \ref{Extoverviewsec} and \ref{framingsec}, $\f \in Z_{d_\infty}(Q_M,R_M)(\bb K)$ be a framing structure on $M$, and $\zeta$ a compatible stability condition. Then we have the following result, which establishes Theorem \ref{Bthm} from the introduction:

\begin{theo}\label{cohareptheo} There exists a natural (twisted $\bb N^{V_Q}$-graded) representation \[ \rho_M: \mc H(Y) \to \End_{\Vect_{\bb N^{V_Q}}}(\V^{\f,\zeta}(M)) \] of the cohomological Hall algebra $\mc H(Y)$ on $\V^{\f,\zeta}(M)$.
\end{theo}
\begin{proof}

The module structure map $\rho:\mc H\otimes^\tw \V \to \V$ for $\V=\V^{\f,\zeta}(M)$ over the cohomological Hall algebra $\mc H=\mc H(Y)$ is defined in components by
\[  \rho_{\aaa,\bbb}:\mc H_\aaa\otimes \V_\bbb \to \V_{\aaa+\bbb}[\chi(\aaa,\bbb)-\chi(\bbb,\aaa) ]  \ .\] 
The components $\rho_{\aaa,\bbb}$ will be constructed in terms of the equivalent dual maps
\[ \rho_{\aaa,\bbb}^\vee:  H^\bullet_{G_{\aaa+\bbb},c}(X^{\f,\zeta}_{\aaa+\bbb},\varphi_{W^\f_{\aaa+\bbb}} )  \to H^\bullet_{G_\aaa,c}(X_{\aaa},\varphi_{W_\aaa})\otimes H^\bullet_{G_\bbb,c}(X^{\f,\zeta}_{\bbb},\varphi_{W^\f_\bbb}) [\chi(\aaa,\bbb)-\chi(\bbb,\aaa)] \ , \]
or after accounting explicitly for the perverse shifts,
\begin{equation}\label{coharepmapeqn}
	\rho_{\aaa,\bbb}^\vee: H^\bullet_{G_{\aaa+\bbb},c}(X^{\f,\zeta}_{\aaa+\bbb},\phi_{W^\f_{\aaa+\bbb}} \underline{\bb Q}_{X^{\f,\zeta}_{\aaa+\bbb}} ) \to H^\bullet_{G_\aaa,c}(X_{\aaa},\phi_{W_{\aaa}} \underline{\bb Q}_{X_\aaa})\otimes H^\bullet_{G_\bbb,c}(X^{\f,\zeta}_{\bbb},\phi_{W^\f_{\bbb}} \underline{\bb Q}_{X^{\f,\zeta}_\bbb})[  2\chi(\aaa,\bbb) ] \ ,
\end{equation}
where the claimed cohomological degree shift again follows from the equality
\[ \chi(\aaa+\bbb,\aaa+\bbb) - \chi(\aaa,\aaa) - \chi(\bbb,\bbb) +\chi(\aaa,\bbb)-\chi(\bbb,\aaa)  = 2\chi(\aaa,\bbb)  \ .\]
We now proceed with the construction of the desired maps:

There is again a restriction of equivariance map
\[ H^\bullet_{G_{\aaa+\bbb},c}(X^{\f,\zeta}_{\aaa+\bbb},\phi_{W^\f_{\aaa+\bbb}} \underline{\bb Q}_{X^{\f,\zeta}_{\aaa+\bbb}} ) \to  H^\bullet_{G_{\aaa,\bbb},c}(X^{\f,\zeta}_{\aaa+\bbb},\phi_{W^\f_{\aaa+\bbb}} \underline{\bb Q}_{X^{\f,\zeta}_{\aaa+\bbb}} ) \]
given by the proper pullback map of Equation \ref{pullbackeqn} induced by the map
\[ X_{\aaa+\bbb}/ G_{\aaa,\bbb} \to  X_{\aaa+\bbb} / G_{\aaa+\bbb} \]
which is proper and schematic, as its fibres are isomorphic to the partial flag variety $G_{\aaa+\bbb} /G_{\aaa,\bbb}$. Similarly, pullback along the map $[X^{\f,\zeta}_{\aaa,\bbb}/G_{\aaa,\bbb}]\to [X^{\f,\zeta}_{\aaa+\bbb}/G_{\aaa,\bbb}]$, which is proper by the compatibility of $\zeta$, gives a map
\[H^\bullet_{G_{\aaa,\bbb},c}(X^{\f,\zeta}_{\aaa+\bbb},\phi_{W^\f_{\aaa+\bbb}} \underline{\bb Q}_{X^{\f,\zeta}_{\aaa+\bbb}} ) \to H^\bullet_{G_{\aaa,\bbb},c}(X^{\f,\zeta}_{\aaa,\bbb},\phi_{W^\f_{\aaa,\bbb}} \underline{\bb Q}_{X^{\f,\zeta}_{\aaa,\bbb}} ) \ . \]

Next, there is again a further restriction of equivariance map
\[ H^\bullet_{G_{\aaa,\bbb},c}(X^{\f,\zeta}_{\aaa,\bbb},\phi_{W^\f_{\aaa,\bbb}} \underline{\bb Q}_{X^{\f,\zeta}_{\aaa,\bbb}} ) \to H^\bullet_{G_{\aaa}\times G_\bbb,c}(X^{\f,\zeta}_{\aaa,\bbb},\phi_{W^\f_{\aaa,\bbb}} \underline{\bb Q}_{X^{\f,\zeta}_{\aaa,\bbb}} )[2(\dim G_{\aaa,\bbb} - \dim G_\aaa\times G_\bbb ) ]  \]
given by the inverse pushforward map, as in Equation \ref{pushfwdinveqn}, induced by the affine fibration
\[ X^{\f,\zeta}_{\aaa,\bbb}/(G_\aaa\times G_\bbb) \to X^{\f,\zeta}_{\aaa,\bbb}/ G_{\aaa,\bbb}   \]
with fibres given by the unipotent group $G_{\aaa,\bbb}/(G_\aaa\times G_\bbb)$.

There is also a pushforward map
\[H^\bullet_{G_{\aaa}\times G_\bbb,c}(X^{\f,\zeta}_{\aaa,\bbb},\phi_{W^\f_{\aaa,\bbb}} \underline{\bb Q}_{X^{\f,\zeta}_{\aaa,\bbb}} ) \to H^\bullet_{G_{\aaa}\times G_\bbb,c}(X^{\f,\zeta_\bbb}_{\aaa,\bbb},\phi_{W^\f_{\aaa,\bbb}} \underline{\bb Q}_{X^{\f,\zeta_\bbb}_{\aaa,\bbb}} )  \]
induced as in Equation \ref{pushfwdeqn} by the map $X^{\f,\zeta}_{\aaa,\bbb} \to X^{\f,\zeta_\bbb}_{\aaa,\bbb}$ since it is an open immersion and thus smooth. Similarly, there is again a pushforward map
\[ H^\bullet_{G_{\aaa}\times G_\bbb,c}(X_{\aaa,\bbb}^{\f,\zeta_\bbb},\phi_{W^\f_{\aaa,\bbb}} \underline{\bb Q}_{X^{\f,\zeta_\bbb}_{\aaa,\bbb}} ) \to H^\bullet_{G_{\aaa}\times G_\bbb,c}(X_{\aaa} \times X^{\f,\zeta}_\bbb, \phi_{W_{\aaa}\boxtimes W^\f_\bbb} \underline{\bb Q}_{X_{\aaa} \times X^{\f,\zeta_\bbb}_\bbb} )[-2(\dim X^\f_{\aaa,\bbb} - \dim X_\aaa\times X^\f_\bbb ) ]  \]
induced by the smooth map
\[ X^{\f,\zeta_\bbb}_{\aaa,\bbb} / (G_\aaa \times G_\bbb) \to (X_{\aaa}\times X^{\f,\zeta}_\bbb) / (G_\aaa\times G_\bbb ) \ .\]
Composing the above sequence of maps with the Thom-Sebastiani isomorphism
\[H^\bullet_{G_{\aaa}\times G_\bbb,c}(X_{\aaa} \times X^{\f,\zeta}_\bbb, \phi_{W_{\aaa}\boxtimes W^\f_\bbb} \underline{\bb Q}_{X_{\aaa} \times X^{\f,\zeta}_\bbb} ) \cong H^\bullet_{G_\aaa,c}(X_{\aaa},\phi_{W_{\aaa}} \underline{\bb Q}_{X_\aaa})\otimes H^\bullet_{G_\bbb,c}(X^{\f,\zeta}_{\bbb},\phi_{W^\f_{\bbb}} \underline{\bb Q}_{X^{\f,\zeta}_\bbb}) \ ,  \]
gives the desired map of Equation \ref{coharepmapeqn}, where the correct shift follows from the equality of Equation \ref{cohaproddegeqn}. Commutativity of the diagram
\[ \xymatrix@C=5pc{ \mc H \otimes^\tw \mc H \otimes^\tw \V \ar[r]^{\id_{\mc H}\otimes \rho  } \ar[d]^{m_{\mc H}\otimes \id_{\V}} & \mc H\otimes^\tw \V \ar[d]^{\rho}  \\  \mc H \otimes^\tw \V \ar[r]^{\rho} & \V} \]
follows from a standard argument, similar to that of the associativity of the product on the cohomological Hall algebra; see for example Section 2.3 of \cite{KS1}.
\end{proof}

\subsection{Enumerative invariants from perverse coherent extensions}\label{invtsec}

The vector spaces $\bb V_M^\zeta(Y)$ underlying the modules constructed in the preceding section should be understood as a family of categorified enumerative invariants of the threefold $Y$, which are heuristically given by counting the number of $\zeta$-stable iterated extensions of the object $M$ with compactly supported perverse coherent sheaves. We now formally define the corresponding enumerative invariants and recall their relation with geometry, which follow from the seminal results of \cite{Beh} and \cite{BehFan}.

Let $M\in \Perv^\p(Y)^T$, $\f$ a framing structure for $M$ of rank $d_\infty$, and $(Q_M^\f,W_M^\f)$ the associated framed quiver with potential, and $\zeta$ a stability condition on $X^\f(Q_M)$ as in the preceding section, and recall that the graded vector space underlying $\bb V_M^\zeta(Y)$ is given by
\[ \V^{\f,\zeta}(M) = \bigoplus_{\dd\in \bb N^{V_Q}} \V^{\f,\zeta}_\dd(M)  \quad\quad \text{with}\quad\quad \V^{\f,\zeta}_\dd(M) =H_\bullet^{G_\dd(Q)}( X_\dd^{\zeta}(Q_M^\f), \varphi_{W_{M,\dd}^{\f}} )  \ . \]

We define the generating functional for the corresponding enumerative invariants by
\begin{equation}\label{genfuneqn}
	 \Zz_M^{\f,\zeta}(\qq) = \sum_{\dd \in \bb N^{V_{Q_Y}}}\qq^\dd \chi(\bb V^{\f,\zeta}_\dd(M))  \quad \in \Pq \ ,
\end{equation}
where we introduce the shorthand notation $\qq^\dd=\prod_{i\in V_{Q_Y}} q_i^{d_i}$ and $\Pq= \bb Z[\![q_i]\!]_{i\in V_{Q_Y}} $, and $\chi$ denotes the Euler characteristic of the cohomologically graded vector space.

Now, suppose that the stability condition $\zeta$ is chosen so that there are no strictly semi-stable points and $\mf M^{ \zeta}_\dd(Q_M^\f)= X^{\zeta}_\dd(Q_M^\f)/G_\dd(Q)$ is a smooth scheme for each $\dd\in \bb N^{V_{Q_Y}}$, so that the function $W_{M,\dd}^\f:X^{ \zeta}_\dd(Q_M^\f)\to \bb K$ descends to a regular function $\overline{W}_{M,\dd}^\f:\mf M^{ \zeta}_\dd(Q_M^\f) \to \bb K $, and we have that
\[ \mf M^\zeta_\dd(Q_M^\f, W_M^\f) = \left[ \Crit(W_{M,\dd}^\f|_{X_{\dd}^{ \zeta}(Q_M^\f)}) / G_\dd(Q_Y) \right] = \Crit(\overline{W}_M^\f) \subset \mf M^{ \zeta}_\dd(Q_M^\f)  \ ,  \]
that is, the moduli space $\mf M^\zeta_\dd(Q_M^\f, W_M^\f)$ of $\zeta$-stable representations of the quiver with potential $(Q_M^\f,W_M^\f)$ is a global critical locus of a regular function on a smooth scheme. In particular, we have
\begin{equation}
\V^{\f,\zeta}_\dd(M) =H_\bullet^{G_\dd(Q)}( X_\dd^{\zeta}(Q_M^\f), \varphi_{W_{M,\dd}^{\f}} ) = H_\bullet( \mf M_\dd^{\zeta}(Q_M^\f), \varphi_{\overline{W}_{M,\dd}^\f})  \ .
\end{equation}

Moreover, the results of \cite{Beh} imply that $\mf M^\zeta_\dd(Q_M^\f, W_M^\f)$ admits a symmetric obstruction theory inducing a virtual fundamental class $[\mf M^\zeta_\dd(Q_M^\f, W_M^\f)]^\textup{vir} \in H_\bullet(\mf M^\zeta_\dd(Q_M^\f, W_M^\f))$ and corresponding invariants
\[  \tilde{\mc Z}^{\f,\zeta}_M(\qq) = \sum_{\dd\in \bb N^{V_{Q_Y}}} \qq^\dd  \int_{ [\mf M^\zeta_\dd(Q_M^\f, W_M^\f)]^\textup{vir}} 1  \quad\in \Pq   \ , \]
which heuristically count the number of points in $\mf M^\zeta_\dd(Q_M^\f, W_M^\f)$, such that these invariants can be computed as
\[\int_{ [\mf M^\zeta_\dd(Q_M^\f, W_M^\f)]^\textup{vir}} 1  = \chi(\mf M^\zeta_\dd(Q_M^\f, W_M^\f), \nu_{M,\dd}^\zeta ):= \sum_{n\in \bb Z} n \  \chi((\nu_{M,\dd}^\zeta)^{-1}\{n\} ) \  , \]
the weighted Euler characteristic of $\mf M^\zeta_\dd(Q_M^\f, W_M^\f)$ with respect to a constructible function
\[\nu_{M,\dd}^\zeta: \mf M^\zeta_\dd(Q_M^\f, W_M^\f) \to \bb Z \ . \]
Moreover, it is shown in \emph{loc. cit.} that in the case of a critical locus, this constructible function can be computed locally in terms of the Euler characteristic of the Milnor fibre of the potential, inducing an identification
\[ \chi(\mf M^\zeta_\dd(Q_M^\f, W_M^\f), \nu_{M,\dd}^\zeta ) = \chi( \V^{\f,\zeta}_\dd(M)) \quad\quad \textup{and thus} \quad\quad   \Zz_M^{\f,\zeta} =  \tilde{\mc Z}^{\f,\zeta}_M \ .\]

We now recall some additional hypotheses that allow for a concrete combinatorial approach to calculating these invariants, and which are valid in many examples of interest. Let $\tilde T\subset T$ denote the maximal torus of $T$ which preserves the Calabi-Yau structure on $Y$, and recall that there is a natural action of $T$ on $\mf M_\dd^{\zeta,\f}(Q_M)$ induced by the action on the underlying threefold $Y$, such that the subtorus $\tilde T$ preserves the potential $W_M^\f$ and thus naturally acts on the moduli space $\mf M^\zeta_\dd(Q_M^\f, W_M^\f)$.

Additionally, let $T_\f$ be a maximal torus of $G^\cn_\f = \textup{Stab}_{\Gl_{d_\infty}(\bb K)}(\f)$ the subgroup of $\Gl_{d_\infty}(\bb K)$ stabilizing the point $\f \in Z_{d_\infty}(Q_\infty,R_\infty)$, and let $A=\tilde T\times T_\f$ which acts naturally on $X_\dd^\f(Q_M)$ such that the potential $W_{M,\dd}^\f$ is $A$ invariant and this action commutes with that of $G_\dd$. Thus, $A$ also acts naturally on $Z_\dd^\f(Q_M,R_M)$, as well as the quotient stacks $\mf M_\dd^\f(Q_M)$ and $\mf M_\dd(Q_M^\f,W_M^\f)$, and similarly we obtain actions on $\mf M_\dd^{\f,\zeta}(Q_M)$ and $\mf M_\dd^\zeta(Q_M^\f,W_{M}^\f)$ if $A$ preserves $\zeta$-stability. Throughout the remainder of this section, we make the following hypothesis:

\begin{hypo}\label{lochypo} We assume that the $A$-fixed subvariety $ \mf M^\zeta_\dd(Q_M^\f, W_M^\f)^A$ is a finite set $\Fp_\dd$ of isolated fixed points.
\end{hypo}

For simplicity of presentation in the following corollary, we also assume that for each $\dd \in \bb N^{V_{Q_Y}}$, the dimensions of the tangent spaces to the $A$ fixed points in the component $\mf M^\zeta_\dd(Q_M^\f, W_M^\f)$ are all of the same parity $(-1)^{k_\dd}$ for $k_\dd\in \bb Z/ 2\bb Z$.

Under the preceding hypotheses, we can compute the above enumerative invariants using the localization results of \cite{BehFan}, so that we obtain:
\begin{corollary}\label{fpcoro} The above generating function is given by
\[ \mc Z_M^\zeta(\qq) = \sum_{\dd\in \bb N^{V_{Q_Y}}} \qq^\dd  (-1)^{k_\dd} \ |\mf M^\zeta_\dd(Q_M^\f, W_M^\f)^A|     \quad \in \Pq \ , \]
where $|\mf M^\zeta_\dd(Q_M^\f, W_M^\f)^A| \in \bb Z$ denotes the cardinality of the set of $A$ fixed points.
\end{corollary}
\begin{proof} Under the given hypotheses, this follows from Theorem 3.4 and Corollaries 3.5 and 3.6 of \emph{loc. cit.}.
\end{proof}

In fact, our description of the set of $T$ fixed points induced by the equivalence of Equation \ref{modulistackfpeqn} together with Theorem \ref{heartthm} gives a natural enumeration of the fixed points in terms of objects in $\filt(\Sigma_M)$, or equivalently in the category of twisted objects $H^0(\tw^0(\mc A_{F\oplus M}))$.

\subsection{Towards representations of shifted quiver Yangians}\label{quiveryangsec}

In this section, we outline the construction of an extension of the geometric action of $\mc H(Y)$ on $\mc V^\zeta(M)$ to the action of a larger associative algebra, called a shifted quiver Yangian in \cite{LiYam,GLY}, following the approach of \cite{RSYZ} and in turn \cite{Ts1}, \cite{SV}, \cite{Var} and \cite{Nak2}.

Let $F$ denote the field of fractions of the polynomial ring
\[  H_A^\bullet(\pt)=\bb K[\hbar_1,\hbar_2,\hbar_3,\xx]/(\hbar_1+\hbar_2+\hbar_3) \ , \]
 where we recall that $A=\tilde T \times T_\f$ is the product of $\tilde T\subset T$ the maximal torus of $T$ which preserves the Calabi-Yau structure on $Y$ and $T_\f$ a maximal torus of $G^\cn_\f = \textup{Stab}_{\Gl_{d_\infty}(\bb K)}(\f)$, and we have introduced the identifications $H_{T_\f}(\pt)=\bb K[\xx]:=\bb K[x_j]_{j=1,..,\rk(T_\f)}$ and $H^\bullet_T(\pt)=\bb K[\hbar_1,\hbar_2,\hbar_3]$ such that the Calabi-Yau structure is of weight $(1,1,1)$ in the corresponding grading.

Throughout this section, all constructions are given over the base ring $F$ by default, though this will be occasionally be suppressed in the notation. In particular, we let
\[\mc H=\mc H(Q_Y,W_Y)^A\in \Alg_\Ass(\Vect_{F,\bb N^{V_Q}}^{\otimes^\tw}) \quad\quad \text{and}\quad\quad \V^\zeta=\V^\zeta(M)^A \in \mc H\Mod(\Vect_{F, \bb N^{V_Q}}^{\otimes^\tw}) \]
denote the localized, $A$-equivariant analogues of the objects introduced under this notation in sections \ref{cohasec} and \ref{coharepsec}, respectively. In particular, for each $\dd\in \bb N^{V_{Q_Y}}$ the underlying vector spaces of the degree $\dd$ components are given by
\[ \mc H_\dd = H_\bullet^{G_\dd\times A}(X_\dd,\varphi_{W_\dd})\otimes_{H^\bullet_A(\pt)} F  \quad\quad \text{and}\quad\quad  \V_\dd^\zeta = H_\bullet^{G_\dd\times A}( X_\dd^{\f,\zeta}, \varphi_{W_{M,\dd}^{\f}} )\otimes_{H^\bullet_A(\pt)} F  \ .\]

We define the universal Cartan algebra $\mc H^0$ as the polynomial algebra
\[ \mc H^0=\mc H^0(Y) = F[\phi_n^{i} ]^{i\in I}_{n\in \bb N}  \quad\quad \textup{and let} \quad\quad \phi^i(z) =1 + \sigma_3 \sum_{n=0}^\infty \phi^i_n z^{-n-1} \ \in \mc H^0 \lP z^{-1} \rP   \  \]
be a formal generating functional for the variables $\phi_n^i$, which we will use as a shorthand to write equations involving these variables.

Recall that the analogue of the Chern polynomial in this setting defines a map
\[c_{1/z}:K_{G_\dd\times A}(X_\dd^{\f,\zeta})\to \End_{H^\bullet_{G_\dd\times A}(\pt)}(H_\bullet^{G_\dd\times A}( X_\dd^{\f,\zeta}, \varphi_{W_{M,\dd}^{\f}} )\otimes_{H^\bullet_A(\pt)} F  )\lP z^{-1}\rP \ ;\]
it can be defined for example on vector bundles $E\in K_{G_\dd\times A}(X_\dd^{\f,\zeta}) $ by
\[ c_{1/z}(E)= \sum_{n=1}^{\textup{rk}(E)} c_n(E) z^{-n} := z^{-\text{rk}(E)} \textup{Eu}(E\otimes q) =  \prod_{i=1}^{\textup{rk}(E)} (1+\frac{\ee_i}{z})  \]
where $\ee_1,...,\ee_{\text{rk}(E)}\in H_\bullet^{G_\dd\times A}( X_\dd^{\f,\zeta}(Q_M), \varphi_{W_{M,\dd}^{\f}} ) $ denote the Chern roots of $E$ and
\[\hspace*{-1cm} \textup{Eu}(E\otimes q)=s^*s_* \  \in \End_{H^\bullet_{G_\dd\times A\times \bb C^\times}(\pt)}(H_\bullet^{G_\dd\times A\times \bb C^\times}(X_\dd^{\f,\zeta},\varphi_{W_{M,\dd}^{\f}} ))\cong \End_{H^\bullet_{G_\dd\times A}(\pt)}(H_\bullet^{G_\dd\times A}(X_\dd^{\f,\zeta},,\varphi_{W_{M,\dd}^{\f}}))[z] \]
denotes the $G_\dd\times A\times \bb C^\times$-equivariant Euler class, defined in terms of $s:X_\dd^{\f,\zeta} \to |E\otimes q|$ the zero section, where $\C^\times$ acts trivially on $X$, $q$ denotes its standard one dimensional, weight one representation, and we identify $H^\bullet_{\bb C^\times}(\pt)\cong \bb C[z]$.

For each $\dd\in \bb N^{V_Q}$, there are $G_\dd\times A$ equivariant tautological bundles $\mc V^i_\dd$ on $X_\dd^{\f,\zeta}$ of rank $d_i$ for each $i\in I_M$ which define classes in $K_{G_\dd\times A}(X_\dd^{\f,\zeta})$, in terms of which we can define the canonical classes $\mc F_\dd^i$ by
\[  \mc F_\dd^i = \ _i \Sigma\otimes_{S_M} \mc V_\dd = \bigoplus_{j\in I_M} \ _i \Sigma_j\otimes \mc V_\dd^j  \ \in K_{G_\dd\times A}(X_\dd^{\f,\zeta}) \]
for each $i\in I_M= V_{Q_Y} \cup \{\infty\}$, where we view $\Sigma$ with its cohomological and $T$ gradings as a complex of representations of $A$.

We now define an action $\mc H^0\to \End_F(\V^\zeta)$ of the universal Cartan algebra $\mc H^0$ on $\V^\zeta$ in components
\[ \mc H^0 \to  \End_{F}(H_\bullet^{G_\dd\times A}(X_\dd^{\f,\zeta},\varphi_{W_{M,\dd}^{\f}}) \otimes_{H^\bullet_A(\pt)} F  ) \ ,\]
by the action of the generators $\phi^i_n$ on $\V^\zeta_\dd(M)$, according to the formula
\[ \phi^i(z) \mapsto c_{-1/z}(\mc F_\dd^i) \ \in \End_{F}(H_\bullet^{G_\dd\times A}( X_\dd^{\f,\zeta}, \varphi_{W_{M,\dd}^{\f}} )\otimes_{H^\bullet_A(\pt)} F )\lP z^{-1}\rP \  , \]
for each $i\in V_{Q_Y}$. 

Recall that the representation structure map $\mc H \to \End(\V^\zeta)$ constructed in Theorem \ref{cohareptheo} is defined in terms of the maps
\[\xymatrix{ & X^{\f,\zeta}_{\aaa,\bbb} \ar[dl]_{\q} \ar[dr]^{\p} \\ X_{\aaa}\times X^{\f,\zeta}_\bbb & & X^{\f,\zeta}_{\aaa+\bbb} }  \ . \]
The relation between the actions of $\mc H$ and $\mc H^0$ on $\V^\zeta$ is determined by the difference
\[ \p^* \mc F_{\aaa+\bbb}^i - \q^*( \mc O_{X_a}\boxtimes \mc F_{\bbb}^i) = q^*(\mc G^i_\aaa \boxtimes \mc O_{ X^{\f,\zeta}_\bbb } )  \]
in terms of a class $\mc G^i_\aaa\in K_{G_\aaa}(X_\aaa)$. In particular, we assume that the Chern polynomial of $\mc G^i_\aaa$ is given by
\[   c_{-1/z}(\mc G_\aaa^i)  =  P^i_\aaa(\hbar_1,\hbar_2,\xx,\ee_k,z)^-  \quad\in \End_{H^\bullet_{G_\dd\times A}(\pt)}(\ H_\bullet^{G_\aaa\times A}(X_\aaa,\varphi_{W_\aaa})\otimes_{H^\bullet_A(\pt)} F ) \lP z^{- 1} \rP  \ ,   \]
the power series expansion in $z^{-1}$ of a rational function
\[ P^i_\aaa=P^i_\aaa(z) = P^i_\aaa(\hbar_1,\hbar_2,\xx,\ee_k,z) \ \in  F(\{\ee_k\}_{k=1}^{\rk(\mc V_\aaa^i)},z) \]
with coefficients in the Chern roots $\ee_k$ of $\mc V_\aaa$ over the base field $F$.

Thus, we similarly define an action $\mc H^0 \to \End_F(\mc H)$ of the universal Cartan algebra $\mc H^0$ on the cohomological Hall algebra $\mc H$ in components
\begin{equation}\label{cartancohaacteqn}
	\mc H^0 \to \End_{H^\bullet_{G_\dd\times A}(\pt)}(H_\bullet^{G_\dd\times A}( X_\dd, \varphi_{W_{\dd}} )\otimes_{H^\bullet_A(\pt)} F ) \quad\quad \text{by}\quad\quad \phi^i(z) \mapsto c_{-1/z} (\mc G_\aaa^i) \ , 
\end{equation}
and we define the \emph{extended cohomological Hall algebra} of $Y$ as the semidirect product algebra
\[\mc H^{\geq 0}= \mc H^0\ltimes \mc H  \ \] with respect to this action. Then we obtain the following Proposition:

\begin{prop} The object $\V^\zeta$ admits a canonical (twisted $\bb N^{V_Q}$-graded) module structure over the extended cohomological Hall algebra $\mc H^{\geq 0}$, that is, there is a natural algebra map
\[ \mc H^{\geq 0} \to \End_F(\V^\zeta)  \ ,\]
extending the representation of $\mc H^0$ in Equation \ref{cartancohaacteqn} and that of $\mc H$ from Theorem \ref{cohareptheo}.
\end{prop}

\begin{proof}
By the construction of $\mc G^i_\aaa$, the proof is a straightforward generalization of that of Proposition 6.2.2 in \cite{RSYZ}.
\end{proof}

We define the subspace of spherical generators $\mc H^1\subset \mc H$ by
\[ \mc H^1 = \bigoplus_{|\dd|=1} \mc H_\dd  = \bigoplus_{i\in V_{Q_Y}} \mc H_{\one_i} \]
where $\one_i\in \bb N^{V_{Q_Y}}$ denotes the $i^{th}$ standard basis vector, and the spherical subalgebra
\[ \mc{SH}=\mc{SH}(Y)= \langle \mc H^1 \rangle  \subset \mc H \ ,\]
as the subalgebra generated by this subspace over the base ring (in the present setting given by $F$), as well as $\mc{SH}^{\geq 0} = \mc H^0\ltimes \mc{SH}$. Note that we have an induced action
\begin{equation}\label{shrepmapeqn}
	\mc{SH} \to \End_{F}(\V^\zeta) \ ,
\end{equation}
and similarly for $\mc{SH}^{\geq 0}$.

Let $\ee_i=c_1(\mc V^i_{\one_i})\in  \mc H_{\one_i}= H_\bullet^{G_{\one_i}\times A}(X_{\one_i},\varphi_{W_{\one_i}})$ denote the first Chern class of the one dimensional tautological bundle $\mc V^i_{\one_i}$, and define the formal generating functional
\[ e^i(z) = \sum_{n=0}^\infty e^i_n z^{-n-1} \in \mc{SH}\lP z^{-1}\rP  \quad\quad \text{where}\quad\quad e_n^i=(\ee_i)^n  \in \mc{SH} \ .  \]
 The induced relations in $\mc{SH}^{\geq 0}$ between the elements $e^j_m$ and $\phi^i_n$ are determined by the action of Equation \ref{cartancohaacteqn} in the case $\aaa=\one_j$, and the rational function
\[P^i_j=P^i_j(z)=P^i_{\one_j}(\hbar_1,\hbar_2,\xx,\ee_j,z) \ \in F(\ee_j,z)  \]
is called the \emph{bond factor} in \cite{LiYam}, where it is denoted $\varphi^{(i\Rightarrow j)}(z)$. The definition of the action in Equation \ref{cartancohaacteqn} determines relations between the generators summarized in terms of the generating functions as
\begin{equation}\label{ephirlneqn}
	\phi^i(z) e^j(w) = e^j(w) \phi^i(z) P^i_j(z-w) \quad\in   \mc{SH}^{\geq 0}\lL z^{-1} \rL \lP w^{-1}\rP \ .  
\end{equation}

For the remainder of this section, we will in additional assume that the $A$-fixed subvariety $ \mf M^\zeta_\dd(Q_M^\f, W_M^\f)^A$ is given by a set $\Fp_\dd$ of isolated fixed points, as in Hypothesis \ref{lochypo}. In particular, the pullback map induced as in Equation \ref{pullbackeqn} gives an isomorphism
\[ \bb V^\zeta_\dd \cong H_\bullet^{A}(\mf M^{\f,\zeta}_\dd, \varphi_{\overline{W}^\f_\dd})\otimes_{H^\bullet_A(\pt)}F  \xrightarrow{\iota^*}  H_\bullet^A(\Fp_\dd,\varphi_{\overline{W}^\f_\dd\circ \iota}) \otimes_{H^\bullet_A(\pt)}F = \bigoplus_{\lambda\in \Fp_\dd} H^A_\bullet(\pt_\lambda)\otimes_{H^\bullet_A(\pt)}F \ , \]
for each $\dd\in \bb N^{V_{Q_Y}}$, and letting $F_\lambda$ denote a copy of the base field $F$, we have an identification
\[  \bigoplus_{\lambda\in \Fp_\dd}  F_\lambda   \xrightarrow{\cong}  \bigoplus_{\lambda\in \Fp_\dd}   H^A_\bullet(\pt_\lambda)\otimes_{H^\bullet_A(\pt)}F \quad\quad \text{defined by} \quad\quad P \mapsto P\cap [ \pt_\lambda ]  \ ,\]
for each $P\in F_\lambda$ so that we obtain a natural basis for the module $\bb V^\zeta_\dd$ given by
\[ \bb V^\zeta_\dd =\bigoplus_{\lambda\in \Fp_\dd}  F_\lambda  \quad\quad \text{and thus}\quad\quad \bb V^\zeta = \bigoplus_{\lambda \in \Fp} F_\lambda  \quad\quad\text{for}\quad\quad  \Fp = \bigsqcup_{\dd\in \bb N^{V_{Q_Y}}} \Fp_\dd\ . \] 
In particular, there is a natural pairing
\begin{equation}\label{pairingeqn}
	 (\cdot,\cdot):\V^\zeta \otimes_F \V^\zeta \to F \quad\quad \text{defined by} \quad\quad  ([\pt_\lambda],[\pt_\mu]) = \delta_{\lambda,\mu}\textup{Eu}_A(T_\lambda)  \ , 
\end{equation}
where $T_\lambda$ denotes the tangent space to $ \mf M^\zeta_\dd(Q_M^\f, W_M^\f)$ at the fixed point $\lambda \in \Fp$ and $\textup{Eu}_A$ denotes the $A$-equivariant Euler class.

Let $\mc{SH}^\op$ denote the opposite algebra of $\mc{SH}$ and note that we have the following proposition:

\begin{prop}\label{factprop} There exists a natural action right action
	\begin{equation}\label{shoprepmapeqn}
		\mc{SH}^\op \to \End_F(\V^\zeta) \quad\quad \text{defined by} \quad \quad f=e^\op \mapsto \rho(e)^*  
	\end{equation}
	for each $f=e^\op \in \mc{SH}^\op$, where $\rho:\mc{SH}\to \End_F(\V^\zeta)$ is the representation of Theorem \ref{cohareptheo} and $(\cdot)^*:\End_F(\V^\zeta)\to \End_F(\V^\zeta)$ denotes the adjoint with respect to the pairing of Equation \ref{pairingeqn} above.
\end{prop}
\begin{proof}
This follows immediately from the proof of Proposition 3.6 of \cite{SV}, \emph{mutatis mutandis}.
\end{proof}

In particular, we define the formal generating functional
\[ f^i(z) = \sum_{n=0}^\infty f^i_n z^{-n-1} \in \mc{SH}^\op[z^{-1}] \quad\quad \text{where}\quad\quad f_n^i=(e_n^i)^\op  \in \mc{SH}^\op \ ,  \]
and note that the induced relations between the generators $f_m^j$ and $\phi^i_n$ are summarized in terms of the generating functions as
\begin{equation}\label{fphirlneqn}
	\phi^i(z) f^j(w) =  f^j(w) \phi^i(z) P^i_j(z-w)^{-1}   \quad\in   \mc{SH}^{\geq 0}\lL z^{-1} \rL \lP w^{-1}\rP\ .  
\end{equation}

Finally, we consider the action by the endomorphisms
\[\tilde \psi^{i,j}_{n,m}=[\rho(e^i_n),\rho(f^j_m)] \ \in \End_F(\V^\zeta) \ .\]
We assume that these endomorphisms vanish for $i\neq j$ and moreover that for $i=j$ the resulting endomorphism $\tilde \psi^i_{n+m}$ depends only on the sum $m+n$ and is determined
\[ \tilde \psi^i(z) := z^{-s_i} +\sum_{n=s_i}^\infty \tilde \psi^i_n z^{-n-1} = Q^i_M(z) c_{-1/z}(\mc F_\dd^i) \quad \in\   z^{-s_i}\End_{H^\bullet_{G_\dd\times A}(\pt)}(H_\bullet^{G_\dd\times A}( X_\dd^{\f,\zeta}, \varphi_{W_{M,\dd}^{\f}} )\otimes_{H^\bullet_A(\pt)} F )\lP z^{-1} \rP  \ ,\]
in terms of a rational function
\begin{equation}\label{}
	Q_M^i=Q_M^i(z)=Q_M^i(\hbar_1,\hbar_2,\xx,e_i,z) \  \in F(e_i,z)
\end{equation}
with coefficients in the first Chern class $\ee_i$ of $\mc V_{\one_i}$ over the base field $F$, where the map is given by taking the power series expansion at $z=\infty$. The rational function $Q_M^i\in F(e_i,z)$ encodes the shift the associated shifted quiver Yangian and is called the ground state charge function in \cite{GLY}, where it is denoted $^\#\psi_0^{(i)}(z)$. Typically, $Q^i_M$ admits a factorization of the form
\[ Q^i_M(z) = \frac{\prod_{l=1}^{s_i^+} (z-q_{i,l}^+(\hbar_1,\hbar_2,\xx,e_i))}{\prod_{l=1}^{s_i^-} (z-q_{i,l}^-(\hbar_1,\hbar_2,\xx,e_i))} \]
and we call the integer $s_i=s_i^+-s_i^-$ the shift of the corresponding root $i\in I$ for the object $M$.

We define the shifted Cartan algebra $\mc H^0_M$ as the polynomial algebra
\[ \mc H^0_M=\mc H^0_M(Y) = F[\psi_n^i]_{n\in \bb N}^{i\in I} \quad\quad \text{and let}\quad\quad \psi^i(z)  = z^{-s_i} +\sum_{n=s_i}^\infty \psi^i_n z^{-n-1}\quad\in z^{-s_i} \mc H^0\lP z^{-1}\rP    \  \]
be the corresponding formal generating functional for the variables $\psi_n^i$. Moreover, we obtain a natural action $\mc H^0_M\to \End_F(\bb V^\zeta)$ of the shifted Cartan algebra $\mc H^0_M$ on $\V^\zeta$ defined in components
\begin{equation}\label{shiftedcartanacteqn}
 \mc H^0_M \to \End_{H^\bullet_{G_\dd\times A}(\pt)}(H_\bullet^{G_\dd\times A}( X_\dd^{\f,\zeta}, \varphi_{W_{M,\dd}^{\f}} )\otimes_{H^\bullet_A(\pt)} F)\quad\quad \text{by}\quad\quad \psi^i(z)  \mapsto  \tilde \psi^i(z)=Q^i_M(z) c_{-1/z}(\mc F_\dd^i)  \ .
\end{equation}
By construction, the relations between the images of the generators $\tilde e^i_n=\rho( e^i_n)$ and $\tilde f^j_M=f^j_M$ under the respective representations are summarized in terms of generating functions as
\[ [\tilde e^i(z),\tilde f^j(w)] = \delta_{i,j} \frac{ \tilde\psi^i(z)-\tilde\psi^i(w)}{z-w} \quad \in \End_F(\V^\zeta)\lL z^{-1} \rL \lP w^{-1}\rP\ .\]

We now recall the definition of (shifted) quiver Yangians given in \cite{LiYam} (and \cite{GLY}):

\begin{defn} The \emph{shifted quiver Yangian} $\mc Y_M=\mc Y_M(Y)$ is the associative algebra over $F$ generated by elements $e^i_n, f^i_n$ and $\psi^i_m$ for $i\in I$, $n \in \bb N$, and $m \in \bb Z_{\geq s_i}$, subject to the relations
\begin{align}
	\psi^i(z) e^j(w) &  = e^j(w) \psi^i(z) P^i_j(z-w)  & \in   \mc Y_M\lL z^{-1} \rL \lP w^{-1}\rP\\
	\psi^i(z) f^j(w)  & = f^j(w) \psi^i(z) P^i_j(z-w)^{-1} & \in    \mc Y_M\lL z^{-1} \rL \lP w^{-1}\rP\\
	[ e^i(z), f^j(w)] & = \delta_{i,j} \frac{ \psi^i(z)-\psi^i(w)}{z-w}  & \in   \mc Y_M\lP z^{- 1},w^{-1}\rP\\
		\psi^i(z) \psi^j(w) &  = \psi^j(w) \psi^i(z)  & \in   \mc Y_M\lP z^{- 1},w^{-1}\rP\\
		e^i(z) e^j(w) &  = e^j(w) e^i(z) P^i_j(z-w)  & \in   \mc Y_M\lP z^{- 1},w^{-1}\rP\\
		f^i(z) f^j(w)  & = f^j(w) f^i(z) P^i_j(z-w)^{-1} & \in    \mc Y_M\lP z^{- 1},w^{-1}\rP  
\end{align}
together with the Serre relations between the various $e^i_n$ generators, and similarly between the $f^i_n$.
\end{defn}

Under the above hypotheses, we make the following conjecture:

\begin{conj}\label{quiveryangconj} The maps of Equations \ref{shrepmapeqn}, \ref{shoprepmapeqn} and \ref{shiftedcartanacteqn} define a representation
\[ \mc Y_M(Y) \to \End_F(\V^\zeta(M)^A)   \]
of the shifted quiver Yangian $\mc Y_M(Y)$ on $\bb V^\zeta(M)^A$.
\end{conj}

In particular, it is implicit in the conjecture that the restriction of the representation to the positive half $\mc Y^+(Y,M)$ agrees with the restriction to the spherical subalgebra $\mc{SH}$ of the representation of $\mc H$ constructed in Theorem \ref{cohareptheo}, as this is the definition of the representation in Equation \ref{shrepmapeqn}. In the case $Y=\bb C^3$, the above conjecture holds in the sense that the above construction is equivalent to that of \cite{SV}, also obtained in unpublished work of Feigin-Tsymbaliuk recalled in \cite{Ts1}, and \cite{RSYZ}, although the shift is trivial in all these cases. Several other examples were also explored in \cite{RSYZ2} in this framework with non-trivial shifts. Moreover, the cohomological analogue of the recent results of \cite{Negquiv}, together with the above construction, appear to imply the above conjecture, but we leave a detailed discussion of this for future work.

\section{Yangians of threefolds and vertex algebras of divisors}\label{funsec}

\subsection{Perverse coherent systems, Donaldson-Thomas theory, and quiver Yangians}\label{DTsec}

In this section, we let $M=\mc O_Y[1]$ be the structure sheaf of the threefold $Y$ shifted down in cohomological degree by 1. As we observed in Example \ref{pervcohsyseg}, the extension group $_\infty \Sigma^1_\infty=\Ext^1(\mc O_Y[1],\mc O_Y[1])=0$ is trivial so that the only possible framing structure for each rank $d_\infty=\rrr$ is the trivial one $\f_\rrr=\{0\}$. Moreover, the notion of $\mc F_{\f_r}^\cn$-framed perverse coherent extension is closely related to the notion of perverse coherent system, in the sense of \cite{NN}, as we explain below.

We begin by recalling the definition of perverse coherent systems from \emph{loc. cit.}:

\begin{defn}\label{pervcohsysdef} A \emph{perverse coherent system} on $Y$ is a tuple $(H,W,S)$ where $H\in \Perv(Y)$ is a perverse coherent sheaf and
	\[ s: \mc O_Y\otimes W  \to H \ ,\]
is a map of complexes of quasicoherent sheaves. A perverse coherent system on $Y$ is called \emph{compactly supported} if the underlying perverse coherent sheaf $H\in \Perv_\cs(Y)$.
\end{defn}

Let $\mf M_{\overline{\textup{Per}}_\cs(Y)}$ denote the moduli stack of compactly supported perverse coherent systems on $Y$, and note there is a natural decomposition
\[ \mf M_{\overline{\textup{Per}}_\cs(Y)} = \bigsqcup_{\rrr\in \bb N} \mf M_{\overline{\textup{Per}}_\cs(Y)}^\rrr \ ,\]
where $\overline{\textup{Per}}_\cs(Y)$ denotes the component for which the vector space $W$ is of dimension $\rrr \in \bb N$.

Recall that we observed in Example \ref{pervcohsyseg} that the extension group
\[_\infty\Sigma_\infty^1=\Ext^1(\mc O_Y[1],\mc O_Y[[1])=0 \ , \]
so that for each rank $d_\infty=\rrr\in \bb N$ there is a unique trivial framing structure $\f_\rrr=\{0\}$. We can now state the desired equivalence:

\begin{prop}\label{pervcohsysprop} There is an equivalence of algebraic $\bb K$ stacks
	\[ \mf M^{\mc F_{\f_\rrr}^\cs}(Y,\mc O_Y[1]) \xrightarrow{\cong }\mf M_{\overline{\textup{Per}}_\cs(Y)}^\rrr \ ,\]
	where we recall that $\mc F_{\f_0}^\cs$ is as in Equation \ref{Fcframedstackeqn}.
\end{prop}
\begin{proof}By definition, the moduli stack $\mf M^{\mc F_{\f_\rrr}^\cs}(Y,\mc O_Y[1])$ parameterizes perverse coherent extensions of $\mc O_Y[1]$ equipped with an isomorphism of the underlying iterated extension of $\mc O_Y[1]$ with $\mc O_Y\otimes W[1]$ for $W$ a vector space of rank $\rrr$, by Example \ref{DTframingeg}.
	
Further, recall from Example \ref{pervcohsyseg} that we observed the extension groups between $\mc O_Y[1]$ and the simple objects $F_i\in \Perv_\cs(Y)$ are given by
\[\ \ _j \Sigma^1_\infty = \Ext^1(F_j,\mc O_Y[1])=0 \quad\quad _\infty \Sigma^1_j=\Ext^1(\mc O_Y[1],F_j)=\begin{cases} \bb K & \textup{if $j=0$} \\ \{0\} & \textup{otherwise} \end{cases} \ , \]
where the only non-trivial extension class is given by
\[\mc O_C \to  \mc I_C[1] \to \mc O_Y[1]  \quad\quad \text{induced by}\quad\quad  \mc I_C[1] \to \mc O_Y[1] \to \mc O_C[1] \ . \] Thus, the only possible extensions of $\mc O_Y[1]$ with the generators $F_i$ is given by the map $\mc O_Y\to F_0$. 

Similarly, note that in this case the potential $W_M^{\f_r}$ is independent of the framing arrows, so there is a canonical map of $\K$ stacks
\[ \mf M^{\mc F_{\f_\rrr}^\cs}(Y,\mc O_Y[1])  \to \mf M(Y) \ , \]
which corresponds to taking the subobject given by the underlying iterated extension of the generators $F_i\in \Perv_\cs(Y)$. For each compactly supported perverse coherent sheaf $H\in \mf M(Y)(\bb K)$, the observations above imply that the only possible perverse coherent extension with subobject $H$ is given by a map from the underlying iterated extension of $\mc O_Y$, which has a fixed isomorphism with the object $\mc O_Y\otimes W$, to the object $H$. This provides the desired map $s:\mc O_Y\otimes W \to F$.

\end{proof}

As a corollary of Proposition \ref{pervcohsysprop} together with Theorem \ref{stackthm}, we obtain that there is a canonical monad presentation for perverse coherent systems, given by that of Example \ref{pervcohsyseg}. In the case of the resolved conifold studied in \cite{Sz1} and \cite{NN}, the monad presentation is that of Equation \ref{conifoldmonadeqn}, though the results also apply in the more general setting studied in \cite{Nag}. Indeed, the corresponding framed quivers with potential $(Q_{\mc O_Y[1]}^{\f_\rrr},W^{\f_\rrr}_{\mc O_Y[1]})$ in this case are precisely those studied in \emph{loc. cit.}.

For each choice of compatible stability condition $\zeta \in \bb R^{V_{Q_Y}}$, the corresponding homology module $\bb V^\zeta_Y:= \bb V^\zeta(\mc O_Y[1])$ defined in Section \ref{coharepsec} is given by
\begin{equation}\label{catDTeqn}\hspace*{-1cm}
	\bb V^\zeta_Y= \bigoplus_{\dd \in \bb N^{V_{Q_Y}}} H_\bullet^{G_\dd(Q_Y)\times A}( X_\dd^{\zeta}(Q_{\mc O_Y[1]}^{\f_\rrr}),\varphi_{W^{\f_\rrr}_{\mc O_Y[1]}})\otimes_{H^\bullet_A(\pt)} F \ ,
\end{equation}
or equivalently, under the hypotheses of Section \ref{invtsec}, which hold in our examples of interest,
\[\bb V^\zeta_Y=  \bigoplus_{\dd \in \bb N^{V_{Q_Y}}} H_\bullet^{ A}( \mf M_\dd^{\zeta}(Q_{\mc O_Y[1]}^{\f_\rrr}),\varphi_{\overline{W}^{\f_\rrr}_{\mc O_Y[1]}})\otimes_{H^\bullet_A(\pt)} F   \ .\]
We also define the corresponding generating functional, as in Equation \ref{genfuneqn}, by
\[ \Zz_Y^\zeta(\qq) = \sum_{\dd \in \bb N^{V_{Q_Y}}}\qq^\dd \chi(\bb V^\zeta_\dd(\mc O_Y[1]))  \ \in \Pq \ ,\]
where we recall the shorthand notation $\qq^\dd=\prod_{i\in V_{Q_Y}} q_i^{d_i}$ and $\Pq= \bb Z[\![q_i]\!]_{i\in V_{Q_Y}} $.

The results of \cite{NN} and \cite{Nag}, in particular Proposition 3.17 and Theorem 3.18 of the latter, imply that for stability conditions $\zeta_\DT,\zeta_\PT,\zeta_\NCDT$ in appropriately chosen chambers, the corresponding homology modules are given by the (categorified) Donaldson-Thomas series \cite{DT}, Pandharipande-Thomas series \cite{PT}, and non-commutative Donaldson-Thomas series \cite{Sz1}, that is, we have
\[ \Zz_Y^{\zeta_\DT} = \Zz_Y^\DT \quad\quad \Zz_Y^{\zeta_\PT} = \Zz_Y^\PT \quad\quad \Zz_Y^{\zeta_\NCDT} = \Zz_Y^\NCDT  \ ,\]
where the latter are defined as the generating functions of the corresponding enumerative invariants of the threefold $Y$. In fact, it was proved in \cite{NN} using the results of \cite{BehFan} that these invariants can be computed via equivariant localization in terms of signed counts of fixed points, so that we have
\[ \Zz_Y^\zeta(\qq)= \sum_{\dd \in \bb N^{V_{Q_Y}}} \qq^\dd (-1)^{k_\dd} | \mf M_\dd^\zeta(Q_{\mc O_Y[1]}^{\f_\rrr},W^{\f_\rrr}_{\mc O_Y[1]})^A |  \ , \]
in keeping with Corollary \ref{fpcoro}, which was written following \emph{loc. cit.}.

\begin{eg} Let $Y=\bb A^3$ and $r=1$, so that the framed quiver with potential $(Q^{\f_1}_{\mc O_{\bb A^3}},W^{\f_1}_{\mc O_{\bb A^3}})$ is given by that on the left in Equation \ref{DTquiverseqn}. It is well known that for a generic stability condition $\zeta$, the corresponding moduli space of $\zeta$-stable representations such that $\dim V_0=n$ is given by
\[\mf M_n^\zeta(Q^{\f_1}_{\mc O_{\bb A^3}},W^{\f_1}_{\mc O_{\bb A^3}})= \Hilb_n(\bb A^3) \ , \]
the Hilbert scheme of points on $\bb C^3$. The set of $T$-fixed points $\Fp_n$ of $\Hilb_n(\bb A^3)$ is labeled by the set of length $n$ plane partitions, so that we have
\[ \bb V_{\bb C^3}^\zeta \cong \bigoplus_{n\in \bb N}\bigoplus_{\lambda \in \Fp_n } F_\lambda  \quad\quad \text{where} \quad\quad F_\lambda= F \]
is a copy of the base field $F$ for each $n\in \bb N$ and $\lambda \in\Fp_n$.
Further, the results of \cite{BehFan} in this case imply that the DT series is given by
\[ \mc Z_{\bb C^3}^\zeta(q_0) = \prod_{k=1}^\infty \frac{1}{(1-(-q_0)^k)^k} \quad   \in \bb Z[\![q_0]\!] \ . \]
We also introduce notation for the MacMahon function
\[ M(x,q)= \prod_{k=1}^\infty \frac{1}{(1-xq^k)^k} \quad\quad\text{and}\quad\quad M(q)= M(1,q)= \prod_{k=1}^\infty \frac{1}{(1-q^k)^k} \ , \]
so that we can write simply
\[\mc Z_{\bb C^3}^\zeta(q)=M(q) \quad\quad \text{for}\quad\quad q=-q_0 \ .\]
\end{eg}

\begin{eg} The approach to studying the DT series of more general threefolds via quivers with potential was pioneered by Szendroi in the seminal paper \cite{Sz1} in the case $Y_{1,1}=|\mc O_{\bb P^1}(-1)^{\oplus 2}|$, and the framed quiver with potential considered in \emph{loc. cit.} is that on the right in Equation \ref{DTquiverseqn}. In this case, it was conjectured in \emph{loc. cit.} based on extensive computational evidence and proved in \cite{Young1} that the non-commutative DT series of $Y_{1,1}$ is given by
\[ \mc Z^{\NCDT}_{Y_{1,1}}(q_0,q_1) = \prod_{k=1}^\infty  \frac{(1+q_0^k(-q_1)^{k-1})^k(1+q_0^k(-q_1)^{k+1})^k}{(1-q_0^k(-q_1)^k)^{2k} }   \ . \]
If we let $q=-q_0q_1$ and $x=q_1$ then this expression can equivalently be written
\begin{equation}\label{conifoldDTeqn}
	\mc Z^{\NCDT}_{Y_{1,1}}(q,x) = \prod_{k=1}^\infty \frac{(1-x^{-1}q^k)^{k}(1-xq^k)^{k}}{(1-q^k)^{2k}} = M(1,q)^2  M(x^{-1},q)^{-1} M(x,q)^{-1}  \ .
\end{equation}
\end{eg}

\begin{eg}
Similar results were obtained in \cite{BrY} for toric quotient singularities. In particular, for $Y_{2,0}=|\mc O_{\bb P^1}\oplus \mc O_{\bb P^1}(-2)|$ with corresponding framed quiver with potential that in the middle of Equation \ref{DTquiverseqn}, the formula for the non-commutative DT series from \emph{loc. cit.} is given by
\begin{equation}\label{A1singDTeqn}
	 \mc Z^{\NCDT}_{Y_{2,0}}(q,x) = \prod_{k=1}^\infty\frac{1}{(1-q^k)^{2k}(1-x^{-1}q^k)^{k}(1-xq^k)^{k}} = M(1,q)^2 M(x^{-1},q)  M(x,q)   \ , 
\end{equation}
where we again let $q=-q_0q_1$ and $x=q_1$.

More generally, for $Y_{m,0}=\tilde A_{m-1}\times \bb A^1$, the analogous formula from \emph{loc. cit.} is given by
\begin{equation}\label{AnsingDTeqn}
		 \mc Z^{\NCDT}_{Y_{m,0}}(q,x_i) = M(1,q)^n \prod_{1\leq a \leq b \leq m-1} M(x_{[a,b]}^{-1},q)  M(x_{[a,b]},q)   \ ,  
\end{equation}
where we let $q=-q_0q_1...q_{m-1}$, $x_i=q_i$ for $i=1,...,m-1$, and $x_{[a,b]}=x_ax_{a+1}\hdots x_{b}$.
\end{eg}

The cohomological Hall algebra of Kontsevich-Soibelman \cite{KS1} was defined so that it naturally acts on the cohomology of DT theory type moduli spaces (see for example the discussion in \cite{Soi} and the review \cite{Szrev}) and in this case the results of Section \ref{coharepsec} simply reproduce this fact:

\begin{corollary} There exists a natural representation 
\begin{equation}\label{cohadtrepeqn}
	 \rho_Y:\mc H(Y)	\to \End_F(\V_Y^\zeta)
\end{equation} of Equation \ref{catDTeqn} is naturally a module for of the cohomological Hall algebra $\mc H(Y)$ on the categorified DT-type series $\bb V^\zeta_Y$ for each compatible stability condition $\zeta \in \bb R^{V_{Q_Y}}$.
\end{corollary}

The extension of this action to a larger associative algebra containing $\mc H(Y)$ as a positive subalgebra in some triangular decomposition was widely anticipated in this setting. The basic idea was already present in the original papers of Lusztig \cite{Lus} and Ringel \cite{Ring}, and was proposed explicitly in the context of cohomological DT theory by Soibelman \cite{Soi} (see also Section 7.3 of the review \cite{Szrev}). Moreover, by the \emph{dimensional reduction} equivalence between the critical and preprojective cohomological Hall algebra  in some examples, established in \cite{YZ1} and the appendix to \cite{RenS}, the results of \cite{Nak2}, \cite{Var} and more recently \cite{SV} and \cite{SV2} suggested a relationship between $\mc H(Y_{m,0})$ and affine Yangian type quantum groups.

Motivated by related considerations in string theory developed in the series of papers \cite{Cos1}, \cite{Cos2}, and \cite{Cos3}, the following conjecture was formulated by Costello. 
\begin{conj}\cite{CosMSRI} \label{Yangactconj} Let $Y_{m,n} \to X_{m,n}$ be a resolution of the affine, toric, Calabi-Yau singularity $X_{m,n}=\{ xy-z^mw^n\}$. Then there exists a natural representation
\[\rho: \mc Y_{-\delta} (\glh_{m|n})  \to \End_F(\V_{Y_{m,n}}^{\zeta_\NCDT})  \ , \]
of the $-\delta$ shifted affine Yangian $\mc Y_{-\delta} (\glh_{m|n})$ of $\gl_{m|n}$ on $\V_{Y_{m,n}}^{\zeta_\NCDT}$, inducing an isomorphism
\[ \rho( \mc Y_{-\delta} (\glh_{m|n})_+) \xrightarrow{\cong}  \rho_{Y_{m,n}} (\mc{SH}(Y_{m,n}))   \ , \]
and such that $\V_{Y_{m,n}}^{\zeta_\NCDT}$ is identified with the vacuum module $V_{m,n}$ for $\mc Y_{-\delta} (\glh_{m|n})$.
\end{conj}

Our expectation is that the proof of this conjecture will follow from Conjecture \ref{quiveryangconj} together with the identification
\[   \mc Y(Y_{m,n}):=\mc Y_{\mc O[1]}(Y_{m,n}) \cong \mc Y_{-\delta} (\glh_{m|n})  \]
so that the induced isomorphism of the positive half with $\mc{SH}$ follows by construction
\[\mc Y_{-\delta} (\glh_{m|n})_+ \cong \mc Y(Y_{m,n})_+ := \mc{SH}(Y_{m,n})  \ . \]

\noindent Indeed, our understanding is that some form of this conjecture is essentially proved along these lines in \cite{LiYam} and \cite{GLY}, and we hope that the present paper will help to develop a more robust translation of their work to the language of geometric representation theory. The conjecture was also checked for $\gl_1$ in \cite{RSYZ2}, following several related results, using precisely the approach described here. Indeed, Section \ref{quiveryangsec} was was written following \emph{loc. cit.}, as well as the references therein and the many indicated here. Closely related results were also obtained in \cite{VarV1}, \cite{Liu1} and \cite{VarV2}.

There are also several more conceptual approaches to understanding the appearance of Yangians in this setting, including \cite{MO} and \cite{DavMein}. We hope to better understand the results presented here in such terms, but this is not addressed in the present work.

As a consequence of Conjecture \ref{Yangactconj}, we expect equality between the non-commutative DT series of the threefold $Y_{m,n}$ and the Poincare
 series of the vacuum module $V_{m,n}$ for $\mc Y_{-\delta} (\glh_{m|n})$:
\begin{corollary} There is a natural grading on $V_{m,n}$ such that
	\[\mc Z_{Y_{m,n}}^{\NCDT}(q) = P_q( V_{m,n})  \ \in \bb Z\lP q\rP \]
where $P_q\in \bb Z\lP q\rP$ denotes the Poincare series.
\end{corollary}

Indeed, in the case $n=0$, the PBW theorem for the affine Yangian of $\spl_n$ proved in \cite{Guay} gives a filtration on $\mc Y(\widehat{\spl}_m)$ such that we identifications of the associated graded
\[ \gr\  \mc Y(\widehat{\spl}_m) \cong \Sym^\bullet( \widehat{\spl}_m[u^{\pm},v]) \quad\quad \text{and}\quad\quad  \gr \ \mc Y_\delta(\widehat{\spl}_m) \cong \Sym^\bullet( \spl_m[t_1,t_2])  \ , \]
where $\widehat{\spl}_m[u^{\pm},v]$ denotes the universal central extension of $\spl_m[u^{\pm 1},v]$  and $\spl_m[t_1,t_2]$ denotes the subalgebra spanned by polynomials in $t_1=u$ and $t_2=u^{-1}v$ with coefficients in $\spl_m$, so that the corresponding associated graded of the vacuum module $V_{m,0}^\spl$ for $\spl_m$ is given by
\[ \gr \ V_{m,0}^{\spl} \cong \bigotimes_{r=0}^\infty \Sym^\bullet( \widehat{\spl}_m[u^{\pm 1}])\otimes_{ \Sym^\bullet( u^{-r}{\spl}_m[u] )} \bb K  \ , \]
the Poincare polynomial of which is given by
\[ P_q(V_{m,0}^\spl) = \prod_{r=0}^\infty \prod_{k=1}^\infty \frac{1}{(1-q^{k+r})^{m^2-1}} = \left( \prod_{k=1}^\infty \frac{1}{(1-q^{k})^{k}}\right) ^{m^2-1}  = M(q)^{m^2-1} \ . \]
Similarly, this implies that in the $\gl_m$ case we have
\begin{equation}\label{DTchieqn}
	 P_q(V_{m,0})=M(1,q)^{m^2} \ , 
\end{equation}
which evidently matches the formula from Equation \ref{AnsingDTeqn} specialized to $x_i=1$. Moreover the additional refinement given by the variables $\{x_i\}$ agrees with the refinement by the grading on $V_{m,0}$ under the Cartan $\mf{h}\subset \spl_m$.

More generally, we conjecture that the modules induced by the construction of Theorem \ref{cohareptheo}, and its extension outlined in Section \ref{quiveryangsec}, in the case of Example \ref{DTmoduleeg} leads to a family of modules
\[\V_{\alpha,\beta,\gamma}^\zeta = \bigoplus_{d\in \bb N^{V_{Q_Y}}} H^A_\bullet(\mf M_\dd^{\zeta}(Q_{M}^{\f_{\alpha,\beta,\gamma}}),\varphi_{\overline{W}_{M,\dd}^{\f_{\alpha,\beta,\gamma}}}) \otimes_{H_A^\bullet(\pt)} F \]
over the affine Yangian of $\gl_1$ which were constructed in the closely related setting of modules over the quantum toroidal algebra in \cite{FJMM}, where $f_{\alpha,\beta,\gamma}$ is the framing structure on the extended quiver described in Example \ref{DTmoduleeg}. Their natural analogue in this setting is the following:

\begin{conj}\label{DTmodconj} There exists a natural representation
	\[ \rho : \mc Y_{-\delta}(\glh_1) \to \End_F( \V^{\zeta_\NCDT}_{\alpha,\beta,\gamma} ) \]
such that $\V^{\zeta_\NCDT}_{\alpha,\beta,\gamma}$ is identified with the cohomological variant of the MacMahon module $\mc M_{\alpha,\beta,\gamma}(u)$ constructed in \cite{FJMM}.
\end{conj}

\subsection{Perverse coherent extensions of divisors, Vafa-Witten theory, and vertex algebras}\label{VWsec}
In this section, we explain the application of our results in the case that $M$ is given by the structure sheaf of a divisor. Let $S$ be an effective toric Cartier divisor in $Y$, $S^\red$ the underlying reduced divisor, and $\mf D_S$ its set of irreducible components, so that we have
\[ S^\red= \bigcup_{d\in \mf D_S} S_d \quad\quad\quad \mc O_{S^\red}^\sss=\bigoplus_{d\in \mf D_S} \mc O_{S_d} \quad\quad \text{and}\quad\quad [S]=\sum_{d\in \mf D_S} r_d [S_d]  \]
for some multiplicities $r_d\in \bb N$ defined for each $d\in \mf D_S$. Let $Q_{\mc O_{S^\red}^\sss[1]}$ denote the extended quiver corresponding to the object $M=\mc O_{S^\red}^\sss[1]$ and note that the data of a Jordan-Holder filtration on $\mc O_S$ with subquotients given by the objects $\mc O_{S_d}$ determines a framing structure $\f_{S}$ of rank $\rr_S=(r_d)_{d\in \mf D_S}$, where we identify $\mf D_S$ with the set of framing nodes of the extended quiver $Q_S$. Thus, we can consider the stack of $\f_S$-framed perverse coherent extensions of $\mc O_{S^\red}^\sss[1]$
\[  \mf M(Y,S):= \mf M^{\f_{S}}(Y,\mc O_{S^\red}^\sss[1])  \quad\quad \text{as well as} \quad\quad  \mf M^0(Y,S):= \mf M^{0_{S}}(Y,\mc O_{S^\red}^\sss[1]) \ , \]
the stack of trivially framed perverse coherent extensions of  $\mc O_{S^\red}^\sss[1]$ of rank $\rr_S$, and their corresponding framed quivers with potential $(Q^{\f_S},W^{\f_S})$ and $(Q^{0_S},W^{0_S})$. We have seen in Examples \ref{ADHMeg} to \ref{NYeg} that certain stable loci $\mf M^\zeta(Y,S)$ or $\mf M^{0,\zeta}(Y,S)$ in these stacks provide models in algebraic geometry for moduli spaces of framed instantons on $S^\red$ of rank $\rr_S$, that is, of rank $r_d$ on the irreducible component $S_d$ for each $d\in \mf D_S$. Indeed, this gives the desired generalization of the ADHM construction described in Section \ref{agintrosec} of the introduction.

In particular, in the case $Y=\C^3$ and $S=S_{r_1,r_2,r_3}=r_1[\C^2_{yz}]+r_2[\C^2_{xz}]+r_3[\C^2_{yz}]$, applying Theorem \ref{stackpotthm} as in Example \ref{spikedeg} implies the desired description in algebraic geometry of the stack $\mf M_{\textup{Nek}}^{r_1,r_2,r_3}(\C^3)$ of rank $\rr =(r_1,r_2,r_3)$ representations of the spiked instantons quiver with potential studied in \cite{NekP}:

\begin{theo} There is an equivalence of algebraic stacks
	\[ \mf M_{\textup{Nek}}^{r_1,r_2,r_3}(\C^3) \xrightarrow{\cong }  \mf M^{0}(\bb C^3,S_{r_1,r_2,r_3})  \ . \]
\end{theo}

Similarly, the special cases of this construction in Examples \ref{KNeg}, \ref{NYeg}, and \ref{chainsaweg} give `three dimensional' variants of the corresponding cases of constructions of \cite{KrNak}, \cite{NY1}, and \cite{FinR}, respectively, analogous to the spiked instantons variant of the ADHM construction. In general, in the case $Y=Y_{m,0}$, this construction conjecturally gives the analogous variant of the relationship between rank $m$, parabolic torsion free sheaves on $\bb P^1\times \bb P^1$ and chainsaw quivers from \emph{loc. cit.}. For divisors $S$ in the spaces $Y_{m,n}$, we obtain a generalization of this variant of their construction to $\gl_{m|n}$. We discuss examples of this form in detail following Conjecture \ref{chainsawconj} below, and in the succeeding Section \ref{MVsec}.

For each choice of compatible stability condition $\zeta \in \bb R^{V_{Q_Y}}$, the corresponding homology module $\bb V_S=\bb V^{\f_S,\zeta}(Y,\mc O_{S^\red}^\sss[1])$ defines a model for a rank $\rr_S$ cohomological Vafa-Witten type series of $S^\red$, given by
\begin{equation}\label{catVWeqn}
\bb V_S^\zeta= \bigoplus_{\dd \in \bb N^{V_{Q_Y}}} H_\bullet^{G_\dd(Q_Y)\times A}( X_\dd^{\zeta}(Q^{\f_S}),\varphi_{W^{\f_S}})\otimes_{H_A^\bullet(\pt)} F  \ , 
\end{equation}
or equivalently, again under the hypotheses of Section \ref{invtsec} which hold in our examples of interest,
\[\bb V_S^\zeta= \bigoplus_{\dd \in \bb N^{V_{Q_Y}}} H_\bullet^{A}( \mf M_\dd^{\zeta}(Q^{\f_S}),\varphi_{\overline{W}^{\f_S}})\otimes_{H_A^\bullet(\pt)} F  \  . \]

We also define the corresponding generating functional, as in Equation \ref{genfuneqn}, by
\[ \Zz_S^\zeta(\qq) = \sum_{\dd \in \bb N^{V_{Q_Y}}}\qq^\dd \chi(\bb V^{\f_S,\zeta}_\dd(\mc O_{S^\red}^\sss[1]))  \ \in \Pq \ ,\]
where we recall the shorthand notation $\qq^\dd=\prod_{i\in V_{Q_Y}} q_i^{d_i}$ and $\Pq= \bb Z[\![q_i]\!]_{i\in V_{Q_sY}} $. In particular, we define
\[ \V_S = \V^{\f_S,\zeta_\VW}(\mc O_{S^\red}^\sss[1])  \quad\quad \text{and}\quad\quad  \Zz_S^{\VW}(\qq) = \sum_{\dd \in \bb N^{V_{Q_Y}}}\qq^\dd \chi(\bb V_{S,\dd}))  \ \in \Pq \ ,\]
for $\zeta=\zeta_\VW$ the Vafa-Witten stability condition, as well as
\[\V^0_S=\V^{0_S,\zeta}(\mc O_{S^\red}^\sss[1])  \quad\quad \text{and}\quad\quad \mc Z_S^{0,\zeta}(\qq) = \sum_{\dd \in \bb N^{V_{Q_Y}}}\qq^\dd \chi(\bb V^0_{S,\dd}))  \ \in \Pq \ , \]
corresponding to the trivial framing structure $0_S$ of rank $\rr_S$.

Note that again the hypotheses necessary to calculate the generating functional in terms of fixed point counts as in Corollary \ref{fpcoro} will hold in the main examples of interest, as we will use below:

\begin{eg}\label{VWAGTeg} Let $Y=\C^3$ and $S^\red=\bb C^2_{xy}$ as in Example \ref{ADHMeg}. For $S=S^\red$ and generic $\zeta$, the moduli space of $\zeta$-stable, trivially-framed perverse coherent extensions of $M=\mc O_{S^\red}^\sss[1]$ of rank 1 and dimension $\dd=n$ corresponds under dimensional reduction to the Hilbert scheme of $n$ points on $\bb C^2$, so that we have
\begin{equation}\label{vacC2eqn}
		 \V_{\C^2} = \bigoplus_{n\in\bb N} H_\bullet^{A}(\Hilb_n(\bb C^2)) \otimes_{H^\bullet_A(\pt)} F  \cong  \bigoplus_{n\in\bb N}  \bigoplus_{\lambda\in \Fp_n}  F_\lambda \ 
\end{equation}
	where $\Fp_n$ denotes the set of $A$-fixed points of $\Hilb_n(\bb C^2)$, which is in bijection with the set of partitions $\lambda$ of length $n$, and $F_\lambda$ denotes a copy of the field of fractions of $H_A^\bullet(\pt)$. It follows that the corresponding generating function is given by
\begin{equation}\label{VWC2eqn}
		 \mc Z_{\C^2}^\VW(q) = \prod_{k=1}^\infty \frac{1}{1-q^k} = \eta(q)^{-1}   \  .
\end{equation}
More generally, for $S=r[\bb C^2]$ the module for the trivial framing $\f=0$ is given by
\begin{equation}\label{vacC20reqn}
	 	\bb V_{r [\bb C^2]}^0   \cong  \bigoplus_{n\in \bb N} H_\bullet^A(M(r,n)) \otimes_{H^\bullet_A(\pt)} F 
	 	\cong \bigoplus_{n\in \bb N}  \bigoplus_{n_1+\hdots + n_r = n}  \bigotimes_{j=1}^r H_\bullet^{A}(\Hilb_{n_j}) \otimes_{H^\bullet_A(\pt)} F 
	 	\cong \bigotimes_{j=1}^r \V_{\C^2} 
\end{equation}
and thus the generating function for the trivially framed extensions is given by that for $r$-tuples of partitions of total length $n$:
\begin{equation}\label{VWr0C2eqn}
 \mc Z^0_{r [ \C^2]}(q) = \prod_{j=1}^r \prod_{k=1}^\infty \frac{1}{1-q^k} = \eta(q)^{-r}    \ .
\end{equation}

For $\f=\f_S$, again following Example \ref{ADHMeg}, the fixed points of the moduli spaces are restricted from arbitrary $r$-tuples of partitions to \emph{nested partitions}, those which satisfy the additional requirement that each partition is a strict subset of the previous; this was studied in \cite{CCDS} in precisely this setting and our examples of this form were inspired by their work.

Under the identification
\[ \bb V_{r [\bb C^2]}^0 \cong \bigoplus_{n\in \bb N}  \bigoplus_{n_1+\hdots + n_r = n} \bigotimes_{j=1}^r \left(  \bigoplus_{\lambda_j \in \Fp_{n_j}}  F_{\lambda_j} \right) \cong  \bigoplus_{n\in \bb N}  \bigoplus_{n_1+\hdots + n_r = n}  \bigoplus_{\lambda_1 , \hdots , \lambda_r; \lambda_j \in \Fp_{n_j}}  F_{\lambda_1,...,\lambda_r}  \]
the module $\V_{r[\C^2]}$ is given by the submodule
\begin{equation}\label{VWrC2eqn}
\bb V_{r[\C^2]} \cong  \bigoplus_{n\in \bb N}  \bigoplus_{n_1+\hdots + n_r = n}  \bigoplus_{\lambda_1 \geq \hdots \geq \lambda_r; \lambda_j \in \Fp_{n_j}}  F_{\lambda_j}  \quad\quad \text{and thus} \quad\quad \mc Z^\VW_{r[\C^2]}(q) = \prod_{j=1}^r \prod_{k=1}^\infty \frac{1}{1-q^{j+k}}  \ . 
\end{equation}

\end{eg}

\begin{eg}\label{VWpiteg} For $Y=\C^3$ and $S=S_{M,N,0}= M[\C^2_{xy}]+N[\C^2_{yz}]$, as in the special case of Example \ref{spikedeg} corresponding to the quiver in Equation \ref{spikedfrmeqn}, we expect that the set of $A$ fixed points of $\mc M(Y,S)$ is in bijection with the set of plane partitions with a pit at location $(M,N)$ and trivial asymptotics, in the sense of \cite{BerFM}.
\end{eg}

\begin{eg}\label{VWpitasseg}
More generally, let $Y=\C^3$, $S=S_{M,N,0}= M[\C^2_{xy}]+N[\C^2_{yz}]$ and
\[ M = \mc O_{S^\red}^\sss[1] \oplus \mc O_{\bb C_x}\oplus \mc O_{\C_y} \oplus \mc O_{\C_z} \ , \]
following Example \ref{VWmoduleeg}, with the framing structure given by that described in \emph{loc. cit.}, and we expect the fixed points of the resulting moduli space to correspond to plane partitions with a pit at location $(M,N)$ and asymptotics $\alpha,\beta,\gamma$ determined as in \emph{loc. cit.}. A special case of the resulting family of quivers is given by the quiver of Equation \ref{VWmodquivereqn}.
\end{eg}

\begin{eg}\label{NYcounteg} Let $Y=|\mc O_{\bb P^1}(-1)\oplus\mc O_{\bb P^1}(-1)|$ and $S=|\mc O_{\bb P^1}(-1)|$ as in Example \ref{NYeg}. The vanishing cycles cohomology defining $\V_S$ is equivalent to the ordinary Borel-Moore homology of the dimensional reduction, and thus by localization we have an isomorphism
	\begin{equation}\label{vacOm1eqn}
 \V_{|\mc O_{\bb P^1}(-1)|} \cong \bigoplus_{k\in\bb Z} \bigoplus_{n\in\bb N}\bigoplus_{n_0+n_1=n}\bigoplus_{\lambda_0\in \Fp_{n_0}, \lambda_1\in \Fp_{n_1}} F_{k,\lambda_0,\lambda_1}  
	\end{equation}
by Proposition 3.2 and Theorem 3.4 in \cite{NY0}, where $\Fp_n$ denotes the set of partitions of $n$.
	
Similarly, Corollary 5.7 of \cite{NY2} applied in the limit $m=\infty$ to each term $k=c_1\in \bb Z$, the partition function is given by the following formula of Corollary 3.15 in \cite{NYLec}
\begin{equation}\label{VWOm1eqn}
	 \mc Z^\VW_{|\mc O_{\bb P^1}(-1)|} (q)= \sum_{k\in \bb Z} q^{\frac{k^2}{2}} \frac{1}{\prod_{k=1}^\infty (1-q^k)^2} \ . 
\end{equation}
\end{eg}

The representation of the cohomological Hall algebra $\mc H(Y)$ on $\V_S$ constructed in Theorem \ref{cohareptheo}, and its extension to the action of the shifted quiver Yangian $\mc Y_M(Y)$ outlined in Section \ref{quiveryangsec}, were essentially defined to generalize the construction of Grojnowski \cite{Groj} and Nakajima \cite{Nak} of the Heisenberg algebra action on the Hilbert scheme of points, following the Schiffmann-Vasserot proof \cite{SV} of the AGT conjecture \cite{AGT} and their respective generalizations in \cite{RSYZ} and \cite{GaiR}. We note that these results are also closely related to those obtained in unpublished work of Feigin-Tsymbaliuk, as explained in \cite{Ts1}, and in turn to many results in the setting of quantum toroidal algebras such as \cite{FFJMM} and \cite{BerFM}; we will elaborate on this relationship further in Section \ref{MVsec} below.

The proposal that there should be analogous vertex algebras corresponding to more general divisors $S$ and threefolds $Y$ was considered already in the original paper of Gaiotto-Rapcak \cite{GaiR}, and explored in some detail in \cite{PrR}, but the analogue of the AGT conjecture in this setting was not known in general, as the relevant moduli spaces generalizing the spiked instantons construction of Nekrasov to threefolds $Y$ other than $\C^3$ had not been constructed previously.

There is also a closely related conjecture of Feigin-Gukov \cite{FeiG} that there should exist vertex operator algebras $\text{VOA}[M_4,\gl_r]$ associated to four manifolds $M_4$, analogously generalizing the AGT conjecture. This appears to coincide with the predictions of \cite{GaiR} discussed in the preceding paragraph in the case that the underlying reduced scheme $S^\red$ is irreducible and smooth, $M_4$ is the analytification of $S^\red$, and $r\in \bb N$ the multiplicity of $S^\red$ in $S$.

However, a mathematical definition of these vertex algebras was also not known in general, for a divisor $S$ nor four-manifold $M_4$, and relatively few examples were known in the non-abelian case. In the companion paper \cite{Bu}, we give a general combinatorial construction of vertex algebras $\V(Y,S)$ as the kernel of screening operators acting on lattice vertex algebras determined by the data of the GKM graph of $Y$ and a Jordan-Holder filtration of $\mc O_S$ with subquotients structure sheaves $\mc O_{S_d}$ of the divisors $S_d$ occurring as irreducible, reduced components of $S$. This construction appears to satisfy the predictions of \cite{GaiR}, \cite{PrR}, and \cite{FeiG}, and in particular we formulate the following analogue of the AGT conjecture in this setting:

\begin{conj}\label{VWconj} There exists a natural representation
	\[\rho: \mc U(\V(Y,S)) \to \End_F(\V_S )\]
of the algebra of modes $\mc U(\V(Y,S))$ of the vertex algebra $\V(Y,S)$ on $\V_S$, inducing an isomorphism
\[ \rho( \mc U(\V(Y,S))_+) \xrightarrow{\cong} \rho_S(\mc{SH}(Y)) \]
and such that $\V_S$ is identified with the vacuum module for the vertex algebra $\V(Y,S)$.
\end{conj}

While we are not currently able to give a proof in general, note that the conjectural identification with the action of $\mc{SH}(Y)$ implies that Theorem \ref{cohareptheo} together with Proposition \ref{factprop} provide the desired construction of the action of the algebra of modes. We expect that the proof will follow from the compatibility between the free field realizations of $\V(Y,S)$ used in their construction in \cite{Bu}, and a family of coproducts on the corresponding shifted quiver Yangians $\mc Y_S(Y)$, under equivalences of the type conjectured in the following Section \ref{MVsec}. We hope to explore this further in future work.

As a consequence of Conjecture \ref{VWconj}, we expect equality between the Vafa-Witten-type series $\mc Z^\VW_S(q)$ of the divisor $S$ defined above and the Poincare series $P_q$ of the vacuum module $\V_{(Y,S)}$ of the vertex algebra $\V(Y,S)$.

\begin{corollary}\label{VWcoro} There is a natural grading on $\V_{(Y,S)}$ such that
	\[ \mc Z^\VW_S(q) =  P_q(\V_{(Y,S)})  \ \in \bb Z \lP q \rP  \ .\]
\end{corollary}

Indeed, the above conjecture and corollary can be verified directly in several of the simplest examples, following the previous work on the subject mentioned above:

\begin{eg}\label{NakGrojeg} In the case $Y=\C^3$ and $S=\C^2$, the constructions we have outlined reduce to the original constructions of Grojnowski \cite{Groj} and Nakajima \cite{Nak}. The vertex algebra
\[\V(\C^3,\C^2) = \pi_k\]
is simply the Heisenberg algebra $\pi_k$ at level $k=-\frac{1}{\hbar_1 \hbar_2}$, the module $\V_{\C^2}$ of Equation \ref{vacC2eqn} is evidently isomorphic to the standard Fock module, and the vacuum character is given by $\eta(q)^{-1}$, which is indeed the Vafa-Witten-type series for $\C^2$ of Equation \ref{VWC2eqn} above.
\end{eg}

\begin{eg}\label{latticeeg}
More generally, the constructions of \emph{loc. cit.} generalize to smooth surfaces $S$ to give the Heisenberg algebra $\pi(S)$ on the cohomology $H_A^\bullet(S;\C)$ over the base field $F$; we have included $A$-equivariance following \cite{NakMLec} in keeping with our present setting. Thus, the vacuum module of $\pi(S)$ is given by
\[  \pi_S \cong \Sym_F^\bullet( z^{-1} (H^A_\bullet(S)\otimes_{H_A^\bullet(\pt)} F) [z^{-1}]) \quad\quad \text{so that}\quad\quad P_q(\pi_S) = \eta(q)^{-\chi(S)}  \ , \]
by the formula of \cite{Gtt}. Indeed, we show in \cite{Bu} that the Heisenberg algebra $\pi(S)$ canonically embeds in the vertex algebra $\V(Y,S)$, but for $H_2(S;\Z)\neq 0$ it is a strict subalgebra.

In fact, we explain in \emph{loc. cit.} that in the rank 1 case, corresponding to our assumption that $S$ is a smooth algebraic surface and in particular reduced when considered as a divisor, the vertex algebra $\V(Y,S)$ is simply the lattice vertex algebra generated by $H_2(S;\Z)$ equipped with negative the intersection pairing, tensored with the Heisenberg algebra generated by $H_0(S)$,
\[ \bb V(Y,S) =  V_{H_2(S;\Z)} \otimes \pi_{H^0(S;\C)}  \ .\]
The additional generators of the lattice vertex algebra correspond to the Hecke modifications of the sheaves along the compact curve classes in $H_2(S;\Z)$, as in the construction outlined in the final chapter of \cite{NakLec}. This is also related to the results of \cite{Nak1} and \cite{GKV}.

We note that the construction of Negut \cite{Neg1} provides a higher rank analogue of the construction of Grojnowski-Nakajima, generalizing Example \ref{AGTeg} below to more general surfaces $S$, but the resulting vertex algebras analogously fail to include the lattice-type extensions corresponding to Hecke modifications along curve classes, which our construction conjecturally provides.
\end{eg}

\begin{eg} As a special case of the preceding Example \ref{latticeeg}, let $Y=|\mc O_{\bb P^1}(-1)\oplus\mc O_{\bb P^1}(-1)|$ and $S=|\mc O_{\bb P^1}(-1)|$ following Examples \ref{NYcounteg} and in turn \ref{NYeg}. Indeed, we have an isomorphism between the vacuum module of Equation \ref{vacOm1eqn} and that of a canonically normalized rank 1 lattice vertex algebra $V_{\bb Z}$ tensored with a Heisenberg algebra,
\[ \bb V_{|\mc O_{\bb P^1}(-1)|} \cong V_{\bb Z} \otimes \pi \quad\quad \text{and in particular}\quad\quad   P_q( V_{\bb Z} \otimes \pi)= \sum_{k\in \bb Z} q^{\frac{k^2}{2}} \frac{1}{\prod_{k=1}^\infty (1-q^k)^2} = \mc Z^\VW_{|\mc O_{\bb P^1}(-1)|}(q) \ ,\]
we obtain equality of the Poincare polynomial of the vacuum module with the Vafa-Witten type series of $|\mc O_{\bb P^1}(-1)|$ given in Equation \ref{VWOm1eqn}. Moreover, we have checked by direct calculation that the conjecture holds following our proposed construction in this example, using results of \cite{NY0}, \cite{NY1}, \cite{NY2}; this result will appear in future work.
\end{eg}

\begin{eg} \label{AGTeg} For $Y=\C^3$ and $S=r[\C^2]$, as in Example \ref{VWAGTeg}, the definition of the vertex algebra $\V(Y,S)$ from \cite{Bu} is simply the Feigin-Frenkel free field realization of the principal affine $W$-algebra of $\gl_r$, so that we have
\[ \V(\C^3,r[\C^2])= W^\kappa_{f_{\textup{prin}} }(\gl_r)\cong W^\kappa_{f_{\textup{prin}} }(\spl_r)\otimes \pi \ , \] 
where we use the notation $W^\kappa_{f_{\textup{prin}} }(\gl_r)$ to denote the vertex algebra over the field of fractions of $H_{\tilde T}^\bullet(\pt)$ defined by \[ \kappa = -h^{\vee} - \frac{\hbar_2}{\hbar_1} \ . \] In this case, our proposed construction was defined to reproduce that of \cite{SV} used in the proof of the AGT conjecture \cite{AGT}, and moreover the variant thereof studied in \cite{CCDS} which produces the vacuum module, as in the statement of Conjecture \ref{VWconj}. In \cite{SV} the authors prove the following theorem:

\begin{theo}\label{SVtheo}\cite{SV} There exists a natural representation
\[ \mc U(W^\kappa_{f_{\textup{prin}} }(\gl_r)) \to \End(\V_{r[\C^2]}^0)  \ ,\]
such that $\V_{r[\C^2]}^0$ is identified with the universal Verma module $\bb M_r$ for $W^\kappa_{f_{\textup{prin}} }(\gl_r)$.
\end{theo}

In \cite{CCDS}, the authors extend these results to give a geometric construction of the vacuum module as the cohomology of the moduli space of stable representations of a variant of the ADHM construction, for which the quiver is precisely that corresponding to taking $A_\f$ equal to a principal nilpotent in Equation \ref{ADHM3deqn} of Example \ref{ADHMeg}. Thus, in the notation we have introduced, they prove:

\begin{theo}\label{CCDStheo}\cite{CCDS} There exists a natural representation
	\[ \mc U(W^\kappa_{f_{\textup{prin}} }(\gl_r)) \to \End(\V_{r[\C^2]})  \ ,\]
	such that $\V_{r[\C^2]}$ is identified with the vacuum module for $W^\kappa_{f_{\textup{prin}} }(\gl_r)$.
\end{theo}

In particular, we have the identification
 \[ P_q(W^\kappa_{f_{\textup{prin}} }(\gl_r))  = \prod_{j=1}^r \prod_{k=1}^\infty \frac{1}{1-q^{k+j}}=\mc Z^\VW_{r[\C^2]}(q) \ ,\]
as desired, where we recall that $\V_{r[\C^2]}$ and $\mc Z^\VW_{r[\C^2]}(q)$ are given by the expressions of Equation \ref{VWrC2eqn}, and we have used the usual abuse of notation denoting the vacuum module by the underlying vertex algebra $ W^\kappa_{f_{\textup{prin}} }(\gl_r)$.

These results are also closely related to the cohomological variant of \cite{FFJMM}, as well as \cite{BerFM} in the case $n=0$. In particular, note that the fixed point counting underlying the decomposition of $\V_{r[\C^2]}$ in Equation \ref{VWrC2eqn} is precisely that considered in \emph{loc. cit.}.
\end{eg}

\begin{eg}\label{GReg} More generally, for $Y=\C^3$ and $S=S_{M,N,0}= M[\C^2_{xy}]+N[\C^2_{yz}]$ as in Example \ref{VWpiteg}, the vertex algebras $\V(\C^3,S_{M,N,0})$ are the Gaiotto-Rapcak $Y$ algebras $Y_{N,0,M}$ of \cite{GaiR}, and the construction from \cite{Bu} reduces to their definition in \cite{RSYZ}. Moreover, \emph{loc. cit.} established the variant of Conjecture \ref{VWconj} proved in \cite{SV}, related to the Verma module rather than the vacuum, for $Y=\C^3$ and arbitrary $S_{M,N,L}$. Their construction was a significant source of inspiration for the results of this paper. Analogous results were also established in the setting of quantum toroidal algebras for $L=0$ in \cite{BerFM}.
	
Following Example \ref{VWpiteg}, the corresponding quivers are given by the special case of Example \ref{spikedeg} in Equation \ref{spikedfrmeqn}, with $A_3^2=A_2^3=0$ and $A_2,A_3$ principal nilpotents in $\gl_{M}$ and $\gl_{N}$, respectively, and the fixed points are in correspondence with the plane partitions with a pit at location $(M,N)$ of \cite{BerFM}, with trivial asymptotics $\alpha=\beta=\gamma=\O$.
\end{eg}

More generally, in analogy with Conjecture \ref{DTmodconj}, we expect that the modules induced by the construction of Theorem \ref{cohareptheo} in the case of Example \ref{VWpitasseg} where $S=S_{M,N,0}=M[\C^2_{xy}]+N[\C^2_{yz}]$ define certain canonical families of modules
\[\V_{\alpha,\beta,\gamma}^{M,N} = \bigoplus_{\dd\in \bb N^{V_{Q_Y}}} H_\bullet^A(\mf M_\dd^{\f_{S;\alpha,\beta,\gamma},\zeta_\VW}(Q_{M}),\varphi_{\overline{W}_{M,\dd}^{\f_{S;\alpha,\beta,\gamma}}})\otimes_{H_A^\bullet(\pt)} F  \]
over the vertex algebras $\V(\C^3,S_{M,N,0})$, with fixed point bases enumerated by plane partitions with a pit and fixed asymptotics as in \cite{BerFM}, determined by the framing structure $\f_{S;\alpha,\beta,\gamma}$ as in Example \ref{VWmoduleeg}. In particular, we make the following conjecture:

\begin{conj}\label{VWmodconj} There exists a natural representation
\[ \rho: \mc U(\V(\C^3;S_{M,N,0})) \to \End_F(\V_{\alpha,\beta,\O}^{M,N} )   \]
of the algebra $\mc U(\V(\C^3;S_{M,N,0}))$ on $\V_{S;\alpha,\beta,\O}$, identifying it with the module $\mc N_{\alpha,\beta,\O}^{M,N}$ of \emph{loc. cit.}.
\end{conj}

We now describe some applications of our results to examples of divisors $S$ in a resolution $Y_{m,n}$ of $X_{m,n}=\{xy-z^mw^n\}$. Let $\mu$ and $\nu$ be partitions of length $m$ and $n$, respectively
\[\mu=\{ \mu_1 \geq \hdots \geq \mu_m \geq 0 \} \quad \quad\text{and}\quad\quad \nu =\{\nu_1 \geq \hdots \geq \nu_n \geq 0 \} \]
and define the corresponding lists of integers
\[ \textbf{M} =(M_i)_{i=0}^{m-1} \quad \textbf{N} =(N_i)_{i=0}^{n-1}  \quad \quad \text{where} \quad \quad M_i=\sum_{k=i+1}^{m} \mu_k \quad N_i=\sum_{k=i+1}^{n} \nu_k \]
 $i=0,...,m-1$, $j=0,...,n-1$; we also write just
 \[M = M_0 = \sum_{k=1}^m \mu_k \quad\quad \text{and}\quad\quad N=N_0=\sum_{k=1}^n \nu_k  \ . \]
 We define $S_{\mu,\nu}$ as the toric divisor corresponding to the labeling of the faces of the moment polytope of $Y_{m,n}$ by the integers $M_i$ and $N_i$ depicted in Figure \ref{fig:Smunufig} below in the case $m=3,n=2$. We also write simply $S_\mu$ or $S_\nu$ if $n=0$ or $m=0$, respectively.

 \begin{figure}[b]
 	\begin{overpic}[width=1.0\textwidth]{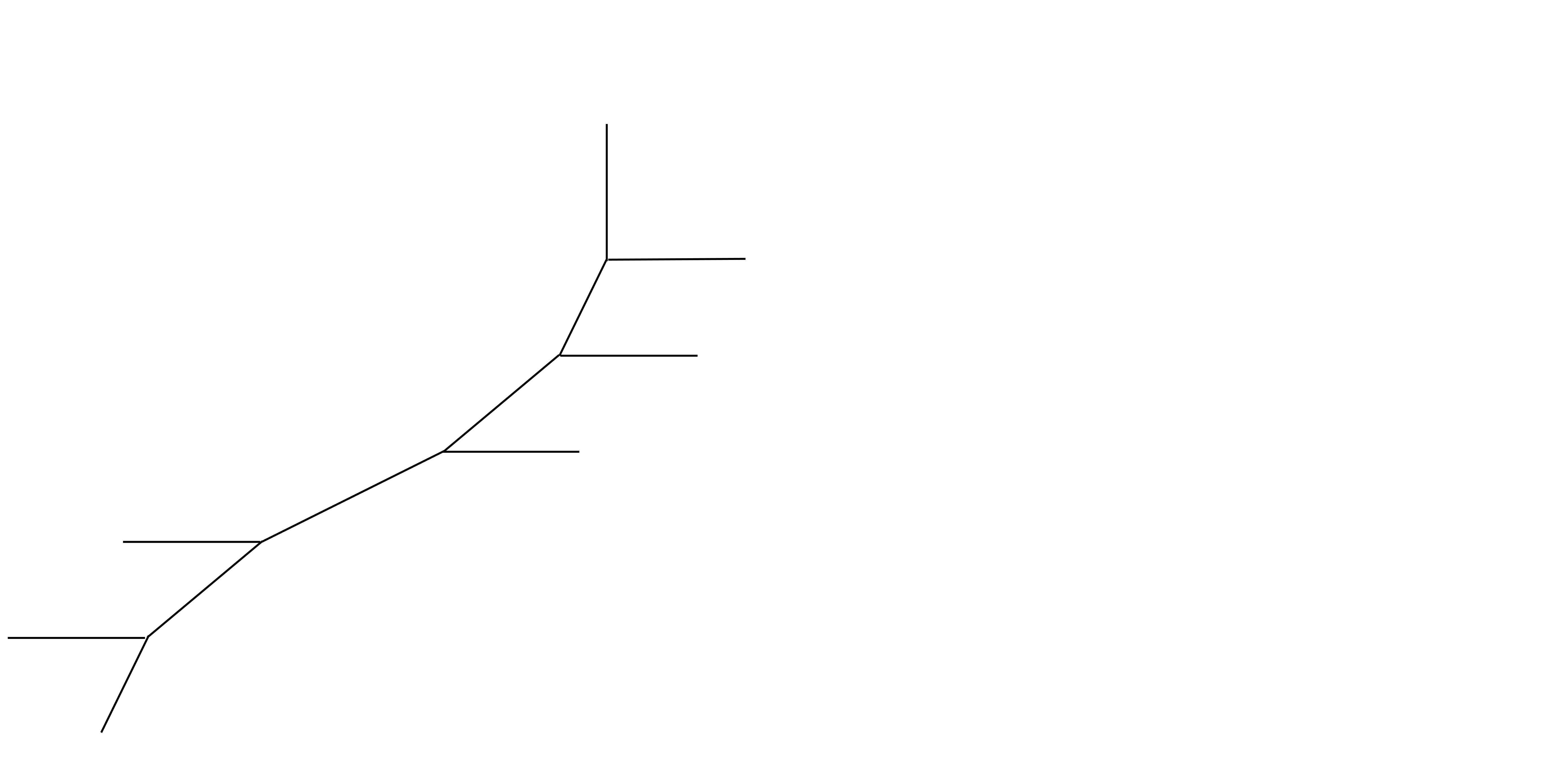}
 		\put(0,43) {$ Y_{3,2}\to X_{3,2} = \{ xy-z^3w^2 \}$}
 		\put (39.5,35) {$M_0= \mu_1+\mu_2+\mu_3$} 
 		\put (37.5,28.5) {$M_1= \mu_2+\mu_3$} 
 		\put(32, 22) {$M_2=\mu_3$} 
 		\put (2,10.5) {$N_1=\nu_2$}
 		\put (4, 16.5) {$N_0=\nu_1+\nu_2$}
 		
 		\put(34,30) {$C_1$}
 		\put(29,25) {$C_2$}
 		\put(23,17) {$C_3$}
 		\put(14,11.5) {$C_4$}
 		
 		\put(70,22) {$ \mleft[ \begin{array}{c c c | c c}  &  &   &   &  \\  \alpha_1 & & & &  \\  &\alpha_2   &  &  &   \\
 				\hline &  & \alpha_3 &  &  \\  & &  &   \alpha_4&  \end{array} \mright] $}
 	\end{overpic}
 	\caption{The resolution $Y_{3,2}$, the toric divisor $S_{\mu,\nu}$, and the compact toric curve classes $C_i$ with their corresponding simple roots $\alpha_i$ of $\gl_{3|2}$}
 	\label{fig:Smunufig}
 \end{figure}
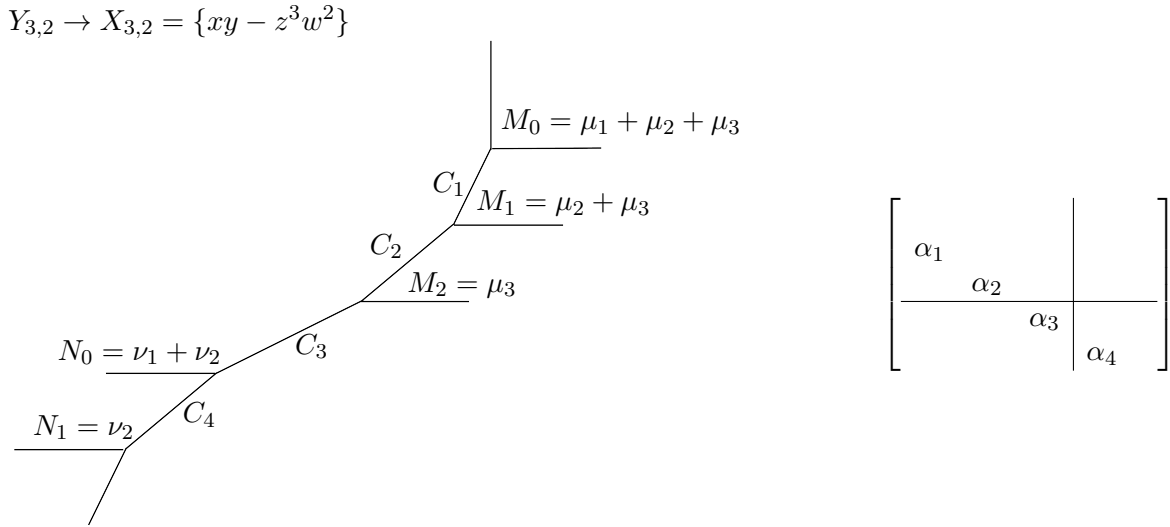
 
 In this setting, following physical predictions from \cite{PrR} we conjecture in \cite{Bu} that the corresponding vertex algebra is given by the affine $W$ algebra of $\gl_{M|N}$ corresponding to the nilpotents $f_\mu$ and $f_\nu$ in $\gl_M$ and $\gl_N$, respectively, determined by $\mu$ and $\nu$:
 
 \begin{conj}\label{WLSconj}\cite{Bu} There is an isomorphism of vertex algebras
 	\[ W_{f_\mu,f_\nu}^\kappa(\gl_{M|N}) \xrightarrow{\cong} \V(Y_{m,n},S_{\mu,\nu})  \ .\]
 \end{conj}

\noindent In fact, we have a proof of this conjecture in the case $n=0$, which will appear in future work:
\begin{theo}\label{WLStheo} There is an isomorphism of vertex algebras
	\[ W^\kappa_{f_\mu}(\gl_M) \xrightarrow{\cong}\V(Y_{m,0}, S_\mu)  \ . \]
\end{theo}

The quivers that arise in this setting are generalizations of the variant of the $\gl_2$ chainsaw quiver considered in Example \ref{chainsaweg}, and in general we conjecture there exists an appropriate framing structure and dimensional reduction relating the resulting module with the ordinary Borel-Moore homology of the moduli space of representations of the corresponding chainsaw quiver. We now formally state this conjecture:

Let $\mf M_\dd^\mu(Q^\textup{Ch}_m,W^\textup{Ch}_m)$ denote the stack of $\dd$ dimensional representations of the $\gl_m$ chainsaw quiver as in \cite{FinR}, with fixed framing dimension given by the integers $\mu_i$, so that following the conjectures of Kanno-Tachikawa \cite{KanT} and in turn \cite{Brav1}, we expect that
\begin{equation}\label{KTequiveqn}
\bb M_{f_\mu} \xrightarrow{\cong}	\bigoplus_{\dd \in \bb N^{V_{Q_Y}}} H^A_\bullet(\mf M_\dd^{\mu,\zeta}(Q^\textup{Ch}_m,I^\textup{Ch}_m))\otimes_{H_A^\bullet(\pt)} F  \ ,
\end{equation}
the $A$-equivariant Borel-Moore homology of the moduli space of stable representations of the $\gl_m$ chainsaw quiver $(Q^\textup{Ch}_m,I^\textup{Ch}_m)$ of framing dimension $\mu$ is identified with the universal Verma module $\bb M_{f_\mu}$ for the affine W-algebra $W_{f_\mu}(\gl_M)$ associated to the nilpotent $f_\mu$ in $\gl_M$.

Let $\f_0$ be the framing structure for $S_\mu$ of rank $\textbf{M}$ generalizing that of Equation \ref{chainsawsl2eq}, that is, such that the map $V_{\infty_{i}}\to V_{\infty_{i-1}}$ is injective for each $i=1,...,m-1$ and all other framing endomorphisms are zero. We denote the corresponding module by
\[ \V_{S_\mu}^{\f_0} = \bigoplus_{\dd \in \bb N^{V_{Q_Y}}} H_\bullet^A(\mf{M}_\dd^{\f_0,\zeta_\VW}(Y_{m,0},S_{\mu}), \varphi_{\overline{W}^{\f_0}})\otimes_{H_A^\bullet(\pt)} F  \ . \]
Following the equivalence outlined at the end of Example \ref{chainsaweg}, we make the following conjecture:

\begin{conj}\label{chainsawconj} There is a canonical isomorphism
\[  H_\bullet^A(\mf{M}_\dd^{\f_0,\zeta}(Y_{m,0},S_{\mu}), \varphi_{\overline{W}^{\f_0}}) \cong H^A_\bullet(\mf M_\dd^{\mu,\zeta}(Q^\textup{Ch}_m,I^\textup{Ch}_m)) \ . \]
\end{conj}

As explained in \emph{loc. cit.}, the expectation is that this follows from dimensional reduction, in the sense of Theorem A.1 of \cite{Dav1}, for example, together with an equivalence given by passing to cokernels of the injective maps determined by the framing structure. In particular, note that the dimensions of the framing vector spaces in the two quiver descriptions of the above equivalence differ in the following way: $\mf{M}^{f_0,\zeta}(Y_{m,0},S_{\mu})$ parameterizes representations of framing dimension $\textbf{M}$ while $\mf M^{\mu,\zeta}(Q^\textup{Ch}_m,I^\textup{Ch}_m)$ parameterizes representations of framing dimension $\mu$, noting that
\[\mu_i=M_{i-1}-M_i \ . \]
This is the generalization of the relation outlined between the framing vector spaces $V_{\infty_0}$ and $\tilde V_{\infty_0}$ in the quivers of Equations \ref{chainsawwpotsl2eqn} and \ref{chainsawsl2eq}, respectively.

In summary, we expect that the preceding Conjecture \ref{chainsawconj} together with an identification of the form in Equation \ref{KTequiveqn} imply that the module constructed by the framing structure $\f_0$ above is identified with the universal Verma module $\bb M_{f_\mu}$ of the corresponding $W$-algebra:
\begin{conj} There exists a natural representation
\[ \mc U(W_{f_\mu}^\kappa(\gl_M))\to \End_F( \V_{S_\mu}^{\f_0}) \ ,\]
such that $\V_{S_\mu}^{\f_0}$ is identified with the universal Verma module $\bb M_{f_\mu}$.
\end{conj}
The preceding equivalence is the analogue of Theorem \ref{SVtheo} recalled from \cite{SV}, and is closely related to several existing mathematical generalizations of the AGT conjecture, such as \cite{BFFR}, \cite{FFNR}, \cite{Negaff}, \cite{Neg2}, and \cite{CDM}. However, we have not considered the modules corresponding to the framing structures that feature in Conjecture \ref{VWconj}. Indeed, the framing structure $\f_{S_\mu}$ can be described by adding to $\f_0$ an additional principal nilpotent $G_\f$ in $\gl_{M_0}$ acting at the framing vertex corresponding to $\C^2$; the modification of the framed quiver with potential in Equation \ref{chainsawwpotsl2eqn}, which corresponds to the framing structure $\f_0$, is depicted in Figure \ref{fig:Walgquivfig} above:

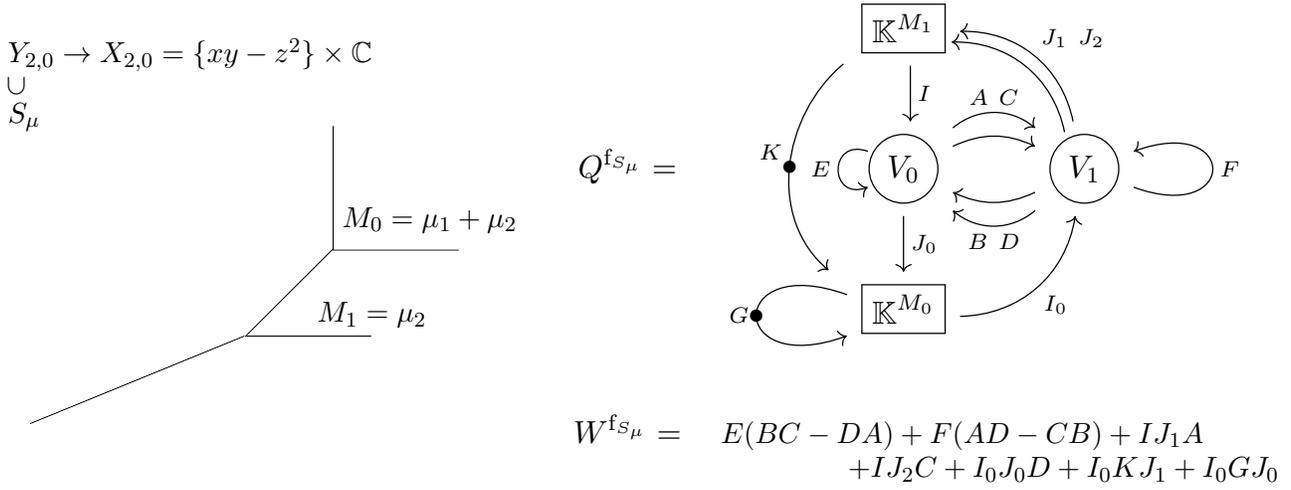
\begin{figure}[t]
	\hspace*{-1cm}
	\begin{overpic}[width=.4\textwidth]{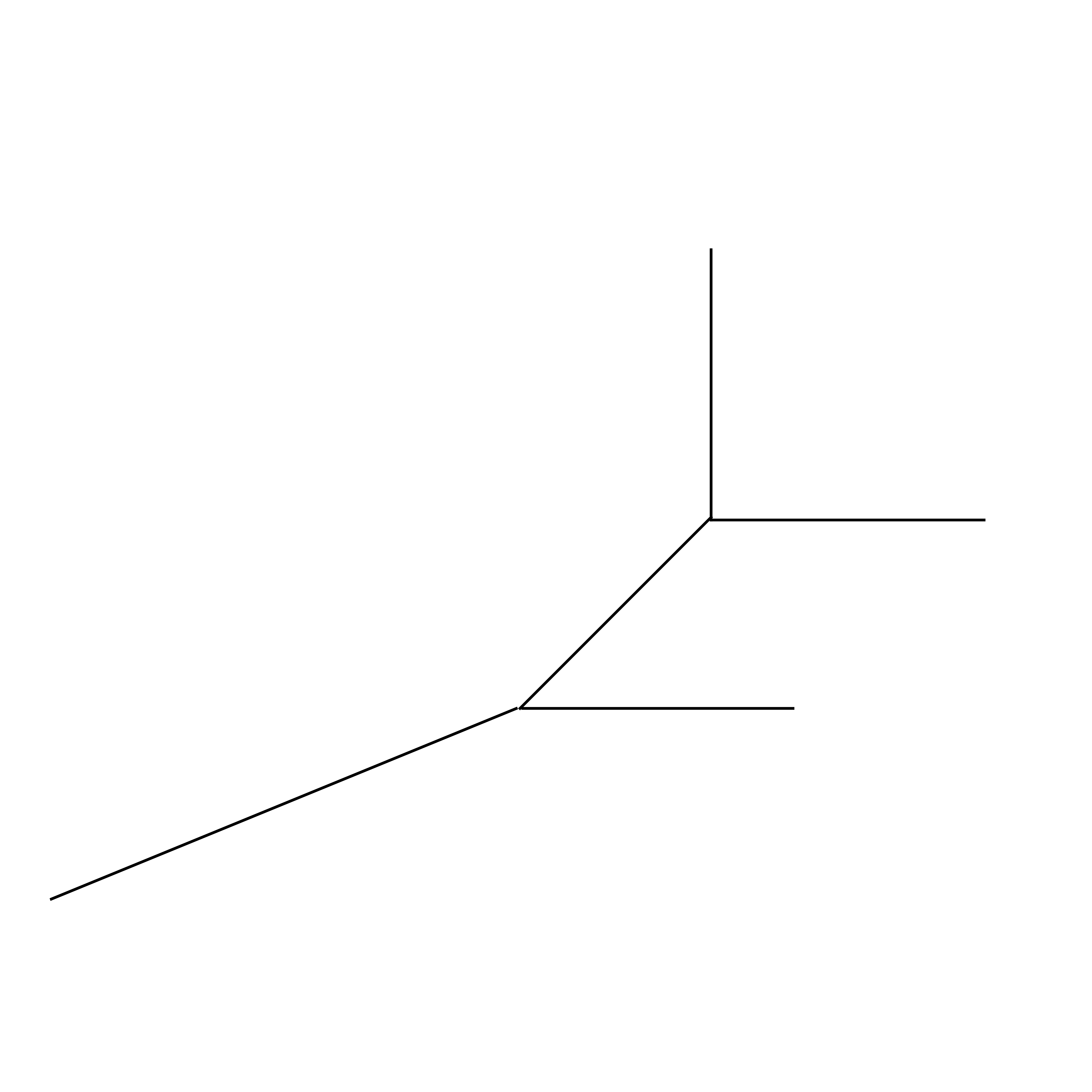}
		\put(0,90) {$ Y_{2,0}\to X_{2,0}=\{xy-z^2\} \times \bb C $}
		\put(0,84) {\rotatebox[origin=c]{90}{$\subset$}}
		\put(0,78) {$S_{\mu}$}
		\put(67,57) {$M_0=\mu_1+\mu_2$}
		\put(62,38) {$M_1=\mu_2$}
	\end{overpic}\quad\quad
		\begin{overpic}[width=.55\textwidth]{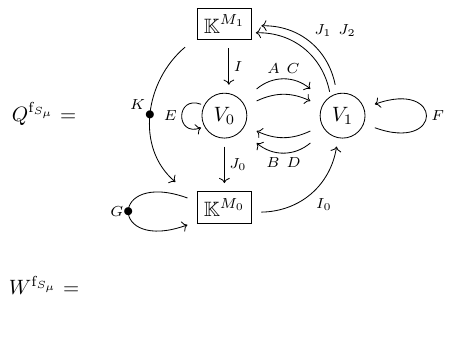}
			\put(22,10) { $  E(BC-DA)+ F(AD-CB) + IJ_1A  $} 
			\put(42,5) {$ + IJ_2C + I_0J_0D+I_0 K J_1 + I_0 G J_0 $}
	\end{overpic}
	\caption{The toric divisor $S_\mu$ in $Y_{2,0}$ and framed quiver with potential $(Q^{\f_{S_\mu}},W^{\f_{S_\mu}})$}
	\label{fig:Walgquivfig}
\end{figure}

We recall that the module associated to the divisor $S$ which features in Conjecture \ref{VWconj} in defined in terms of the corresponding quiver with potential as
\[\V_{S_\mu} = \V_{S_\mu}^{\f_{S_\mu},\zeta_\VW} =  \bigoplus_{\dd \in \bb N^{V_{Q_Y}}} H_\bullet^{G_\dd(Q_Y)\times A}( X_\dd^{\zeta_\VW}(Q^{\f_{S_\mu}}),\varphi_{W^{\f_{S_\mu}}})\otimes_{H_A^\bullet(\pt)} F \ .   \]
The statement of \emph{loc. cit.} in this setting, by Theorem \ref{WLStheo}, gives the following conjecture:
\begin{conj}\label{Walgvacconj}
There exists a natural representation
\[ \mc U(W_{f_\mu}^\kappa(\gl_M)) \to \End_F( \V_{S_\mu}) \ ,\]
such that $\V_{S_\mu}$ is identified with the vacuum module $W_{f_\mu}^\kappa(\gl_M)$.
\end{conj}

More generally, following Conjecture \ref{WLSconj}, we make the analogous conjectures in the $\gl_{m|n}$ case:
\begin{conj}\label{Wsupalgvacconj}
	There exists a natural representations
	\[ \mc U(W_{f_\mu,f_\nu}^\kappa(\gl_{M|N})) \to \End_F( \V_{S_{\mu,\nu}}^0) \quad\quad \text{and}\quad\quad \mc U(W_{f_\mu,f_\nu}^\kappa(\gl_{M|N})) \to \End_F( \V_{S_{\mu,\nu}}) \ ,\]
	such that $\V_{S_{\mu,\nu}}^0$ and $\V_{S_{\mu,\nu}}$ are identified with the verma and vacuum modules of $W_{f_\mu,f_\nu}^\kappa(\gl_{M|N})$.
\end{conj}

\subsection{Towards isomorphisms between $W$-superalgebras and Yangians for $\widehat{\mf{gl}}_{m|n}$}\label{MVsec}

As at the end of the preceding section, let $Y_{m,n} \to X_{m,n}$ be a resolution of the affine, toric, Calabi-Yau singularity $X_{m,n}=\{ xy-z^mw^n\}$. In this section, we explain the analogue of Conjecture \ref{Yangactconj} for $\V_{S}$, the module constructed from a divisor $S$ in $Y_{m,n}$ in the preceding section, and moreover explain its expected relationship with Conjecture \ref{VWconj}. This gives a geometric approach to prove a variant of the main theorem of \cite{BrKl}: an isomorphism between affine $W$-superalgebras and truncated shifted Yangians of $\glh_{m|n}$.

To begin, recall from Proposition \ref{Picprop} that
\[ \Pic(\hat Y) \xrightarrow{\cong} \bb \Z^{I_+} \quad\quad \text{by} \quad\quad \mc L \mapsto \deg \iota_i^* \mc L \]
where $I_+$ is the index set of the irreducible components $C_i$ of the fibre $C$ of the resolution over the $T$-fixed point $x\in X$, and $\iota_i:C_i\to Y$ are their the inclusions. For $Y=Y_{m,n}$, we thus have identifications of the Picard group of $\hat Y_{m,n}$ with the root lattice of $\spl_{m|n}$ and the index set $I_+$ with the set of positive simple roots for $\spl_{m|n}$,
\[ \Pic( \hat Y_{m,n}) \cong Q_{\spl_{m|n}}  \quad\quad \text{and}\quad\quad I_+ \cong \Pi_{\spl_{m|n}} \ .\]
For example, in Figure \ref{fig:Smunufig} we have depicted the four compact toric curve classes $C_i$ in $Y_{3,2}$ and the corresponding positive simple roots $\alpha_i$ in $\spl_{m|n}$.

Following \cite{BrKl}, the shifted Yangian is defined in terms of a square matrix $\sigma$ of size $|\Pi|+1$ called the \emph{shift matrix}, which is uniquely determined by its entries $s_{i,i+1}$ and $s_{i+1,i}$ immediately above and below the diagonal, for $i=1,...,|\Pi|$. Now, given a divisor $S$ in $Y_{m,n}$, we define the corresponding shifts of real roots as the intersection numbers
\[ s_{i+1,i} = \langle [C_i], S \rangle  \quad \quad \text{for}\quad\quad i=1,..., |\Pi_{\spl_{m|n}}|=m+n-1  \]
 where we have assumed for simplicity that all these numbers are non-negative, as will be the case for $S_{\mu,\nu}$ as long as $\mu_m\geq \nu_1$. We let $\sigma_S$ denote the induced shift matrix, and $\mc Y_{\sigma_S}(\glh_{m|n})$ the corresponding shifted Yangian of $\glh_{m|n}$ with respect to the shift of real roots determined by $\sigma_S$; for example, the divisor $S_{\mu,\nu}$ in $Y_{3,2}$ from the example of Figure \ref{fig:Smunufig} induces the shift matrix given by
\[ \sigma_{S_{\mu,\nu}} =   \mleft[\begin{array}{c c c | c c} 0 & 0 & 0 & 0 &0  \\
	\mu_1-\mu_2& 0 & 0 & 0 & 0\\
		\mu_1 - \mu _3& \mu_2-\mu_3 & 0 & 0 & 0 \\
\hline	\mu_1 - \nu_1 & \mu_2 - \nu_1 & \mu_3-\nu_1 & 0 & 0\\
\mu_1 - \nu_2 & \mu_2 - \nu _2 &\mu_3-\nu_2&\nu_1-\nu_2 &0
\end{array}\mright]  \ . \]

We can now state the analogue of the Conjecture \ref{Yangactconj} of Costello \cite{CosMSRI} for the modules $\V_S$ corresponding to divisors $S$ in $Y_{m,n}$:

\begin{conj}\label{Yangsurfconj} There exists a natural representation
	\[\rho:\mc Y_{\sigma_S}(\glh_{m|n}) \to  \End_F(\V_{S}) \ ,\]
of the $\sigma_S$ shifted affine Yangian $\mc Y_{\sigma_S}(\glh_{m|n})$ of $\gl_{m|n}$ on $\bb V_S$, inducing an isomorphism
\[ \rho(\mc Y_{\sigma_S}(\glh_{m|n})_+ ) \xrightarrow{\cong} \rho_{S}(\mc{SH}(Y_{m,n}))  \ . \]
\end{conj}

As for Conjecture \ref{Yangactconj}, we expect that the proof of the preceding conjecture will follow from Conjecture \ref{quiveryangconj} together with an identification
\[ \mc Y_S(Y_{m,n}):= \mc Y_{\mc O_S[1]}(Y_{m,n}) \cong \mc Y_{\sigma_S}(\glh_{m|n}) \ . \]
In fact, for general $Y$ and $S$, we expect that this action factors as a surjection 
\[ \mc Y_S(Y) \to \mc U(\V(Y,S))  \ , \]
or a map with dense image though we will ignore this subtlety in what follows, composed with the representation $\mc U(\V(Y,S))\to  \End_F(\V_{S}) $ of Conjecture \ref{VWconj}. In particular, we conjecture that for $Y=Y_{m,n}$ and $S=S_{\mu,\nu}$ this induces a surjection 
\begin{equation}\label{YSevaleqn}
	\mc Y_{\sigma_{S_{\mu,\nu}} }(\glh_{m|n}) \to  \mc U(\V(Y_{m,n},S_{\mu,\nu}) )
\end{equation}
with kernel given by the $\ell^{th}$-truncation ideal for $\ell=\mu_1$, so that we have
\[ \mc Y_{\sigma_{S_{\mu,\nu}}}^\ell(\glh_{m|n}) \xrightarrow{\cong} \mc U(\V(Y_{m,n},S_{\mu,\nu})) \ . \] 
In summary, together with Conjecture \ref{WLSconj}, we obtain:
\begin{conj}\label{affineBKconj} There is a canonical isomorphism
\begin{equation}\label{affineBKeqn}
	 \mc Y_{\sigma_{S_{\mu,\nu}}}^\ell(\glh_{m|n}) \xrightarrow{\cong}\mc U(W_{f_\mu,f_\nu}^\kappa(\gl_{M|N})) \ , 
\end{equation}
where $\mu,\nu$ are the unique dominant coweights such that $\sigma=\sigma_{S_{\mu,\nu}}$ and $\ell=\mu_1$.
\end{conj}

\noindent This is precisely the analogue for $\glh_{m|n}$ of the main theorem of \cite{BrKl}, which was originally proved for the Yangian of $\gl_N$. Some results along these lines have been obtained recently in \cite{Ue1}, \cite{Negrect} and \cite{KoUe} in the case $\sigma=0$ and $\mu$ and $\nu$ are given by rectangular partitions, but the statement of the conjecture for general shift appears to be new. There is also an example in the non-rectangular case given in \cite{Ue2}, but given the apparent absence of a shift of real roots in this result, we are unsure whether the map constructed in \emph{loc. cit.} has an analogous geometric origin.

We also propose the existence of a natural limiting vertex algebra, generalizing the usual $W_{1+\infty}$ algebra, as follows: let $\sigma$ be a shift matrix for $\gl_{m|n}$ and let $\mu^1,\nu^1$ be the minimal in the dominance order among the coweights corresponding to $\sigma$ as in Conjecture \ref{affineBKconj}, or equivalently those such that $\ell^1=\mu^0_1$ is minimal, where we have assumed $\mu_n\geq \nu_1$ for simplicity. More generally, let $\mu^\nn$ and $\nu^\nn$ denote the $\nn^{th}$ smallest coweights corresponding to $\sigma$, and let $M_\nn=|\mu^\nn|$ and $N_\nn=|\nu^\nn|$. Then we propose to consider the large $r$ limit algebra, conjecturally enhanced to include $r$ as a parameter,
\[W^\kappa_{1+\infty,\sigma}(\gl_{m|n}) :=  \ \text{``}  \ \lim_{\nn \to\infty} W^\kappa_{f_{\mu^\nn},f_{\nu^\nn}}( \gl_{M_\nn,N_\nn}) \ \text{''} \ , \]
which defines a $\sigma-$shifted variant of the $\gl_{m|n}$-extended $W_{1+\infty}$ vertex algebra, which was introduced in \cite{Cos2} for $n=0$, and studied further and generalized to $n\neq 0$ in \cite{EGR}, \cite{CrH}, \cite{EbP}, \cite{CrHU} and \cite{Rap}. In terms of this vertex algebra, the preceding conjecture implies the following, which has an interpretation in string theory as an example of the twisted holographic principle of Costello-Li \cite{CosL} for M5 branes in the $\Omega$-background, following the approach of \cite{Cos2}:
 
 \begin{conj}\label{twholoconj} There is a canonical isomorphism of associative algebras
\[\mc Y_{\sigma }(\glh_{m|n}) \xrightarrow{\cong} \mc U(W^\kappa_{1+\infty,\sigma}(\gl_{m|n})) \ ,\]
inducing isomorphisms as in Equation \ref{affineBKeqn} for each compatible $l\in \bb N$, $\mu$, $\nu$, and an identification
\[ V_{m,n;\sigma} \xrightarrow{\cong } W^\kappa_{1+\infty,\sigma}(\gl_{m|n}) \]
between the vacuum module $V_{m,n;\sigma} $ of the $\sigma$-shifted affine Yangian of $\glh_{m|n}$ and that of $W^\kappa_{1+\infty,\sigma}(\gl_{m|n})$.
\end{conj}

We hope to provide a proof of this conjecture in future work, using the geometric approach developed here. In what follows, we explain the conjecture in the case of $\gl_1$ for which it follows from \cite{SV}, and provide some numerical evidence for the general statement. In particular, we note the following corollary of the preceding conjecture, which we verify holds in the examples below:

\begin{corollary}\label{twholochicoro} Let $\sigma$ be a shift matrix for $\gl_{m|n}$, and for each $\nn\in\bb N$ let $\mu^\nn$ and $\nu^\nn$ be the dominant coweights defined as above. Then
\begin{equation}\label{twholochieqn}
	 P_q(V_{m,n;\sigma}) = \lim_{\nn \to\infty} P_q(W^\kappa_{f_{\mu^\nn},f_{\nu^\nn}}( \gl_{M^\nn|N^\nn})) \ . 
\end{equation}
\end{corollary}

Note that the choice of shift matrix $\sigma$ together with $\nn\in \bb N$ determines not only the dominant coweights $\mu^\nn$ and $\nu^\nn$ but an enhancement of them to a \emph{pyramid}, in the sense of \cite{EK}. This determines a good grading with respect to which we define $W^\kappa_{f_{\mu^\nn},f_{\nu^\nn}}( \gl_{M^\nn,N^\nn})$ and the Poincare polynomial on the right hand side of Equation \ref{twholochieqn}.

We now proceed to the discussion of examples:

\begin{eg}
Let $m=1$ and $n=0$ so that $Y=Y_{1,0}=\C^3$. In this case, there are no real roots with which to shift the Yangian, and the isomorphism
\[ \mc Y(\glh_1) \xrightarrow{\cong} \mc U( W_{1+\infty}^\kappa(\gl_1) ) \]
was explained in \cite{GGLP}, following the proof of the AGT conjecture from \cite{SV} and the unpublished results of Feigin-Tsymbaliuk explained in \cite{Ts1}. Moreover, the results of \cite{SV} and their generalizations in \cite{RSYZ}, which we recall establish special cases of Conjecture \ref{VWconj} as outlined in Examples \ref{AGTeg} and \ref{GReg}, respectively, also prove the special cases of Conjecture \ref{Yangsurfconj} in these examples. Further, in these cases, the proofs of the two conjectures are related precisely as described in the discussion following \emph{loc. cit.}.

In particular, for each $r\in \bb N$ we can compute the Poincare polynomial of the vacuum module of $\V(\C^3,r[\C^2])=W^\kappa_{f_{\textup{prin}} }(\gl_r)$ following the spectral sequence argument from Section 15.2.9 of \cite{FBZ}, as we recall below; this is summarized in Figure \ref{fig:C3wfig}.

\begin{figure}[b]
	\begin{overpic}[width=1\textwidth]{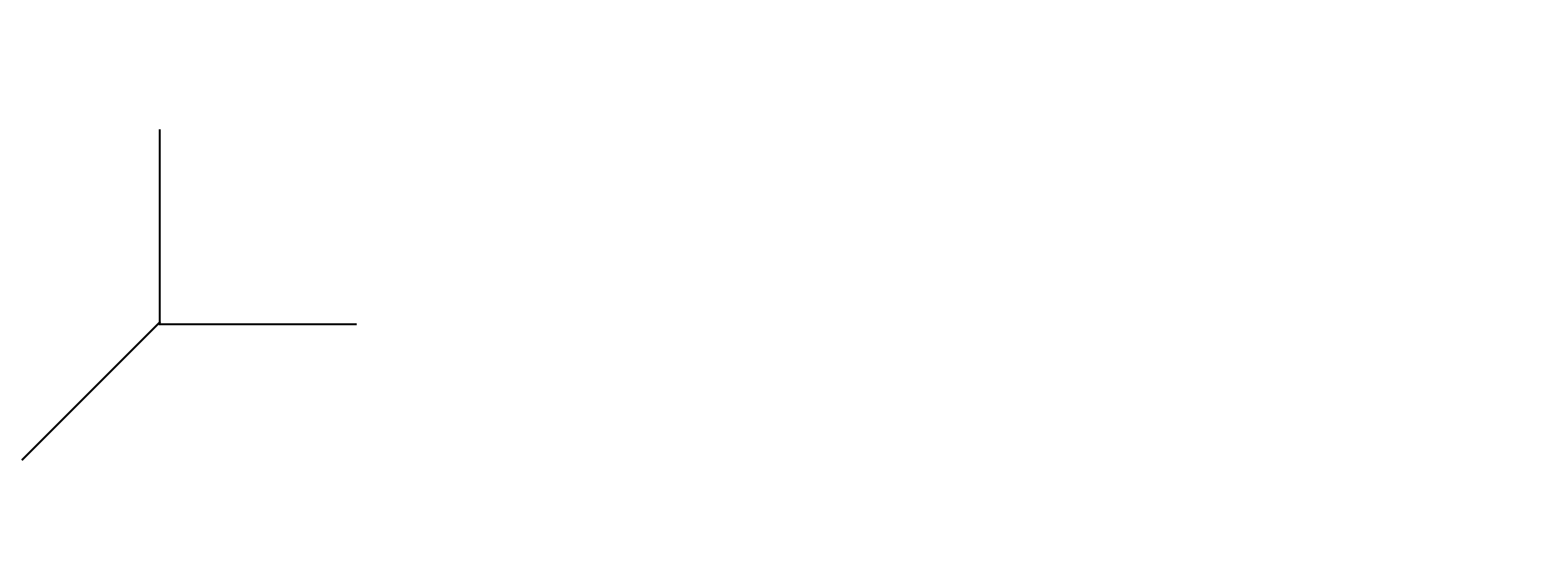}
				\put(0,33) {$Y_{1,0}=X_{1,0}= \C^3$}
		\put(13,19) {$r$}
		\put(40,30)   {$ \pi = \ $  \ytableausetup{mathmode, boxsize=2 em}
			\begin{ytableau}
				1 & 2 & 3 & \none[\dots]& \scriptstyle \nn - 1 & \nn \\
		\end{ytableau} }
		\put(40,17) {$\Gamma=\scriptsize{ \mleft[\begin{array}{c c c c c }
					\textbf{0} & \textbf{2} & \textbf{4} & \hdots & \textbf{2r-2} \\
					-2& 0 & 2 & \hdots & 2r-4  \\
					\vdots & \vdots &\vdots & \ddots & \vdots   \\
					4-2r & 6-2r & 8-2r & \hdots &2 \\
					2-2r & 4-2r & 6-2r & \hdots &0 
				 \end{array}\mright] }$}
		\put(40,3) {$ P_q(W_{f_{\textup{prin}} }^\kappa(\gl_{\nn})) = \prod_{j=0}^{r-1} \prod_{k=1}^\infty \frac{1}{1-q^{k+j}} $}
	\end{overpic}
	\caption{The divisor $r[\C^2]$ in $\C^3$, with associated pyramid $\pi$, induced grading $\Gamma$, and Poincare polynomial $P_q$ of the vacuum module of $W_{f_{\textup{prin}} }^\kappa(\gl_{\nn})$ }
	\label{fig:C3wfig}
\end{figure}

The computation of \emph{loc. cit.} proceeds as follows: given the divisor $S=r[\C^2]$, consider the corresponding pyramid $\pi$ as explained above, the grading $\Gamma$ induced by $\pi$, and compute the degree with respect to $\Gamma$ of the highest weights of each of the irreducible representations occurring in the decomposition of the adjoint of $\gl_r$ under the $\spl_2$ embedding, representatives of which are written in bold in Figure \ref{fig:s0Wfig}. Each such highest weight can be lifted to a field of $W^\kappa(\gl_r)$ of the corresponding conformal weight, and it is proved in \emph{loc. cit.} that these fields form a set of strong generators, in terms of which we can compute the Poincare polynomial of the vacuum module, usually called the \emph{character} of the vertex algebra.

In general, the character of $W^\kappa_{f_{\textup{prin}} }(\gl_r)$ computed as outlined above is given by
\[ P_q(W^\kappa_{f_{\textup{prin}} }(\gl_r)) =\prod_{j=0}^{r-1} \prod_{k=1}^\infty \frac{1}{1-q^{j+k}}  \]
and we find a direct verification of Corollary \ref{twholochicoro} above
\[ \lim_{r\to \infty} P_q(W^\kappa_{f_{\textup{prin}} }(\gl_r)) =\lim_{r\to \infty} 
\prod_{j=0}^{r-1}\prod_{k=1}^\infty \frac{1}{1-q^{j+k}}  = \prod_{k=1}^\infty \frac{1}{(1-q^{k})^k} = M(q)= P_q(V_{1,0}) \ , \]
where we have used the equality of Equation \ref{DTchieqn} above in the case $m=1$.
\end{eg}

\begin{eg}\label{sl2unshifteg} Let $m=2$ and $n=0$, so that $Y=Y_{2,0}\to X_{2,0}=\{ xy-z^2\}\times \bb C$. The corresponding lattice of real roots is given by $\Pic(\hat Y_{2,0}) \cong Q_{\spl_2} \cong \bb Z$, and to begin we suppose the shift is trivial, that is $\sigma = \scriptsize{\begin{bmatrix} 0 & 0 \\ 0 & 0 \end{bmatrix}}$. For each $r\in \bb N$, there is a dominant coweight $\mu^r$ corresponding to $\sigma$ determined as above, and the corresponding divisor $S_{\mu^r}$ is pictured on the left of Figure \ref{fig:s0Wfig} below. Moreover, the corresponding pyramid $\pi$ and good grading $\Gamma$, together with the character of the vacuum module for the $W$-algebra $W_{f_{\mu^\nn}}(\gl_{2r})$ determined by this grading, are pictured on the right.

\begin{figure}[b]
	\begin{overpic}[width=1\textwidth]{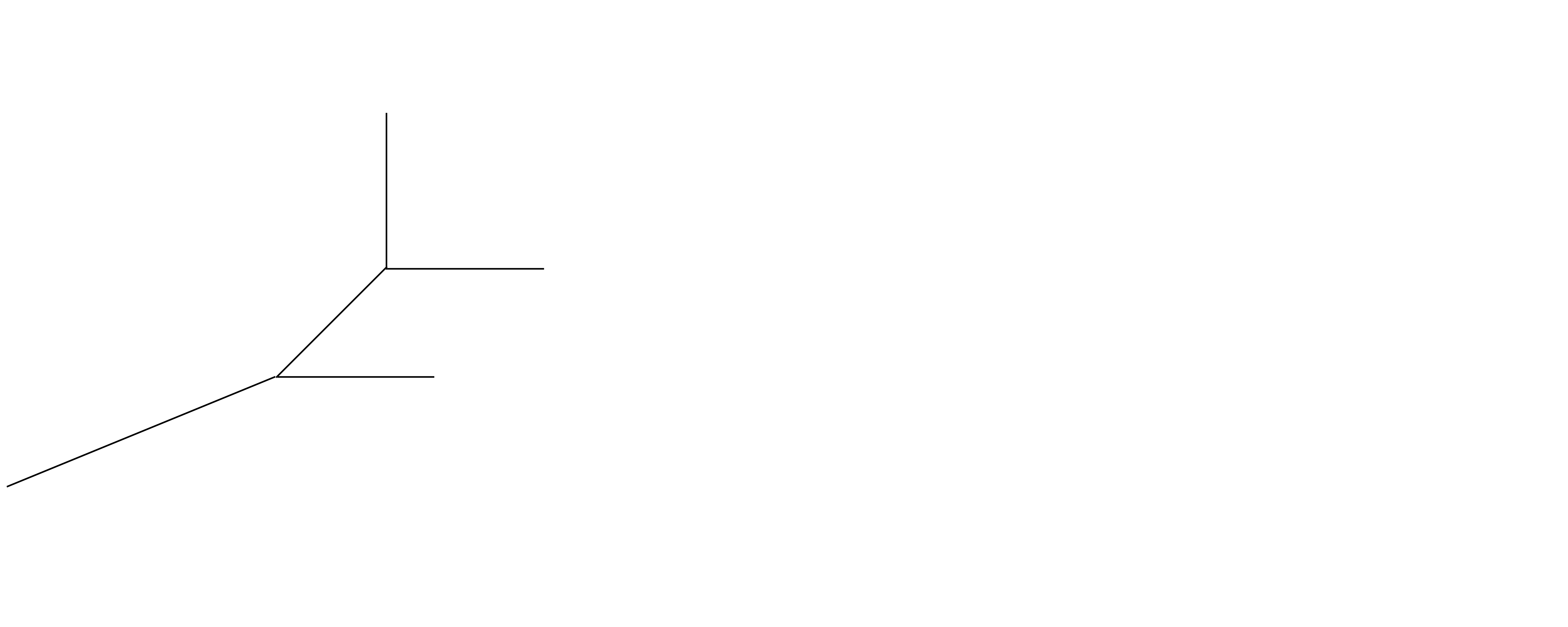}
		\put(0,39) {$ Y_{2,0}\to X_{2,0}=\{xy-z^2\} \times \bb C $}
		\put(26,25) {$M_0=2r$}
		\put(23,18) {$M_1=r$}
		\put(40,37)   {$ \pi =$ \ytableausetup{mathmode, boxsize=2 em}
			\begin{ytableau}
				\scriptstyle	\nn+1 & \scriptstyle \nn+2 &\scriptstyle \nn+3 & \none[\dots] 	& \scriptstyle 2r - 1 & 2\nn \\
				1 & 2 & 3 & \none[\dots]& \scriptstyle \nn - 1 & \nn \\
		\end{ytableau} }
		\put(40,17) {$\Gamma=\scriptsize{ \mleft[\begin{array}{c c c c c | c c c c c}
						\textbf{0} & \textbf{2} & \textbf{4} & \hdots & \textbf{2r-2} & 	\textbf{0} & \textbf{2} & \textbf{4} & \hdots & \textbf{2r-2} \\
					-2& 0 & 2 & \hdots & 2r-4   &-2& 0 & 2 & \hdots & 2r-4 \\
					\vdots & \vdots &\vdots & \ddots & \vdots & 	\vdots & \vdots &\vdots & \ddots & \vdots \\
					4-2r & 6-2r & 8-2r & \hdots &2  & 	4-2r & 6-2r & 8-2r & \hdots &2  \\
					2-2r & 4-2r & 6-2r & \hdots &0 & 	2-2r & 4-2r & 6-2r & \hdots &0\\
					\hline
						\textbf{0} & \textbf{2} & \textbf{4} & \hdots & \textbf{2r-2} & 	\textbf{0} & \textbf{2} & \textbf{4} & \hdots & \textbf{2r-2} \\
					-2& 0 & 2 & \hdots & 2r-4   &-2& 0 & 2 & \hdots & 2r-4 \\
					\vdots & \vdots &\vdots & \ddots & \vdots & 	\vdots & \vdots &\vdots & \ddots & \vdots \\
					4-2r & 6-2r & 8-2r & \hdots &2  & 	4-2r & 6-2r & 8-2r & \hdots &2  \\
					2-2r & 4-2r & 6-2r & \hdots &0 & 	2-2r & 4-2r & 6-2r & \hdots &0\\
				\end{array}\mright] }$}
		\put(40,0) {$ P_q(W_{f_{\mu^\nn}}(\gl_{2r})) = \prod_{j=0}^{r-1} \prod_{k=1}^\infty \frac{1}{(1-q^{k+j})^4} $}
	\end{overpic}
	\caption{The toric divisor $S_{\mu^\nn}$ in $Y_{2,0}$ for $\sigma=\scriptsize{\begin{bmatrix}
				0 & 0 \\ 0 & 0 \end{bmatrix}}$,  with associated pyramid $\pi$, induced grading $\Gamma$, and Poincare polynomial $P_q$ of the vacuum module of $W_{f_{\mu^\nn}}(\gl_{2\nn})$}
	\label{fig:s0Wfig}
\end{figure}

 For example, for $r=1$ the coweight $\mu^1$ corresponds to the trivial nilpotent $f_{\mu^1}=0$, so that the corresponding $W$-algebra is the affine Kac-Moody algebra $V^\kappa(\gl_2)$, and the conjectural map $\mc Y(\glh_2)\to \mc U(V^\kappa(\gl_2))\cong \mc U(\glh_2)$ of Equation \ref{YSevaleqn} is simply the evaluation homomorphism constructed in \cite{Guay}. The vertex algebra $V^\kappa(\gl_2)$ is strongly generated by a collection of fields of conformal weight 1 enumerated by any basis of the underlying Lie algebra, and thus the character of the vacuum module is given by $\eta(q)^{-4}$. Evidently this agrees with the $r=1$ case of the general formula
\[ P_q(W_{f_{\mu^\nn}}(\gl_{2r})) = \prod_{j=0}^{r-1} \prod_{k=1}^\infty \frac{1}{(1-q^{k+j})^4}  \ , \]
from Figure  \ref{fig:s0Wfig}. Moreover, we can again verify Corollary \ref{twholochicoro} directly in this case:
\[ \lim_{r\to \infty}  P_q(W_{f_{\mu^\nn}}(\gl_{2r})) =  \lim_{r\to \infty}  \prod_{j=0}^{r-1} \prod_{k=1}^\infty \frac{1}{(1-q^{k+j})^4} = \prod_{k=1}^\infty \frac{1}{(1-q^k)^{4k}} = M(q)^4 = P_q(V_{2,0}) \ , \]
where again we have used the equality of Equation \ref{DTchieqn}.
\end{eg}

\begin{eg}\label{s1sl2eg} Let $m=2$ and $n=0$, so that $Y=Y_{2,0}\to X_{2,0}=\{ xy-z^2\}\times \bb C$, as in the preceding example, and suppose the shift matrix is given by $\sigma_1=\scriptsize{\begin{bmatrix}	0 & 0 \\ 1 & 0 \end{bmatrix}}$. Again, the sequence of toric divisors, and their corresponding pyramids, good gradings, and vacuum module characters are pictures in Figure \ref{fig:s1Wfig} below.

For example, in the case $r=2$ the coweight $\mu^2$ corresponds to the subregular nilpotent $f_{\mu^2}\in \gl_3$, so that the corresponding vertex algebra is given by the Bershadsky-Polyakov $W_3^{(2)}$ algebra, the subregular $W$-algebra for $\spl_3$, tensored with a Heisenberg algebra. It is well known that the $W_3^{(2)}$ vertex algebra is strongly generated by four fields $J,L,G^+$ and $G^-$ of conformal weights $1$, $2$, $\frac{3}{2}$ and $\frac{3}{2}$, respectively.
	
\begin{figure}[b]
	\hspace*{-1cm}
	\begin{overpic}[width=1\textwidth]{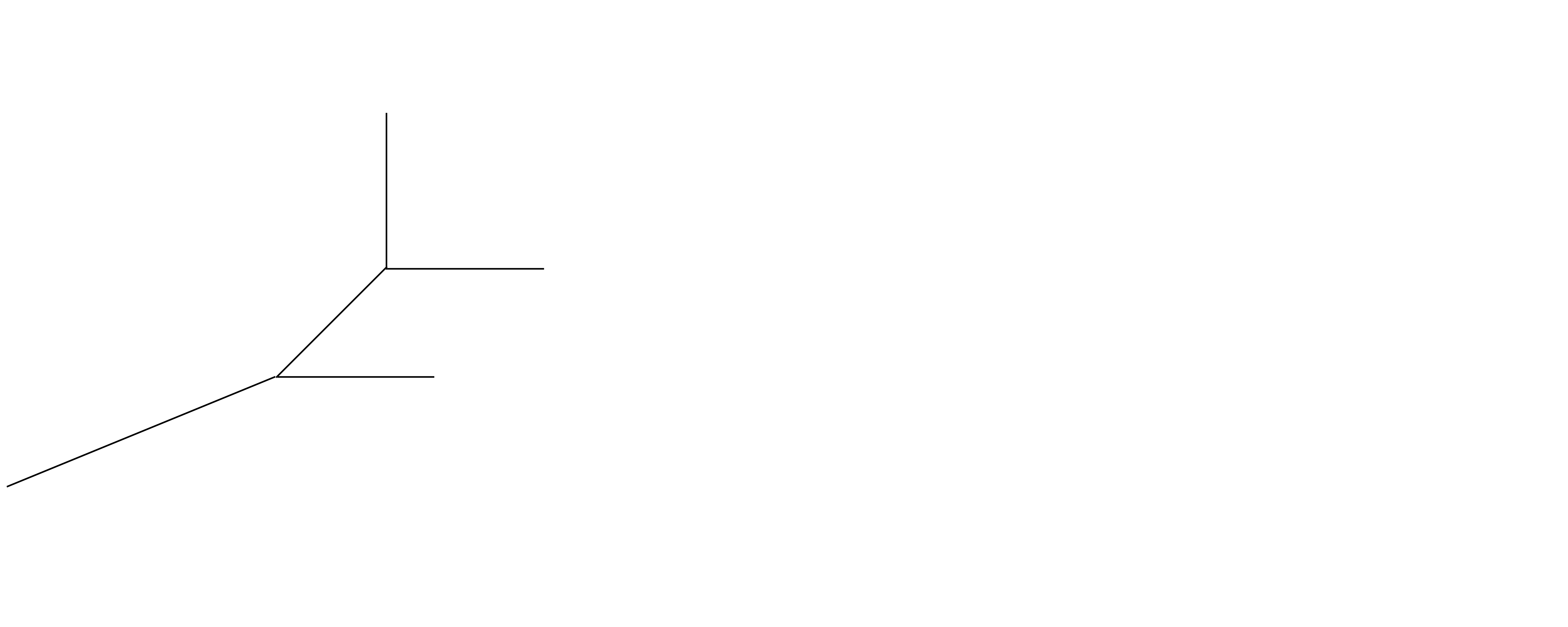}
				\put(0,39) {$ Y_{2,0}\to X_{2,0}=\{xy-z^2\} \times \bb C $}
		\put(26,25) {$M_0=2r-1$}
		\put(23,18) {$M_1=r-1$}
		\put(40,40)   {$ \pi =$ \ytableausetup{mathmode, boxsize=2 em}
			\begin{ytableau}
				\none[ ] & \scriptstyle	\nn+1 & \scriptstyle \nn+3  & \none[\dots] & \scriptstyle 2\nn -2	& \scriptstyle \scriptstyle 2\nn -1  \\
				1 & 2 & 3 & \none[\dots]& \scriptstyle \nn - 1 & \nn  \\
		\end{ytableau} }
		\put(40,20) {$\Gamma=\scriptsize{ \mleft[\begin{array}{c c c c c | c c c c c}
					\textbf{0} & \textbf{2} & \textbf{4} & \hdots & \textbf{2r-2}& \textbf{2} & \textbf{4} & \textbf{6}  & \hdots & \textbf{2r-2} \\
					-2& 0 & 2 & \hdots & 2r-4 &  0 & 2 & 4 & \hdots &  2r-4 \\
					\vdots & \vdots &\vdots & \ddots & \vdots & 	\vdots & \vdots &\vdots & \ddots & \vdots \\
					4-2r & 6-2r & 8-2r & \hdots &2  & 	 6-2r & 8-2r & 10-2r& \hdots &2  \\
					2-2r & 4-2r & 6-2r & \hdots &0 &  4-2r & 6-2r & 8-2r&  \hdots &0\\
					\hline -2& \textbf{ 0} &  \textbf{2} & \hdots &  \textbf{2r-4} & 	\textbf{0} & \textbf{2} & \textbf{4} & \hdots & \textbf{2r-4} \\
					-4 & -2& 0  & \hdots & 2r-6 & -2 & 0 & 2 & \hdots & 2r-6\\
					\vdots & \vdots &\vdots & \ddots & \vdots & 	\vdots & \vdots &\vdots & \ddots & \vdots \\
					4-2r & 6-2r & 8-2r & \hdots &2  &  6-2r & 8-2r & 10-2r& \hdots &2  \\
					2-2r & 4-2r & 6-2r & \hdots &0 &  4-2r & 6-2r & 8-2r&  \hdots &0\\
				\end{array}\mright] }$}
		\put(40,3) {$ P_q(W_{f_{\mu^\nn}}(\gl_{2r-1})) = \prod_{j=0}^{r-1} \prod_{k=1}^\infty \frac{(1-q^k)(1-q^{k+(r-1)})^2}{(1-q^{k+j})^4} $}
	\end{overpic}
	\caption{The toric divisor $S_{\mu^\nn}$ in $Y_{2,0}$ for $\sigma=\scriptsize{\begin{bmatrix}
				0 & 0 \\ 1 & 0 \end{bmatrix}}$,  with associated pyramid $\pi$, induced grading $\Gamma$, and Poincare polynomial $P_q$ of the vacuum module of $W_{f_{\mu^\nn}}(\gl_{2\nn-1})$}
	\label{fig:s1Wfig}
\end{figure}

In fact, the traditional conformal weights for the Bershadsky-Polyakov algebra correspond to the Dynkin grading on $W_{f_{\mu^2}}(\spl_3)$, while with respect to the good grading induced by the pyramid above, the fields $G^+$ and $G^-$ are of degree $1$ and $2$. Thus, including the auxiliary Heisenberg generator, the character of $\V(Y_{2,0},S_{\mu^2})=W_{f_{\mu^2}}(\gl_3)$ is given by
\[ P_q(W_{f_{\mu^2}}(\gl_3)) = \prod_{k=1}^\infty \frac{1}{(1-q^k)^3 (1-q^{k+1})^2 } \ , \]
in agreement with the general formula given in Figure \ref{fig:s1Wfig} above.

We can again verify Corollary \ref{twholochicoro} directly in this case:
\[ \lim_{r\to \infty}  P_q(W_{f_{\mu^\nn}}(\gl_{2r-1})) =  \lim_{r\to \infty} \prod_{j=0}^{r-1} \prod_{k=1}^\infty \frac{(1-q^k)(1-q^{k+(r-1)})^2}{(1-q^{k+j})^4} = \prod_{k=1}^\infty \frac{(1-q^k)}{(1-q^k)^{4k}} = P_q(V_{2,0;\sigma_1}) \ . \]
\end{eg}

\begin{eg}\label{s2sl2eg} Let $m=2$ and $n=0$, so that $Y=Y_{2,0}\to X_{2,0}=\{ xy-z^2\}\times \bb C$, as in the preceding two examples, and suppose now that the shift matrix is given by $\sigma_2=\scriptsize{\begin{bmatrix}	0 & 0 \\ 2 & 0 \end{bmatrix}}$. The sequence of toric divisors, and their corresponding pyramids, good gradings, and vacuum module characters are pictures in Figure \ref{fig:s2Wfig} below.
	
We can yet again verify Corollary \ref{twholochicoro} directly in this case:
\begin{align*}
		\lim_{r\to \infty}  P_q(W_{f_{\mu^\nn}}(\gl_{2r-1}))  & =  \lim_{r\to \infty} \prod_{j=0}^{r-1} \prod_{k=1}^\infty \frac{(1-q^k)(1-q^{k+1})(1-q^{k+(r-1)})^2}{(1-q^{k+j})^4(1-q^{k+r})^2}  \\
		& = \prod_{k=1}^\infty \frac{(1-q^k)(1-q^{k+1})}{(1-q^k)^{4k}} = P_q(V_{2,0;\sigma_2}) \ . 
\end{align*}

\begin{figure}[b]
	\hspace*{-1cm}
	\begin{overpic}[width=1\textwidth]{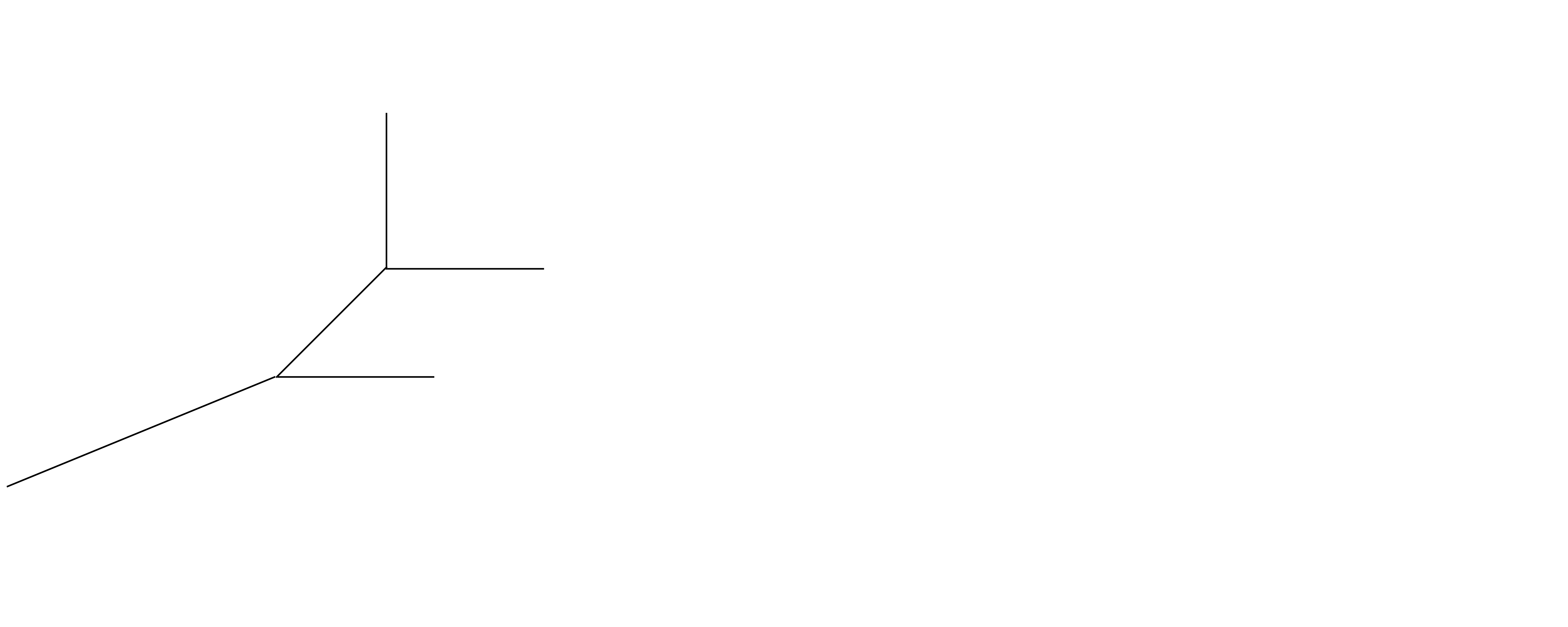}
				\put(0,39) {$ Y_{2,0}\to X_{2,0}=\{xy-z^2\} \times \bb C $}
		\put(26,25) {$M_0=2r$}
		\put(23,18) {$M_1=r-1$}
		\put(40,40)   {$ \pi =$ \ytableausetup{mathmode, boxsize=2 em}
			\begin{ytableau}
				\none[ ] & 	\none[ ] & \scriptstyle \nn+2  & \none[\dots] & \scriptstyle 2\nn -2	& \scriptstyle \scriptstyle 2\nn -1  & 2\nn\\
				1 & 2 & 3 & \none[\dots]& \scriptstyle \nn - 1 & \nn & \scriptstyle  \nn +1  \\
		\end{ytableau} }
		\put(40,20) {$\Gamma=\scriptsize{ \mleft[\begin{array}{c c c c c | c c c c c}
					\textbf{0} & \textbf{2} & \textbf{4} & \hdots & \textbf{2r} & \textbf{4} & \textbf{6} & \textbf{8}  & \hdots & \textbf{2r} \\
					-2& 0 & 2 & \hdots & 2r-2 &   2 & 4 & 6 & \hdots &  2r-2 \\
					\vdots & \vdots &\vdots & \ddots & \vdots & 	\vdots & \vdots &\vdots & \ddots & \vdots \\
					2-2r & 4-2r & 6-2r & \hdots &2  & 	 6-2r & 8-2r & 10-2r& \hdots &2  \\
					2r & 2-2r & 4-2r & \hdots &0 &  4-2r & 6-2r & 8-2r&  \hdots &0\\
					\hline -4& -2 &  \textbf{0} & \hdots &  \textbf{2r-4} & 	\textbf{0} & \textbf{2} & \textbf{4} & \hdots & \textbf{2r-4} \\
					-6 & -4& -2  & \hdots & 2r-6 & -2 & 0 & 2 & \hdots & 2r-6\\
					\vdots & \vdots &\vdots & \ddots & \vdots & 	\vdots & \vdots &\vdots & \ddots & \vdots \\
					4-2r & 6-2r & 8-2r & \hdots &2  &  6-2r & 8-2r & 10-2r& \hdots &2  \\
					2-2r & 4-2r & 6-2r & \hdots &0 &  4-2r & 6-2r & 8-2r&  \hdots &0\\
				\end{array}\mright] }$}
		\put(40,3) {$ P_q(W_{f_{\mu^\nn}}(\gl_{2r})) = \prod_{j=0}^{r-1} \prod_{k=1}^\infty \frac{(1-q^k)(1-q^{k+1})(1-q^{k+(r-1)})^2}{(1-q^{k+j})^4(1-q^{k+r})^2} $}
	\end{overpic}
	\caption{The toric divisor $S_{\mu^\nn}$ in $Y_{2,0}$ for $\sigma=\scriptsize{\begin{bmatrix}
				0 & 0 \\ 2 & 0 \end{bmatrix}}$,  with associated pyramid $\pi$, induced grading $\Gamma$, and Poincare polynomial $P_q$ of the vacuum module of $W_{f_{\mu^\nn}}(\gl_{2\nn})$}
	\label{fig:s2Wfig}
\end{figure}

\end{eg}

\begin{eg}
	Let $m=n=1$ so that $Y_{1,1}= | \mc O_{\bb P^1}(-1)\oplus  \mc O_{\bb P^1}(-1)| \to X_{1,1}=\{ xy-zw\}$, so that the lattice of real (super) roots is given by $\Pic(Y_{1,1})\cong Q_{\spl_{1|1}}\cong \Z$, and we consider the case that the shift is trivial, so that $\sigma=\scriptsize{\begin{bmatrix}	0 & 0 \\ 0 & 0 \end{bmatrix}}$.
	
	As in Example \ref{sl2unshifteg}, in the unshifted case with $r=1$ the corresponding $W$-algebra is the affine Kac-Moody vertex superalgebra $V^\kappa(\gl_{1|1})$, which is strongly generated by two odd and two even fields all of conformal weight $1$, so that the character of the vacuum module is given by
	\[ P_q(V^\kappa(\gl_{1|1})) =  \prod_{k=1}^\infty \frac{(1+q^{k})^2}{(1-q^{k})^2} \ .\]
	More generally, the vacuum character of the rectangular $W$-superalgebra which corresponds to the divisor $S_{\mu^r,\nu^r}$ via Conjecture \ref{Wsupalgvacconj} is computed analogously, as summarized in Figure \ref{fig:s0SWfig}:
	\[P_q(W_{f_{\mu^\nn},f_{\nu^\nn}}(\gl_{r|r})) = \prod_{j=0}^{r-1} \prod_{k=1}^\infty \frac{(1+q^{k+j})^2}{(1-q^{k+j})^2} \ .\]
	Again, we can verify Corollary \ref{twholochicoro} directly in this case:
\[	\lim_{r\to \infty} 	P_q(W_{f_{\mu^\nn},f_{\nu^\nn}}(\gl_{r|r}))  = \lim_{r\to \infty}  \prod_{j=0}^{r-1} \prod_{k=1}^\infty \frac{(1+q^{k+j})^2}{(1-q^{k+j})^2} \\
 = \prod_{k=1}^\infty \frac{(1+q^{k})^{2k}}{(1-q^{k})^{2k}} = P_q(V_{1|1}) \  \ ,\]
 in agreement with the computations of the character of the $\mc N=2$ superconformal $W_{1+\infty}$ algebras studied in \cite{GWP} and \cite{GWPZ}. More generally, the analogous computations to those of Examples \ref{s1sl2eg} and \ref{s2sl2eg}  can evidently be checked similarly for divisors on $Y_{1|1}$ with non-trivial shift.
\end{eg}

\begin{figure}[b]
	\begin{overpic}[width=1\textwidth]{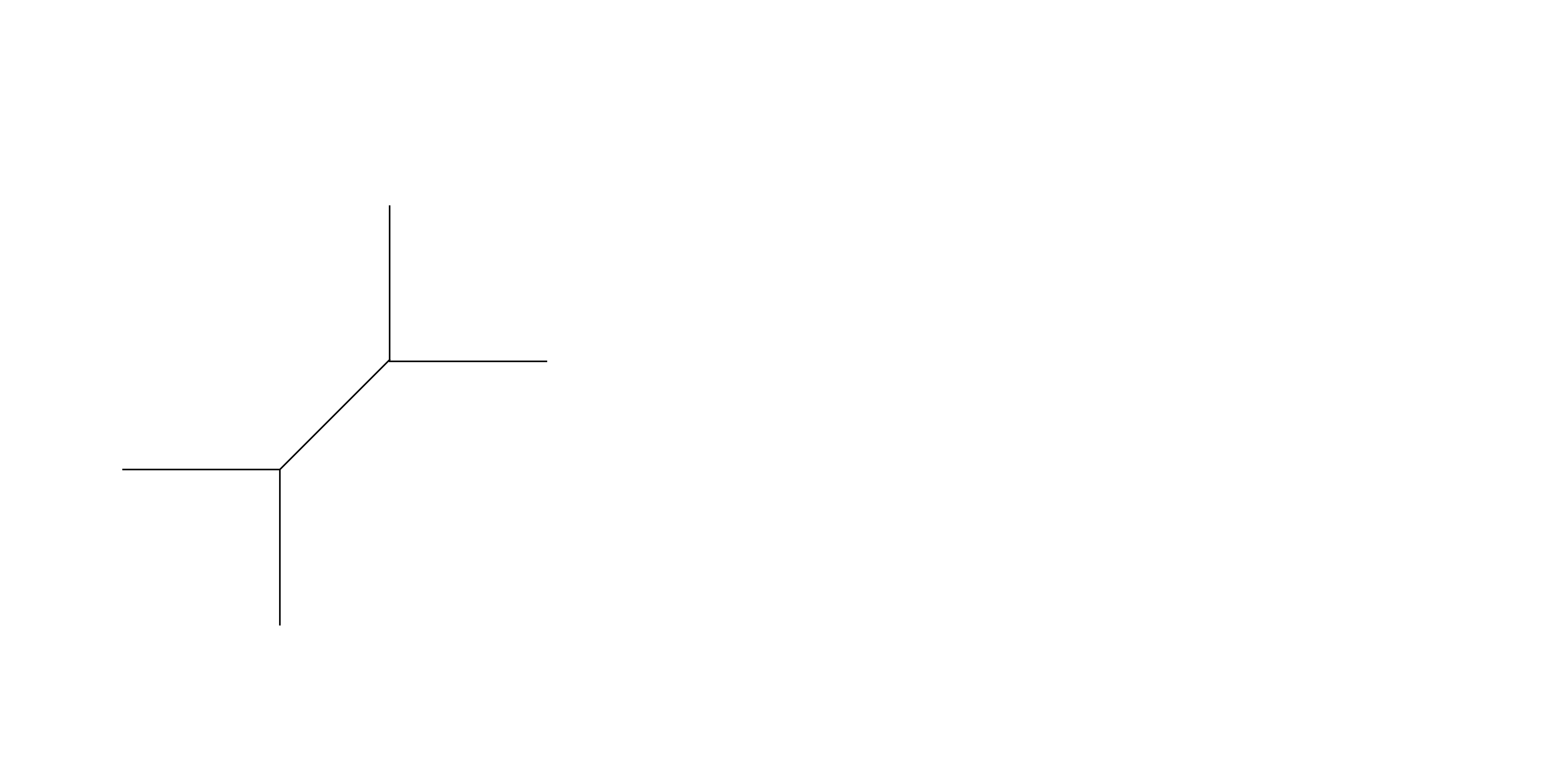}
		\put(0,44) {$ Y_{1,1}\to X_{1,1}=\{xy-zw\} $}
		\put(26,29) {$M_0=r$}
		\put(15,25) {$N_0=r$}
		\put(40,42)   {$ \pi =$ \ytableausetup{mathmode, boxsize=2 em}
			\begin{ytableau}
				\scriptstyle	\nn+1 & \scriptstyle \nn+2 &\scriptstyle \nn+3 & \none[\dots] 	& \scriptstyle 2r - 1 & 2\nn \\
				1 & 2 & 3 & \none[\dots]& \scriptstyle \nn - 1 & \nn \\
		\end{ytableau} }
		\put(40,22) {$\Gamma=\scriptsize{ \mleft[\begin{array}{c c c c c | c c c c c}
					\textbf{0} & \textbf{2} & \textbf{4} & \hdots & \textbf{2r-2} & 	\textbf{0} & \textbf{2} & \textbf{4} & \hdots & \textbf{2r-2} \\
					-2& 0 & 2 & \hdots & 2r-4   &-2& 0 & 2 & \hdots & 2r-4 \\
					\vdots & \vdots &\vdots & \ddots & \vdots & 	\vdots & \vdots &\vdots & \ddots & \vdots \\
					4-2r & 6-2r & 8-2r & \hdots &2  & 	4-2r & 6-2r & 8-2r & \hdots &2  \\
					2-2r & 4-2r & 6-2r & \hdots &0 & 	2-2r & 4-2r & 6-2r & \hdots &0\\
					\hline
					\textbf{0} & \textbf{2} & \textbf{4} & \hdots & \textbf{2r-2} & 	\textbf{0} & \textbf{2} & \textbf{4} & \hdots & \textbf{2r-2} \\
					-2& 0 & 2 & \hdots & 2r-4   &-2& 0 & 2 & \hdots & 2r-4 \\
					\vdots & \vdots &\vdots & \ddots & \vdots & 	\vdots & \vdots &\vdots & \ddots & \vdots \\
					4-2r & 6-2r & 8-2r & \hdots &2  & 	4-2r & 6-2r & 8-2r & \hdots &2  \\
					2-2r & 4-2r & 6-2r & \hdots &0 & 	2-2r & 4-2r & 6-2r & \hdots &0\\
				\end{array}\mright] }$}
		\put(40,5) {$ P_q(W_{f_{\mu^\nn},f_{\nu^\nn}}(\gl_{r|r})) = \prod_{j=0}^{r-1} \prod_{k=1}^\infty \frac{(1+q^{k+j})^2}{(1-q^{k+j})^2} $}
	\end{overpic}
	\caption{The toric divisor $S_{\mu^\nn,\nu^\nn}$ in $Y_{1,1}$ for $\sigma=\scriptsize{\begin{bmatrix}
				0 & 0 \\ 0 & 0 \end{bmatrix}}$,  with pyramid $\pi$, induced grading $\Gamma$, and Poincare polynomial $P_q$ of the vacuum module of $W_{f_{\mu^\nn},f_{\nu^\nn}}(\gl_{r|r})$}
	\label{fig:s0SWfig}
\end{figure}

\newpage

\bibliographystyle{alpha}

\bibliography{reCYsnewest}

\end{document}